%% file: tese.tex
\documentclass[oneside,openright,titlepage,numbers=noenddot,headinclude, lineheaders,footinclude=true,cleardoublepage=empty,BCOR=5mm,paper=a4,fontsize=12pt,enabledeprecatedfontcommands]{scrbook} 
\usepackage[usenames,dvipsnames,svgnames,table]{xcolor}
\usepackage[unicode]{hyperref}
\usepackage[english,brazil,portuguese]{babel}
\usepackage[parts,pdfspacing,dottedtoc,eulerchapternumbers]{classicthesis}
\usepackage{estilo}
\usepackage{lmodern}
\usepackage[T1]{fontenc}
\usepackage{tgadventor}
\usepackage{enumerate}
\usepackage{hyperref}
\usepackage{amsmath}
\usepackage{amssymb}
\usepackage{amsthm}
\usepackage{mathtools}
\usepackage{stmaryrd}
\usepackage{mathrsfs}
\usepackage{bbm}
\usepackage{upgreek}
\usepackage{setspace}
\usepackage{breqn}
\usepackage{pdfpages}
\usepackage{makeidx}
\usepackage{blindtext}
\usepackage{cite}
\usepackage{comment}
\usepackage[capitalise]{cleveref}

\usepackage{lipsum}

\usepackage{constants}

\newconstantfamily{small}{
symbol=\mathtt{c}
}

\usepackage{graphicx}
\usepackage{float}
\usepackage[dvipsnames]{xcolor}
\usepackage{tikz}
\usetikzlibrary{calc,3d,angles}

\emergencystretch=\maxdimen
\usepackage{tocloft}

\DeclareFontFamily{U}{dutchcal}{\skewchar\font=45 }
\DeclareFontShape{U}{dutchcal}{m}{n}{<-> s*[1.0] dutchcal-r}{}
\DeclareFontShape{U}{dutchcal}{b}{n}{<-> s*[1.0] dutchcal-b}{}
\DeclareMathAlphabet{\mathlcal}{U}{dutchcal}{m}{n}
\SetMathAlphabet{\mathlcal}{bold}{U}{dutchcal}{b}{n}

\DeclareOldFontCommand{\rm}{\normalfont\rmfamily}{\mathrm}
\DeclareOldFontCommand{\sf}{\normalfont\sffamily}{\mathsf}
\DeclareOldFontCommand{\tt}{\normalfont\ttfamily}{\mathtt}
\DeclareOldFontCommand{\bf}{\normalfont\bfseries}{\mathbf}
\DeclareOldFontCommand{\it}{\normalfont\itshape}{\mathit}
\DeclareOldFontCommand{\sl}{\normalfont\slshape}{\@nomath\sl}
\DeclareOldFontCommand{\sc}{\normalfont\scshape}{\@nomath\sc}
\DeclareOldFontCommand{\sfb}{\normalfont\sffamily\bfseries}{\@nomath\sfb}

\makeatletter
\newcommand{\ltext}[2]{%
  \@bsphack
  \csname phantomsection\endcsname 
  \def\@currentlabel{#1}{\label{#2}}%
  \@esphack
}

\newcommand\RedeclareMathOperator{%
  \@ifstar{\def\rmo@s{m}\rmo@redeclare}{\def\rmo@s{o}\rmo@redeclare}%
}
\newcommand\rmo@redeclare[2]{%
  \begingroup \escapechar\m@ne\xdef\@gtempa{{\string#1}}\endgroup
  \expandafter\@ifundefined\@gtempa
     {\@latex@error{\noexpand#1undefined}\@ehc}%
     \relax
  \expandafter\rmo@declmathop\rmo@s{#1}{#2}}
\newcommand\rmo@declmathop[3]{%
  \DeclareRobustCommand{#2}{\qopname\newmcodes@#1{#3}}%
}
\@onlypreamble\RedeclareMathOperator
\makeatother

\newcommand\blfootnote[1]{%
  \begingroup
  \renewcommand\thefootnote{}\footnote{#1}%
  \addtocounter{footnote}{-1}%
  \endgroup
}

\numberwithin{equation}{chapter}

\usepackage{accents}
\newcommand\thickbar[1]{\accentset{\rule{.4em}{.5pt}}{#1}}

\theoremstyle{plain}
\newtheorem{theorem}{Theorem}[chapter]
\newtheorem{lemma}[theorem]{Lemma}
\newtheorem{proposition}[theorem]{Proposition}
\newtheorem{proposition*}{Proposition}
\newtheorem{corollary}[theorem]{Corollary}
\newtheorem{claim}{Claim}

\theoremstyle{definition}
\newtheorem*{definition}{Definition}
\newtheorem{example}{Example}[chapter]

\theoremstyle{remark}
\newtheorem{remark}{Remark}[chapter]

\numberwithin{equation}{chapter}

\DeclareMathOperator{\N}{\mathbb{N}}
\DeclareMathOperator{\Z}{\mathbb{Z}}

\DeclareMathOperator{\R}{\mathbb{R}}
\DeclareMathOperator{\Gr}{\Gamma}
\DeclareMathOperator{\tht}{\upvartheta}
\DeclareMathOperator{\Om}{\Upomega}
\DeclareMathOperator{\F}{\mathscr{F}}
\DeclareMathOperator{\A}{\mathscr{A}}
\DeclareMathOperator{\p}{\mathbb{P}}
\DeclareMathOperator{\E}{\mathbb{E}}
\DeclareMathOperator{\g}{\mathfrak{g}}
\DeclareMathOperator{\valg}{\mathfrak{v}}

\DeclareMathOperator{\bull}{{\scriptscriptstyle\bullet}}
\DeclareMathOperator{\ab}{\operatorname{ab}}
\DeclareMathOperator{\GHto}{\xrightarrow{\text{\footnotesize GH}}}
\DeclareMathOperator{\tor}{\operatorname{tor}}
\DeclareMathOperator*{\argmin}{arg\,min}
\DeclareMathOperator{\G}{\mathcal{G}}
\DeclareMathOperator{\Hh}{\mathcal{H}}
\DeclareMathOperator{\Poi}{\mathcal{P}}

\DeclareMathOperator{\K}{\mathtt{K}}
\DeclareMathOperator{\Kprime}{\mathtt{K}^\prime}

\newcommand{\autor}{Lucas Roberto de Lima}
\newcommand{\titulo}{Asymptotic Shape of Subadditive Processes on Groups and on Random Geometric Graphs}
\newcommand{\orientador}{Cristian Favio Coletti}
\newcommand{\coorientador}{Daniel Rodrigues Valesin}
\newcommand{\fomento}{ \href{https://bv.fapesp.br/en/bolsas/190142/shape-theorem-for-the-contact-process-in-random-environment-on-groups-with-polynomial-growth/}{\#2019/19056-2} and \href{https://bv.fapesp.br/en/bolsas/195701/limiting-shape-for-the-contact-process-on-random-geometric-graphs/}{\#2020/12868-9} São Paulo Research Foundation (FAPESP).}

\def\versaofinal{}

\newcommand{\ano}{2024}
\newcommand{\centro}{Center for Mathematics, Computation, and Cognition \xspace}
\newcommand{\titulacao}{Doctor in Mathematics \xspace}
\newcommand{\palavraschaves}{cociclos subaditivos, forma limite, grupos virtualmente nilpotentes, grafos de Cayley, percolação, teoria ergódica, disco de Gilbert, crescimento aleatório, árvores geradoras, geodésicas, competição, coexistência.}
\newcommand{\keywords}{subadditive cocycles, limiting shape, virtually nilpotent groups, Cayley graphs, percolation, ergodic theory, moderated deviations, Gilbert disk model, random growth, spanning trees, geodesics, competition, coexistence.} 
\makeindex

\begin{document}
\frontmatter
\input{capa}  

\mainmatter
\input{folha-de-rosto}



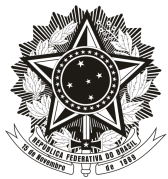

**MINISTÉRIO DA EDUCAÇÃO**
Fundação Universidade Federal do ABC
Avenida dos Estados, 5001 - Bairro Santa Terezinha - Santo André - SP
CEP 09210-580 · Fone: (11) 4996-0017

**Ata de Defesa de Tese de Doutorado e Folha de Assinaturas**

No dia 16 de Agosto de 2024 às 16:15, no local: Auditório 801 no 8º andar do Bloco B da Universidade Federal do ABC, realizou-se a Defesa da Tese de Doutorado, que constou da apresentação do trabalho intitulado **"Asymptotic Shape of Subadditive Processes on Groups and on Random Geometric Graphs"** de autoria do candidato, **LUCAS ROBERTO DE LIMA**, RA nº 23202010272, discente do Programa de Pós-Graduação em MATEMÁTICA da UFABC, sob orientação do Profº CRISTIAN FAVIO COLETTI. Concluídos os trabalhos de apresentação e arguição, o candidato foi considerado _Aprovado_ pela Banca Examinadora.

E, para constar, foi lavrada a presente ata e folha de assinaturas assinada pelos membros da Banca.

*Documento assinado digitalmente*
**CRISTIAN FAVIO COLETTI**
Data: 19/08/2024 12:55:21-0300
Verifique em https://validar.iti.gov.br

**Dr. CRISTIAN FAVIO COLETTI, UFABC**
Presidente - Interno ao Programa

*Documento assinado digitalmente*
**LUIS ENRIQUE RAMIREZ**
Data: 19/08/2024 13:05:54-0300
Verifique em https://validar.iti.gov.br

**Dr. LUIS ENRIQUE RAMIREZ, UFABC**
Membro Titular - Examinador(a) Interno ao Programa

**Dra. INES ARMENDARIZ**
Membro Titular - Examinador(a) Externo à Instituição

**Dr. LUIZ RENATO GONÇALVES FONTES**
Membro Titular - Examinador(a) Externo à Instituição

*Documento assinado digitalmente*
**LUIZ RENATO GONCALVES FONTES**
Data: 20/08/2024 11:40:23-0300
Verifique em https://validar.iti.gov.br

**UFABC - Fundação Universidade Federal do ABC**

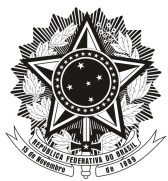

**MINISTÉRIO DA EDUCAÇÃO**
**Fundação Universidade Federal do ABC**
Avenida dos Estados, 5001 - Bairro Santa Terezinha - Santo André - SP
CEP 09210-580 · Fone: (11) 4996-0017

_assinatura_
Signer ID: JUYEVQAJ13...

**Dr. CHRISTOPHE FREDERIC GALLESCO**
Membro Titular - Examinador(a) Externo à Instituição

**Dr. JEFERSON CASSIANO, UFABC**
Membro Suplente - Examinador(a) Externo ao Programa

**Dr. GLAUCO VALLE DA SILVA COELHO**
Membro Suplente - Examinador(a) Externo à Instituição

**Dr. ALEXSANDRO GIACOMO GRIMBERT GALLO**
Membro Suplente - Examinador(a) Externo à Instituição

**UFABC - Fundação Universidade Federal do ABC**

Declaração de Atendimento às Observações da Banca

Este exemplar foi revisado e alterado em relação à versão original, de acordo com as observações levantadas pela banca examinadora no dia da defesa, sob responsabilidade única do autor(a) e com a anuência do(a) (co)orientador(a).

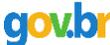

Lucas R. de Lima
Autor

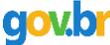

Cristian Favio Coletti
Orientador


%
%
\input{agradecimentos}
%
\input{epigrafe}
\input{abstract}
\input{resumo}
\selectlanguage{english}
\addtocontents{toc}{\protect\thispagestyle{empty}}
\tableofcontents
\thispagestyle{empty}
%

\chapter{Introduction}

\nocite{*}

\input{texts/introduction}

\part{Limiting Shape for Subadditive Random Processes on Finitely Generated Groups with Polynomial Growth Rate} \label{part:groups}

\chapter{Preliminaries and Notation} \label{basic:def}

\input{texts/preliminaries}

\chapter{Asymptotic Shape of Subadditive Processes on Groups with Polynomial Growth} \label{ch:shape.groups}

\input{texts/shape_groups}

\part{Asymptotic Shape Theorem for First-Passage Percolation Models on Random Geometric Graphs and its Speed of Convergence} \label{part:rgg}

\input{texts/fpp_rgg}

\input{texts/moderate_deviation}

\input{texts/competition}

\part{Concluding Remarks and Comparative Discussion} \label{part:conclusion}

\chapter{Conclusion}

\input{texts/conclusion}

\appendix
\input{texts/appendix}

\selectlanguage{english}
\bibliographystyle{acm}
\bibliography{references.bib}

\clearpage
\phantomsection
\addcontentsline{toc}{chapter}{Index}
\printindex
\end{document}

%% file: capa.tex
\thispagestyle{empty}
\begin{center}

  {\large \normalfont \sffamily FEDERAL UNIVERSITY OF ABC

  \vspace{-0.1cm}
  {\normalsize\renewcommand{\sfdefault}{lmss} \normalfont \sffamily \centro}
  
  \renewcommand{\sfdefault}{lmss} \normalfont \sffamily Graduate Program in Mathematics}

\vspace{-2.8cm}
\end{center}

\hspace{-0.7cm} \includegraphics[height=2.7cm, keepaspectratio=true]{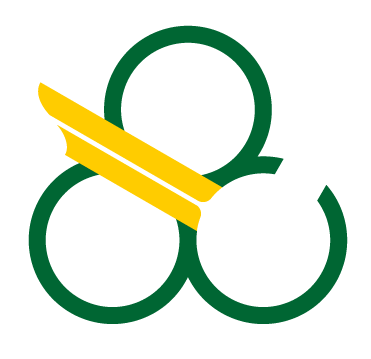} \hfill \raisebox{1.2\height}{\includegraphics[height=0.85cm, keepaspectratio=true]{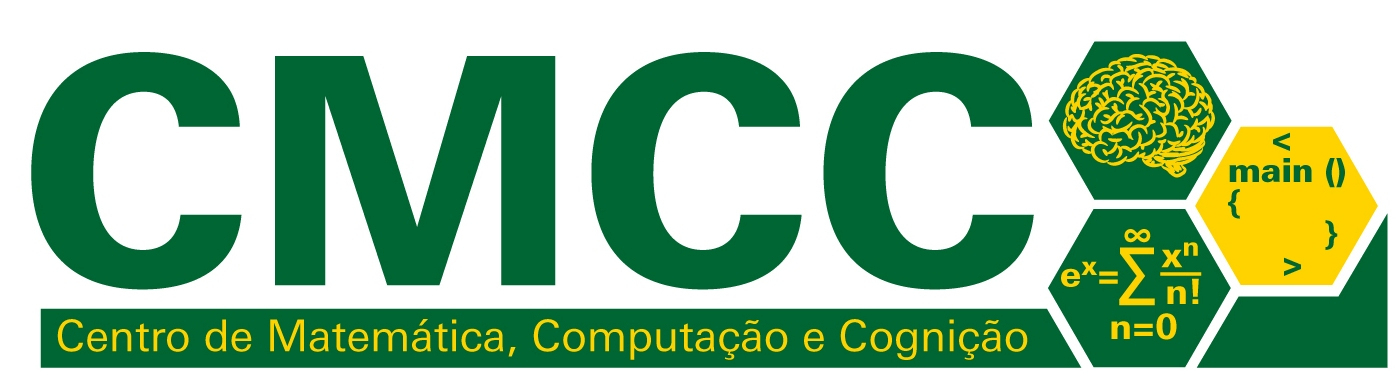}}

\vspace{2cm}
\begin{center}
  {\Large \normalfont \sffamily  \autor}
\end{center}
\vspace{2cm}
\begin{center}
  \begin{onehalfspace}
    \LARGE \normalfont \sffamily \textbf{\titulo}
  \end{onehalfspace}
\end{center}
\vfill
\begin{center}
  {\renewcommand{\sfdefault}{lmss} \normalfont \sffamily Research supported by grants \fomento \\}
  \vspace{0.8cm}
  {\normalfont \sffamily \textbf{Santo  André, SP, Brazil \\\ano}}
\end{center}

%% file: folha-de-rosto.tex
\thispagestyle{empty}
    


\begin{center}
  {\Large \normalfont \sffamily \autor}
\end{center}
\vspace{4cm}
\begin{center}
\begin{onehalfspace}
  \LARGE  \normalfont \sffamily \textbf{\titulo}
\end{onehalfspace}
\end{center}
\vspace{1cm}
{\normalfont \sffamily
\noindent
 \textbf{Supervisor:} Prof. Dr. \orientador
\vspace{.25cm}

\ifx\coorientador\undefined
\else
\noindent
{\bfseries Co-supervisor}\ifx\femaleCoorientador\undefined
\else
a\fi\textbf{:} Prof\ifx\femaleCoorientador\undefined
\else
a\fi. Dr\ifx\femaleCoorientador\undefined
\else
a\fi. \coorientador
\fi
}

\vspace{1cm}
\begin{flushright}
  \begin{minipage}[c]{.6\textwidth}
    \begin{flushleft} 
      \ifx\mestrado\undefined
      \noindent Doctoral thesis
      \else
    \noindent   Master's dissertation
      \fi
      presented to \centro in order to obtain the\\ \noindent  title of
      \titulacao
    \end{flushleft}
  \end{minipage}
\end{flushright}
\vspace{1cm}
  \ifx\versaofinal\undefined
\noindent 
{\footnotesize \scshape
This is the original version of the thesis, as\\
submitted to the Evaluation Committee.\\
}
\else
\noindent 
{\footnotesize \scshape
This copy corresponds to the final version of the \\
\ifx\mestrado\undefined
thesis
\else
dissertation%
\fi 
defendend by the author.
}
\fi

\begin{center}
   \includegraphics[height=3cm, keepaspectratio=true]{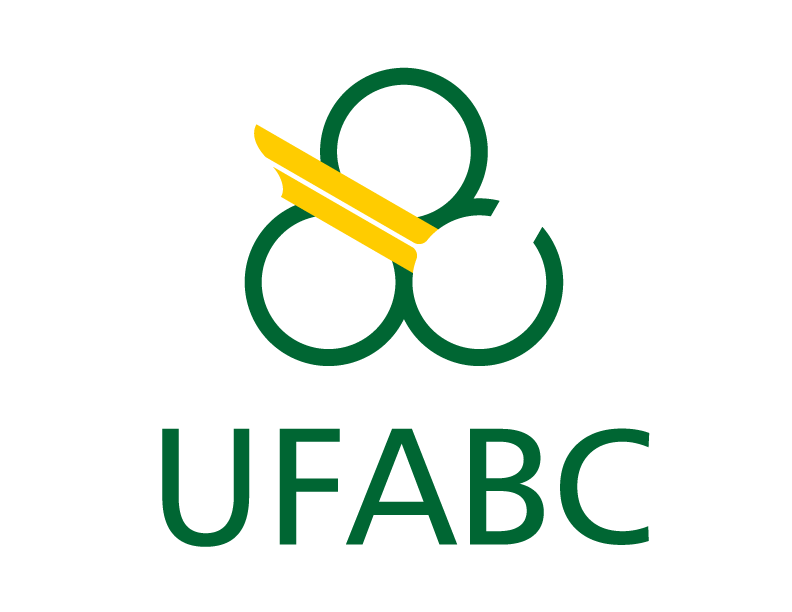}  
\end{center}

\vspace{-0.2cm}

\vfill
\begin{center}
  {\normalfont \sffamily \textbf{Santo André, SP, Brazil\\ \ano}}
\end{center}

%% file: agradecimentos.tex
\chapter*{Acknowledgements}
\thispagestyle{empty}

I would like to take this opportunity to express my deepest gratitude for all those who have  helped me in my academic journey. Without their support, guidance, and encouragement, it would not have been possible for me to see this endeavor through to completion.

I am grateful to \textit{the S\~ao Paulo Research Foundation (FAPESP)}  that provided funding for this project. FAPESP's support and financial assistance were essential in providing me with the opportunities I needed to pursue challenging projects and take more risks.  I am truly grateful for their belief in my potential.

I would like to express my heartfelt gratitude to my supervisor, \textit{Cristian F. Coletti}, for his support and invaluable feedback throughout this entire journey. His exceptional expertise and mentorship played a pivotal role in this study and guiding me through various challenges along the way. I am deeply grateful for the valuable insights beyond academia that I gained from him, which I will carry with me for a lifetime. Furthermore, I want to acknowledge the assistance received from the \textit{Federal University of ABC (UFABC)} and its faculty who shaped my academic and my professional growth.

It is with great appreciation that I thank my co-supervisor \textit{Daniel Valesin} for his guidance and hospitality during my research internship at the \textit{University of Groningen}, for a year-long project, and later at the \textit{University of Warwick}. In addition, I am thankful for his counseling and collaborative efforts, which significantly contributed to my enhanced professional growth and provided clarity throughout the research process.I would also like to acknowledge the support and resources provided by both universities. I want to extend my sincere thanks to \textit{Benedikt Jahnel}, from the Weierstrass Institute of Berlin and Braunschweig University, and his former PhD student \textit{Alexander Hinsen}, with whom we collaborated on a joint work.

I want to thank \textit{Pablo Groisman} for his warm welcome and support during my technical research visit to the \textit{University of Buenos Aires}. More locally, I extend my appreciation to \textit{Denis A. Luiz}, \textit{Diego S. de Oliveira}, \textit{Daniel Miranda Machado} and \textit{Marcus A.M. Marrocos} —colleagues whose friendship and professionalism made this experience more enjoyable and enriching. Their constructive criticism and helpful insights were extremely beneficial in refining my ideas and arguments. I am grateful to \textit{Ana Carolina Boero} for introducing me to diverse mathematical topics and for her decisive role in shaping my development as a mathematician.

I also wish to express my gratitude to \textit{Quasar Space} for giving me the opportunity to work on ambitious projects with them, in particular, \textit{Caio Fagonde} and \textit{Thais C. Franco}. Thais, a dear friend, played a significant role in showcasing the attainability of numerous goals.

The steadfast support of numerous friends has been an indispensable source of strength during my academic journey. Special mention goes to \textit{Karen}, whose unwavering presence brought lightness to the most challenging moments. Also, I thank \textit{Alana} for bringing magic amidst uncertainty. The camaraderie of my lab friends \textit{Aline}, \textit{Denilson}, and \textit{Denis} has been also vital. Their shared enthusiasm, constructive criticism, and insightful discussions not only fostered a stimulating intellectual environment but have also played a significant role in honing my ideas and arguments.

Lastly, I want to express my gratitude to my family and friends for their belief in me and their constant encouragement have been instrumental in my life journey. Among them all, my mother, \textit{Raquel}, who is an exemplary educator, has always provided me with great educational foundations and incentives for personal freedom. This was the cornerstone for building my character and my academic career.

\thispagestyle{empty}

\blfootnote{{\bfseries Funding:} Research supported by grants \fomento This study was financed in part by the Coordenação de Aperfeiçoamento de Pessoal de Nível Superior – Brasil (CAPES) – Finance Code 001.}

%% file: epigrafe.tex
\chapter*{}
\thispagestyle{empty}
\vfill
\hspace{30pt}\includegraphics[scale=0.55]{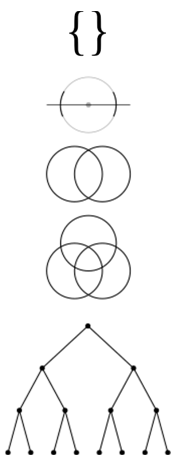}

\begin{flushright}
\vspace{1cm}

\textit{``The infinite is that in which one sees nothing else, hears nothing else,\phantom{''}\\
and knows nothing else. But the finite is that in which one sees\phantom{''}\\
something else, hears something else, and knows something else.\phantom{''}\\
That which is infinite is immortal, and that which is finite is mortal.''}

\vspace{0.3cm}
{\normalsize\renewcommand{\sfdefault}{lmss} \normalfont\sffamily -- Chandogya Upanishad,  7.24.1}



\end{flushright}

%% file: abstract.tex
\selectlanguage{english}

\chapter*{Abstract}
\thispagestyle{empty}

This doctoral thesis undertakes an in-depth exploration of limiting shape theorems across diverse mathematical structures, with a specific focus on subadditive processes within finitely generated groups exhibiting polynomial growth rates, as well as standard First-Passage Percolation (FPP) models applied to Random Geometric Graphs (RGGs). Employing a diverse range of techniques, including subadditive ergodic theorems and tailored modifications suited for polygonal paths within groups, the thesis examines the asymptotic shape under varying conditions. The investigation extends to subadditive cocycles characterized by at least and at most linear growth. Moreover, the study delves into moderate deviations for FPP models on RGGs, refining previous results with theorems that quantify its speed of convergence to the limiting shape, the fluctuation of the geodesics, and its spanning trees. Additionally, we apply the obtained results in a competition model to verify the positive probability of coexistence of two species competing for territory in a random geometric graph.

\vspace{.5cm}
\textbf{Keywords}: \keywords

\selectlanguage{portuguese}

%% file: resumo.tex
\chapter*{Resumo}
\thispagestyle{empty}

Esta tese de doutorado apresenta uma investigação aprofundada sobre teoremas da forma limite em diversas estruturas matemáticas, com um foco especial nos processos subaditivos em grupos finitamente gerados que exibem taxas de crescimento polinomial, além dos modelos padrão de Percolação de Primeira Passagem (FPP) aplicados aos Grafos Geométricos Aleatórios (RGGs). Utilizando uma ampla gama de técnicas, que vão desde teoremas ergódicos subaditivos até modificações adaptadas para caminhos poligonais dentro de grupos, a tese investiga a forma assintótica sob diferentes condições. O estudo se estende a cociclos subaditivos caracterizados por crescimento linear no mínimo e no máximo. Ademais, o estudo se estede a desvios moderados para modelos FPP em RGGs, refinando resultados anteriores com teoremas que quantificam sua velocidade de convergência para a forma limite, a flutuação das geodésicas e suas árvores geradoras. Por fim, aplicamos os resultados obtidos em um modelo de competição para verificar a probabilidade positiva de coexistência de duas espécies disputando território em um grafo geométrico aleatório.

\vspace{.5cm}
\textbf{Palavras-chave}: \palavraschaves

%% file: texts/introduction.tex
\section{Background and Motivation}

In this doctoral thesis, we have studied the limiting shape phenomenon for subadditive processes on groups and random geometric graphs. We have obtained several results on the existence, uniqueness and convergence of the limiting shape for various models of interest. We have also explored some connections between the limiting shape and other topics in geometric group theory. Our methods combine probabilistic, combinatorial and geometric techniques. It led us to contribute with advances of research in this area and we hope that them inspire further studies. 

The asymptotic shape theorem for random processes is a fundamental concept in probability theory and stochastic processes. It provides a framework for understanding the long-term behavior of random processes and has found applications in various areas such as mathematics, ergodic theory, and contact processes. The theorem establishes conditions under which the set of points reachable within a given time from the origin, once appropriately rescaled, converges to a compact and convex limiting shape. This result is significant as it allows for the characterization of the macroscopic behavior of random processes, describing their global properties and providing insights into their long-term dynamics.

Historically, the development of the asymptotic shape theorem has been marked by significant advancements in the understanding of random processes. Early work by Richardson \cite{richardson1973} and Cox and Durrett \cite{cox1981} laid the foundation for the Shape Theorem, providing precise conditions for the convergence of rescaled sets of points in First-Passage Percolation to a deterministic compact and convex set. Subsequent research extended the applicability of the theorem to various models, such as the contact process in random environment, where Asymptotic Shape Theorems were proven, demonstrating the existence of a norm such that the hitting time is asymptotically equivalent to the norm when the process survives.

In recent years, the investigation of the limiting shape theorem for random processes has seen significant advancements. 
Garet and Marchand \cite{garet2012} provided an asymptotic shape theorem for the contact process in random environment, demonstrating the convergence of rescaled sets of points to a compact and convex limiting shape. Benjamini and Tessera \cite{benjamini2015} focused on first passage percolation on nilpotent Cayley graphs, proving an asymptotic shape theorem for the first-passage percolation in the case of independent and identically distributed weights on the edges. Cantrell and Furman \cite{cantrell2017} extended this concept to ergodic families of metrics on nilpotent groups. Additionally, Coletti and de Lima \cite{coletti2021} focused on the frog model on finitely generated abelian groups. Ahlberg et al. \cite{ahlberg2016} investigated inhomogeneous first-passage percolation, showing that the asymptotic growth of the resulting process obeys a shape theorem. Björklund \cite{bjoerklund2010} contributed to the field by proving further structure theorems and providing rates of convergence for the asymptotic shape theorem in certain classes of generalized first passage percolation. These recent investigations have expanded the understanding of the asymptotic shape theorem for random processes, demonstrating its applicability across various models and providing insights into their long-term dynamics.

In particular, one can also consider the speed of convergence for the limiting shape, also referred to as the quantitative asymptotic shape theorem. This result, delineated by Kesten \cite{kesten1993} and later subject of interest in many other studies (see, for instance, \cite{bjoerklund2010,pimentel2011,tessera2018}), represents a refinement of the standard shape theorem, necessitating sophisticated techniques to effectively govern the stochastic growth of these processes.

Overall, the evolution of the asymptotic shape theorem has been developed by contributions from diverse mathematical subfields, playing a pivotal role in elucidating the long-term behavior of random processes. Its historical development underscores the interdisciplinary essence of probability theory, manifesting in its multifaceted applications across various realms of research. The methodologies employed in proving such results hinge upon subadditive ergodic theorems, concentration inequalities, or moderated deviations of random growth.

In the context of this doctoral thesis, the geometric properties of the limiting space sometimes do not admit a direct application of ergodic theorems. This challenge leads us to explore alternative techniques. By employing distinct methods, we extend the results found in the literature to study this phenomenon in groups and within a family of random graphs. The approaches are shaped by the geometric structures we are investigating, as highlighted in the next section.

\section{Objectives and Contributions}

The primary objective of this thesis is to explore limiting shape theorems using various techniques applied to structures of interest in contemporary mathematical studies. These techniques range from the straightforward application of subadditive ergodic theorems to modified versions of the ergodic theorem tailored for polygonal paths in groups. Additionally, the thesis presents a study of moderated convergence for First-Passage Percolation (FPP) models on Random Geometric Graphs (RGGs), resulting in a quantitative shape theorem that refines previous results under stronger conditions.

In particular, we first extend shape theorems found in \cite{benjamini2015, cantrell2017, coletti2021} by obtaining results for subadditive cocycles with conditions of at least and at most linear growth, replacing the hypotheses of i.i.d. or $L^\infty$ random processes. The application of these new theorems is illustrated through examples. Furthermore, in the case of FPP models on RGGs, we present the limiting shape result that we obtained in \cite{coletti2023}, representing an advancement over prior works such as \cite{hirsch2015, yao2013, yao2011}. However, we focus on refining the results obtained in \cite{coletti2023} under stronger conditions. 
The study we conducted in the last chapters of the thesis delved into moderate deviations and related topics, leading to an investigation of the speed of convergence. This enabled us to quantify rates of convergence, estimate fluctuations of geodesics and spanning trees within FPP models on RGGs. Moreover, we study a two-species competition model on the infinite connected component of an RGG, where the growth and competition of the species are determined by Richardson's and Voter's models. We show that there is a positive probability of coexistence of the two species at any time.

The theoretical implications of this work may lead to advancements in the field, paving the way for future research related to this subject matter.

\section{Overview of the Thesis}

The main topics examined are divided into two parts, with the third section dedicated to concluding remarks and a comparative discussion.

\paragraph{\cref{part:groups}} This part of the doctoral thesis dedicated to the study of the limiting shape for subadditive random processes on finitely generated groups with polynomial growth rates. It begins by laying down the foundational groundwork in \cref{basic:def} with preliminary concepts and notation, encompassing basic algebraic background, groups with polynomial growth rates, subadditive random processes, and the metric geometry of locally compact groups. Through these foundational elements, the thesis establishes a solid framework for the subsequent exploration.

Moving forward, \cref{ch:shape.groups} delves into the heart of the matter with the investigation of the asymptotic shape. It commences with an introduction, and intermediate results are then presented, focusing on the establishment of a candidate for the limiting shape and the approximation of admissible curves along polygonal paths. \cref{sec:limiting.shape} includes detailed proofs of the first and second theorems, accompanied by an additional result for First-Passage Percolation (FPP) models. The results obtained in this part extend the theorems found in the literature \cite{benjamini2015,cantrell2017,coletti2021}. Throughout these discussions, the thesis showcases applications to random growth models in \cref{sec:examples}, illustrating the broader relevance and implications of the derived results. This part corresponds to the findings reported in \cite{coletti2023asymptotic}.

\paragraph{\cref{part:rgg}} Our focus shifts towards investigating the asymptotic shape theorem for standard FPP models on RGGs, with a particular emphasis on studying their speed of convergence and geodesics. \cref{ch:rgg} lays the groundwork by providing essential definitions and auxiliary results on RGGs.

In \cref{ch:shape.fpp.rgg}, we delve into the existence of the limiting shape of FPP models on RGGs. This chapter includes a proof of the standard shape theorem, demonstrating its validity under specified conditions using classical techniques. While this work builds upon the findings presented in our previous publication \cite{coletti2023}, we selectively omit certain intermediate details to allow a more focused exploration of the improved version presented in the subsequent chapter.

\cref{ch:moderate} further explores moderate deviation for FPP models on RGGs and its implications. Here, we impose stronger conditions on the random variables involved. By utilizing moderate deviations, we ascertain the speed of convergence to the asymptotic shape, offering a quantification of convergence that does not rely on ergodic theorems. Additionally, the chapter delves into fluctuations of geodesics and spanning trees within the context of FPP models on RGGs, providing valuable insights into their behavior.

As a consequence of the theorems obtained in the previous chapters, in in \cref{ch:competition}, we apply our results to the study of a two-species competition model. The growth and competition dynamics are determined by Richardson's and Voter's models, and we demonstrate that the probability of coexistence of both species is positive. This is due to the limiting shape being a Euclidean ball and the moderate deviations, along with several other properties studied in this thesis. 

\paragraph{\cref{part:conclusion}} We conclude by commenting and comparing the main results with an additional discussion on possible further research topics in this area.

%% file: texts/preliminaries.tex
\vspace{0.5cm}
In this chapter, we delve into the fundamental concepts of geometric group theory, a field that provides tools to comprehend the relationship between algebraic properties and geometric structures. We begin by establishing the basic definitions that serve as the cornerstone for our exploration. Central to our analysis is the construction of the asymptotic cone, a powerful tool that reveals the geometric behavior of groups at infinity. To illustrate the versatility of our framework, we present concrete examples of groups that satisfy the conditions under consideration.

To simplify notation, we use $\N$ to represent the set $\{1,2, \dots\}$ and $\N_0=\N\cup\{0\}$. Let us denote $a\wedge b := \min\{a,b\}$ and $a \vee b := \max\{a,b\}$ for $a,b \in \R$. Consider $\mathcal{O}\big(f(t)\big)$ as the class of functions given by the big $\mathcal{O}$ notation. In other words, $g(t)\in \mathcal{O}\big(f(t)\big)$ as $t\to a$ exactly when $\limsup\limits_{t\to a} \left\vert\frac{g(t)}{f(t)}\right\vert < +\infty$. Additionaly, the cardinality of a set $A$ is denoted by $|A|$.

\section{Basic Algebraic Background} \label{sec:basic.algebra}

This section explores fundamental concepts essential to our study. We commence with a brief review of the basic notions surrounding groups, establishing a common language and notation. Subsequently, we introduce key definitions that serve as building blocks for the subsequent development in our thesis. While some concepts may appear elementary, their explicit formulation ensures clarity and coherence in our exploration of more advanced topics. 

For a comprehensive treatment of group theory and abstract algebra, we recommend consulting in-depth texts such as \cite{dummit2004,lang2002}.

\begin{definition}[Groups ans Subgroups]
A \textit{group} $(\Gr, .)$ is a set $\Gr$ equipped with a binary operation $.:\Gr\times\Gr\to\Gr$ satisfying the following properties:
\begin{itemize}
    \item (\textit{Associativity}) For all $x, y, z \in \Gr$, $(x . y) . z = x . (y . z)$.
    \item (\textit{Neutral Element}) There exists $e \in \Gr$ such that, for all $x \in \Gr$, $x . e = e . x = x$.
    \item (\textit{Inverse Element}) For each $x \in \Gr$, there exists $x^{-1} \in \Gr$ such that $x . x^{-1} = x^{-1} . x = e$.
\end{itemize}

A \textit{subgroup} $H$ of a group $\Gr$ is a subset of $\Gr$ that is a group with the same binary operation restricted to $H \times H$, denoted as $H \leq \Gr$.
\end{definition}

The pair $(\Gr, \cdot)$ is commonly denoted as $\Gr$, omitting the binary operator. In our operations, we often skip the use of the dot when working with elements of the group. Here, the symbol $+$ serves as a binary operator to indicate that the group is abelian, i.e., when $x+y=y+x$ for all $x, y \in \Gr$. For $n\in \N$, the repeated operation of $x \in \Gr$ $n$ times is given by 
\[x^n := \underbrace{x.x. \cdots .x}_{n \text{ times}} \quad\text{and} \quad n\cdot x := \underbrace{x+x+ \cdots +x}_{n \text{ times}}.\]

Furthermore, the \textit{trivial group} is a group $(\Gr, \cdot)$ with $\Gr=\{e\}$. We continue defining basic structures of group theory.

\begin{definition}[Cosets and Normal Subgroups]
For a subgroup $H \leq \Gr$, the \textit{left coset} of $H$ in $\Gr$ containing $x \in \Gr$ is given by $x.H = \{x . h : h \in H\}$. Similarly, $H.x = \{ x. g : h \in H\}$ is the \textit{right coset}. A subgroup $N \leq \Gr$ is \textit{normal}, denoted as $N \unlhd \Gr$, if $x.N = N.x$ for all $x \in \Gr$.
\end{definition}

Normal subgroups play a central role in the decomposition of the group, and their properties are repeatedly applied throughout this text. An immediate consequence of the definition above is that $N\unlhd \Gr$ if, and only if, for all $x \in \Gr$, one has $N= x.N.{x^{-1}}$.

\begin{definition}[Quotients]
The \textit{quotient} of a group $\Gr$ by its subgroup $H$ is $\Gr/H$, which consists of left cosets of $H$ in $\Gr$, \textit{i.e.},
\[\Gr/H:=\{x.H : x \in \Gr\}.\]
\end{definition}

In particular, when $H \unlhd \Gr$, the quotient group $(\Gr/H, \star)$ emerges, defined by the binary operation $x.H \star y.H = xy.H$. Additionally, for $N \unlhd \Gr$, the cosets $xN$ and $yN$ are either equal or disjoint for all $x, y \in \Gr$. Consequently, $\Gr/N$ forms a partition of $\Gr$. Observe that, if $H \leq \Gr$, then for all $x\in\Gr$ and $y\in x.H$, we have $x.H=y.H$ and, therefore, every element of a coset is called a representative of the coset.

The finite index property will be a recurring necessity throughout the text. The index of a subgroup can be defined as follows.

\begin{definition}[Subgroup Index]
The \textit{index} of a subgroup $H \leq \Gr$ is the cardinality of the set of left cosets (or right cosets) of $H$ in $\Gr$ and it is denoted by $[\Gr : H]$, \textit{i.e.},
\[[\Gr : H] := \left\vert \Gr/H \right\vert.\]
\end{definition}

The groups and graph structure linked to our random processes are characterized by finite symmetric generating sets. Please find the respective definitions outlined below.

\begin{definition}[Symmetric Set]
A subset $S \subseteq \Gr$ is \textit{symmetric} if $S = S^{-1}$, where $S^{-1} = \{s^{-1} \colon s \in S\}$.
\end{definition}

\begin{definition}[Generated Subgroup]
The \textit{subgroup generated} by a (non-empty) subset $S \subseteq \Gr$ is denoted as $\langle S \rangle$ and is the smallest subgroup of $\Gr$ containing $S$. A set $S \subseteq \Gr$ is a \textit{generating set of ~$\Gr$} when $\Gr=\langle S \rangle$.
\end{definition}

A straightforward consequence of the definition above is that every element $x \in \langle S \rangle$ is so that,  when $S$ is symmetric, $x=s_1.s_2.\cdots s_m$ with $s_i \in S$ for $i \in \{1, 2, \dots, m\}$ and $m\in\N$.

\begin{definition}[Order and Torsion Subgroup]
The \textit{order} of an element $x \in \Gr$ is the smallest positive integer given by $o(x) := \inf\{n\in\N : x^n = e\}$ The \textit{torsion subgroup} of $\Gr$, denoted as $\tor \Gr$, is the set of all elements in $\Gr$ with finite order. Alternatively, \[\tor\Gr:=\langle x\in\Gr: o(x)<+\infty\rangle.\]

The group $\Gr$ is called \textit{torsion-free} when $\tor \Gr$ is the trivial group.
\end{definition}

Group actions are fundamental for defining subadditive cocycles, a cornerstone in the study undertaken in this thesis.

\begin{definition}[Group Action]
A \textit{group action} of a group $\Gr$ on a set $\Om$ is a map $\tht: \Gr \times \Om \rightarrow \Om$ satisfying:
\begin{itemize}
    \item (\textit{Identity}) For all $\omega \in \Om$, $\tht(e, \omega)=\omega$;
    \item (\textit{Compatibility}) For all $x,y \in \Gr$ and each $\omega \in \Om$, $\tht(x,\tht(y,\omega))=\tht(xy,\omega)$.
\end{itemize}
\end{definition}
To simplify notation, we write $\tht:\Gr\curvearrowright\Om$ for a group action of $\Gr$ on $\Om$ and $\tht_x(\omega):=\tht(x,\omega)$. In particular, we also write $\tht:\Gr\curvearrowright(\Om,\F,\p)$ for a group action of $\Gr$ on $\Om$ when more properties of $\tht$ are associated with the probability space $(\Om,\F,\p)$.

\begin{definition}[Ergodic and Probability Measure-Preserving Group Actions]
A group action on a probability measure space $\tht \curvearrowright (\Om,\F,\p)$ is \textit{ergodic} \[\text{if } E \in \F \text{ is such that }\tht_x(E)=E\quad\text{ for all }x \in \Gr, \text{  then }\p(E)\in\{0,1\}.\] 
The group action $\tht \curvearrowright (\Om,\F,\p)$ is $\p$-preserving (\textit{probability measure preserving} or \textit{p.m.p.}) if, for all $x \in \Gr$ and every $E \in \F$, one has $\p\left(\tht_x^{-1}(E)\right)=\p(E)$.
\end{definition}

\begin{definition}[Group Homomorphisms and Isomorphisms]
Let $(\Gr, .)$ and $(\Psi, \ast)$ be groups. A \textit{group homomorphism} $\psi: \Gr \to \Psi$ is a map such that for all $x, y \in \Gr$, \[\psi(x . y) = \psi(x) \ast \psi(y).\]

An \textit{isomorphism} between groups $\Gr$ and $\Psi$ is a bijective group homomorphism. Furthermore, $\Gr$ and $\Psi$ are said to be isomorphic when there exists a group isomorphism between them, we denote it by $\Gr \simeq \Psi$.
\end{definition}

\begin{definition}[Commutator Element]
The \textit{commutator} of elements $x, y \in \Gr$ is defined as \[[x, y] := xyx^{-1}y^{-1}.\]
\end{definition}

The commutator elements are essential to describe the connection between algebraic and geometric properties of the structures considered in this text. The aforementioned properties are going to be explored by the following sequence of normal subsets. Let us first consider $U,V \subseteq \Gr$, then we write
$[U,V]:=\big\langle[u,v] : u \in U, v\in V\big\rangle$.

\begin{definition}[Commutator Subgroup and Lower Central Series]
Set $\Gr_0:=\Gr$ and let $\Gr_n := [\Gr,\Gr_{n-1}]$ for all $n \in \N$. Thus $\{\Gr_n\}_{n\in\N}$ is the \textit{lower central series} with $\Gr_n \unlhd \Gr_{n-1}$ for all $n\in\N$. 
The group $\Gr_1 = [\Gr,\Gr]$ is called the \textit{commutator subgroup} of $\Gr$.
\end{definition}

One can easily verify that $\Gr_n$ is indeed a normal subgroup of $\Gr_{n-1}$. In particular, the quotient of the group by its commutator subgroup is called the \textit{abelianization} of $\Gr$, which is denoted by $\Gr^{\ab}:=\Gr/[\Gr, \Gr]$.

\begin{definition}[Nilpotent and Virtually Nilpotent Groups]
A group $\Gr$ is \textit{nilpotent} when there is an $n \in \N$ such that $\Gr_n = \{e\}$, \textit{i.e.}, when its lower central series stabilizes in the trivial group. More specifically, $\Gr$ is nilpotent of class $n$ when $n$ is the minimal value such that $\Gr_n=\{e\}$.

The group $\Gr$ is \textit{virtually nilpotent} when there exists a normal subgroup $N \unlhd \Gr$ that is nilpotent with finite index $\kappa=[\Gr: N] < +\infty$. In particular, we are also interested in torsion-free nilpotent groups associated with a given virtually nilpotent group $\Gr$. Hence, let $\Gr':= N / \tor N$ where $N \unlhd \Gr$ is a nilpotent with finite index.
\end{definition}

Later in the text, we will observe that the groups under consideration are both finitely generated and virtually nilpotent. In  particular, they are also going to be associated with \textit{Lie groups}. Lie groups are mathematical structures that combines the smoothness of a manifold with the group structure. Intuitively, it is a space that, near each of its points, resembles the Euclidean space. 

We do not intend to introduce the concepts of Lie Theory here, for a broad introduction we recommend reading books for this purpose such as Hall \cite{hall2015}. We continue with a brief basic explanation of relevant elements of this theory.

\begin{definition}[Lie Group]
  A Lie group $G$ is a group that is also a smooth manifold, such that the group operations $(g,h) \mapsto gh$ and $g\mapsto g^{-1}$ are smooth.
\end{definition}

A Lie group $G$ is said to be a real Lie group when $G$ is also a real manifold. Similarly, $G$ is simply connected when the manifold is simply connected and its dimension is also determined by the manifold.

\begin{definition}[Lie Algebra]
  The \emph{Lie algebra} \(\g\) is algebra with a  anti-symmetric bilinear product $[-,-]:\g\times\g\to\g$ that satisfies the Jacobi identity,\textit{i.e.}, for all $X,Y,Z \in \g$,
  \[[X,[Y,Z]] + [Y,[Z,X]] + [Z,[X,Y]] =0.\]
\end{definition}

There is a correspondence between Lie groups and Lie algebras. In particular, one can say that the elements of the Lie algebra represent group elements that are infinitesimally close to the identity element of the Lie group. The Lie algebra provides a linear approximation to the group structure around the identity, consisting of tangent vectors at the identity equipped with a Lie bracket operation that captures the group commutator.

Lie's third theorem establishes that every finite-dimensional real Lie algebra $\g$ is associated to a Lie group $G$. This correspondence is given by the \emph{exponential Lie map}, denoted by $\exp: \mathfrak{g} \to G$. Moreover, if $G$ is a simply connected nilpotent Lie Group, then the exponential map is a bijective diffeomorphism. The \emph{logarithm map} \(\log: G \to \mathfrak{g}\) serves as the inverse of the exponential map. A remarkable property in this case is that the Lie algebra $\g$ is isomorphic to the tangent space $T_eG$.



A \emph{Carnot group} is a special type of Lie group that arises in the study of sub-Riemannian geometry. Carnot groups are also known as stratified Lie Groups, they have applications in the analysis of hypoelliptic differential operators and serve as fundamental examples in geometric group theory.

\begin{definition}[Carnot Group]
A \emph{Carnot group} $H$ is a simply connected, finite-dimensional Lie group whose corresponding Lie algebra $\mathfrak{h}$ admits a stratification by non-trivial linear subspaces $V_1, \dots, V_l$ such that
  \[\mathfrak{h}= \bigoplus_{i=1}^l V_i\]
  where $[V_1, V_i] = V_{i+1}$ for all $i \in {1, \dots, l-1}$, and $[V_1, V_l] = {0}$.
\end{definition}

In particular, the limiting spaces of interest in Part I of this thesis are Carnot groups, as highlited in \cref{sec:polynomial.growth.group}.

\section{Groups with Polynomial Growth Rate}

The interplay among finitely generated groups, Cayley graphs, word metrics, and the convergence of metric spaces establishes a bridge between the algebraic properties of groups and geometric structures.

\begin{definition}[Right-invariant Cayley graph]
Let $(\Gr,.)$ be a group generated by a finite symmetric set $S$. The associated Cayley graph $\mathcal{C}(\Gr, S)$ represents elements of $G$ as vertices, with edges connecting $x$ and $y$ if and only if $y = sx$ for some $s \in S$. Formally, the right-invariant Cayley graph $\mathcal{C}(\Gr,S)=(V,E)$ is defined by
\[
    V= \Gr \quad \text{and} \quad E=\big\{\{x,sx\}:x \in \Gr, s \in S\big\}.
\]
\end{definition}

Cayley graphs provide a visual representation of the group structure and are fundamental in the study of geometric group theory. Denote by $u\sim v$ the relation $\{u,v\} \in E$. 

\begin{definition}[Self-avoiding paths and its Length]
    Let $\mathscr{P}(x,y)$ be the set of self-avoiding paths from $x$ to $y$, where each $\gamma \in \mathscr{P}(x,y)$ follows $\gamma=(x_0,\dots, x_m)$ with $m \in \N$, $x_i\sim x_{i+1}$, $x_0=x$, $x_m=y$, and $x_i \neq x_j$ for all $i \neq j$. We write $\mathtt{e} \in \gamma$ for $\mathtt{e} = \{x_{i}, x_{i+1}\} \in E$, and $|\gamma|=m$ represents the length of the path.
\end{definition}

\begin{definition}[Word norm and Word Metric]
The word length on $\Gr$ with respect to $S$ is defined as follows: For any $x \in \Gr$, the length of the shortest word (or self-avoiding path) in $S$ that represents $x$ is its word length, denoted by \[\|x\|_S \ = \inf_{\gamma\in\mathscr{P}(e, x)}\vert\gamma\vert.\] The word metric $d_S$ on $\Gr$ is given by $d_S(x, y) = \|yx^{-1}\|_S$. Observe that, since $S$ generates $\Gr$, $\|x\|_S< +\infty$ for all $x\in \Gr$.
\end{definition}

Throughout this text, various distinct metrics will be considered. Therefore, let us consider a (semi-pseudo-quasi) metric \(d_\diamondsuit\) on a non-empty set \(\mathbb{X}\), where the metric is indexed by \(\diamondsuit\). We define \(B_\diamondsuit(x, r) := \{y \in \mathbb{X} : d_\diamondsuit(x, y) < r\}\) as the open \(d_\diamondsuit\)-ball centered at \(x \in \mathbb{X}\).

A finitely generated group $\Gr$ has polynomial growth with respect to $S$ when $|B_S(e, r)| \in \mathcal{O}\big(n^{D'}\big)$ for a $D'\in \N_0$ as $n\uparrow+\infty$. The growth is associated with the Cayley graph $\mathcal{C}(\Gr,S)$. The polynomial growth rate of $\Gr$ is a constant $D \in \N_0$ such that there exists $\mathtt{k}\in (1,+\infty)$ for all $r>1$ satisfying \[\mathtt{k}^{-1}r^D \leq \vert B_S(e,r) \vert \leq \mathtt{k}r^D.\]
Thus $D = \min \big\{D'\in\N_0: \vert B_S(e,r)\vert \in \mathcal{O}\big(n^{D'}\big)\big\}$. Moreover, one can verify that the polynomial growth rate of $\mathcal{C}(\Gr,S)$ does not depend on the choice of $S$.

A noteworthy result obtained by Gromov \cite{gromov1981} is that a finitely generated group has polynomial growth exactly when it is virtually nilpotent. Therefore, the growth established by word metrics is strongly related to algebraic properties of the group.

\section{Subadditive Random Processes}\label{sec:subadditive.cocycles}

The study of subadditive processes began with Hammersley and Welsh \cite{hammersley1965}, who examined problems related to percolation on graphs. They identified a subadditive relation in these problems, noting that if this relation were additive, it would produce a result similar to the Strong Law of Large Numbers through Birkhoff's Ergodic Theorem. This work led to the development of a new branch within Ergodic Theory, initially developed by Kingman \cite{kingman1968,kingman1976}, known as \textit{Subadditive Ergodic Theory}.

Before stating the main elements of the theory, let us define the deterministic counterpart of subadditive elements. A real sequence $\{a_n\}_{n\in\N}$ is considered subadditive when, for all $n, m \in \N$, it satisfies the inequality
\[a_{n+m} \leq a_m + a_n.\]
A notable result for subadditive sequences is given by Fekete's lemma, which ensures that $\lim_{n\uparrow+\infty}\frac{a_n}{n}= \inf_{n\in\N}\frac{a_n}{n}$.

Extending this concept to functions defined on a group, a function $f:(\Gr,.) \to (\mathbb{R},+, \cdot)$ is said to be subadditive if it satisfies, for all $x,y\in\Gr$,
\[f(xy) \leq f(y) + f(x).\]
Consequentially, if $+$ is the binary operation in $\Gr$, and $f:(\Gr,+) \to (\mathbb{R},+, \cdot)$ is such that $f(x+ y) \leq f(x) + f(y)$ for all $x,y\in \Gr$, then $f$ is subadditive.

Building upon these definitions, we introduce the concept of a subadditivity of the following random elements:

\begin{definition}[Subadditive Cocycles]
    Let $\{X_n\}_{n\in\mathbb{N}}$ be a collection of random variables and let $\theta:\Om\to\Om$ be a probability measure preserving measurable map on a $(\Om,\F,\p)$. Then $\{X_n\}_{n\in\mathbb{N}}$ is said to be a \textit{subadditive random process} (or \textit{subadditive cocycle}) on $(\Om, \F,\p, \theta)$ when it satisfies, $\p$-a.s., for all $n,m\in \N$,
    \[X_{n+m} \leq X_m + X_n \circ \theta^m.\]

    Consider a group $(\Gr,.)$ and let $\tht:\Gr \curvearrowright (\Om,\F,\p)$ a p.m.p. group action. A function $c:\Gr\times\Om\to(\R,+,\cdot)$ that satisfies, $\p$-a.s., for all $x,y\in \Gr$,
    \[c(xy) \leq c(y) + c(x)\circ\tht_y\]
    is denominated \textit{subadditive cocycle function}.
\end{definition}

For simplicity, we refer to \textit{subadditive cocycle functions} exclusively as \textit{subadditive cocycles} throughout the text. Note that $\{c(x^n)\}_{n\in\N}$ is a subadditive random process when $c$ is a subadditive cocycle. Furthermore, $\big\{\E[X_n]\big\}_{n\in\N}$ is a subadditive sequence and $\thickbar{c}:\Gr\to\R$ with $\thickbar{c}(x):= \E[c(x)]$ is a subadditive function. Hence, the correspondence between the random and deterministic subadditive properties can be explored.

Below, we present a version of Kingman's Subadditive Ergodic Theorem, as outlined in \cite{steele1989} and, to some extent, in \cite{liggett1985}, among other sources.

\begin{theorem}[Kingman's Subadditive Ergodic Theorem] \label{thm:kingman}
    Let $(\Om, \F,\p, \theta)$ be a probability space with a measurable map $\theta:\Om\to\Om$ which is probability measure preserving. Consider $\{X_n\}_{n\in\N}$ a collection of $L^1$ random variables so that, $\p$-a.s., for all $n,m \in \N$,
    \[
        X_{n+m} \leq X_m + X_n\circ \theta^m
    \]
    Then there exists a $\theta$-invariant random variable $Y \in [-\infty,+\infty)$ such that
    \begin{align*}
        Y &= \lim_{n\uparrow+\infty} \frac{1}{n}X_n \quad \p\text{-a.s. and in }L^1, \quad \text{and}\\
        \E[Y] &= \lim_{n \uparrow +\infty} \frac{1}{n}\E[X_n]= \inf_{n \in \N} \frac{1}{n}\E[X_n].
    \end{align*}

    Additionally, if $\theta$ is ergodic, then $Y$ is constant.
\end{theorem}

This theorem provides an exceptional methodology for studying the asymptotic behavior of subadditive cocycles. It demonstrates the feasibility of studying the asymptotic shape through subadditive cocycles, as their behavior can be correlated with that of a norm in space. Nevertheless, it remains essential to manage convergence in all directions or undertake the approximation of geodesic curves.

There are various versions of \cref{thm:kingman}, along with several improved subadditive ergodic theorems, such as those derived by Derriennic \cite{derriennic1983}, Liggett \cite{liggett1985}, and Kesten (refer to comments in \cite{kesten_discussion}). Our primary focus will be on examining processes that fit the hypotheses of Kingman's theorem. Nevertheless, these supplementary results suggest potential pathways for future investigations into the object under study.

\section[Metric Geometry of Locally Compact Groups]{Metric Geometry of Locally Compact \\ Groups}

A key focus of our investigation lies in the construction of the norm in the limiting space, which will be denoted as $G_\infty$, providing the foundation for defining the limiting shape. We explore crucial results and properties in the following subsections. For an in-depth discussion on this topic, we refer interested readers to \cite{breuillard2014,burago2001,decornulier2011,decornulier2016,raghunathan1972}. Through this lens, we gain a deeper understanding of the interplay between algebraic properties and geometric structures.

\subsection{Volume Growth of Cayley Graphs} \label{sec:polynomial.growth.group}

Let $[U]_\varepsilon$ to be the \textit{$\varepsilon$-neighborhood} of $U \subseteq \mathbb{X}$ of in a metric space $(\mathbb{X}, d_\diamondsuit)$, \textit{i.e.}, the set $[U]_\varepsilon = \bigcup_{u \in U} B_\diamondsuit(e,\varepsilon)$. The Hausdorff distance $d_H$ detects the largest variations between sets with respect to the given metric
\[d_H(U,V) := \inf\{\varepsilon>0 : U \subseteq [V]_\varepsilon \text{ and } V \subseteq [U]_\varepsilon\}.\]  
We define the convergence of metric spaces used in the main theorems employing the Hausdorff distance. Let $(\mathbb{X}_n, d_{\diamondsuit_n}, o_n)_{n \in \N}$ be a sequence of centered, locally compact metric spaces. Consider $\{\psi_n\}_{n \in \N}$ as a family of isometric embeddings $\psi_n : (\mathbb{X}_n, d_{\diamondsuit_n}, o_n) \to (\mathbb{X}, d_\diamondsuit, o)$.

The \textit{pointed Gromov-Hausdorff convergence} of $(\mathbb{X}_n, d_{\diamondsuit_n}, o_n)$ to $(\mathbb{X}, d_\diamondsuit, o)$ is denoted by
\[
    (\mathbb{X}_n, d_{\diamondsuit_n}, o_n) \GHto (\mathbb{X}, d_\diamondsuit, o)
\]
and it implies, for all $r>0$,
\[
    \lim_{n \uparrow +\infty} d_H \Big( \psi_n\big( B_{\diamondsuit_n}(o_n,r) \big), ~B_\diamondsuit(o,r)\Big) = 0.
\]
The definitions above are immediately extended to random semi-pseudo-quasi metrics, as employed in the main theorems (see \cref{shape.thm,thm:shape.polynomial,cor:fpp.virt.nil}). Here, a pseudo-quasi metric refers to a metric where the properties of positivity and symmetry do not necessarily hold true. It would also be possible to replace the triangle inequality with a weaker condition; however, this is beyond the scope of our discussion here.

The assumption of almost sure local compactness is also maintained. We are now prepared to present Pansu's theorem on the convergence of finitely generated virtually nilpotent groups.

\begin{theorem}[Pansu \cite{pansu1983}] \label{thm:Pansu}
    Let $\Gr$ be a virtually nilpotent group generated by a symmetric and finite $S \subseteq \Gr$. Then
    \[
        \left( \Gr, \frac{1}{n}d_S, e \right) \GHto (G_\infty, d_\infty, \mathlcal{e}),
    \]
    where $G_\infty$ is a simply connected real graded Lie group (Carnot group). The metric $d_\infty$ is a right-invariant sub-Riemannian (Carnot-Caratheodory) metric which is homogeneous with respect to a family of homotheties $\{\delta_t\}_{t>0}$, \textit{i.e.}, $d_\infty\big(\delta_t(\mathlcal{g}),\delta_t(\mathlcal{h})\big) = t~d_\infty(\mathlcal{g},\mathlcal{h})$ for all $t>0$ and $\mathlcal{g},\mathlcal{h} \in G_\infty$. 
\end{theorem}

Note that \cref{shape.thm,thm:shape.polynomial} are generalizations of the theorem above. Therefore, the shape theorems under investigation can be interpreted as the convergence of random metric spaces in large-scale geometry. The next subsection is dedicated to the construction of the asymptotic cone $G_\infty$ and related results.

\subsection{Rescaled Distance and Asymptotic Cone} \label{sec:rescaled.dist}

Consider for now $\Gr$  as a nilpotent and torsion-free group, unless stated otherwise. We also assume that its abelianization is torsion-free. In this subsection, we use $\Gr$ instead of $\Gr'$ to simplify notation, but we will subsequently extend the results to the more general case.

Let $G$ denote the real Mal'cev completion of $\Gr$. The group $G$ can be defined as the smallest simply connected real Lie group such that $\Gr \leq G$ and, for all $z \in \Gr$ and $n \in \N$, there exists $\mathlcal{z} \in G$ with $\mathlcal{z}^n = z$. In this case, $G$ is nilpotent of the same order of $\Gr$ and it is uniquely defined. Furthermore, $G$ is simply connected it is associated with the Lie algebra $(\g, [\cdot,\cdot]_1)$ where $\Gr$ is cocompact in $G$. We write $\log:G \to \g$ for the Lie logarithm map.

Define $\g^1 := \g$  and $\g^{i+1} := [\g,\g^i]_1$. It follows from the nilpotency of $\Gr$ that threre exists $l \in \N$ such that $\Gr_l = \{e\}$. Thus $\g^{l+1} = (0)$. Since $[\g^{i},\g^{j}]_1 \subseteq \g^{i+j}$ and, in particular, $[\g^{i+1},\g^{j}]_1,[\g^{i},\g^{j+1}]_1 \subseteq \g^{i+j+1}$, the Lie bracket on $\g$ determines a bilinear map
\[(\g^{i}/\g^{i+1}) \otimes (\g^{j}/\g^{j+1}) \longrightarrow \g^{i+j}/\g^{i+j+1}\]
which in turn defines a Lie bracket $[\cdot,\cdot]_\infty$ on

\[\g_\infty := \bigoplus\limits_{i=1}^l \valg_i~~~\mbox{with}~~ \valg_i := \g^i/\g^{i+1}.\]

Consider the decomposition $\g = V_1 \oplus \cdots \oplus V_l$ given by $\g^i:= V_i \oplus \cdots \oplus V_l$. Thus, $(\g_\infty, [\cdot,\cdot]_\infty)$ is a graded Lie algebra. Let us define a family of linear maps $\delta_t : \g_\infty \to \g_\infty$ given by
\[\delta_t(v_1+v_2+ \dots + v_l) = tv_1+t^2v_2+ \dots+ t^lv_l\]
for each $t >0$ and $v_i \in \mathfrak{v}_i$ with $i \in \{1,\dots,l\}$. It follows from the definition of $\delta_t$ that $\delta_t([u,v]_\infty) = [\delta_t(u),\delta_t(v)]_\infty$ and $\delta_{tt'} = \delta_t \circ \delta_{t'}$ for all $u,v\in \g_\infty$ and $t,t'>0$. Hence, $\{\delta_t\}_{t>0}$ defines a family of automorphisms in the graded Lie group $G_\infty := \exp_\infty\left[ \g_\infty \right]$. Here we write $\exp_\infty:\g_\infty \to G_\infty$ and $\exp:\g \to G$ to differentiate the distinct exponential maps of $\g_\infty$ and $\g$. Similarly, $\log_\infty$ and $\log$ stand for their correspondent Lie logarithm maps.

Let $\g = V_1 \oplus \cdots \oplus V_l$ be the decomposition given by $\g^i= V_i \oplus \cdots \oplus V_l$. Set $L: \g \to \g_\infty$ to be an linear map such that $L(V_i) = \valg_i$. Consider now $\sigma_t$ to be the linear automorphism on $\g$ so that $\sigma_t(v_i)=t^iv_i$ for each $v_i\in V_i$ and $i \in \{1, \dots, l\}$. Define the Lie brackets $[\cdot,\cdot]_t$ on $\g$ by
\[[v,w]_t = \sigma_{1/t}\big([\sigma_t(v),\sigma_t(w)]\big),~\mbox{ for all }t>0,\]
thus $(\g, [\cdot,\cdot]_t)$ is isomorphic to $(\g, [\cdot,\cdot]_1)$ via $\sigma_t$. Furthermore,
\[[L(v),L(w)]_\infty = \lim\limits_{t\uparrow+\infty} [v,w]_t\]
since, given $v \in V_i$ and $w \in V_j$, one has that the main term belongs to $V_{i+j}$, the other terms of superior order belong to $V_{i+j+1} \oplus \cdots \oplus V_l$ and it makes them insignificant in the rescaled limit (see \cite{breuillard2014,pansu1983} for a detailed discussion). Set
\[
	\frac{1}{t_n}\bull x_n := (\exp_\infty \circ L \circ \sigma_{1/t_n}\circ \log)(x_n).
\]
The convergence established by \cref{thm:Pansu} determines the metric $d_\infty$ such that
\[\left(\Gamma,\frac{1}{n}d_S,e\right) \GHto \left(G_\infty,d_\infty,\mathlcal{e}\right).\]
Hence, $\lim_{n\uparrow +\infty} \frac{1}{t_n}\bull x_n = \mathlcal{g}$ exactly when $\frac{1}{t_n}\bull x_n$ converges to $\mathlcal{g}$ in $(G_\infty,d_\infty)$. The corresponding metric statement shows that, given sequences $\{x_n\}_{n\in\N}$, $\{x_n'\}_{n\in\N}$ in $\Gamma$, and $t_n\uparrow+\infty$ as $n \uparrow +\infty$ with $\lim_{n\uparrow +\infty} \frac{1}{t_n}\bull x_n = \mathlcal{g}$ and $\lim_{n\uparrow +\infty} \frac{1}{t_n}\bull x_n' = \mathlcal{g}'$,
\begin{equation*}
    d_\infty (\mathlcal{g},\mathlcal{g}') = \lim_{n\uparrow +\infty} \frac{1}{t_n}d_S(x_n,x_n').
\end{equation*}

The abelianized Lie algebras are defined by $\g_\infty^{\ab} := \g_\infty/[\g_\infty,\g_\infty]_\infty \cong \mathfrak{v}_1$ and $\g^{\ab}:=\g/[\g,\g]_1$. In particular, $\g^{\ab}\cong\g_\infty^{\ab}$.  According to the Frobenius integrability criterion, the integrable curves in $G_\infty$ are those for which the tangent vectors at each point of the curve belong to $\mathfrak{v}_1$. An \textit{admissible} (or curve) in $G_\infty$ is a Lipschitz curve $\upgamma:[t_0,t_1] \to G_\infty$ such that the tangent vector $\upgamma' (t) \in \mathfrak{v}_1$ for all $t \in [t_0, t_1]$. Let $\phi:\g_\infty^{\ab}\to[0,+\infty)$ be a norm in the abelianized algebra. Then the $\ell_\phi$-length of the admissible $\upgamma$ is
\[\ell_\phi(\upgamma) := \int_{t_0}^{t_1}\phi\big(\upgamma'(t)\big)dt.\]
Set $d_\phi$ to be the inner metric of the length space $(G,\ell_\phi)$ given by
\begin{equation} \label{eq:def.d.phi}
    d_\phi(\mathlcal{g},\mathlcal{g}') :=\inf\big\{\ell_\phi(\upgamma): \upgamma \text{ is an admissible curve from } \mathlcal{g} \text{ to }\mathlcal{g}' \text{ in }G_\infty\big\}.
\end{equation}

In fact, the construction of $d_\phi$ can be employed to define $d_\infty$. The bi-Lipschitz property is a consequence of the results in \cref{sec:norms.and.mean}. One can also verify that the metric $d_\infty$ is right-invariant and homogeneous with respect to $\delta_t$. Let us define the projections
\[\pi:\g\to\g^{\ab} \quad \text{and}\quad\pi_\infty:\g_\infty \to \mathfrak{v}_1\cong \g_\infty^{\ab}\]
so that, if $v=\sum_{i=1}^l v_i \in \g_\infty$ with $v_i\in\mathfrak{v_i}$, then $\pi_\infty(v)=v_1$ and $\pi = L^{-1}\circ \pi_\infty\circ L$. The next lemma compiles several well-known results that will be employed throughout the text. We state the results and their proofs can be found in \cite{cantrell2017,decornulier2016,pansu1983}.

\begin{lemma} \label{lm:conv.seq.Gr}

    Consider $\Gr$ a finitely generated torsion-free nilpotent group, then all of the following hold true:
    \begin{itemize}
    \item[(i)] Let $\mathlcal{g} \in G_\infty$. Then there exists a sequence $\{x_n\}_{n \in \N} \subseteq \Gr$ such that \[\lim_{n\uparrow + \infty}\frac{1}{n}\bull x_n= \mathlcal{g}.\]

    \item[(ii)]
    Let $\{x_n\}_{n \in \N}, \{y_n\}_{n \in \N} \subseteq \Gr$, $\mathlcal{g}, \mathlcal{h} \in G_\infty$, and  $t_n \uparrow +\infty$ as $n \uparrow +\infty$ be such that\break
    $\lim_{n\uparrow + \infty}\frac{1}{t_n}\bull x_n = \mathlcal{g}$ and $\lim_{n\uparrow + \infty}\frac{1}{t_n}\bull y_n = \mathlcal{h}$. Then
    \[\lim_{n\uparrow + \infty}\frac{1}{t_n}\bull x_n y_n = \mathlcal{gh}.\]

    \item[(iii)]
    Let $x \in \Gr$, then
    \begin{align*}
        \lim_{n\uparrow + \infty}\frac{1}{n}\bull x^n &= \left(\exp_\infty \circ L \circ \pi \circ \log\right)(x)\\
        &= \left(\exp_\infty \circ ~\pi_\infty \circ L\circ \log\right)(x).
    \end{align*}
    \end{itemize}
\end{lemma}
\begin{remark}
    The conditions imposed on $\Gr$ might appear somewhat restrictive. However, we will subsequently regain many properties by making necessary adjustments for virtually nilpotent $\Gr$ through the quotient $\Gr' = N/\tor N$ (see \cref{sec:virt.nilpotent.proofs}).
\end{remark}

The item (iii) in \cref{lm:conv.seq.Gr} has direct implications for the application of subadditive ergodic theorems. To address this constraint, we overcome it by approximating the lengths using polygonal curves. We present, without proof, Lemma 3.7 from \cite{cantrell2017}, which will be employed in the approximation.

\begin{lemma} \label{lm:small.perturbations}
    Consider $\Gr$ nilpotent. Let $\{y_i\}_{i=1}^m \subseteq \Gr$ and $\varepsilon>0$ be given. Then there exist $\xi>0$ and $\thickbar{n} \in \N$ so that, for all $n>\thickbar{n}$, for all $n_j \in \{0, 1, \dots, \lfloor \xi \thickbar{n}\rfloor\}$,
    \[
        \frac{1}{n}d_S\left(y_m^{n-n_m}y_{m-1}^{n-n_{m-1}}\dots y_1^{n-n_1},~y_m^{n}y_{m-1}^{n}\dots y_1^{n}\right)< \varepsilon.
    \]
\end{lemma}

One standout example that exemplifies several properties presented above is the discrete Heisenberg group. As a prime example of a nilpotent group, it offers valuable insights into the fusion of algebraic structures with geometric phenomena in both geometric group theory and metric geometry.

\begin{example}[The discrete Heisenberg group] \label{ex:Heisenberg.group}
    The discrete Heisenberg group can be visualized as a collection of integer lattice points in a three-dimensional space, with a unique group structure derived from matrix multiplication. The nilpotent nature is the key to understand its intricate geometric properties. Let $R$ be a commutative ring with identity and set $H_3(R) := \{(\mathtt{x}, \mathtt{y}, \mathtt{z}): \mathtt{x},\mathtt{y},\mathtt{z} \in R\}$ to be the set of upper triangular matrices with
    \[
        (\mathtt{x}, \mathtt{y}, \mathtt{z}) := 
        \begin{pmatrix}
        \mathtt{1} & \mathtt{x} & \mathtt{z}\\
        \mathtt{0} & \mathtt{1} & \mathtt{y} \\
        \mathtt{0} & \mathtt{0} & \mathtt{1}
        \end{pmatrix}.
    \]
    
    \begin{figure}[htb!]
    \centering
    \includegraphics[scale=0.16]{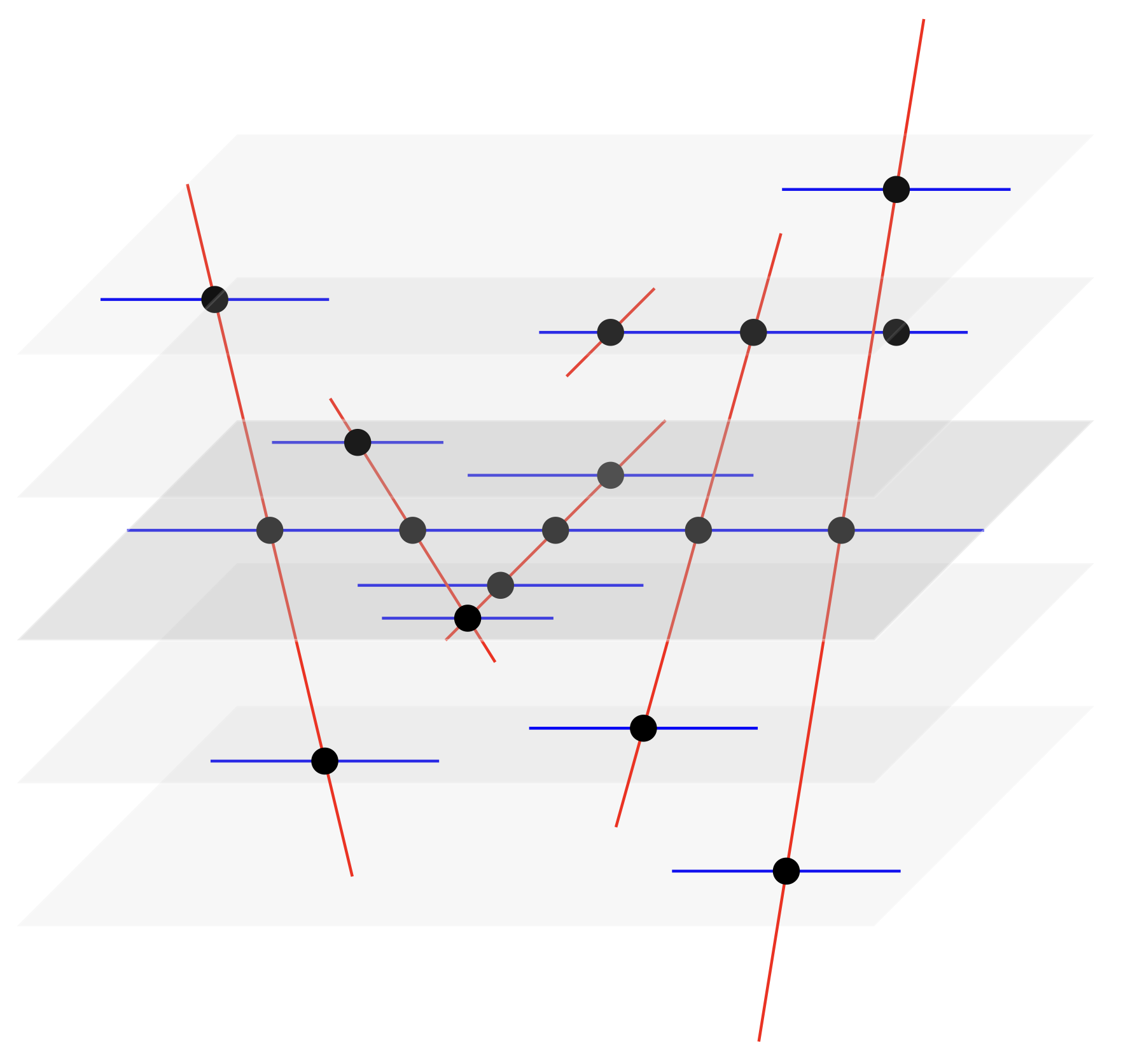}
    \caption{A section of the Heisenberg discrete Cayley graph $\mathcal{C}(H_3(\mathbb{Z}),S)$ embedded in $\R^3$. Fig. from \cite{coletti2023asymptotic}.}
    \end{figure}

    The Heisenberg group on $R$ is $H_3(R)$ with the matrix multiplication. In particular, $H_3(\mathbb{Z})$ is known as discrete Heisenberg group. Let $\Gr = H_3(\mathbb{Z})$, $\mathtt{X} = (1,0,0)$, $\mathtt{Y} = (0,1,0)$, $\mathtt{Z} = (0,0,1)$, and $S =\{\mathtt{X}^{\pm 1},\mathtt{Y}^{\pm 1}\}$. Observe that
    \begin{align*}
        (\mathtt{x}, \mathtt{y}, \mathtt{z}).(\mathtt{x}', \mathtt{y}', \mathtt{z}') &= (\mathtt{x} + \mathtt{x}', ~\mathtt{y} + \mathtt{y}', ~\mathtt{z} + \mathtt{z}' + \mathtt{x}\mathtt{y}'),\\
        (\mathtt{x}, \mathtt{y}, \mathtt{z})^{-1} &= (-\mathtt{x}, -\mathtt{y}, \mathtt{xy}-\mathtt{z}), \text{ and}\\
        \big[(\mathtt{x}, \mathtt{y}, \mathtt{z}),~(\mathtt{x}', \mathtt{y}', \mathtt{z}')\big] &= (0,0,\mathtt{xy}'-\mathtt{x}'\mathtt{y}).
    \end{align*}

    Therefore, for all $m,n\in \mathbb{Z}$,
    \begin{equation} \label{eq:Heisenberg.op}
        \mathtt{X}^m = (m,0,0), \quad \mathtt{Y}^n=(0,n,0),~~\text{and} \quad [\mathtt{X}^m,\mathtt{Y}^n]=\mathtt{Z}^{m\cdot n} = (0,0,m\cdot n).
    \end{equation}

    One can easily see that $S$ is a finite generating set of $\Gr$. Furthermore, $\Gr_1 = [\Gr,\Gr] = \langle\mathtt{Z}\rangle$ and $\Gr_2=[\Gr,\Gr_1] = \{e\}$. Hence, $\Gr$ is nilpotent of class $2$ and $S \subseteq \Gr \setminus [\Gr,\Gr]$. Consider $\|-\|_{S}$ the word norm of $\mathcal{C}(\Gr,S)$. It follows from \eqref{eq:Heisenberg.op} that
    \[\|\mathtt{Z}^{m}\|_S \in \mathcal{O}(\sqrt{m}) \quad \text{as }m \uparrow +\infty.\]
    It highlights how the rescaled norm $\frac{1}{n}\|x^n\|_S$ vanishes as $n \uparrow +\infty$ when $x \in [\Gr,\Gr]$. 
    
    Due to the properties above, one can write $(\mathtt{x},\mathtt{y})=(\mathtt{x},\mathtt{y},\mathtt{z})[\Gr,\Gr]$. Note that $S^{\ab} = \left\{(\pm 1,0), (0,\pm 1)\right\}$ is a finite generating set of the abelianized group $\Gr^{\ab}= \Gr/[\Gr,\Gr]$ which yields an isomorphism of $\mathcal{C}(\Gr^{\ab},S^{\ab})$ and the square $\mathbb{Z}^2$ lattice.

    By construction of the asymptotic cone, the Mal'cev completion $G$ of $\Gr\simeq\Gr'$ is the continuous Heisenberg group $H_3(\R)$ with its associated Lie algebra $\mathfrak{h}= \g$, in this case, $\g\simeq \g_\infty$ and $G \simeq G_\infty$. The Heisenberg algebra $\mathfrak{h}$ is given by $\mathfrak{h} = \operatorname{span}_{\R}\{e_{12},e_{13}, e_{23}\}$ with $\big\{e_{ij}:i,j \in \{1,2,3\}\big\}$ the canonical basis of $M_{3\times3}(\R)$.

    Since for all $\mathtt{A},\mathtt{B} \in \mathfrak{h}$ one has $[\mathtt{A},\mathtt{B}]_\infty = \mathtt{A}\mathtt{B}-\mathtt{B}\mathtt{A}\in\operatorname{span}_{\R}\{e_{13}\}$ by matrix multiplication, it then follows that $\mathfrak{h} = \mathfrak{v}_1 \oplus \mathfrak{v}_2$ with $\mathfrak{v}_1 \simeq \operatorname{span}_{\R}\{e_{12},e_{23}\}$ and $\mathfrak{h}^{\ab}\simeq \g^{\ab}_\infty \simeq\mathfrak{v}_1$. 
    
    Let $\mathtt{A} = \mathtt{u}\cdot e_{12} + \mathtt{v}\cdot e_{23} + \mathtt{w}\cdot e_{13}$, then $\exp_\infty (\mathtt{A})= \left(\mathtt{u},\mathtt{v}, \mathtt{w}+ \frac{1}{2}\mathtt{uv}\right)$. 
    
    Since $(\mathtt{x},\mathtt{y},\mathtt{z})^{n} = \left(n\mathtt{x},n\mathtt{y},n\mathtt{z} + \frac{n(n-1)}{2}\mathtt{xy}\right)$ one can verify by the procedure defined in this section that $\frac{1}{n} \bull (\mathtt{x},\mathtt{y},\mathtt{z})^{n} = \left(\mathtt{x},\mathtt{y},\frac{1}{n}\mathtt{z}-\frac{1}{2n}\mathtt{xy} + \frac{1}{2}\mathtt{xy}\right) \in G_\infty$. It implies that, for all $\mathtt{x},\mathtt{y},\mathtt{z} \in \mathbb{Z}$,
    \[
        \lim_{n \uparrow +\infty} \frac{1}{n} \bull (\mathtt{x},\mathtt{y},\mathtt{z})^{n} = \left(\mathtt{x},\mathtt{y},\frac{1}{2}\mathtt{xy}\right) = \exp_\infty \left(\pi_\infty\Big( \log(\mathtt{x},\mathtt{y},\mathtt{z}) \Big)\right).
    \]
\end{example}

\subsection{Some Examples of Virtually Nilpotent Groups} \label{sec:examples.virt.nil}
In this subsection, our focus shifts to examples of virtually nilpotent groups that can be constructed through direct and outer semidirect products. The discussion of the virtually nilpotent case will be explored more extensively later in the text.

Let $\mathrm{L}$ be a nilpotent group and consider $M$ a finite group. Then the direct product
\[
    K = \mathrm{L} \times M
\]
is a group with the binary operation given by $(x,m).(y,m') = (xy,mm')$. Note that the commutator is $\big[(x,m),(y,m')\big] = \big([x,y], [m,m']\big) $. It follows that, for all $A,A' \subseteq \mathrm{L}$ and $B,B' \subseteq M$,
\[
    \big[A \times B,A' \times B'\big]  = [A, B] \times [A', B'].
\]

Hence, $K$ is a nilpotent group if, and only if, $M$ is nilpotent. On the other hand, for all finite group $M$, $K$ is virtually nilpotent.

Set $S_{\mathrm{L}}$ and $S_M$ to be finite symmetric generating sets of $\mathrm{L}$ and $M$, respectively. \[\big(S_{\mathrm{L}}\times\{e\}\big) \cup \big(\{e\}\times S_M)\] is a finite generating set of $K$. We will consider another useful example of generating set of $K$. Let $S_\square^e$ stand for $S_\square\cup\{e\}$. Then 
\[S=S_{\mathrm{L}}\times S_M^e\]
is also a symmetric generating set of $K$. In \cref{ch:shape.groups}, we will define a set $\llbracket S \rrbracket$. Under the assumption that $\mathrm{L}$ is torsion-free, the set $\llbracket S \rrbracket$ is analogous to $S_{\mathrm{L}}$, where $\Gr' \simeq \mathrm{L}$.

\begin{example}
    Let $\mathrm{SL}(2,3)$ be the of degree two over a field of three elements determined by 
    \[\mathrm{SL}(2,3) = \left\langle \rho_1,\rho_2,\rho_3 : \rho_1^3=\rho_2^3=\rho_3^3 =\rho_1 \rho_2 \rho_3 \right\rangle\] 
    A remarkable property of $\mathrm{SL}(2,3)$ is that it is the smallest group that is not nilpotent. Let $\Z_m = \langle  \rho_0\rangle$ the cyclic group with $\rho_0^m=e$ and consider $H_3(\Z)$ to be the discrete Heisenberg group, as defined in \cref{ex:Heisenberg.group}. Set
    \[\Gr = \big(H_3(\Z)\times\Z_m\big)\times\mathrm{SL}(2,3).\]

    Then $\Gr$ is virtually nilpotent with $N=H_3(\Z)\times\Z_m\times\{e\} \unlhd \Gr$ such that $\kappa=[\Gr:N] = |\mathrm{SL}(2,3)| = 24$. Hence, considering this notation:
    \[N \simeq H_3(\Z)\times\Z_m, \quad \tor N = \{e\}\times\Z_3\times\{e\} \simeq \Z_3 \quad \Gr'=N/\tor N \simeq H_3(\Z).\]
    Let us write $\mathrm{SL}(2,3) = \{z_j\}_{j=1}^{24}$ and fix $z_{(j)} = (e,e,z_j)$ as representatives for each coset in $\Gr/N$. Thus,
    \[\uppi_N(x,y,z)=(x,y, e), \text{ and} \quad \big\llbracket (x,y,z) \big\rrbracket = \{x\}\times\Z_3\times\{e\} \cong x \in H_3(\Z).\]
    Now, set
    \[S_{H_3(\Z)}= \big\{\mathtt{X}^{\pm1},\mathtt{Y}^{\pm1}\big\}, \quad S_{\Z_m}=\left\{\rho_0^{\pm1}\right\}, \text{ and}\quad S_{\mathrm{SL}(2,3)}=\left\{\rho_1^{\pm1},\rho_2^{\pm1},\rho_3^{\pm1}\right\}.\]
    Then \[S=S_{H_3(\Z)} \times S_{\Z_m}^e \times S_{\mathrm{SL}(2,3)}^e\]
    is a finite symmetric generating set of $\Gr$. Moreover, the Cayley graph $\mathcal{C}(\Gr,S)$ is homomorphically equivalent to $\mathcal{C}(\Gr',\llbracket S \rrbracket)$, which is isomorphic to $\mathcal{C}\left(H_3(\Z),S_{H_3(\Z)}\right)$.
\end{example}

More generally, one can also obtain a virtually nilpotent group by the outer semidirect product. Consider $N$ a nilpotent and $H$ a finite group. Let $\varphi$ be a group homomorphism $\varphi:H \to \operatorname{Aut}(N)$, where $\operatorname{Aut}(N)$ is the automorphism group of $N$. Then the semidirect group is
\[
    \Gr = N \rtimes_\varphi H
\]
whose elements are the same of $N \times H$ but the binary operation is characterized by 
\begin{align*}
    (x,h).(y,h') &= \big(x\varphi_h(y), hh'\big),\\
    (x,h)^{-1} &= \big(\varphi_{h^{-1}}(x^{-1}), h^{-1}\big), \text{ and}\\
    \Big[(x,h),(y,h')\Big] &= \Big(x\varphi_h(y)\varphi_{hh'h^{-1}}(x^{-1})\varphi_{[h,h']}(y^{-1}), ~[h,h']\Big).
\end{align*}

Let $S_N$ and $S_H$ be finite symmetric generating sets of $N$ and $H$, respectively. Hence, similarly to the direct product, \[\big(S_N\times\{e\}\big)\cup\big(\{e\}\times S_H\big)\] is a finite symmetric generating set of $\Gr$. Moreover, $S_N \times S_H^e$ is also a finite generating set, but not necessarily symmetric. However,
\[\left(\bigcup_{h\in H}\varphi_h(S_N)\right)\times H\]
is finite, symmetric, and generates $\Gr$. The next example illustrates how some properties of the outer semidirect product groups change in comparison to the direct product.

\begin{example}[Generalized dihedral group] \label{ex:dihedral}
Let $(N,+)$ be a finitely generated abelian group with polynomial growth rate $D\geq 1$ and $(\Z_2,+)$ with $\Z_2=\{0,1\}$. Fix $\varphi:\Z_2\to\operatorname{Aut}(N)$ such that $\varphi_0=id$ and $\varphi_1=-id$. The generalized virtually nilpotent diheral group is \[\operatorname{Dih}(N):= N \rtimes_\varphi \Z_2.\]

Consider $\Gr=\operatorname{Dih}(N)$, then for all $(x,r),(y,r') \in \Gr$,
\begin{align*}
    (x,r).(y,r') &= \big(x+\varphi_r(y), r+r'\big),\\
    (x,r)^{-1} &= \big((-1)^{r+1}x, r\big),\\
    \Big[(x,r),(y,r')\Big] &= \Big(\big(1 -(-1)^{r'}\big)x - \big(1-(-1)^{r}\big)y,\ 0\Big).
\end{align*}

Therefore, $\Gr$ is non-abelian and $\Gr_1 =[\Gr,\Gr]= 2N\times\{0\}$. One can easily verify that all elements of $\Gr_2 = [\Gr, \Gr_1]$ are
\[\Big[(x,r),(2y,0)\Big] = \Big(2\big((-1)^{r}-1\big)y ,\ 0\Big).\]

Hence, for all $n\in \N$, one has $\Gr_{n} \simeq 2^nN$. We can conclude that $\Gr$ is not nilpotent while it is virtually nilpotent since $N \unlhd \Gr$.
\end{example}

\section{First-Passage Percolation Models} \label{sec:fpp} Hammersley and Welsh \cite{hammersley1965} introduced the First-Passage Percolation (FPP) as a mathematical model in 1965 to study the spread of fluid through a porous medium. In FPP models, a graph with random edge weights is considered, where these weights represent the time taken for the fluid to pass through the corresponding edge. These concepts will be revisited in \cref{sec:additional.FPP} and illustrated with examples in \cref{sec:examples}. Furthermore, the random processes considered in Part III are i.i.d. FPP models.

Let $\mathcal{G}=(V,E)$ be a graph and set $\tau=\{\tau(\mathtt{e})\}_{\mathtt{e} \in E}$ to be a collection of non-negative random variables. We may regard each $\tau(u,v)$ as random length (also \textit{passage time} or \textit{weight}) of an edge $\{u,v\} \in E$. It turns $(\mathcal{G},\tau)$ into a random length space and it motivates the following construction.

The random passage time of a path $\gamma \in \mathscr{P}(x,y)$ is given by $T(\gamma) = \sum_{\mathtt{e} \in \gamma} \tau(\mathtt{e})$. Let us now define the first-passage time of $y$ with the process starting at $x$ by
\[
    T(x,y) := \inf_{\gamma \in \mathscr{P}(x,y)} T(\gamma).
\]

The random variable $T(x,y)$ is also known as \textit{first-hitting time}. Observe that $T(x,y)$ is a random intrinsic pseudometric, \textit{i.e.}, $x\neq y$ does not imply in $T(x,y)>0$. We can now consider the group action $\tht:\Gr\curvearrowright(\Om,\F,\p) $ as a translation such that $c(x):= T(e,x)$ is a subadditive cocycle (see \eqref{eq:cocycle.metric}) with $\tau(x,sx)\circ\tht_{y} = \tau(xy^{-1},sxy^{-1})$ for all $x,y\in \Gr$ and $s \in S$ when $\mathcal{G}=\mathcal{C}(\Gr,S)$.

By requiring $\tht$ to be ergodic, we obtain for the FPP model that, for all $x \in \Gr$ and $s \in S$,
\[c(s)\circ\tht_x=\tau(x,sx)\sim  \tau(e,s)=c(s).\]
It also follows that $\tau(e, s) \sim \tau(e, s^{-1})$. Therefore, each direction of $\mathcal{C}(\Gr,S)$ determines a common distribution for its random lengths in a FPP model. \cref{ex:color} portraits an FPP model with dependend and identically distributed random lengths. While the random variables of the FPP model presented in \cref{ex:richardson} are independent but not identically distributed.

In the next chapter, we will introduce conditions to study a family of random processes on groups. Specifically, we will examine conditions \eqref{all}, \eqref{aml}, \eqref{innerness}, \eqref{aml2}, and \eqref{innerness2} (see \cref{ch:shape.groups}). Since passage times are preserved under translation, condition \eqref{innerness} is immediately satisfied when $S = F(\varepsilon)$, as it suffices to consider a geodesic path. However, other examples of subadditive interacting particle systems do not exhibit these properties. For instance, the Frog Model (see \cref{ex:frog}) can be described by a subadditive cocycle satisfying \eqref{all}, \eqref{aml}, and \eqref{aml2}. If we denote $\tau(x,sx) = |T(x)-T(sx)|$, then $\tau$ describes the growth of the process, and
\[ \tau(x,sx) \circ \theta_x \sim \tau(e,s) \quad\text{while} \quad \tau(x,sx) \not\sim \tau(e,s).\]

The results and properties highlighted above will be crucial in the study of the asymptotic shape and its applications in the subsequent discussions.

%% file: texts/shape_groups.tex
\vspace{2.1cm}
\section{Introduction}

The investigation of the asymptotic shape for subadditive processes on groups with polynomial growth, often synonymous with virtually nilpotent groups, has recently gained significant attention in the mathematical community. This is in part due to the fact that the usage of subadditive ergodic theorems for the limiting shape relies on vertex-transitive properties that are natural for group actions. Typically, these actions involve translations of the underlying space, providing motivation for the investigation of random processes defined on groups. Our study brings to light the algebraic structures inherent in a class of subadditive processes, offering a generalization beyond the fundamental settings of previously studied models.

The findings presented in this work hold the potential to deepen our comprehension of various mathematical and scientific phenomena. For instance, they could be instrumental in exploring the geometry of random surfaces or modeling the propagation of information or diseases through networks. The techniques used in this study could also be applied to other types of random processes on graphs or manifolds. 

Benjamini and Tessera \cite{benjamini2015} were the first to establish an asymptotic shape theorem for First-Passage Percolation (FPP) models on finitely generated groups of subexponential growth with i.i.d. random variables having finite exponential moments. Recently, Auffinger and Gorski \cite{auffinger2023} demonstrated a converse result, revealing that a Carnot-Carath\'eodory metric on the associated graded nilpotent Lie group serves as the scaling limit for certain FPP models on a Cayley graph under specified conditions. Broadening the investigation, Cantrell and Furman \cite{cantrell2017} explored the limiting shape for subadditive random processes on groups of polynomial growth, focusing on a class of processes satisfying an almost-surely bi-Lipschitz condition. From a probabilistic standpoint, there is significant interest in relaxing this bi-Lipschitz condition, which requires random variables to be in $L^\infty$. Here, we modify this hypothesis by replacing it with conditions of at least and at most linear growth. These new conditions widen the range of distributions of interest and improve the applicability to random processes within $L^p$ spaces for finite $p$. The implications and applicability of this new result are illustrated through examples presented at the end of the chapter. Notably, we enhance our previous result from \cite{coletti2021} on a limiting shape theorem obtained for the Frog Model, now extended to a broader class of non-abelian groups.

Addressing this challenge is primarily approached through the utilization of techniques from metric geometry and geometric group theory.
The existence of the limiting shape can be viewed as an extension of Pansu's theorem to random metrics. The primary strategy involves considering the subadditive cocycle determining a pseudo-quasi-random metric, with the standard case on $\Z^D$ and $\R^D$ extensively covered in the literature (see, for instance, \cite{bjoerklund2010, boivin1990}).

We describe the process and the obtained theorem below, more detailed definitions can be found in the next section.

\subsection*{Basic description and main results}

Let $(\Om,\F,\p)$ be a probability space and $(\Gr,.)$ a finitely generated group with polynomial growth rate. Set  $\tht: \Gr \curvearrowright (\Om,\F,\p)$ to be a $\p$-preserving (p.m.p.) 
ergodic group action. Consider the family $\{c(x)\}_{x \in \Gr}$ of non-negative random variables such that, $\p$-a.s.,
\begin{equation} \label{eq:subadditivity}
    {c}(xy) \leq c(y) + c(x)\circ\tht_y
\end{equation}

Write $c(x,\omega)$ for $c(x)(\omega)$ and let $z\cdot\omega:=\tht_z (\omega)$. Recall that function $c:\Gr\times\Om\to \R_{\geq 0}$ satisfying \eqref{eq:subadditivity} is referred to as a \emph{subadditive cocycle} (see \cref{sec:subadditive.cocycles}). Once given a subadditive cocycle $c$, there is a correspondent random pseudo-quasi metric $d_\omega$ defined by
\[
    d_{z\cdot\omega}(x,y):= \big(c(yx^{-1})\circ\tht_{x}\big)(z\cdot \omega),
\]
which is $\Gr$-right equivariant, \emph{i.e.}, for all $x,y,z \in \Gr$, and for every $ \omega \in \Om$,
\[
    d_\omega(x,y) = d_{z\cdot\omega}(xz^{-1},yz^{-1}).
\]
The correspondence is one-to-one since given a $\Gr$-right equivariant random pseudoquasimetric $d_\omega$, one can easily verify that
\begin{equation} \label{eq:cocycle.metric}
    c(x,\omega):= d_\omega(e, x)
\end{equation}
is a subadditive cocycle.

To avoid dealing with unnecessary technicalities, we initially consider $\Gr$ as a group of polynomial growth, which is nilpotent and torsion-free. Later, we address the more general case where $\Gr$ is virtually nilpotent. The essential definitions and notation are introduced as we proceed with the text. The group will be associated with a finite symmetric generating set $S \subseteq \Gr$. We write $\|-\|_S$ and $d_S$ for a word length and a word metric, respectively. The following conditions will be needed throughout the paper. We assume the existence of $\upbeta>0$ and $\upkappa>1$ such that, for all $x \in \Gr$,%
\begin{equation} \label{all} \tag{\scshape i}
    \p\big( c(x) \geq  t \big) \leq {g}(t) \quad \text{for all } t >\upbeta\|x\|_S
\end{equation}
where $g(t) \in \mathcal{O}\left(1/t^{2D+\upkappa}\right)$ as $t \uparrow +\infty$. 

Let $[\Gr,\Gr]$ be the commutator subgroup of $\Gr$ and set $\|x\|_S^{\ab} := \inf_{y \in x[\Gr,\Gr]}\|y\|_S$. Suppose that there exists $a >0$ such that, for all $x \in \Gr \setminus [\Gr,\Gr]$ there is a sequence $\{n_j\}_{j\in \N}$ of positive integers depending on $x[\Gr,\Gr]$ with $\lim_{j \uparrow + \infty} n_j = +\infty$ and, for all $y \in x[\Gr,\Gr]$ and every $j \in \N$,
\begin{equation} \label{aml} \tag{\scshape ii}
    a \|y^{n_j}\|_S^{\ab} \leq \E\left[c\left(y^{n_j}\right)\right].
\end{equation}
We say that the process grows \emph{at least linearly} when condition \eqref{all} is satisfied. Condition \eqref{aml} provides a lower bound for the norm of the rescaled process $\phi$, which will be defined later.

To obtain the asymptotic result, we will introduce an \textit{innerness assumption}. Specifically, for each $\upvarepsilon>0$, we require the existence of a finite generating set $F(\upvarepsilon) \subseteq \Gr\setminus[\Gr,\Gr]$ such that, for $\p$-\textit{a.s.} $\omega \in \Om$ and for every $x \in \Gr$, we can write $x = z_nz_{n-1}\dots z_1$  with $z_n,z_{n-1},\dots, z_1 \in F(\upvarepsilon)$ satisfying
\begin{equation} \label{innerness} \tag{\scshape iii}
    \sum_{i=1}^n c(z_i,{z_{i-1}\dots z_1}\cdot\omega) \leq (1+\upvarepsilon)c(x,\omega).
\end{equation}
When considering First-Passage Percolation models where $S \subseteq \Gr \setminus [\Gr,\Gr]$, condition \eqref{innerness} is automatically fulfilled (see \cref{sec:fpp}). Additionally, in the case where $\Gr$ is abelian, we can eliminate the need for hypothesis \eqref{innerness} in the main theorem altogether.

\begin{theorem}[Limiting Shape for Torsion-Free Nilpotent Groups] \label{shape.thm}
    Let $(\Gr,.)$ be a torsion-free nilpotent finitely generated group with polynomial growth rate $D \geq 1$ and torsion-free abelianization. Consider $c:\Gr\times\Om \to \R_{\geq0}$ to be a subadditive cocycle associated with $d_\omega$ and a p.m.p. ergodic group action $\tht$. 
    
    Suppose that conditions \eqref{all}, \eqref{aml}, and \eqref{innerness} are satisfied for a finite symmetric generating set $S \subseteq \Gr$. Then
    \begin{equation} \label{eq:GH:conv.THM}
        \quad \quad \left(\Gr,\frac{1}{n}d_\omega,e\right) \GHto \left(G_\infty,d_\phi,\mathlcal{e}\right) \quad\quad \p\text{-a.s.}
    \end{equation}
    where $G_\infty$ is a simply connected graded Lie group, and $d_\phi$ is a quasimetric homogeneous with respect to a family of homotheties $\{\delta_t\}_{t>0}$. Moreover, $d_\phi$ is bi-Lipschitz equivalent to $d_\infty$ on $G_\infty$.
    
     In addition, if $\Gr$ is abelian, then \eqref{eq:GH:conv.THM} remains true even when condition \eqref{innerness} is not valid.
\end{theorem}

The limit space $G_\infty$ is also known as a Carnot group and $d_\infty$ coincides with the Carnot-Carath\'eodory metric obtained by the asymptotic cone of $\Gr$ as the limit of $\frac{1}{n}d_S$. More details about its construction and properties can be found in \cref{basic:def,sec:norms.and.mean} along with the definitions of $\delta_t$ and $d_\phi$. The usage of the pointed Gromov-Hausdorff convergence arises naturally from its correspondence with geometric group theory.

Let now $(\Gr,.)$ be a finitely generated group with polynomial growth rate. Gromov's Theorem \cite{gromov1981} establishes the equivalence of polynomial growth and virtual nilpotency in finitely generated groups. Then there exists a normal nilpotent subgroup $N \unlhd \Gr$ with finite index $\kappa:=[\Gr:N] < + \infty$. Set $\tor N$ to be the torsion subgroup of $N$ and define \[\Gr' :=N/\tor N.\]

Pansu \cite{pansu1983} showed that $\Gr$ and $\Gr'$ share the same asymptotic cone. Let us fix $z_{(j)}$ as a representative of the coset $N_{(j)}= z_{(j)}N$ such that $\Gr = \bigcup_{j=1}^{\kappa} N_{(j)}$. Consider $z_{(j)}=e$ when $N_{(j)}=N$.  Set $\uppi_N :\Gr \to N$ to be given by $\uppi_N( x ) = z_{(j)}^{-1}x$ for $x \in N_{(j)}$. Define now $\llbracket - \rrbracket : \Gr \to \Gr'$ to be given by
\[\llbracket x \rrbracket := \uppi_N( x).\tor N.\]

To refine the first main theorem, let us introduce some new conditions. Suppose that there exists $a >0$ such that, for all $x \in \Gr$ there is a sequence $\{n_j\}_{j\in \N}$ of positive integers depending on $\llbracket x \rrbracket. [\Gr',\Gr']$ with $n_j \uparrow +\infty$ as $j \uparrow + \infty$,
\begin{equation} \label{aml2} \tag{\scshape ii$^\prime$}
    a \|x^{n_j}\|_S \leq \E\left[c\left(x^{n_j}\right)\right].
\end{equation}

Let $c':\Gr'\times\Om \to \R_{\geq 0}$ by 
\begin{equation} \label{eq:def.c.prime}
    c'\big(\llbracket x \rrbracket\big) := \max_{\substack{y \in \llbracket x \rrbracket\\ z \in \tor N}} c(y)\circ\tht_z.
\end{equation}

Fix, for each $\llbracket x\rrbracket \in \Gr'$, a $\upupsilon_x \in \llbracket x \rrbracket$ and consider $\theta:\Gr' \curvearrowright (\Om,\F,\p)$ given by $\theta_{\llbracket x\rrbracket} \equiv \tht_{\upupsilon_x}$ and $\theta_z(\omega) = z\ast\omega$ (see \cref{sec:virt.nilpotent.proofs} and \cref{rmk:c.prime} for a detailed discussion). We consider a similar \textit{innerness assumption} to replace \eqref{innerness}. Suppose that, for each $\upvarepsilon>0$, there exists a finite $F(\upvarepsilon) \subseteq N\setminus[N,N]$ which is a generating set of $\Gr'$ such that, $\p$-\textit{a.s.}, for every $x \in \Gr$, one can write $\llbracket x \rrbracket = z_nz_{n-1}\dots z_1$  with $z_n,z_{n-1},\dots, z_1 \in F(\upvarepsilon)$ satisfying
\begin{equation} \label{innerness2} \tag{\scshape iii$^\prime$}
    \sum_{i=1}^n c'(z_i,~{z_{i-1}\dots z_1}\ast\omega) \leq (1+\upvarepsilon)c'\big(\llbracket x \rrbracket,~\omega\big).
\end{equation}

Similar to \eqref{innerness}, First-Passage Percolation models satisfy \eqref{innerness2} under specific conditions. In the case where $\Gr=N$ is nilpotent, it suffices to have $S \subseteq N \setminus \big([N,N] \cup \tor N \big)$ for an FPP model to satisfy \eqref{innerness2}. The virtually nilpotent case is treated separately in \cref{sec:additional.FPP} with additional conditions imposed on $\llbracket S \rrbracket$ and $\tht$. Moreover, when $\Gr'$ is abelian, hypothesis \eqref{innerness2} is not required to verify the theorem below.

\begin{theorem}[Limiting Shape for Groups with Polynomial Growth] \label{thm:shape.polynomial}
    Let $(\Gr,.)$ be a finitely generated group with polynomial growth rate $D \geq 1$ and $\Gr'/[\Gr',\Gr']$ torsion-free. Consider $c:\Gr\times\Om \to \R_{\geq0}$ to be a subadditive cocycle associated with $d_\omega$ and a p.m.p. ergodic group action $\tht$. 
    
    Suppose that conditions \eqref{all}, \eqref{aml2}, and \eqref{innerness2} are satisfied for a finite symmetric generating set $S \subseteq \Gr$ such that $\llbracket S \rrbracket$ generates $\Gr'$. Then
    \begin{equation} \label{eq:GH:conv.THM2}
        \quad \quad \left(\Gr,\frac{1}{n}d_\omega,e\right) \GHto \left(G_\infty,d_\phi,\mathlcal{e}\right) \quad\quad \p\text{-a.s.}
    \end{equation}
    where $G_\infty$ is a simply connected graded Lie group, and $d_\phi$ is a quasimetric homogeneous with respect to a family of homotheties $\{\delta_t\}_{t>0}$. Moreover, $d_\phi$ is bi-Lipschitz equivalent to $d_\infty$ on $G_\infty$.
    
     Furthermore, if $\Gr'$ is abelian, then \eqref{eq:GH:conv.THM2} remains true even when condition \eqref{innerness2} is not valid.
\end{theorem}

The primary technique employed in this work involves the approximation of admissible curves through the use of polygonal paths and ergodic theory. In \cref{sec:limiting.shape}, we introduce and delve into these tools, presenting their application in proving the theorems and a corollary for FPP models. \cref{sec:examples} showcases examples dedicated to illustrating the applicability of the theorems.

\section{Preparatory and Intermediate Results}

This section is primarily dedicated to the establishment of a norm within $G_\infty$, a critical step for defining the subsequent limiting shape. The formulation of this norm draws upon insights from subadditive ergodic theorems, coupled with properties highlighted in \cref{basic:def}. This approach enables us to delve into the asymptotic behavior of sequences within the group through the examination of expected values associated with subadditive cocycles, a topic explored further in the subsequent sections. The convergence is not directly established as an uniform convergence in $\R^D$ because of the constraints imposed by admissible curves.

From this point until the proof of the first theorem in \cref{sec:main.proofs}, let us once again regard $\Gr$ as a finitely generated torsion-free nilpotent group.

\subsection{Establishing a Candidate for the Limiting Shape} \label{sec:norms.and.mean}

Set $\thickbar c(x) := \E[c(x)]$, due to the subadditivity of the cocycle
\begin{equation*}
\thickbar{c}(xy) \leq \thickbar{c}(y) + \thickbar{c}(x),
\end{equation*}
for all $x,y \in \Gr$. Thus $\thickbar{c}(x) \leq b\|x\|_S$
with $b= \max_{s \in S}\big\{\thickbar{c}(s)\big\}$. It follows from \eqref{aml} that there exists a subsequence of $c(x^n)/n$ such that $c(x^{n_j})/n_j \geq a \|x\|_S^{\ab}$ $\p$-\textit{a.s.} for sufficiently large $j$.

Recall that $\Gr^{\ab}= \Gr/[\Gr,\Gr]$ and consider $x^{\ab} = x[\Gr,\Gr]$, To simplify notation, we also use $x^{\ab}$ interchangeably with $(\pi_\infty \circ L \circ \log)(x)$ when it is clear from the context. Let
\[
    \|x\|_S^{\ab} := \inf_{y \in x[\Gr,\Gr]} \|y\|_S.
\]

Since $\|-\|_S^{\ab}$ is discrete, there exists $y \in x[\Gr,\Gr]$ such that $\|x\|_S^{\ab} = \|y\|_S$. Hence, for all $x,y\in \Gr$, there exist $x',x'' \in x[\Gr,\Gr]$ and $y',y'' \in y[\Gr,\Gr]$ such that
\[\|xy\|_S^{\ab} = \|x'y'\|_S, \quad \|x\|_S^{\ab} = \|x''\|_S, \text{ and} \quad \|y\|_S^{\ab}= \|y''\|_S;\]
which implies the subadditivity
\[\|xy\|_S^{\ab} = \|x'y'\|_S \leq \|x''y''\|_S \leq \|x''\|_S + \|y''\|_S =\|x\|_S^{\ab} + \|y\|_S^{\ab}.\]

Now, regarding $\|x\|_S^{\ab} = 0$ whenever $x \in [\Gr,\Gr]$, one has for all $x \in [\Gr,\Gr]$ and $y \in \Gr$, $\|xy\|_S^{\ab} = \|y\|_S^{\ab}$. Let $y = s_m \dots s_{j+1}s_js_{j-1} \dots s_1$ with $s_j\in[\Gr,\Gr]$. Since $[\Gr,\Gr]$ is a normal subgroup of $\Gr$, $\Bar{s}_j= (s_{j-1} \dots s_1)^{-1} s_j(s_{j-1} \dots s_1) \in [\Gr, \Gr]$ is such that $y = s_m \dots s_{j+1}s_{j-1} \dots s_1\Bar{s}_j$. Hence
\[\|y\|_S^{\ab} = \|s_m \dots s_{j+1}s_{j-1} \dots s_1\|_S^{\ab}.\]

Therefore, $\|y\|_S^{\ab} = \|y\|_S=m$ if, and only if, there exists $\{s_i\}_{i=1}^m \subseteq S\setminus[\Gr,\Gr]$ such that $y=s_m \dots s_1$. Observe that $\Gr^{\ab}$ is a topological lattice of $G^{\ab}$ and $G^{\ab} \simeq \Gr^{\ab} \otimes \R \simeq \R^{\operatorname{dim}\mathfrak{v}_1} \simeq \g^{\ab}$. Let $\|-\|$ be an Euclidean norm on $G^{\ab}$ and fix $\thickbar{a},\thickbar{b} >0$ such that
\[\thickbar{a}:= \min\big\{\|s[\Gr,\Gr]\| : s \in S \setminus[\Gr,\Gr]\big\}, \text{ and} \quad \thickbar{b}:= \max\big\{\|s[\Gr,\Gr]\| : s \in S \setminus[\Gr,\Gr]\big\},\]

Due to the properties of a normed vector space, one has, for all $x \in \Gr$,
\begin{equation} \label{eq:biLipschitz.abelian.norm}
    \thickbar{a} \|x\|_S^{\ab} \leq \big\|x[\Gr,\Gr]\big\| \leq \thickbar{b}\|x\|_S^{\ab}.
\end{equation}

Set $f : \Gr^{\ab} \to \R_{\geq 0}$ to be given by
\begin{equation*}
    f(x^{\ab}) = \inf\{\thickbar{c}(y): y\in x[\Gamma,\Gamma]\}.
\end{equation*}
It follows immediately from the subadditivity of $\thickbar{c}$ and the definition of $f$ that
\[
    f(x^{\ab}y^{\ab}) \leq f(x^{\ab})+ f(y^{\ab}).
\]

We are now able to state the following a subadditive ergodic theorem obtained by Austin \cite{austin2016} and improved by Cantrell and Furman \cite{cantrell2017}.

\begin{proposition}[Subadditive Ergodic Theorem] \label{prop:subadditive.ergodic.thm}
    Let the subadditive cocycle ${c:\Gr\times\Om \to \R_{\geq0}}$ associated with a p.m.p. ergodic group action $\tht:\Gr\curvearrowright\Om$ be such that $c(x) \in L^1(\Om,\F,\p)$ for all $x \in \Gr$. Then there exists a unique homogeneous subadditive function $\phi: \mathfrak{g}_ \infty^{\ab} \to \R_{\geq 0}$ such that, for every $x \in \Gr$, 
    \[
        \lim_{n \uparrow +\infty}\frac{1}{n}c(x^n)= \phi(x^{\ab}) \quad \p\text{-a.s. and in }L^1.
    \]
    Moreover, $\phi$ is given by
    \begin{equation} \label{eq:def_phi}
        \phi(x^{\ab}) = \lim_{n\uparrow+\infty} \frac{1}{n}f\left(n \cdot x^{\ab}\right) = \inf_{n \geq 1}  \frac{1}{n}f\left(n \cdot x^{\ab}\right).
    \end{equation}
\end{proposition}

\begin{remark}
    The function $\phi$ obtained above is naturally associated with the abelianized space considering the well-known fact of the convergence of $\frac{1}{n}\bull x^n$ to the projection of $x$ onto a subspace isomorphic to $\g_\infty^{\ab}$. It will allow us to measure distances in $G_\infty$ with $\ell_\phi$ by considering the rescaling of the subadditive cocyle.
\end{remark}

The bi-Lipschitz property established in the following lemma is crucial for the main results.

\begin{lemma} \label{lm:phi.bi-Lipschitz}
    Let $c: \Gr\times\Om\to\R_{\geq 0}$ be a subadditive cocycle under the assumptions of \cref{prop:subadditive.ergodic.thm}. Set $\phi$ as in \eqref{eq:def_phi}. Consider $c$ satisfying \eqref{all} and \eqref{aml}. Then there exist $a',b'>0$ such that, for all $x \in \Gr$,
    \[
    a'\|x^{\ab}\| \leq \phi(x^{\ab}) \leq b'\|x^{\ab}\|.
    \]
\end{lemma}
\begin{proof}
    Observe that condition \eqref{all} implies $c \in L^1(\Om,\F,\p)$. Consider $\thickbar{a},\thickbar{b}> 0$ as in \eqref{eq:biLipschitz.abelian.norm} and fix $a':= a/\thickbar{b}$. By \eqref{aml}, one has
    \[
        f(n_j\cdot x^{\ab}) = \inf_{y \in x[\Gr,\Gr]} \thickbar{c}(y) \geq a \|x^{n_j}\|_S^{\ab} \geq a'n_j\|x^{\ab}\|.
    \]

    We know by \cref{prop:subadditive.ergodic.thm} that $\phi$ exists and \[\phi(x^{\ab}) = \inf_{n\in\N}\frac{1}{n}\thickbar{c}(x^{\ab}) = \lim_{j\uparrow+\infty}
    \frac{1}{n_j}f(n_j\cdot x^{\ab}) \geq a'\|x^{\ab}\|.\]

    It follows from \eqref{all} and subaditivity that there exists $b>0$ such that, for all $x \in \Gr$,
    \[
        \thickbar{c}(x) \leq b \|x\|_S. 
    \]

    Let us fix $b':= b/\thickbar{a}$, then
    \[
        \phi(x^{\ab}) = \inf_{n\in\N}\frac{1}{n} \inf_{y\in x^n[\Gr,\Gr]}\thickbar{c}(y) \leq b \inf_{n\in\N}\frac{1}{n} \|x^n\|_S^{\ab} \leq b'\|x^{\ab}\|,
    \]
    which is our assertion.
\end{proof}

\begin{remark}
    The Subadditive Ergodic Theorem guarantees the $\p$-\textit{a.s.} existence of the $\lim_{n \uparrow +\infty}c(x^n)/n$. By combining this fact with previous assertions and the $L^1$ convergence, we obtain the existence of $0 <a \leq b<+\infty$ such that
    \begin{equation} \label{eq:mean.double.bound}
        a \|x\|_S^{\ab} \leq \thickbar{c}(x) \leq b \|x\|_S.
    \end{equation}
    
    Furthermore, one has from \eqref{eq:mean.double.bound} that $a \|x\|^{\ab}_S \leq \phi(x^{\ab})\leq  b \|x\|_S$ for all $x \in \Gr$. Since there exists $y \in x[\Gr,\Gr]$ with $\|x\|_S^{\ab} = \|y\|_S$ and $x^{\ab} = y^{\ab}$, one has by \eqref{eq:biLipschitz.abelian.norm}
    \[\frac{ ~ a ~ }{\thickbar{b}}\|x^{\ab}\| \leq a \|x\|^{\ab}_S \leq \phi(x^{\ab})\leq  b \|x\|_S^{\ab} \leq \frac{~ b ~}{\thickbar{a}}\|x^{\ab}\|.\]
\end{remark}

Recall the definition of $d_\phi$ in \eqref{eq:def.d.phi}. Therefore, there is a bi-Lipschitz relation between $d_\infty$ and $d_\phi$. We now define $\Phi: G_\infty \to [0,+\infty)$ by 
\[\Phi(\mathlcal{g}) :=d_\phi(\mathlcal{e},\mathlcal{g}).\]

\subsection{Approximation of Admissible Curves along Polygonal Paths}

The proof strategy for the main theorem involves approximating geodesic curves with polygonal paths. Throughout the following discussion, we assume that $c$ is a subadditive cocycle, and $\Gr$ is finitely generated by the symmetric set $S$ with polynomial growth rate $D \geq 1$. To set the stage, we begin by stating Proposition 3.1 from \cite{cantrell2017}.

\begin{proposition} \label{polygonal.path} Let $\upgamma:[0,1] \to G_\infty$ be a Lipschitz curve and let $\Hat{\varepsilon} \in (0, 1)$. Then there exists $k_0=k_0(\upgamma,\Hat{\varepsilon})>0$ so that one can find, for all $k>k_0$, $\{y_j\}_{j=1}^k \subseteq \Gr$, $p>0$ and $n_0>0$ such that,  for all $n > n_0$,
\[
    \sum_{j=1}^kd_\infty\left(\frac{1}{np}\bull y_j^n y_{j-1}^n\dots y_1^n,\upgamma\left(\frac{j}{k}\right) \right) < \Hat{\varepsilon}
\]
 Moreover, for $\upphi:\g_\infty^{\ab} \to \R_{\geq0}$ a subadditive homogeneous function bi-Lipschitz with respect to $\|-\|$, one has that
\[
\left|\frac{1}{p} \big(\upphi(y_k^{\text{ab}})+ \cdots + \upphi (y_1^{\text{ab}})\big)-\ell_\upphi(\upgamma)\right| < \Hat{\varepsilon}.
\]
\end{proposition}

The approximation technique outlined in the upcoming proposition will be utilized in the subsequent subsections. It extends the guarantees of the subadditive ergodic theorem for the decomposition of polygonal paths under certain properties. In what follows, we write \(t \vee t' := \max\{t, t'\}\) and \(t \wedge t' := \min\{t, t'\}\).

\begin{proposition} \label{prop:ergodic.thm.path}
Let $\Gr$ be a torsion-free nilpotent finitely generated group with torsion-free abelianization. Consider $c:\Gr\times\Om \to\R_{\geq 0}$ a subadditive cocycle associated with an ergodic group action $\tht$ satisfying \eqref{all}. Then for all integer $j>1$ and  $\{y_i\}_{i=1}^j \subseteq \Gr$,
\[
	\lim_{n \uparrow \infty}\frac{1}{n}c(y_j^n) \circ\tht_{{y_{j-1}^n \dots y_1^n}} = \phi(y_j^{\text{ab}}) \quad \p-a.s.
\]
In particular, if we let $\Check{\varepsilon} \in (0,1)$, then there exists, $\p$-a.s., a random $M_0>0$ depending on $\Check{\varepsilon}$ and on $\sum_{i=1}^j\|y_i^{\ab}\|$ such that, for all $n > M_0$,
\[\left| \frac{1}{n}c(y_j^n, {y_{j-1}^n \dots y_1^n} \cdot \omega) ~- ~\phi(y_j^{\text{ab}})\right| < \Check{\varepsilon}.\]
\end{proposition}

Before proving \cref{prop:ergodic.thm.path} we show the following lemma.

\begin{lemma} \label{lm:local.bound}
    Let $\varepsilon \in (0,1)$ and consider a subadditive cocycle $c$ that satisfies condition \eqref{all}. There exists, $\p$-a.s., $M_1>0$ such that if $\{x_n\}_{n\in \N}$, $\{y_n\}_{n\in \N}$, $\{u_n\}_{n\in\N}$, and $\{v_n\}_{n \in \N}$ are sequences in $\Gamma$ satisfying, for a $n_0=n_0(\varepsilon) \in \N$ and all $n>n_0$:
    \begin{enumerate}[(i)]
        \item There exist elements $\mathlcal{x}, \mathlcal{u} \in G_\infty$ and $\mathtt{c}_{\mathlcal{x},\mathlcal{u}}>0$ such that
        \[
            d_\infty\left(\frac{1}{n}\bull x_n,\mathlcal{x}\right)< \varepsilon,  \quad d_\infty\left(\frac{1}{n}\bull u_n,\mathlcal{u}\right)< \varepsilon,
        \]
        and $ \
        \|x_n\|_S, \|u_n\|_S < \mathtt{c}_{\mathlcal{x},\mathlcal{u}}\cdot n$;
        
        \item $d_S(u_n,v_n)\leq n\varepsilon \quad \text{and}\quad d_S(x_nu_n,y_nv_n) \leq n\varepsilon$.
    \end{enumerate}
    Then
    \begin{equation*} 
        \big|c(x_n) \circ \tht_{u_n} - c(y_n)\circ\tht_{v_n}\big| < 2 \upbeta n \varepsilon
    \end{equation*}
    for all $n > \max\big\{n_0, M_1, \exp\big((2 \mathtt{c}_{\mathlcal{x},\mathlcal{u}} + 3)^D\big)\big\}$.
\end{lemma}

\begin{proof}
    Fix $a_n :=v_nu_n^{-1}$ and $b_n := y_nv_n(x_nu_n)^{-1}$. Then $\|a_n\|_S= \|a_n^{-1}\|_S = d_S(u_n,v_n)$ and $\|b_n\|_S= \|b_n^{-1}\|_S = d_S(x_nu_n,y_nv_n)$.
    Observe that $y_n = b_n x_n a_n^{-1}$ and $x_n = b_n^{-1}y_na_n$. We thus obtain the $\p$-almost surely inequalities below:
    \begin{align}
        c(y_n)\circ \tht_{v_n} &\leq c(b_n)\circ\tht_{x_nu_n} + c(x_n)\circ \tht_{u_n} + c(a_n^{-1})\circ \tht_{v_n} \label{cocycle:lm:p1}\\
        c(x_n)\circ \tht_{u_n} &\leq c(b_n^{-1})\circ\tht_{y_nv_n} + c(y_n)\circ\tht_{v_n} + c(a_n)\circ\tht_{u_n}. \label{cocycle:lm:p2}
    \end{align}
    Observe now that, by items \textit{(i)} and \textit{(ii)}, for $n > n_0(\varepsilon)$,
    \[
        x_n u_n,y_n v_n, u_n, v_n \in B_S\left(e, 2\mathtt{c}_{\mathlcal{x},\mathlcal{u}}\cdot n + 3n\varepsilon\right) .
    \]

    Hence, by combining \eqref{cocycle:lm:p1} and \eqref{cocycle:lm:p2},
    
    \begin{equation} \label{eq:difference.cocycle.parallel}
            |c(x_n)\circ \tht_{u_v}- c(y_n)\circ \tht_{v_n}| \leq \hspace{10pt}2 \hspace{-15pt} \sup_{\substack{\|y\|_S \leq n\varepsilon \\ \|z\|_S \leq \sqrt[D]{\log(n)}n}} \{c(y)\circ\tht_{z}\}
    \end{equation}
    for all $n > \max\big\{n_0, \exp\big((2 \mathtt{c}_{\mathlcal{x},\mathlcal{u}} + 3)^D\big)\big\}$. It follows from \eqref{all} that there exists $\mathtt{C}>0$ such that
    \begin{equation} \label{prob.sup.parallel.substitute}
    \p\left( \sup_{\substack{\|y\|_S \leq n\varepsilon \\ \|z\|_S \leq \sqrt[D]{\log(n)}n}} \{c(y)\circ\tht_{z}\} \geq \upbeta n\varepsilon  \right) \leq \mathtt{C} {n^{2D}}\log(n)g(\upbeta\varepsilon n) \in \mathcal{O}(\log(n)/n^\upkappa),
    \end{equation}
    for $n >\max\big\{n_0, \exp\big((\|\mathlcal{x}\|_\infty + \|\mathlcal{u}\|_\infty + 3)^D\big)\big\}$. Since $\sum_{n=1}^{+\infty}\frac{\log(n)}{n^\upkappa} = -\upzeta'(\upkappa) < +\infty$ for $\upkappa>1$ where $\upzeta'$ is the derivative of the Riemann zeta function, the proof is completed by applying Borel-Cantelli Lemma to  \eqref{eq:difference.cocycle.parallel} and \eqref{prob.sup.parallel.substitute}.
\end{proof}

\begin{remark}
    If $\{x_n\}_{n\in\N},\{u_n\}_{n\in\N} \subseteq \Gr$ are such that $\lim_{n\uparrow+\infty}\frac{1}{n}\bull x_n = \mathlcal{x}$ and $\lim_{n\uparrow+\infty}\frac{1}{n}\bull u_n = \mathlcal{u}$ in $(G_\infty,d_\infty)$, then item (i) of \cref{lm:local.bound} is immediately satisfied (see \cref{sec:rescaled.dist}).
\end{remark}

Using the lemma above, the \cref{prop:ergodic.thm.path} becomes a straightforward extension of Theorem 3.3 of \cite{cantrell2017}. The result can be verified by replacing the Parallelogram inequality with \cref{lm:local.bound}. To be self-contained, let us first define, for each $E \in \F$,  $\omega \in \Om$, $x\in \Gr$, $\xi>0$, and $n\in\N$,
\[\mathsf{N}_{x,n}^\xi(E,\omega) := \#\big\{n' \in\{0, 1, \dots, 
\lceil\xi n\rceil-1\} \colon \tht_{x^{n-n'}}(\omega) \in E\big\}.\]
Set
\begin{align*}
&\Xi^\star(E,x,\xi) := \left\{ \omega \in \Om \colon \liminf_{n\uparrow+\infty}\frac{\mathsf{N}_{x,n}^\xi(E,\omega)}{\xi n}>0 \right\}, \quad\text{and}\\ &\Xi_m^\star(E,x,\xi) := \left\{ \omega \in \Om \colon \forall n\ge m\left(\frac{\mathsf{N}_{x,n}^\xi(E,\omega)}{\xi n}>0\right) \right\}.
\end{align*}

We now state Lemma 3.6 of \cite{cantrell2017} without proof before the proving \cref{prop:ergodic.thm.path}.

\begin{lemma} \label{lm:measure.approx}
    Let $x\in\Gr$, $\xi >0$, and $E \in \F$. Then, for all $\varepsilon\in(0,1)$, there is $m_0>0$ such that, for $m>m_o$,
    \[\p\big(\Xi^\star(E,x,\xi)\big) \geq \p(E)\quad\text{and}\quad \p\big(\Xi_m^\star(E,x,\xi)\big) > \p\big(\Xi^\star(E,x,\xi)\big)-\varepsilon.\]
\end{lemma}

We proceed below with the proof of ergodic subadditive approximation via polygonal paths.

\begin{proof}[Proof of \cref{prop:ergodic.thm.path}]
    Consider $\varepsilon' \in (0,1)$ and  $\{y_i\}_{i=1}^j \subseteq \Gr$ fixed. Let $\xi>0$ and $\thickbar{n}\in\N$ be given by \cref{lm:small.perturbations} for $\varepsilon=\varepsilon'$. Set $\eta \in (0,\frac{1}{2j})$ and $m\in \N$ sufficiently large so that, for each $i \in \{1, \dots, j\}$, one has by \cref{prop:subadditive.ergodic.thm},
    \[\mathcal{X}_i:=\left\{\omega\in\Om \colon \forall n>m \left(\left\vert\frac{1}{n} c(y_i^n,\omega) - \phi (y_i^{\ab})\right\vert< \Check{\varepsilon}\right)\right\} \ \text{ and } \ \p(\mathcal{X}_i) > 1- \eta.\]

    Fix $\mathcal{Y}_j := \mathcal{X}_j$ and define inductively $\mathcal{Y}_{i-1} := \mathcal{X}_{i-1} \cap \Xi_{m_i}^\ast(\mathcal{Y}_i,y_i,\xi)$ so that, for each $i \in \{2, \dots, j\}$,  $m_i \in \N$ is given by \cref{lm:measure.approx} satisfying \[\p\big(\Xi_{m_i}^\ast(\mathcal{Y}_i,y_i, \xi)\big) \geq \p(\mathcal{Y}_i)- \eta.\]

    Therefore, 
    \[\p(\mathcal{Y}_1) > \p(\mathcal{Y}_2) - 2\eta > \cdots > \p(\mathcal{Y}_j) - 2(j-1) \eta > 1 - (2j-1) \eta. \]

    Let now $\Check{m} := \max \{m, m_1, \dots, m_j\}$. Thence, for all $\varpi_i \in \mathcal{Y}_i$ and every $n>\Check{m}$ with $i \in \{1, \dots, j-1\}$, if $n_i<\xi n$, then $\tht_{y^{n-n_i}}(\varpi_i)= y^{n-n_i} \cdot \varpi_i \in \mathcal{Y}_{i+1}$, and
    \[\left\vert \frac{1}{n} c(y_{i+1}^n, \varpi_i) - \phi(y_{i+1}^{\ab}) \right\vert< \varepsilon'.\]

    It follows that, for all $n>\Check{m}$, there exist non negative integers $n_1, \dots, n_{j-1} < \xi n$ such that, for all $\omega \in \mathcal{Y}_1$,
    \[\left\vert \frac{1}{n} c(y_{j}^n, y_{j-1}^{n-n_{j-1}}\cdots y_1^{n-n_1}\cdot \omega) - \phi(y_{j}^{\ab}) \right\vert< \varepsilon'.\]

    By \cref{lm:small.perturbations}, for each $i \in \{2, \dots, j\}$ and every $n>\thickbar{n}$,
    \begin{eqnarray*}
        d_S\left(y_{i-1}^n\cdots y_1^n, ~y_{i-1}^{n-n_{i-1}}\cdots y_1^{n-n_1}\right)< n \varepsilon' \quad \text{and}\\
        d_S\left(y_i^n y_{i-1}^n\cdots y_1^n, ~y_i^n y_{i-1}^{n-n_{i-1}}\cdots y_1^{n-n_1}\right)< n \varepsilon'.
    \end{eqnarray*}

    Hence, by \cref{lm:conv.seq.Gr,lm:local.bound}, there exists $\Om_j \in \F$ with $\p(\Om_j)=1$, and a random $\thickbar{M}_i\geq \thickbar{n}$ depending on $\varepsilon'$, and $\|y_i^{\ab}y_{i-1}^{\ab}\dots y_1^{\ab}\|_\infty \vee \|y_{i-1}^{\ab}\dots y_1^{\ab}\|_\infty$ such that for all $n> \thickbar{M}_i$ and all $\omega \in \Om_j$, 
    \[\left\vert c(y_i^n, ~y_{i-1}^n \cdots y_1^n \cdot \omega) - c(y_i^n, ~y_{i-1}^{n-n_{i-1}} \cdots y_1^{n-n_1} \cdot \omega) \right\vert< 2 \upbeta n \varepsilon'.\]

    Therefore, for all $\omega \in \Om_j \cap \ \mathcal{Y}_1$, every $n> \check{m} \vee \thickbar{M}_i$ and all $i \in \{2, \dots, j\}$
    \begin{equation} \label{eq:erg.ineq.polygonal}
        \left\vert \frac{1}{n}c(y_i^n, ~y_{i-1}^n \cdots y_1^n \cdot \omega) - \phi(y_i^{\ab}) \right\vert< (2 \upbeta +1) \Check{\varepsilon},
    \end{equation}
    and $\p\left(\Om_j \cap \ \mathcal{Y}_1\right) > 1-(2j-1)\eta$. It suffices to consider $\eta_n \downarrow 0$ replacing $\eta  \in (0, \frac{1}{2j})$ with $\sum_{n \in \N} \eta_n < +\infty$, then there exists, $\p$-a.s.,  $M_0 \geq \check{m} \vee \thickbar{M}_i$ by Borel-Cantelli Lemma such that \eqref{eq:erg.ineq.polygonal} is satisfied for all $n > M_0$, which is our assertion with $i=j$ and $\Check{\varepsilon}= \frac{1}{2\upbeta+1}\varepsilon'$.
\end{proof}

\section{The Limiting Shape} \label{sec:limiting.shape}

In this section, we undertake the task of proving the asymptotic shape theorems by utilizing the tools meticulously developed in preceding sections. The concluding subsection is dedicated to exploring a corollary specifically tailored for FPP models.

We initiate our proof by addressing the case of finitely generated torsion-free nilpotent groups. Subsequently, we extend these results to encompass the virtually nilpotent property.

\subsection{Proof of the First Theorem} \label{sec:main.proofs}

This subsection is dedicated to proving \cref{shape.thm}. Therefore, consider all conditions and notations established in the first main theorem for the subsequent results. For instance, here $\Gr$ is torsion-free nilpotent with torsion-free abelianization. Before turning to the proof of the theorem, let us refine the techniques of approximation as outlined in the upcoming propositions and lemmas.

\begin{proposition} \label{prop:asymptotic.approx}
    Let $\mathlcal{g} \in G_\infty$ and $\epsilon \in(0,1)$. Consider $\{y_j\}_{j=1}^k \subseteq \Gr$ and $p>0$ given by \cref{polygonal.path} for a $d_\infty$-geodesic curve $\upgamma:[0,1] \to G_\infty$ from $\mathlcal{e}$ to $\mathlcal{g}$ and $\Hat{\varepsilon}=\epsilon/2$.
    
    If conditions \eqref{all}, \eqref{aml}, and \eqref{innerness} are satisfied, then there exists, $\p$-a.s., $M_2>0$ depending on $\mathlcal{g}$, $\epsilon$, and $\omega \in \Om$, 
such that, for all $n> M_2$,
\[
    \left|\frac{1}{pn}c(y_k^n \cdots y_1^n) - \ell_\phi(\upgamma)\right| <  \epsilon.
\]
\end{proposition}
\begin{proof}
    Let us write $\mathbbmtt{y}_n:= y_k^n\dots y_1^n$ and consider $n_0>0$ for $\Hat{\varepsilon}=\epsilon/2$
    given by \cref{polygonal.path}. It follows from subadditivity that
    \[
        c(\mathbbmtt{y}_n) \leq  \sum_{j=1}^k c(y_{j}^n) \circ \tht_{y_{j-1}^n \dots y_1^{n}}  \quad \p\text{-a.s.}
    \]
    Then, one has by \cref{prop:ergodic.thm.path} with $\Check{\varepsilon} \le \frac{p}{2k}\epsilon$ that, $\p$-a.s., for all $n>M_0 \vee n_0$,
    \begin{equation} \label{eq:upper.bound.path.cocycle}
        \frac{1}{pn}c(\mathbbmtt{y}_n) \leq \frac{1}{p}\sum_{j=1}^k \phi(y_n^{\ab}) + \frac{\epsilon}{2} < \ell_\phi(\upgamma)+ \epsilon.
    \end{equation}
    Set $\upvarepsilon \in (0,1)$ to be defined later and apply condition \eqref{innerness} to obtain
    \[
        \sum_{j=1}^{k_n}c_{n,j} \leq (1+\upvarepsilon) \frac{1}{pn} c(\mathbbmtt{y}_n)
    \]
    where $c_{n,j} = \frac{1}{pn} c(z_{n,j},{z_{n,j-1}\dots z_{n,1}}\cdot \omega)$ with $z_{m,i} \in F(\upvarepsilon)$. Define a sequence of piecewise $d_\infty$-geodesic curves $\zeta_n$ between each $\frac{1}{pn} \bull z_{n,j} \dots z_{n,1}$ and $\frac{1}{pn} \bull z_{n,j-1} \dots z_{n,1}$ for $j \in \{1, \dots, k_n\}$ such that $z_{n,0}=e$ and  \[
        \zeta_n(\tau_j)= \frac{1}{pn} \bull z_{n,j} \dots z_{n,1} \quad \text{for }\tau_j = \sum_{i=1}^j c_{n,i}\Big/\sum_{i=1}^{k_n}c_{n,i}.
    \]
    
    Set $\mathtt{m}_\upvarepsilon:= \min_{z \in F(\upvarepsilon)}\|z^{\ab}\|>0$. It follows from \cref{lm:phi.bi-Lipschitz} that $\E[ c_{n,j} ] \geq a'\mathtt{m}_\upvarepsilon \frac{1}{pn}$ and, due the the $L^1$ convergence in \cref{prop:subadditive.ergodic.thm}, there exists $n_1>n_0$ such that, for all $n >M_0 \vee n_1$,
    \[
        a' \mathtt{m}_\upvarepsilon\frac{1}{pn}\E[k_n] \leq (1 + \upvarepsilon)~ \ell_\infty(\upgamma) + k\upvarepsilon.
    \]
    
    Fix $\mathtt{C}_{\upgamma,\upvarepsilon} > \frac{2p}{ a'\mathtt{m}_\upvarepsilon}\big((1 + \upvarepsilon) \ell_\infty(\upgamma) + k\upvarepsilon\big)$ so that, for all $n \in \N$, $\E[k_n] \leq \mathtt{C}_{\upgamma,\upvarepsilon} n/2$. By Chernoff bound, $\p(k_n \geq \mathtt{C}_{\upgamma,\upvarepsilon}n) \leq \exp(-2n)$. It then follows from an application of Borel-Cantelli Lemma that, $\p$-a.s., there exists $M_0' \ge M_0$ such that, for every $n>M_0'$,
    \begin{equation} \label{eq:kn.upperbound}
        k_n \leq \mathtt{C}_{\upgamma,\upvarepsilon} n.
    \end{equation}
    Let $\mathtt{M}_\upvarepsilon := \max\limits_{z \in F(\upvarepsilon)} \|1\bull z\|_\infty$ Observe now that, for every $t,t'\in[0,1]$, $\p$-a.s., for  $n > M_0'$,
    \[
        d_\infty\big(\zeta_n(t),\zeta_n(t')\big) \leq \frac{k_n}{pn}\mathtt{M}_\upvarepsilon|t-t'| \leq \frac{1}{p}\mathtt{C}_{\upgamma,\upvarepsilon}\mathtt{M}_\upvarepsilon|t-t'|.
    \]
    
    Hence, one has by Arzelà–Ascoli Theorem that a subsequence of $\zeta_n$ converges uniformly to a Lipschitz curve $\zeta:[0,1]\to G_\infty$ such that $\zeta(0)=\mathlcal{e}$ and $\zeta(1)=\mathlcal{g}$. 
    
    We apply \cref{polygonal.path} once again fot the curve $\zeta$ with $\Hat{\varepsilon} = \upvarepsilon/2$ to obtain $p'>0$, $\{w_i\}_{i=1}^{k'} \subseteq \Gr$, $t_n=\lfloor np/p' \rfloor$, and $n_2>0$ such that, for all $n>n_2$
    \[
        \sum_{i=1}^{k'}d_\infty\left(\frac{1}{p' t_n} \bull w_i^{t_n} \dots w_1^{t_n}, \zeta\left(\frac{i}{k'}\right)\right) < \frac{\upvarepsilon}{2}.
    \]
    
    Recall that $F(\upvarepsilon)$ is a generating set of $\Gr$ and $\mathcal{C}\big(\Gr,F(\upvarepsilon)\big)$ shares the polynomial growth rate of $\mathcal{C}(\Gr,S)$. Then there exists $\mathtt{C}'>0$ such that, for a given $\varepsilon'\in(0,1)$,
    \begin{align*}
        \p \left( \sup_{z \in F(\upvarepsilon)}\Big\{c(z)\circ\tht_{z'}:z' \in B_{F(\upvarepsilon)}(e, \mathtt{C}_{\upgamma,\upvarepsilon}n)\Big\} \geq \varepsilon' n\right) \leq \mathtt{C}' |F(\upvarepsilon)| n^D g(\varepsilon' n) \\ \in \mathcal{O}_{\varepsilon'}(1/n^\upkappa)
    \end{align*}
    as $n \uparrow +\infty$.
    
    It thus follows by an application of Borel-Cantelli Lemma and by \eqref{eq:kn.upperbound} that for all $\varepsilon'\in(0,1)$, there exist, $\p$-a.s., $M_0'' \ge M_0'$ and a subdivision function $\mathsf{d}_n:\{0,1, \dots, k'\} \to \{0,1, \dots,k_n\}$ with $\mathsf{d}_n(0)=0 < \mathsf{d}_n(1) < \cdots < \mathsf{d}_n(k')=k_n$ such that, for all $n>M_0''$,
    \[
        \left\vert\frac{1}{k'} -\left.\sum_{i=\mathsf{d}_n(j-1)}^{\mathsf{d}_n(j)-1}c_{n,i+1}\right/\sum_{i=1}^{k_n}c_{n,i}\right\vert < \varepsilon'.
    \]
    Let
    \[
        \mathlcal{g}_{n, j} := \frac{1}{p't_n} \bull z_{n,\mathsf{d}_n(j)} z_{n,\mathsf{d}_n(j)-1} \dots z_{n,1} \quad \text{for } j\in\{1, \dots, k'\}.
    \]
    Then there exist $n_3 \geq n_2$ and, $\p$-a.s., $M_2'>M_0''$ such that, for all $n>M_2' \vee n_3$, $\vert\mathlcal{g}_{n,j} - \zeta(j/k')\vert< \frac{\upvarepsilon}{2}$. Hence, for $n>M_2'\vee n_3$, $\sum_{j=1}^{k'} d_\infty\left(\frac{1}{p'{t_n}}\bull w_j^{t_n} \dots w_1^{t_n}, \mathlcal{g}_{n,j}\right) < \upvarepsilon$. 
    
    It follows that there exists, $\p$-a.s. $M_2''>M_2' \vee n_3$ so that, for every $n>M_2''$, 
    \[
        \sum_{j=1}^{k'} \frac{1}{p't_n} d_S\left( w_j^{t_n} \dots w_1^{t_n}, z_{n,\mathsf{d}_n(j)}z_{n,\mathsf{d}_n(j)-1} \dots z_{n,1} \right) < \upvarepsilon
    \]
    We thus get from \cref{lm:local.bound} that there exists, $\p$-a.s., $M_1'> M_2''$, such that, for each $n> M_1'$,
    \[
        \sum_{j=1}^{k'} \left| c(w_j^{t_n}, w_{j-1}^{t_n} \dots w_1^{t_n} \cdot \omega) - c(z_{n,\mathsf{d}_n(j)}, z_{n,\mathsf{d}_n(j)-1} \dots z_{n,1}\cdot \omega) \right| < 4 \upbeta p't_n\upvarepsilon.
    \]
    Set $M_2 \geq M_1' \vee \frac{p'}{p \upvarepsilon^2}$. Therefore, one has, $\p$-a.s., for all $n>M_2$,
    \begin{align}
        \frac{1}{pn} c(\mathbbmtt{y}_n,\omega) &> \frac{1}{(1+\upvarepsilon)pn} \sum_{j=1}^{k'} ~\sum_{i=d_n(j-1)}^{d_n(j)}c(z_{n,i}, z_{n,i-1}\dots z_{n,1}\cdot\omega) \nonumber\\
        &\geq \frac{1}{(1+\upvarepsilon)pn} \sum_{j=1}^{k'} c(z_{n,\mathsf{d}_n(j)}, z_{n,\mathsf{d}_n(j)-1} \dots z_{n,1}\cdot \omega) \nonumber\\
        & \geq \frac{1}{1+\upvarepsilon} \frac{p't_n}{pn}\left(\frac{1}{p't_n} \sum_{j=1}^{k'} c(w_j^{t_n}, w_{j-1}^{t_n} \dots w_1^{t_n} \cdot \omega)-4\upbeta\upvarepsilon\right) \nonumber\\
        &> (1-\upvarepsilon)\Big(\ell_\phi(\zeta) - (4\upbeta+1)\upvarepsilon \Big). \nonumber
     \end{align}
    
    Fix $\upvarepsilon = \frac{1}{2\big(\ell_\phi(\upgamma) + 4\upbeta+1 \big)}\epsilon$. Hence, since $\ell_\phi(\zeta)\geq\ell_\phi(\upgamma)- \epsilon/2$, one has, $\p$-a.s., for all $n >M_2$,
    \begin{equation} \label{eq:lower.bound.path.cocycle}
         \frac{1}{pn} c(\mathbbmtt{y}_n,\omega) > \ell_\phi(\upgamma) - \epsilon.
    \end{equation}
    We complete the proof by combining \eqref{eq:upper.bound.path.cocycle} and \eqref{eq:lower.bound.path.cocycle}.
\end{proof}

\begin{lemma} \label{lm:lim.phi.g}
    Let  $\{x_n\}_{n\in\N}$ be a sequence in $\Gr$ and let $\{t_n\}_{n \in\N}$ be an increasing sequence in $\R$ such that $\lim_{n\uparrow +\infty}\frac{1}{t_n}\bull x_n = \mathlcal{g} \in G_\infty$. 
    
    Consider a subadditive cocycle $c:\Gr \times \Om \to \R_{\geq 0}$ satisfying conditions \eqref{all} and \eqref{aml}. If condition \eqref{innerness} is satisfied or if $\Gr$ is abelian, then, for all $\upepsilon \in (0,1)$, there exists, $\p$-a.s., a random $M = M(\mathlcal{g},\epsilon)>0$
    such that, for $t_n > M$,
	\[
		\left| \frac{1}{t_n}c(x_n) - \Phi(\mathlcal{g})\right| <\upepsilon%
	\]
\end{lemma}
\begin{proof}
	Set $\varepsilon>0$ to be defined later. Consider $y_k, \dots, y_1 \in \Gr$ and $p$ given by \cref{polygonal.path} for $\Hat{\varepsilon}=\varepsilon/2$ and a $d_\infty$-geodesic curve $\upgamma:[0,1] \to G_\infty$ from $\mathlcal{e}$ to $\mathlcal{g}$. Let  $t_n':=\lfloor t_n/p\rfloor$. Since $\frac{1}{t_n}\bull x_i$ converges to $\mathlcal{g}$. It follows from the Borel-Cantelli Lemma applied to \eqref{all} that there exists, $\p$-a.s., $M'>0$ so that, for every $t_n>M'$,
	\begin{equation}
		\left(\frac{1}{pt_n'}-\frac{1}{t_n}\right)c(x_n) < \frac{p-1}{pt_n'}\upbeta \frac{\|x_n\|_S}{t_n} < \frac{\upepsilon}{4}. \label{eq.xfloor}
	\end{equation}
	
    Let us write $\mathbbmtt{y}_i':=y_k^{t_n'}\dots y_1^{t_n'}$. Since $\lim_{n\uparrow+\infty}d_{\infty}\left(\frac{1}{p t_n'}\bull\mathbbmtt{y}_n' , \mathlcal{g}\right) =0$, one can easily see that there exists $n_1'>0$ such that, for all $t_n'>n_1'$,  \[\frac{1}{pt_n'}d_\infty(x_n,\mathbbmtt{y}_n')<\varepsilon.\] 
    Then there exists $n_2'\geq n_1'$ such that, for all $t_n' >n_2'$, one has
    $\|x_n(\mathbbmtt{y}_n')^{-1}\|_S< pt_n' \varepsilon$.
    
    Let now $t= \upbeta p t_n'\varepsilon$ in \eqref{all}. Since $c(x)$ is identically distributed  to $c(x)\circ\tht_y$, we have by Borel-Cantelli Lemma that there exists, $\p$-a.s., $M'' \geq M' \vee n_2'$ such that, for $t_n > M''$,
	\begin{align}
		\frac{1}{pt_n'}\big|c(x_n)- c(\mathbbmtt{y}_n')\big| &\leq \frac{1}{pt_n'}\max\big\{ c(\mathbbmtt{y}_n'x_n^{-1})\circ\tht_{x_n},c(x_n(\mathbbmtt{y}_n' )^{-1})\circ\tht_{\mathbbmtt{y}_n'} \big\} \nonumber \\
		&< 2\upbeta \varepsilon. \label{eq.ylim}
	\end{align}
	
	Set $\varepsilon\leq \frac{\upepsilon}{8\upbeta}$ and combine \eqref{eq.xfloor} and \eqref{eq.ylim}. We thus obtain that, $\p$-a.s., for all $t_n >M''$,
	\begin{equation} \label{eq:asymp.pt1}
		\left| \frac{1}{t_n}c(x_n) - \frac{1}{p t_n'}c\left(\mathbbmtt{y}_n'\right) \right|<\frac{\upepsilon}{2}.	
	\end{equation}
	
	Consider $\Gr$ abelian, then $\mathbbmtt{y}_n'= (y_k\dots y_1)^{t_n'}$. In fact, it is straightfoward that $k=1$ by the standard approach for commutative groups. Then by \cref{prop:subadditive.ergodic.thm}, there is, $\p$-a.s., $M^\ast\geq M''$  such that, for all $t_n>M^\ast$,
	\begin{equation} \label{eq:asymp.pt2}
	    \frac{1}{p}\left|\frac{1}{t_n'}c\left( \mathbbmtt{y}_n'\right) - \phi(y_1^{\ab})\right|<\frac{\upepsilon}{2}.
	\end{equation}
    
    Furthermore, we have $\phi(y_1^{\ab})/p = \ell_\phi(\upgamma) = \Phi(\mathlcal{g})$. Combining the two previous inequalities with \eqref{eq:asymp.pt1}, we can establish the result for the commutative case with  $M=M^\ast$. Now, let's consider the non-abelian case, assuming that \eqref{innerness} holds true. Notably, by \cref{prop:asymptotic.approx} with $\epsilon/2$ and $M= M'' \vee M_2$, for all $t_n>M$,
    \[\left\vert \frac{1}{p t_n'}c\left(\mathbbmtt{y}_n'\right) -\ell_\phi(\upgamma) \right\vert < \frac{\upepsilon}{2}.\]
    This result, when combined with \eqref{eq:asymp.pt1}, completes the proof.
\end{proof}

We now proceed to demonstrate the proof of the first main theorem.

\begin{proof}[Proof of \cref{shape.thm}]
    We begin by proving the $\p$-a.s. asymptotic equivalence given, which is given by 
    \begin{equation} \label{eq:asymp.equiv}
        \lim_{\|x\| \uparrow +\infty} \frac{\vert c(x) - \Phi(1\bull x) \vert}{\|x\|_S} = 0 \quad \p\text{-a.s.}
    \end{equation}

    Suppose, by contradiction, that \eqref{eq:asymp.equiv} is not true. Consider $\{v_n\}_{n\in \N} \subseteq \Gr$ to be such that $\|v_n\|_S \uparrow +\infty$ as $n\uparrow+\infty$. Let $\mathlcal{S}_r$ stand for $\overline{B_\infty(\mathlcal{e},r)}$, the closure of the $d_\infty$-ball or radius $r>0$ in $G_\infty$. Due to the compactness of $\mathlcal{S}_1$ with respect to $d_\infty$, there exists a subsequence $\{y_n\}_{n\in\N} \subseteq \{v_n\}_{n\in\N}$ such that, for $t_n:=\|y_n\|_S$
    \[\lim_{n\uparrow+\infty} \frac{1}{t_n}\bull y_n = \mathlcal{h} \in \mathlcal{S}_1.\]

    By construction, $\Delta := \bigcup_{n\in\N}\left(\frac{1}{n}\bull\Gr\right)$ is a countable dense subset of $G_\infty$. Fix, for each $\mathlcal{g} \in \Delta$, $\sigma(\mathlcal{g})= \{x_n\}_{n \in \N}$ such that $\frac{1}{n}\bull x_n$ converges to $\mathlcal{g}$ under $d_\infty$ (see \cref{lm:conv.seq.Gr}). Let $\Om_{\mathlcal{g}} \in \F$ be the event with $\p(\Om_{\mathlcal{g}})=1$ given by \cref{lm:lim.phi.g} for $\sigma(\mathlcal{g})$. Hence, $\Om_\Delta := \bigcap_{\mathlcal{g} \in \Delta} \Om_{\mathlcal{g}}$ is such that $\p(\Om_\Delta)=1$.

    The compactness of $\mathlcal{S}_r$ implies the existence of a finite $\Delta_{r,\varepsilon}\subseteq \mathlcal{S}_r \cap \Delta$ such that $\bigcup_{\mathlcal{g}\in\Delta_{r,\varepsilon}}B_\infty(\mathlcal{g}, \varepsilon)$ covers $\mathlcal{S}_r$. Thus there exists $\mathlcal{g} \in \Delta_{1,\varepsilon}$ so that $\mathlcal{h} \in B_\infty(\mathlcal{g}, \varepsilon)$. Consider $\sigma(\mathlcal{g}) = \{x_n\}_{n\in\N}$ as defined above and let $\varepsilon>0$ to be determined later. Then, there exists $m(\varepsilon)>0$ so that, for all $t_n>m(\varepsilon)$,
    \begin{align*}
        d_\infty\left(\frac{1}{t_n}\bull x_{t_n}, \frac{1}{t_n}\bull y_n\right) &\leq d_\infty\left(\frac{1}{t_n}\bull x_{t_n}, \mathlcal{g}\right) + d_\infty\left(\frac{1}{t_n}\bull y_n, \mathlcal{h}\right) + d_\infty(\mathlcal{g}, \mathlcal{h})\\
        &\leq 3\varepsilon.
    \end{align*}
    and $d_S(x_{t_n},y_n) < 7\varepsilon=: \eta_\varepsilon$.

    Let $M_1(\mathlcal{g}, \eta_\varepsilon)>0$ be given by \cref{lm:local.bound} on $\Uptheta_{\mathlcal{g}} \in \F$ with $\p(\Uptheta_{\mathlcal{g}})=1$ satisfying, for all $t_n> M_1(\mathlcal{g}, \eta_\varepsilon)$  and $u_n \in B_S(x_{t_n}, t_n\eta_\varepsilon)$,
    \[|c(x_{t_n})-c(u_n)|< 14\|y_n\|_S\upbeta\varepsilon.\]
    
    Fix, for $M(\mathlcal{g}, \varepsilon)$ given by \cref{lm:lim.phi.g},  
    \[\widehat{M}(\varepsilon) := \max_{\mathlcal{g}\in \Delta_{\mathlcal{S}_1,\varepsilon}}\lbrace M(\mathlcal{g}, \varepsilon), M_1(\mathlcal{g},\eta_\varepsilon)\rbrace,\]
    which is finite on $\Uptheta_\Delta:= \bigcap_{\mathlcal{g} \in \Delta} (\Om_\Delta \cap \ \Uptheta_{\mathlcal{g}})$ with $\p(\Uptheta_\Delta)=1$. 
    
    Set $\hat{m}(\varepsilon)>m(\varepsilon)$ to be such that $\left|\Phi(\frac{1}{t_n}\bull y_n) - \Phi(\mathlcal{h})\right|<\varepsilon$ for all $n>\hat{m}(\varepsilon)$. Hence, for all $t_n > \widehat{M}(\varepsilon) \vee \hat{m}(\varepsilon)$ on $\Uptheta_\Delta$,
    \begin{align*}
        \frac{|c(y_n) -\Phi(1\bull y_n)|}{\|y_n\|_S} &\leq  \frac{1}{t_n}|c(y_n)-c(x_{t_n})| + \left| \frac{1}{t_n}c(x_n) - \Phi(\mathlcal{g}) \right|  \\ & \hspace{55pt}+ |\Phi(\mathlcal{g}) - \Phi(\mathlcal{h})|+ \left| \Phi(\mathlcal{h}) -\Phi\left(\frac{1}{t_n}\bull y_n\right)\right| \\
        &\leq (14\upbeta + 3) \varepsilon,
    \end{align*}
	which contradicts the above assumption proving that \eqref{eq:asymp.equiv} holds true.
 
    It remains to show how $\frac{1}{n}d_\omega$ converges to $d_\phi$ in the asymptotic cone. Recall that $d_\omega(x,y)= \big(c(yx^{-1})\circ\tht_{x}\big)(\omega)$. Consider now any given $\mathlcal{h}, \mathlcal{h}' \in G_\infty$ and $\{u_n\}_{n \in \N}$ a sequence with $\{t_n'\}_{n\in\N} \subseteq \N$ such that $t_n' \uparrow +\infty$ and $\frac{1}{t_n'}\bull u_n \to \mathlcal{h'h}^{-1}$. Then $\|u_n\|_S/t_n'$ converges to $d_\infty(\mathlcal{h}, \mathlcal{h}')$ and $\frac{1}{\|u_n\|_S}\bull u_n$ converges as above. In particular, one can fix any $r'>d_\infty(\mathlcal{h}, \mathlcal{h}')$ to find $\mathtt{k}_{r'}>0$ such that  $\|u_n\|_S/t_n' < r'$ for all $t_n'>\mathtt{k}_{r'}$.
    
    Let us define $\mathtt{K}_{r'} = (14\upbeta +3)r'$ and $m_{r'}(\varepsilon) = \hat{m}(\varepsilon) \vee \mathtt{k}_{r'}$. The asymptotic equivalence \eqref{eq:asymp.equiv} implies the existence of a random $\widehat{M}(\varepsilon)>0$ for $ \varepsilon \in (0, \frac{1}{14\upbeta+3})$ such that, for all $t_n' > \widehat{M}(\varepsilon) \vee m_{r'}(\varepsilon)$ on $\Uptheta_\Delta$,
    \[\left|\frac{1}{t_n'}c(u_n)  - \Phi\big(\mathlcal{h'h}^{-1}\big) \right|< \mathtt{K}_{r'} \varepsilon.\]

    Due to the fact that $\tht$ is a p.m.p. group action, one can repeat all arguments above also in \cref{prop:asymptotic.approx,lm:lim.phi.g} to obtain $\widehat{M}\big(\varepsilon, \sigma(\mathlcal{g})\big)$ and $\Uptheta_\Delta\big(\sigma(\mathlcal{g})\big)$ for each $\sigma(\mathlcal{g}) = \{x_n\}_{n \in \N}$ with $\mathlcal{g}\in G_\infty$ and $\p\big(\Uptheta_\Delta(\mathlcal{g})\big)=1$ so that, for all converging $\frac{1}{t_n'}\bull u_n$ as above and every $t_n' > \widehat{M}\big(\varepsilon, \sigma(\mathlcal{g})\big) \vee m_{r'}(\varepsilon)$ on $\Uptheta_\Delta\big(\sigma(\mathlcal{g})\big)$,
    \begin{equation} \label{eq:c.phi.1}
        \left|\frac{1}{t_n'}c(u_n)\circ\tht_{x_{t_n'}}  - \Phi(\mathlcal{h'h}^{-1}) \right|< \mathtt{K}_{r'} \varepsilon.
    \end{equation}

    Let now $\{v_n\}_{n\in\N}$ be a sequence that $\frac{1}{n} \bull v_n \to \mathlcal{h}$ and choose $r\geq d_\infty(\mathlcal{e}, \mathlcal{h})$. Fix $\mathlcal{g} \in \Delta_{r,\varepsilon}$ so that $\mathlcal{g} \in B_\infty(\mathlcal{h}, \varepsilon)$. By \cref{lm:local.bound}, one can find a random $\thickbar{M}_{r,r'}\big(\varepsilon,\sigma(\mathlcal{g})\big)> 0$ and $\Upxi_{\sigma(\mathlcal{g})}$ with $\p\big(\Upxi_{\sigma(\mathlcal{g})}\big) =1$ such that, for all $n> \thickbar{M}_{r,r'}\big(\varepsilon,\sigma(\mathlcal{g})\big)$ on $\Upxi_{\sigma(\mathlcal{g})}$,
    \begin{equation} \label{eq:c.phi.2}
        |c(w_n)\circ\tht_{x_n} - c(w_n)\circ\tht_{v_n}| < 2 \upbeta n \varepsilon,
    \end{equation}
    where $\{w_n\}_{n\in\N}$ is any convergent sequence $\frac{1}{n} \bull w_n \to \mathlcal{w}\in B_\infty(\mathlcal{e}, r')$. Let us fix 
    \[ 
        \Upxi_\Delta := \bigcap_{\mathlcal{g} \in \Delta}\left( \Upxi_{\sigma(\mathlcal{g})} \cap \Uptheta_\Delta\big(\sigma(\mathlcal{g}) \big)\right),
    \]
    and set 
    \[
        M_{r,r'}(\varepsilon):= \max_{\mathlcal{g} \in \Delta_{r,\varepsilon}}\left\{ \thickbar{M}_{r,r'}\big(\varepsilon,\sigma(\mathlcal{g})\big), \widehat{M}\big(\varepsilon, \sigma(\mathlcal{g})\big) \right\}.
    \]
    Then $M_{r,r'}(\varepsilon)$ is finite on $\Upxi_\Delta$ and  $\p\big(\Upxi_\Delta\big) = 1$. It follows from \eqref{eq:c.phi.1} and \eqref{eq:c.phi.2} that, for all $t_n'> M_{r,r'}(\varepsilon) \vee m_{r'}(\varepsilon)$ on $\Upxi_\Delta$,
    \[
        \left|\frac{1}{t_n'}c(u_n)\circ\tht_{v_{t_n'}}  - \Phi(\mathlcal{h'h}^{-1}) \right|< (\mathtt{K}_{r'} + 2\upbeta) \varepsilon.
    \]
    This establishes the $\p$-a.s. convergence of $\frac{1}{t_n'}d_\omega(v_{t_n'},u_nv_{t_n'})$ to $\Phi\big(\mathlcal{h'h}^{-1}\big) =d_\phi(\mathlcal{h},\mathlcal{h}')$ for $\omega\in \Upxi_\Delta$ as $n\uparrow+\infty$. Observe that the bi-Lipschitz equivalence is a straightforward consequence of \cref{lm:phi.bi-Lipschitz}, and this completes the proof.
\end{proof}

\subsection{Proof of the Second Theorem} \label{sec:virt.nilpotent.proofs}

With the first main theorem now established, we have determined the asymptotic shape for finitely generated torsion-free nilpotent groups. The objective of this subsection is to extend this result to a finitely generated virtually nilpotent group $\Gr$.

Recall that the nilpotent subgroup $N \unlhd \Gr$ has a finite index $\kappa=[\Gr:N]$, and for each coset $N_{(j)} =z_{(j)}N \in \Gr/N$, we designate a representative $z_{(j)} \in N_{(j)}$. Also, define $\uppi_N( x ) = z_{(j)}^{-1}x$ for all $x\in N_{(j)}$ and $j \in \{1, \dots, \kappa\}$.

We commence by presenting results concerning the properties of p.m.p. ergodic group actions of $\Gr$ with respect to $N$ and $\Gr'$. We adopt the notation $\cup \mathcal{A} := \bigcup_{A \in\mathcal{A}}A$.

\begin{lemma} \label{lm:erg.finite.index}
    Let $\Gr$ be a discrete group $\Gr$ and $N \unlhd \Gr$ a finite normal subgroup with finite index $[\Gr:N]=\kappa$. Consider that $\tht: \Gr \curvearrowright (\Om, \F, \p)$ is a p.m.p. ergodic group action. Then there exists a finite $\mathfrak{B}_N \subseteq \F$ such that, for all $B \in \mathfrak{B}_N$, $\p(B) \geq 1/\kappa$ and $\tht\big\vert_N,$ the restriction of $\tht$ on $N$, induces a p.m.p. ergodic group action on $\big(B, \F_{\cap B}, \p(~\cdot \mid B)\big)$. Furthermore, $|\mathfrak{B}_N|\leq \kappa$ and $\p(\cup\mathfrak{B}_N)=1$. 
\end{lemma}
\begin{proof}
    Set $\mathfrak{A}_N \subseteq \F$ to be the family of all non-empty $N$-invariant events under $\tht$. Then, for all $A \in \mathfrak{A}_N$, \[\p\left(\bigcup_{j =1}^{\kappa}z_{(j)}\cdot A\right) =1\] which implies $\p(A) \geq 1/\kappa$. Observe that $\mathfrak{A}_N$ is closed under countable unions and non-empty countable intersections. Let us fix $A_0 \in \mathfrak{A}_N$  such that $ \p(A_0)= \inf_{A \in \mathfrak{A}}\p(A)$. Define $\mathfrak{B}_N = \{z_{(j)}\cdot A_0\}_{j=1}^{\kappa}$. 
    
    Since $N$ is a normal subgroup of $\Gr$, $N$ acts ergodically on $\big(B, \F_{\cap B}, \p(~\cdot \mid B)\big)$ for all $B \in \mathfrak{B}_N$ and it inherits the measure preserving property.
\end{proof}

We use \cref{lm:erg.finite.index} to write $(B,\F_B,\p_B)$ with $\F_B:=\F_{\cap B} = \{E \cap B : E \in \F\}$ and $\p_B(E):=\p(E \mid B)$ for each $B \in \mathfrak{B}_N$. Let us denote by $[\omega] = \tor N \cdot \omega$, the orbit of $\omega\in B$ under the action on $\tor N$. Set
\[\big([B], \F_B', \p_B'\big) := (B, \F_B, \p_B)/\tor N\]
where $\F_B' = \big\{[E]:E \in \F_B\big\}$ and $\p_B'\big([E]\big)$ is the induced probability measure $(\tor N)_\ast\p_B(E) = \p_B\left(\cup[E]\right)$. Let us fix $\upupsilon_x = \upupsilon_{\llbracket x \rrbracket} \in \llbracket x \rrbracket$ for each $\llbracket 
x \rrbracket \in \Gr'$. Define $\theta: \Gr' \curvearrowright \big([B], \F_B', \p_B'\big)$ so that 
\[
    \theta_{\llbracket x \rrbracket}\big([\omega]\big) = \big[\tht_{\upupsilon_x}(\omega)\big].
\]

\begin{lemma} \label{lm:erg.torsion.quotient}
    Let $\mathfrak{B}_N$ be the set obtained in \cref{lm:erg.finite.index}. Then, for each $B \in \mathfrak{B}_N$, $\theta: \Gr' \curvearrowright \big([B], \F_B', \p_B'\big)$ is a p.m.p. ergodic group action.
\end{lemma}
\begin{proof}
    The measure preserving property is immediately inherited from $\tht$. Let $\tht_v(\omega) = v\cdot\omega$. Due to the normality of $\tor N \unlhd N$, for all $A \in \F_B$ and each $v' \in v.\tor N$, 
    \[\cup[v \cdot A] = v'\cdot\left(\cup[A]\right).\]
    
    Hence, if for all $v.\tor N \in \Gr'$, one has $[v \cdot A] = [A]$. Then, for all $x \in N$
    \[x\cdot \left(\cup[A]\right)=\cup[A].\]
    It follows from the ergodicity of $\tht:N \curvearrowright (B, \F_B, \p_B)$ that $\p_B'\big([A]\big)\in\{0,1\}$, which is the desired conclusion.
\end{proof}

\begin{remark} \label{rmk:c.prime}
    Recall that definition \eqref{eq:def.c.prime} determines
    \[c'\big(\llbracket x \rrbracket\big) := \max_{\substack{y \in \llbracket x \rrbracket\\ z \in \tor N}}c(y)\circ\tht_z.\]

    It is straightforward to see that $c'$ is compatible with the probability space $\big([B],\F_b', \p_B'\big)$ for each $B \in \mathfrak{B}_N$. Futhermore, it is a subadditive cocycle associated with $\theta$. Additionally, $c'$ is well defined on $(B,\F_B,\p_B)$. Let $\Om':= \cup \mathfrak{B}_N$ and $\p(\Om')=1$. Consequently, one can investigate $c'$ on $([B],\F_B',\p_B')$, and the results can be naturally extended $\p$-a.s. to $(\Om,\F,\p)$.
\end{remark}

In preparation for the asymptotic comparison between cocycles $c$ and $c'$, the following lemmas provide essential insights into their respective properties and relationships.

\begin{lemma} \label{lm.T.bounded.nilpotent.as}
    Let $\varepsilon, r>0$ and consider a subadditive cocycle $c$ that satisfies condition \eqref{all}. Then there exists, $\p$-a.s., $M_N=M_N(\varepsilon,r)>0$ such that, for all $n > M_N$ and every $x \in B_S(e,rn)$,
    \[
        \big|c(x) - c\big(\uppi_N( x )\big)\big| < \varepsilon n.
    \]
\end{lemma}
\begin{proof}
    It follows from subadditivity that, for $x \in N_{(j)}$,
    \[
        |c(x) - c(z_{(j)}^{-1}x)| \leq \max\left\{c(z_{(j)})\circ\tht_{z_{(j)}^{-1}x},\; c(z_{(j)}^{-1})\circ\tht_{x}\right\} \quad \p\text{-a.s.}
    \]
    for every $j \in \{1,\dots, \kappa\}$. Let $\mathtt{m}_\kappa=\max\left\{\|z_{(j)}\|_S:1 \leq j \leq \kappa\right\}$. Hence, one has by \eqref{all} and a $\mathtt{C}>0$ that
    \begin{align*}
        \p\left(\max_{x \in B_S(e,rn)}\big\{ |c(x)-c(\uppi_N(x))| \big\} \geq \varepsilon n \right) &\leq |B_S(e,rn)|\sum_{j=1}^\kappa\p\left(c(z_{(j)}^{\pm 1}) \geq \varepsilon n\right)\\
        & \leq \mathtt{C}r^Dn^D g(n \varepsilon) \in \mathcal{O}_{\varepsilon, r}\big(1/n^{D+\upkappa}\big)
    \end{align*}
    for $n > \upbeta \mathtt{m}_\kappa/\varepsilon$. The result is derived through the application of the Borel-Cantelli Lemma.
\end{proof}

\begin{lemma} \label{lm.T.bounded.torsion.as}
    Let $\varepsilon, r>0$ and consider a subadditive cocycle $c$ that satisfies condition \eqref{all}. Then there exists, $\p$-a.s., $M_q=M_q(\varepsilon,r)>0$ such that, for all $n > M_q$ and every $x \in B_S(e,rn)$,
    \[
        \left|c\left(x_1\right)\circ\tht_{y_1} - c(x_2)\circ\tht_{y_2}\right| < \varepsilon n
    \]
    where $x_1,x_2 \in \llbracket x\rrbracket$ and $y_1,y_2 \in \tor N$.

\end{lemma}
\begin{proof}
    Since $\tor N$ is a normal subgroup of $N$, the exists $v_2\in \tor N$ such that $x_1=v_2x_2y_3$ with $y_3= y_2y_1^{-1}$. Thus
    \begin{align*}
        c(x_1)\circ\tht_{y_1} &\leq c(y_3)\circ \tht_{y_1} + c(v_2x_2)\circ\tht_{y_2}\\
        &\leq c(y_3)\circ \tht_{y_1} + c(x_2)\circ\tht_{y_2} + c(v_2)\circ\tht_{x_2y_2} ~~~\p\text{-a.s.}
    \end{align*}
    We apply the same reasoning for $c(x_2)\circ\tht_{y_2}$ obtaining that
    \[
        \left|c(x_1)\circ\tht_{y_1} - c(x_2)\circ\tht_{y_2}\right| \leq \max_{y,z \in \tor N}\{c(y)\circ\tht_z\} + \max_{y,z \in \tor N}\{c(y)\circ \tht_{x_1z}\} \quad \p\text{-a.s.}
    \]
    By \eqref{all} and the finitness of $\tor N$, there exists a constant $C' >0$ such that
     \begin{eqnarray*}
        \mathbb{P}\left( \sup\limits_{\substack{x\in B_S(e,rn)\\x_1,x_2 \in \llbracket x \rrbracket \\ y_1,y_2 \in \tor N}}\hspace{-3pt} \big|c(x_1)\circ\tht_{y_1} - c(x_2)\circ\tht_{y_2}\big| \geq \varepsilon n \right) &\leq& 2|\tor N|^4 {|B_S(e,rn)|^2} g(\varepsilon n)\\
        &\leq& C'{(rn)^{2D}}g(\varepsilon n) \in \mathcal{O}_{\varepsilon,r}(1/n^\upkappa)
     \end{eqnarray*}
     for $n > \upbeta \max\{ \|z\|_S: z \in \tor N\}/\varepsilon$. The desired conclusion follows from an application of Borel-Cantelli Lemma.
\end{proof}

Let us define, for all $\llbracket x \rrbracket \in \Gr'$,
\[
    \big\vert \llbracket x \rrbracket \big\vert_S^{\inf} := \min_{1 \leq i,j\leq \kappa} ~\min_{y \in (z_{(j)} .\llbracket x \rrbracket.z_{(i)}^{-1})} \|y\|_S,
\]
and
\[
    \big\vert \llbracket x \rrbracket \big\vert_S^{\sup} := \max_{1 \leq i, j\leq \kappa}  ~\max_{y \in (z_{(j)} .\llbracket x \rrbracket.z_{(i)}^{-1})} \|y\|_S.
\]
Set \[\mathtt{m}_{\kappa,q} := \max_{1 \leq i, j \leq \kappa} ~\max_{z \in (z_{(j)} .\llbracket e \rrbracket. z_{(i)}^{-1})} \|z\|_S.\]

Thus, one has, for all $y \in z_{(j)}.\llbracket x \rrbracket$ with $j \in \{1, \dots, \kappa\}$,
\begin{equation} \label{eq:equiv.nom.inf.sup}
    \big\vert \llbracket x \rrbracket \big\vert_S^{\inf} \leq \|y\|_S \leq \big\vert \llbracket x \rrbracket \big\vert_S^{\sup} \leq \big\vert \llbracket x \rrbracket \big\vert_S^{\inf} + 2\cdot\mathtt{m}_{\kappa,q}.
\end{equation}

By the same arguments employed in Section \ref{sec:norms.and.mean}, the discrete norm
\begin{equation} \label{eq:norm.inf.abelian}
    \big\vert \llbracket x \rrbracket\big\vert_S^{\ab} := \inf_{\llbracket y \rrbracket \in \big(\llbracket x \rrbracket. [\Gr',\Gr']\big)} \big\vert \llbracket y \rrbracket \big\vert_S^{\inf}
\end{equation}
exhibits the same properties as $\|-\|_S^{\ab}$ when $\llbracket S \rrbracket$ is a generating set of $\Gr'$.

Consider $\sigma(\mathlcal{g}) = \big\{\llbracket x \rrbracket_n\big\}_{n\in\N} \subseteq \Gr'$ to be the sequences fixed for each $\mathlcal{g}\in G_\infty$ in the proof of  \cref{shape.thm}. Set $x_n:= \upupsilon_{\llbracket x \rrbracket_n}$ with $\upupsilon$ defined by the group action $\theta$. Then 
\[\llbracket x_n \rrbracket= \llbracket x \rrbracket_n\]
when $\sigma(\mathlcal{g})$ is given. Let us write $\upupsilon_\sigma(\mathlcal{g}) = \{x_n\}_{n \in \N}$ for each $\sigma(\mathlcal{g}) = \big\{\llbracket x \rrbracket_n\big\}_{n\in\N} \subseteq \Gr'$. Also, one can easily verify that
    \[\lim_{n\uparrow+\infty}\frac{\|x_n\|_S}{n} = \lim_{n\uparrow+\infty}\frac{\big\vert\llbracket x_n \rrbracket \big\vert_S^{\inf}}{n} = \lim_{n\uparrow+\infty}\frac{\big\vert\llbracket x_n \rrbracket \big\vert_S^{\sup}}{n} = d_\infty(\mathlcal{e},\mathlcal{g}) .\]

The proposition below shows us that $c$ and $c'$ share the same linear asymptotic behaviour.

\begin{proposition} \label{prop:asympt.equiv.c.cPrime}
    Let $\Gr$ be a virtually nilpotent group, and let $c:\Gr \times \Om \to \R_{\geq0}$ be a subadditive cocycle associated with $\tht$. 

    If condition \eqref{all} is satisfied, then  $c$ and $c'$ are asymptotically equivalent, \textit{i.e.}, there exists, $\p$-a.s., $M'(\varepsilon)>0$ such that, for all $x \in \Gr$ with $\|x\|_S> M'(\varepsilon)$,
    \begin{equation} \label{eq:asymp.equiv.c.c.prime}
        \big| c(x)-c'(\llbracket x\rrbracket)\big| < \varepsilon\|x\|_S.
    \end{equation}

    In particular, \eqref{all} implies the $\p$-a.s. existence of $M'\big(\varepsilon, r, \upupsilon_\sigma(\mathlcal{g})\big)>0$ so that, for all $n> M'\big(\varepsilon,r, \upupsilon_\sigma(\mathlcal{g})\big)$ and every $y \in B_S(e, rn)$,
    \begin{equation} \label{eq:aproxx.seq.c.c.prime}
        \big| c(y)-c'(\llbracket y\rrbracket)\big| \circ \tht_{x_n} < n \varepsilon.
    \end{equation}
\end{proposition}
\begin{proof}
    From \cref{lm.T.bounded.nilpotent.as,lm.T.bounded.torsion.as}, we can deduce that, for every $\varepsilon>0$,  one can fix $M'(\varepsilon)= M_N(\frac{\varepsilon}{2},1) \vee M_q(\frac{\varepsilon}{2},1)$ so that, $\p$-a.s., for all $n> M'(\varepsilon)$ and every $x \in B_S(e,n+1)\setminus B_S(e,n)$,
    \[
        \frac{|c(x)-c'(x)|}{\|x\|_S} < \frac{\left|c(x)-c\big(\uppi(x)\big)\right|}{n} + \frac{\left|c\big(\uppi(x)\big)-c'(x)\right|}{n} < \varepsilon.
    \]
    The inequality above implies the asymptotic equivalence of $c$ and $c'$ on $\Gr$.

    Since $\tht$ is p.m.p. group action, one can obtain from \cref{lm.T.bounded.nilpotent.as,lm.T.bounded.torsion.as} the random variables $M_N>0$ and $M_q>0$ depending on $\upupsilon_\sigma(\mathlcal{g})\big)>0$ determining
    \[
        M'\big(\varepsilon,r,\upupsilon_\sigma(\mathlcal{g})\big) = M_N\big({\varepsilon}/{2},r,\upupsilon_\sigma(\mathlcal{g})\big) \vee M_q\big({\varepsilon}/{2},r,\upupsilon_\sigma(\mathlcal{g})\big)
    \]
     so that \eqref{eq:aproxx.seq.c.c.prime} holds true.
\end{proof}

The following result extends the subadditive ergodic theorem to $c'$ with respect to $|-|_S^{\ab}$. 
\begin{lemma} \label{lm:subaddtive.ergodic}
    Consider $\Gr$ to be a virtually nilpotent group generated by a finite symmetric set $S \subseteq \Gr$ with $\llbracket S \rrbracket$ a generating set of $\Gr'$.
    
    If the subadditive cocycle $c$ satisfies \eqref{all} and \eqref{aml2} with respect to the word norm $\|-\|_S$, then $c'$ satisfies \eqref{all} and \eqref{aml} with respect to $\vert-\vert_S^{\inf}$. In particular, \cref{lm:phi.bi-Lipschitz} is still valid with $x^{\ab}=\llbracket x \rrbracket^{\ab}$ and
    \[\phi(x^{\ab}) = \inf_{n \in \N}\frac{\E[c'(\llbracket x\rrbracket^n)]}{n}.\]
\end{lemma}
\begin{proof}
    First, observe that \eqref{all} and \eqref{aml2} imply, for all $x \in \Gr$ and 
    \begin{equation*} 
        \p\Big( c'\big(\llbracket x \rrbracket\big) \geq  t \Big) \leq \kappa ~\vert\tor N\vert~ g(t), \quad \text{for all } t >\upbeta\big|\llbracket x \rrbracket\big|_S^{\sup},
    \end{equation*}
    and
    \begin{equation*} 
        \E\Big[ c'\big(\llbracket x \rrbracket\big)\Big] \geq a \big|\llbracket x \rrbracket\big|_S^{\inf}.
    \end{equation*}

    Therefore, it follows from \eqref{eq:equiv.nom.inf.sup} that $c'$ satisfy \eqref{all} and \eqref{aml} with respect to $\vert -\vert_S^{\inf}$ for a new $g'(t) \in \mathcal{O}\big(t^{2D+\upkappa}\big)$ and $\upbeta'>0$. The proof is complete by replacing $\|-\|_S$ with $\vert - \vert_S^{\inf}$ and applying \eqref{eq:norm.inf.abelian} in the proof of \cref{lm:phi.bi-Lipschitz}.

\end{proof}

Having established the aforementioned results, we now move forward to prove the second theorem.

\begin{proof}[Proof of Theorem \ref{thm:shape.polynomial}]
    Observe that it follows from \cref{lm:erg.finite.index,lm:erg.torsion.quotient,lm:subaddtive.ergodic}, and \cref{rmk:c.prime} that, for each $B \in \mathfrak{B}_N$, \cref{shape.thm} holds true for $c'$ on $(B,\F_B,\p_B)$. Therefore, it suffices to extend the results to $(\Om,\F,\p)$ and compare $c$ with $c'$. 
    
    The asymptotic equivalence is an immediate consequence of \eqref{eq:asymp.equiv} and \eqref{eq:asymp.equiv.c.c.prime}, we focus on the second part of the proof of \cref{shape.thm}. Recall de definition of $\Delta$ as a dense subset of $G_\infty$, the finite $\Delta_{r,\varepsilon}$. Similarly, we consider $\{u_n\}_{n\in\N} \subseteq \Gr$ and $\{t_n'\}_{n\in\N}\subseteq\N$ with $t_n\uparrow+\infty$ as $n\uparrow+\infty$ and $\frac{1}{t_n'}\bull u_n \to \mathlcal{h}'\mathlcal{h}^{-1}$. Note that we may regard $\llbracket u \rrbracket_n = \llbracket u_n \rrbracket$ to replace the orifinal sequence in the proof of Thm. \ref{shape.thm} and let $\mathtt{K}_{r'}$ and $m_{r'}(\varepsilon)$ be defined as before with $r'>d_\infty(\mathlcal{h},\mathlcal{h}')$. 
    
    Set $\widehat{M}\big(\varepsilon,\sigma(\mathlcal{g}),B\big)$ and $\Uptheta_\Delta\big( \sigma(\mathlcal{g}), B \big)$ to be defined by \eqref{eq:c.phi.1} for each $B \in \mathfrak{B}_N$ with $\p\left(\Uptheta_\Delta\big( \sigma(\mathlcal{g})\big), B \big) \mid B\right)=1$ so that, for all $t_n'>\widehat{M}\big(\varepsilon,\sigma(\mathlcal{g}),B\big) \vee m_{r'}(\varepsilon)$,
    \begin{equation} \label{eq:c.phi.1.prime}
        \left\vert \frac{1}{t_n'}c' \big(\llbracket u_n \rrbracket\big)\circ \tht_{x_{t_n}} - \Phi(\mathlcal{h}'\mathlcal{h}^{-1}) \right\vert < \mathtt{K}_{r'}\varepsilon.
    \end{equation}
    on $\Uptheta_\Delta\big( \sigma(\mathlcal{g})\big), B \big)$ with $\upupsilon_\sigma(\mathlcal{g})=\{x_n\}_{n\in\N}$. Fix
    \[
        \widehat{M}'\big(\varepsilon,\sigma(\mathlcal{g})\big) := \sum_{B\in\mathfrak{B}_N}\widehat{M}\big(\varepsilon,\sigma(\mathlcal{g}),B\big)\mathbbmtt{1}_B + \mathbbmtt{1}_{\Om\setminus(\cup\mathfrak{B}_N)}.
    \]

    Consider $\{y_n\}_{n\in\N}$ with $\|y_n\|_S/n < r'$ for every $n> m_{r'}(\varepsilon)$. Then \cref{prop:asympt.equiv.c.cPrime} ensures the existence of $M'\big(\varepsilon,r', \upupsilon_\sigma(\mathlcal{g})\big)>0$ and $\Uplambda_{\sigma(\mathlcal{g})} \in \F$ with $\p(\Uplambda_{\sigma(\mathlcal{g})})=1$ so that, for all $n>M'\big(\varepsilon,r', \upupsilon_\sigma(\mathlcal{g})\big)$ on $\Uplambda_{\sigma(\mathlcal{g})}$,
    \begin{equation} \label{eq:new.approx.c.c.prime}
        \frac{1}{n}\left\vert c(y_n) - c'\big( \llbracket y_n \rrbracket \big) \right\vert \circ \tht_{x_n} < \varepsilon.
    \end{equation}

    Let now $\{v_n\}_{n\in\N}\subseteq \Gr$ be a sequence such that $\frac{1}{n} \bull v_n \to \mathlcal{h}$ and choose $r\geq d_\infty(\mathlcal{e}, \mathlcal{h})$. Fix $\mathlcal{g} \in \Delta_{r,\varepsilon}$ so that $\mathlcal{g} \in B_\infty(\mathlcal{h}, \varepsilon)$. Observe that \eqref{eq:c.phi.2} is still valid for $c$. Hence, by \cref{lm:local.bound}, one can find  $\thickbar{M}_{r,r'}'\big(\varepsilon,\sigma(\mathlcal{g})\big)> 0$ and $\Upxi_{\sigma(\mathlcal{g})}'$ with $\p\big(\Upxi_{\sigma(\mathlcal{g})}'\big) =1$ such that, for all $n> \thickbar{M}_{r,r'}'\big(\varepsilon,\sigma(\mathlcal{g})\big)$ on $\Upxi_{\sigma(\mathlcal{g})}'$,
    \begin{equation} \label{eq:c.phi.2.prime}
        |c(w_n)\circ\tht_{x_n} - c(w_n)\circ\tht_{v_n}| < 2 \upbeta n \varepsilon,
    \end{equation}
    where $\{w_n\}_{n\in\N}\subseteq\Gr$ is any convergent sequence $\frac{1}{n} \bull w_n \to \mathlcal{w}\in B_\infty(\mathlcal{e}, r')$. Let us fix 
    \[
        \Uplambda_\Delta := \bigcap_{\mathlcal{g} \in \Delta}\left( \Uplambda_{\sigma(\mathlcal{g})} \cap \Upxi_{\sigma(\mathlcal{g})}  \cap \left(\bigcup_{B \in\mathfrak{B}_N} \Uptheta_\Delta\big(\sigma(\mathlcal{g}),B \big)\right)\right),
    \]
    and set
    \[
        M_{r,r'}'(\varepsilon):= \max_{\mathlcal{g} \in \Delta_{r,\varepsilon}}\left\{ M'\big(\varepsilon,r', \upupsilon_\sigma(\mathlcal{g})\big), ~\thickbar{M}_{r,r'}'\big(\varepsilon,\sigma(\mathlcal{g})\big), ~\widehat{M}'\big(\varepsilon, \sigma(\mathlcal{g})\big) \right\}.
    \]
    Then $M_{r,r'}'(\varepsilon)$ is finite on $\Uplambda_\Delta$ and  $\p\big(\Uplambda_\Delta\big) = 1$. It follows from \eqref{eq:c.phi.1.prime}, \eqref{eq:new.approx.c.c.prime}, and \eqref{eq:c.phi.2.prime} with $u_n=w_{t_n'}=y_{t_n'}$ that, for all $t_n'> M_{r,r'}'(\varepsilon) \vee m_{r'}(\varepsilon)$ on $\Uplambda_\Delta$,
    \[
        \left|\frac{1}{t_n'}c(u_n)\circ\tht_{v_{t_n'}}  - \Phi(\mathlcal{h'h}^{-1}) \right|< (\mathtt{K}_{r'} + 2\upbeta +1) \varepsilon.
    \]
    This establishes the $\p$-a.s. convergence of $\frac{1}{t_n'}d_\omega(v_{t_n'},u_nv_{t_n'})$ to $\Phi\big(\mathlcal{h'h}^{-1}\big) =d_\phi(\mathlcal{h},\mathlcal{h}')$ for $\omega\in \Uplambda_\Delta$ as $n\uparrow+\infty$.
\end{proof}

\subsection{An Additional Result for FPP Models} \label{sec:additional.FPP}

In the preceding sections, we delved into the asymptotic behavior of $c$ and $c'$. The definition of $c'$ depends only  on the action of $\tht$ restricted to $N \unlhd \Gr$, ensuring that we can systematically investigate the group action of $\Gr'$ within a fixed $B \in \mathfrak{B}_N$. 

To broaden the scope of our findings and establish the validity of \eqref{innerness2} for FPP models on virtually nilpotent groups, we will introduce a new random variable induced by a graph homomorphism. Let us now define, for all $\llbracket x \rrbracket \in \Gr' \setminus\{\llbracket  e\rrbracket\}$,
\[c''\big(\llbracket x \rrbracket\big) := \max_{1 \leq i, j \leq \kappa}\max_{\substack{y \in z_{(j)}.\llbracket x \rrbracket\\ z \in z_{(i)}.\tor N}}c(y)\circ\tht_z\]
and consider $c''\big( \llbracket e \rrbracket \big) :=0$. Note that $c''$ restricted to $B \in \mathfrak{B}_N$ is not well-defined when there exists another set $B' \in \mathfrak{B}_N$ distinct from $B$. This inherent limitation prompts the necessity for specific conditions in the subsequent result.

The following lemma outlines the criteria under which $c''$ inherits the FPP property from $c$. Before presenting this result, we establish the notation:
\[\llbracket S \rrbracket^\pm := \big\{ \llbracket s \rrbracket^{\pm1} \colon s \in S \big\}.\]

\begin{lemma} \label{lm:N.ergodic}
    Let $(\Gr,.)$ be a virtually nilpotent group generated by a finite symmetric set $S \subseteq \Gr$ with $\llbracket S \rrbracket$ a generating set of $\Gr'$. Consider a subadditive cocycle $c:\Gr\times\Om\to\R_{\geq 0}$ determining a FPP model on $\mathcal{C}(\Gr, S)$ which satisfies \eqref{all}. Suppose that the restriction $\tht\big\vert_{N}: N \curvearrowright (\Om,\F,\p)$ is a p.m.p. ergodic group action.

    If, for all $s \in S$, $\llbracket s^{-1} \rrbracket = \llbracket s \rrbracket^{-1}$, then $c''$ determines a FPP model on $\mathcal{C}(\Gr, \llbracket S \rrbracket^\pm)$ and condition \eqref{innerness2} is satisfied when $\llbracket S \rrbracket^\pm \subseteq \Gr'\setminus[\Gr',\Gr']$.
\end{lemma}
\begin{proof}
    Define, for each $x \in \Gr$ and every $\llbracket s \rrbracket \in \llbracket S \rrbracket^\pm$,    \[
        \tau\big(\llbracket x \rrbracket, \llbracket s \rrbracket\llbracket x \rrbracket\big) := \max_{1 \leq i, j \leq \kappa} \max_{\substack{~y \in z_{(j)}.\llbracket x \rrbracket~\\ h \in z_{(i)}.\llbracket s \rrbracket}} \tau(y, hy)
    \]
    and note that $\tau$ preserves the symmetry
    \[
        \tau\big(\llbracket x \rrbracket, \llbracket s \rrbracket\llbracket x \rrbracket\big)= \tau\Big(\llbracket s \rrbracket\llbracket x \rrbracket, \llbracket s \rrbracket^{-1}\big(\llbracket s \rrbracket\llbracket x \rrbracket\big)\Big) =\tau\big(\llbracket s \rrbracket\llbracket x \rrbracket, \llbracket x \rrbracket\big).
    \]

    Condition \eqref{all} imply that $c''$ is $\p$-a.s. finite and there exists of a (finite) geodesic path. Observe that $\llbracket s^{-1} \rrbracket = \llbracket s \rrbracket^{-1}$ for all $s \in S$ induces a graph homomorphism of $\mathcal{C}(\Gr,S)$ and $\mathcal{C}(\Gr',\llbracket S\rrbracket^\pm)$. In other words, for all $w_1, w_2 \in \llbracket w \rrbracket$ and $i,j \in \{1, \dots, \kappa\}$, $z_{(j)}.w_1 \not\sim z_{(i)}.w_2$ and if $x \sim y$ in $\mathcal{C}(\Gr, S)$, then $\llbracket x \rrbracket \sim \llbracket y \rrbracket$ in $\mathcal{C}\big(\Gr',\llbracket S \rrbracket^{\pm}\big)$.  Hence, one can easily verify by the minimax property that
    \begin{align*}
        c''(\llbracket x \rrbracket) &:= \max_{1 \leq i,j \leq \kappa}\max_{\substack{y \in z_{(j)}.\llbracket x \rrbracket\\ z \in z_{(i)}.\tor N}}\left( \inf_{\gamma \in \mathscr{P}(e, y)} \sum_{\{u,v\} \in \gamma} \tau(u,v)\right)\circ\tht_z\\
        &\phantom{:}= \inf_{\gamma \in \mathscr{P}(\llbracket e\rrbracket, \llbracket x \rrbracket)}\left( \sum_{\{\llbracket u\rrbracket,\llbracket v\rrbracket\} \in \gamma} \max_{1 \leq i, j \leq \kappa} \max_{\substack{~u' \in z_{(j)}.\llbracket u \rrbracket~\\ s' \in z_{(i)}.\llbracket vu^{-1} \rrbracket}}\tau(u',s'u')\right) \quad \p\text{-a.s.}
    \end{align*}

    This is a direct consequence of the graph homomorphism. Property \eqref{innerness2} arises naturally from the given definition when $\llbracket S \rrbracket^\pm \subseteq \Gr'\setminus[\Gr',\Gr']$.
\end{proof}

\begin{proposition} \label{prop:c.doubleprime}
    Under the same hypotheses stated in \cref{lm:N.ergodic}, it follows that the results in \cref{lm.T.bounded.nilpotent.as,lm.T.bounded.torsion.as,prop:asympt.equiv.c.cPrime,lm:subaddtive.ergodic,rmk:c.prime} also hold when replacing $c'$ with $c''$.
\end{proposition}
\begin{proof}
    Notice that 
    \[\p\left(\max_{x \in B_S(e,n)}\max_{~y \in \bigcup_{j=1}^\kappa z_{(j)}.\tor N}c(y)\circ \tht_{x} > \sqrt{n}\right) \in \mathcal{O}(1/n^\upkappa).\]

    Consequently, $\max\limits_{x \in B_S(e,n)}\max\limits_{~y \in \bigcup_{j=1}^\kappa z_{(j)}.\tor N}c(y)\circ \tht_{x} \in o(n)$, $\p$-a.s., as $n \uparrow +\infty$. Therefore, defining $c''\big( \llbracket e \rrbracket \big) = 0$ is a suitable choice for investigating the asymptotic cone of $c''$ in comparison to $c$.
    
     The arguments in the proofs of \cref{lm.T.bounded.nilpotent.as,lm.T.bounded.torsion.as,prop:asympt.equiv.c.cPrime,lm:subaddtive.ergodic} can be repeated for $c''$, yielding the same properties up to a constant factor.
\end{proof}

\begin{corollary} \label{cor:fpp.virt.nil}
    Let $(\Gr,.)$ be a finitely generated group with polynomial growth rate $D \geq 1$ and $\Gr'/[\Gr',\Gr']$ torsion-free. Consider $c:\Gr\times\Om \to \R_{\geq0}$ to be a subadditive cocycle associated with $d_\omega$ and a p.m.p. ergodic group action $\tht\big\vert_N:N \curvearrowright (\Om,\F,\p)$. 
    
    Suppose that $c$ describes a FPP model which satisfies conditions \eqref{all} and \eqref{aml2} for a finite symmetric generating set $S \subseteq \Gr$ so that 
    \begin{itemize}
        \item[(i)] For all $s \in S$, $\llbracket s^{-1} \rrbracket = \llbracket s \rrbracket^{-1}$, and
        \item[(ii)] $\llbracket S \rrbracket^{\pm} \subseteq \Gr' \setminus [\Gr',\Gr']$ generates $\Gr'$.
    \end{itemize}
    
    Then
    \begin{equation*} 
        \quad \quad \left(\Gr,\frac{1}{n}d_\omega,e\right) \GHto \left(G_\infty,d_\phi,\mathlcal{e}\right) \quad\quad \p\text{-a.s.}
    \end{equation*}
    where $G_\infty$ is a simply connected graded Lie group, and $d_\phi$ is a quasimetric homogeneous with respect to a family of homotheties $\{\delta_t\}_{t>0}$. Moreover, $d_\phi$ is bi-Lipschitz equivalent to $d_\infty$ on $G_\infty$.
\end{corollary}
\begin{proof}
    First, according to \cref{prop:c.doubleprime}, the random variables $c'$ and $c''$ share similar properties. Observe that $|\mathfrak{B}_N|=1$, ensuring that $c''$ is well-defined and a suitable replacement of $c'$ in the proof of \cref{thm:shape.polynomial}, which establishes the result.
\end{proof}

The next example highlights a case where $\tht\big\vert_N$ acts ergodically on the probability space followed by an example of virtually nilpotent group with generating set satisfying items (i) and (ii) of \cref{cor:fpp.virt.nil}.

\begin{example}[Independent FPP models] \label{exmpl:independent.FPP}

The subadditive cocycle $c$ exhibits equivariance. Recall properties discussed in \cref{sec:fpp} for FPP models and notice that, for all $x,y \in \Gr$ and $s \in S$,
\[\tau(x,sx) \sim \tau\big(y,s^{\pm1}y\big).\]

Consider that the random weights are independent, but not necessarily identically distributed (see \cite{benjamini2015} for FPP with i.i.d. random variables). Let us define $S' := \big\{ \{s,s^{-1}\}: s \in S \big\}$ and set $\varsigma(s'):= s \in s'$ for $s' \in S'$, \textit{i.e.}, the function $\varsigma$ fixes one element of each $s' \in S'$.

Suppose that, for all $s \in S$, $s^2 \neq e$ and consider $\nu^{(s')}$ to be the law of $\tau(x,\varsigma(s')x)$ with $x \in \Gr$ and  $s \in s' \in S'$. Thus, one can write
\[\p \equiv \left(\bigotimes_{s' \in S'}\nu^{(s')}\right)^{\otimes\Gr} = \left(\bigotimes_{j=1}^\kappa\bigotimes_{s' \in S'}\nu^{(s')}\right)^{\otimes N} \equiv \bigotimes_{x \in N} \nu^{(x)},\]
where, for each $x \in N$, $\nu^{(x)} \equiv \bigotimes_{j=1}^\kappa\bigotimes_{s' \in S'}\nu^{(s')}$. Let $E \in \F$ be such that, for all $x \in N$, $\tht_x(E)=E$. Then, for all $x,y \in N$,
\[\nu^{(x)}(E) = \nu^{(y)}(E)=: \mathtt{k}_E \in [0,1].\]

The condition of polynomial growth rate $D \geq 1$ ensures that $N$ is countably infinite. Consequently,
\[\p(E) = \prod_{x \in N} \mathtt{k}_E \in \{0,1\}.\]

Therefore, $\tht\big\vert_N$ as defined in \cref{sec:fpp} constitutes a probability measure-preserving (p.m.p.) ergodic group action for independent FPP models.
\end{example}

\begin{example}[Direct product]
    Consider $\mathrm{L}$ a torsion-free nilpotent group with torsion-free abelianization and a symmetric finite generating set $S_{\mathrm{L}} \subseteq \mathrm{L} \setminus [\mathrm{L},\mathrm{L}]$. Set $M$ to be a finite group. Recall the properties highlighted in \cref{sec:examples.virt.nil}. Let us define
    \[\Gr = \mathrm{L}\times M, \quad \text{and} \quad S = S_{\mathrm{L}} \times M.\]
    
    Then $S$ is a symmetric finite generating set of $\Gr$. Fix $\uppi_N(x,m) = (x, e)$ for all $x\in \mathrm{L}$ and $m \in M$. One can easily see that $\Gr' \cong \mathrm{L}$ with $\llbracket S \rrbracket = \llbracket S \rrbracket^{\pm} \cong S_{\mathrm{L}}$.

    Furthermore, for any $(x,m) \in \Gr$, the inverse $(x,m)^{-1}$ is given by $(x^{-1},m^{-1})$, leading to  \[\llbracket (x,m)^{-1} \rrbracket \cong x^{-1} \cong \llbracket (x,m) \rrbracket^{-1}.\]

    As a consequence, both items (i) and (ii) of \cref{cor:fpp.virt.nil} hold when $\Gr$ is the direct product equipped with the generating set $S$ defined above.
\end{example}

\section{Applications to Random Growth Models} \label{sec:examples}

In this section, we delve into three distinct examples that serve as applications of the main results outlined in this monograph for a random growth on $\mathcal{C}(\Gr,S)$. These examples have been deliberately chosen to address scenarios that fall outside the scope of previous works, thereby offering a nuanced examination of the versatility and robustness of our established theorems.

The first example considers a First-Passage Percolation (FPP) model with dependent random variables, challenging the assumption of $L^\infty$, since we allow random weights to be zero with a strict positive probability. Transitioning to the second example, we investigate a FPP model with independent random variables that are not identically distributed and also not $L^\infty$. The third example shifts focus to an interacting particle system that is not a FPP model. Notably, this model fails to meet the conditions found in the literature.

\begin{example}[First-Passage Percolation for a Random Coloring of $\Gr$]\label{ex:color}
Let us now consider a dependent Bernoulli FPP model based on the random coloring studied by Fontes and Newman \cite{fontes1993}. Set $\{X_x\}_{x \in \Gr}$ to be a family of i.i.d. random variables taking values in a finite set of colors $\mathlcal{F}$. The model generates color clusters by assigning weight $0$ to edges between sites with same color and weight $1$ otherwise. We define for every edge $u \sim v$
\[
    \tau(u,v) = \mathbbmtt{1}(X_{u} \neq X_{v}) ,
\]
Set for each self-avoiding path $\gamma \in \mathscr{P}(x,y)$ the random length $T(\gamma)= \sum_{\mathtt{e} \in \gamma}\tau(\mathtt{e})$. The first-passage time is
\[T(x,y) := \inf_{\gamma \in \mathscr{P}(x,y)}T(\gamma)\]

Let $p_{\mathlcal{s}}:= \p(X_x = \mathlcal{s})$ then one can verify that $T(x,y)$ is a FPP model with dependent identically distributed passage times $\tau(x,y) \sim \operatorname{Ber}\left(1-\sum_{\mathlcal{s} \in \mathlcal{F}}p_{\mathlcal{s}}^2\right)$. One can easily see that $c(x):= T(e,x)$ is a subadditive cocycle and the translations $\tht$ are ergodic due to the fact that $\{X_x\}_{x \in \Gr}$ are i.i.d. random variables. 

Observe that $c(x)$ is bounded above by the word norm $\|x\|_S$, items \eqref{all} and \eqref{innerness} are immediately satisfied. Consider $p_{\mathlcal{s}} \in (0,1)$ for all $\mathlcal{s} \in \mathlcal{F}$. Set
\begin{equation*}
    p:= \max_{\mathlcal{s} \in \mathlcal{F}} ~p_{\mathlcal{s}}~, \quad q :=  \max_{\mathlcal{s} \in \mathlcal{F}} ~(1-p_{\mathlcal{s}}), \quad\text{ and }\quad p' :=\frac{p}{p+q}.
\end{equation*}
The lemma below establishes a sufficient condition for \eqref{aml} and \eqref{aml2}.

\begin{lemma} \label{lm:aml.colors}
    Consider the Random Coloring Model of $\Gr$ on $\mathcal{C}(\Gr,S)$ satisfying
    \begin{equation} \label{eq:bound.colors}
        p < \frac{1}{|S|-1},
    \end{equation}
    then \eqref{aml} and \eqref{aml2} hold true.
\end{lemma}
\begin{proof}
    Let $\gamma = (x_0=e, x_1, \dots, x_n) \in \mathscr{P}_{n}$ with $\mathscr{P}_n$ the set of all self-avoiding paths in $\mathcal{C}(\Gr,S)$ of graph length $n$ starting at $e$. Fix $[n]:=\{1, \dots, n\}$, then
    \begin{align*}
    \p\big(T(\gamma) = m\big) &\leq \sum_{\substack{A \subseteq [n]\\ |A|=m}}\prod_{i \in [n]\setminus A}\mathbb{P}(X_{x_i}=X_{x_{i-1}}|X_{x_{i-1}}) \prod_{j \in A}\mathbb{P}(X_{x_i}\neq X_{x_{i-1}}|X_{x_{i-1}})\\ &\leq \binom{n}{m}p^{n-m}q^m = (p+q)^n P(Y= m)
    \end{align*}
    where $Y \sim \operatorname{Binomial}(n, 1-p')$ with respect to $P$. Let us regard $\|x\|_S=n$, thus
    \begin{align*}
        \p\big(c(x) \leq \upalpha \|x\|_S\big) &\leq \p\big(\exists \gamma \in \mathscr{P}_n:T(\gamma) \leq \upalpha n\big)\\ &\leq \big|\mathscr{P}_n\big|(p+q)^n \cdot P(Y \leq \upalpha n).
    \end{align*}
    It is a well-known fact that $|\mathscr{P}_n| \leq |S|(|S|-1)^{n-1}$. Therefore, there exists $\mathtt{C}>0$ such that $|\mathscr{P}_n| \leq \mathtt{C}(|S|-1)^n$. By Chernoff bound, one can obtain
    \begin{align}
        P(Y \leq \upalpha n) &\leq \exp\left( n\left(\upalpha-1)\log\frac{1-\upalpha}{1-p'} -\upalpha \log\frac{\upalpha}{p'}\right) \right) \nonumber \\
        & =\left((p')^\upalpha(1-p')^{(1-\upalpha)}\upalpha^{-\upalpha}(1-\upalpha)^{\alpha-1}\right)^n. \label{eq:binom.color}
    \end{align}
    Observe that the base of \eqref{eq:binom.color} converges to $p'$ as $\upalpha \downarrow 0$. Hence, there exist $\upalpha, p''>0$ such that $P(Y \leq \upalpha n) \leq \left(p''\right)^n$ with $p'<p''< 1/((p+1)(|S|-1))$ when $p$ satisfies \eqref{eq:bound.colors}. It then follows that there exists $\mathtt{C}'>0$ such that
    \[\p\big(c(x) \leq \upalpha \|x\|_S\big) \leq \mathtt{C} \big(p''(p+q)(|S|-1) \big)^n = \mathtt{C} \exp(-\mathtt{C}'n).\]

    Let now $a:=\upalpha/2$ and choose $\|x\|_S \gg 1$ so that $\p\big(c(x) \leq \upalpha\|x\|_S\big) \leq 1/2$, then $a\|x\|_S \leq \E[c(x)]$, which yields \eqref{aml} and \eqref{aml2} as a consequence.
\end{proof}

Similarly to \cref{exmpl:independent.FPP}, let $\nu$ be the law of the random coloring of a vertex. Then \[\p\equiv \nu^{\otimes\Gr} = \Bigg(\bigotimes_{j=1}^\kappa \nu\Bigg)^{\otimes N} \equiv\bigotimes_{x\in N}\nu^{(x)}\] with $\nu^{(x)} \equiv \bigotimes_{j=1}^\kappa \nu$. By the same reasoning employed for $\nu^{(x)}$ in \cref{exmpl:independent.FPP}, we verify that $\tht\big\vert_N$ acts ergodically on $(\Om, \F, \p)$.

Hence, under the assumption of \eqref{eq:bound.colors} and based on the aforementioned results, the Shape Theorems \ref{shape.thm} and \ref{thm:shape.polynomial} are applicable to the random coloring of $\Gr=N$ nilpotent with a finite generating set $S \subseteq N \setminus \big( [N,N] \cup \tor N\big)$ or in the case where $\Gr'$ is abelian. Moreover, under the fulfillment of conditions (i) and (ii) in \cref{cor:fpp.virt.nil}, the existence of the limiting shape is also guaranteed when $\Gr$ is virtually nilpotent.

\begin{remark}
    Observe that \eqref{eq:bound.colors} provides a lower bound for the critical probability of site percolation on $\mathcal{C}(\Gr,S)$ (see for instance \cite{grimmett1998}). To verify that, fix a color $\mathlcal{s} \in \mathlcal{F}$ and we say that a site $x \in \Gr$ is open when $X_x=\mathlcal{s}$. Therefore, one can write
    \[\tau_{\operatorname{site}}^{(\mathlcal{s})}(x,y)= \mathbbmtt{1}(X_x\neq \mathlcal{s}\text{ or }X_y\neq \mathlcal{s}).\]
    Note that it stochastically dominates with $\tau \leq \tau_{\operatorname{site}}^{(\mathlcal{s})}$ $\p$-a.s. The open edges are the new edges of length zero. By \cref{lm:aml.colors}, we can apply \cref{shape.thm} to obtain that, $\p$-a.s., there is no infinite open cluster in $\mathcal{C}(\Gr,S)$ when $p_{\mathlcal{s}}< \frac{1}{\vert S \vert-1}$.
\end{remark}

\end{example}

\begin{example}[Richardson's Growth Model in a Translation Invariant Random Environment] \label{ex:richardson}

In this example, we define a variant of the Richardson's Growth Model which is commonly employed to describe the spread of infectious diseases. This version of the model involves independent random variables that are not identically distributed (see \cite{garet2012,richardson1973} for similar models).  Specifically, we consider that the transmission rate of the disease between neighboring sites is randomly chosen. The distribution of this variable will vary depending on the directions of the Cayley graph.

Consider that the infection rates between neighbors are determined by a random environment taking values in $\Uplambda:=(0,+\infty)^E$.  Let $S':=\big\{\{s,s^{-1}\}:s \in S\big\}$ be the set of directions of $\mathcal{C}(\Gr,S)$. Consider $\{\uplambda_{s'}\}_{s' \in S'}$ a set of strictly positive random variables that are independent over a probability measure $\nu$. Set $\big(\uplambda(\mathtt{e})\big)_{\mathtt{e} \in E}$ to be a collection of independent random variables over $\nu$ such that
\[
    \uplambda(x,sx) \sim \uplambda_{s'} \quad \text{with }s'= \{s^{\pm 1}\}.
\]

Let us regard $\uplambda \in \Uplambda$ as a fixed realization of the random environment. The growth process is defined by the family of independent random variables ${\{\tau(x,sx): x\in \Gr, s \in S\}}$ such that 
\begin{equation} \label{eq:t.Richardson}
    \tau(x,sx) \sim \text{Exp}\big(\uplambda(x,sx)\big).
\end{equation}

Set $\p_\uplambda$ to be the quenched probability law of \eqref{eq:t.Richardson}. We write, for each path $\upgamma \in \mathscr{P}(x, y)$ with $x,y \in \Gr$, its random length $T(\upgamma):= \sum_{\mathtt{e} \in \upgamma}\tau(\mathtt{e})$.

The first-passage time is
\[
c(x):= \inf_{\upgamma \in \mathscr{P}(e,x)}T(\upgamma).
\]
It is straightforward to see that $c(x)$ is subadditive. However, the group action $\tht$ is not ergodic over $\p_\uplambda$ for a given $\uplambda \in \Uplambda$. Let $\p(\cdot) = \int_\Uplambda \p_\upgamma(\cdot)d\nu(\uplambda)$ be the annealed probability. It then follows that $\tht$ preserves the measure $\p$ and it is ergodic.

Note that $c(x)$ defines a First-Passage Percolation (FPP) model, which we refer to as Richardson's Growth Model in a Translation Random Environment (RGTRE). In the following, we establish that conditions \eqref{all}, \eqref{aml}, and \eqref{aml2} are met.

\begin{lemma} \label{lm:richardson.all}
    Consider the RGTRE defined as above. Then there exist $\upbeta, \mathtt{C},\mathtt{C}'>0$ such that, for all $x \in \Gr$,
    \[
        \p\big(c(x) \geq t\big) \leq \mathtt{C} \exp\big(-\mathtt{C}'t\big)
    \]
    for all $t \geq \upbeta \|x\|_S$.
\end{lemma}
\begin{proof}
    Let $\upgamma \in \mathscr{P}(e,x)$ be a $d_S$-geodesic with $\|x\|_S=n$. Then one has by Chernoff bound and the independence of $\{\tau(\mathtt{e})\}_{\mathtt{e}\in E}$ that 
    \[
        \p_\uplambda\big(c(x) \geq t\big) \leq \p_\uplambda\big(T(\upgamma) \geq t\big) \leq \frac{\prod_{\mathtt{e} \in \gamma}\E_\uplambda\left[e^{\alpha \tau(\mathtt{e})}\right]}{e^{\alpha t}}.
    \]
    where
    \[
        \E_\uplambda\left[e^{\alpha \tau(\mathtt{e})}\right]= \sum_{m=0}^{+\infty} \left(\frac{\alpha}{\uplambda(e)}\right)^m
    \]
    Let $\thickbar{\uplambda}_{\min} := \min_{s'\in S'} \E[\uplambda_{s'}]$ and set $\alpha= \thickbar{\uplambda}_{\min}/2$. Thus, by the Dominated Convergence Theorem,
    \[
         \p_\uplambda\big(c(x) \geq t\big) \leq \frac{2^n}{e^{\thickbar{\uplambda}_{\min}t/2}}.
    \]

    Therefore, it suffices to choose $\upbeta> 2\log(2)/\thickbar{\uplambda}_{\min}$ to complete the proof.
\end{proof}

\begin{lemma} \label{lm:richardson.aml}
     Consider the RGTRE defined as above. The there exists $a>0$ such that, for all $x \in \Gr$,
     \[
        a \|x\|_S \leq \E[c(x)].
     \]
\end{lemma}
\begin{proof}
    It is a well-known fact that, for all $\uplambda \in \Uplambda$ and every $\mathtt{e}\in E$, that $\p_\uplambda\big(\tau(\mathtt{e})=0\big) =0$ and, therefore,  $\p\big(\tau(\mathtt{e})=0\big) =0$. By the right continuity of the cumulative distribution function, one can find $\delta>0$  and $p \geq 0$ such that, for every $\mathtt{e} \in E$,
    \[
        \p\big(\tau(\mathtt{e}) < \delta\big) =p < \frac{1}{|S|-1}.
    \]
    
    We use similar arguments as those employed in the proof of \cref{lm:aml.colors}, we may consider $Y \sim \operatorname{Binomial}(n, 1-p)$ over $P$. Then there exists $\upalpha>0$ such that, for any $\upgamma \in \mathscr{P}_n$,
    \[\p\big(T(\upgamma)\leq \upalpha n \big) \leq P\big(Y \leq \upalpha n /\delta\big) \leq p^n.\]

    It follows that there exists $C>0$ such that, for $\|x\|_S=n$,
    \[
        \p\big(c(x) \leq \alpha \|x\|_S \big) \leq \big\vert \mathscr{P}_n \big\vert \cdot P\big(Y \leq \upalpha n /\delta\big) \leq C \big((|S|-1)p\big)^n.
    \]
    Since $(|S|-1)p <1$, we can complete the proof by following the same steps as in  \cref{lm:aml.colors}.
\end{proof}

It follows from \cref{lm:richardson.all,lm:richardson.aml} that conditions \eqref{all}, \eqref{aml}, and \eqref{aml2} are satisfied. Observe that $\tht\big\vert_N$ acts ergodically on the probability space (see \cref{exmpl:independent.FPP}). 

Therefore, building upon the preceding results, the Limiting Shape Theorems \ref{shape.thm} and \ref{thm:shape.polynomial} apply to the RGTRE with $\Gr=N$ nilpotent with a finite generating set $S \subseteq N \setminus \big( [N,N] \cup \tor N\big)$ or in scenarios where $\Gr'$ is abelian. Additionally, when $\Gr$ is virtually nilpotent and conditions (i) and (ii) from \cref{cor:fpp.virt.nil} are satisfied, the existence of the limiting shape is also assured.

\end{example}

\begin{example}[The Frog Model]\label{ex:frog}

The Frog Model, originally introduced by Alves et al. \cite{alves2002} and previously featured as an example in \cite{telcs1999},  is a discrete-time interacting particle system determined by the intersection of random walks on a graph. In this model, particles, often representing individuals, are distributed across the vertices and they can be in either active (awake) or inactive (sleeping) states. At discrete time steps, active particles perform simple random walks, while inactive ones remain stationary. The activation of an inactive particle occurs when its vertex is visited by an active counterpart, thereby characterizing an awakening process. This straightforward yet potent model serves as a valuable tool for analyzing diverse dynamic processes, such as the spread of information and disease transmission.

In our previous study Coletti and de Lima \cite{coletti2021},  we investigated the Frog Model on finitely generated groups.  We can now extend our findings to virtually nilpotent groups as a consequence of \cref{thm:shape.polynomial}. Let us define the model in detail. The initial configuration of the process at time zero begins with one particle at each vertex and the only active particle lies on the origin $e \in \Gr$.

Set $S_n^x$ to be the simple random walk on $\mathcal{C}(\Gr,S)$ of a particle originally placed at $x \in \Gr$ and let $t(x, y)$
be the first time the random walk $S_n^x$ visits $y\in\Gr$, \textit{i.e.}, it defined the random variable $t(x, y) = \inf\{n \in \N_0: S_n^x=y\}$. Note that $t(x, y) = +\infty$ with strictly positive probability when $D \geq 3$.

The activation time of the particle originally positioned at $x$ is given by the random variable
\[T(x) = \inf\left\{ \sum_{i=1}^m t(x_{i-1},x_i) \colon m\in\N, ~\{x_i\}_{i=1}^m \subseteq\Gr, ~x_0=e \right\}.\]
Observe that $x_{i-1}$ and $x_i$ are not necessarily neighbours. We proved in \cite{coletti2021} that $c(x)=T(x)$ is a subadditive cocycle with respect to the translation $\tht$, which is p.m.p. and an ergodic group action. Futhermore, $\tau^{(e)}(x,sx)= |T(x)-T(sx)|$ is not identically distributes as in the FPP models (see \cref{sec:fpp}).

Due to the discrete time random walks, $T(x) \geq \|x\|_S$ and therefore \eqref{aml2} is immediately satisfied. The at least linear growth in virtually nilpotent group was already investigated in \cite{coletti2021}. Hence, condition \eqref{all} is a consequence of the following result.

\begin{lemma}[Prop. 2.10 of \cite{coletti2021}]
    Let $\Gr$ be a group of polynomial growth rate $D \geq 3$ with a symmetric finite generating set $S \subseteq \Gr \setminus\{e\}$. Then there exists $\mathtt{C},\varkappa>0$ and $\upbeta>1$ such that, for all $x \in \Gr$ and every $t > \upbeta \|x\|_S$, one has
    \[\p\big(T(x) \geq t\big) \leq \mathtt{C} \exp(-t^\varkappa).\]
\end{lemma}

Consider now $\Gr$ as a group with polynomial growth rate $D \geq 3$ generated by a symmetric finite set $S \subseteq \Gr\setminus\{e\}$. According to \cref{thm:shape.polynomial} and the preceding results, it can be inferred that the Frog Model on $\mathcal{C}(\Gr,S)$ exhibits a limiting shape when $\big\langle \llbracket S \rrbracket\big\rangle$ generates an abelian group $\Gr'$. This phenomenon can be exemplified by the generalized dihedral group $\Gr = \operatorname{Dih}(N)$ when $N$ is a finitely generated abelian group with polynomial growth rate $D \geq 3$ (see \cref{ex:dihedral}).

\end{example}

%% file: texts/fpp_rgg.tex
\chapter{On Random Geometric Graphs} \label{ch:rgg}

\vspace{1.3cm}

In the early 1960s, Gilbert \cite{gilbert1961} introduced a mathematical model for wireless networks, laying the groundwork for continuum percolation theory and extending the concepts of discrete percolation.  This model, often referred to as the \textit{Gilbert disk model}, is characterized by the uniform distribution of points across the infinite plane $\mathbb{R}^2$ through a homogeneous Poisson point process (PPP) with intensity $\lambda>0$. The point process determines the vertices of the graph and the edges are defined between any pair of vertices that are within an Euclidean distance smaller than a fixed threshold $r>0$. In this context, we define the random graph in $\mathbb{R}^d$ with $d \geq 2$, and following Penrose \cite{penrose2003}, we refer to it as \textit{random geometric graph} (RGG). This class of graphs is associated to the Poisson--Boolean model in continuum percolation, and it can also be seen as a particular case of the random-connection model (see for instance Meester and Roy \cite{meester1996}).

In this chapter, we present the definition and parameters for random geometric graphs (RGGs) and the existence of the infinite connected component. We also show some results about its geometry in order to study the asymptotic shape in the following chapters. Throughout the text, we assume $\|\cdot\|$ to be the Euclidean norm on $\R^d$, and let $\|\cdot\|_1$ and $\|\cdot\|_\infty$ stand for the $\ell^1$ and $\ell^\infty$ norms, respectively.

\section{Basic Definitions and Auxiliary Results} \label{sec:basic.defs_RGGs}

Let $\Poi_\lambda$ be the random set of points determined by the homogeneous PPP on $\R^d$ with intensity $\lambda>0$. The RGG $\G=(V,\mathcal{E})$ on $\R^d$ is defined by
\[
  V = \Poi_\lambda \quad \text{and} \quad
  \mathcal{E} = \big\{\{u,v\} \subseteq V: \|u-v\|<r,~ u \neq v\big\}.
\]
Since $\lambda^{-\frac{1}{d}}\mathcal{P}_\lambda \sim \mathcal{P}_1$, we consider $\lambda$ as a fixed parameter and allow $r$ to vary due to the homogeneity of the norm. Thus, unless specified otherwise, we set $\lambda=1$ and denote $\mathcal{P}_1$ as $\mathcal{P}$. Let $\left(\Upxi,\F,\mu\right)$ denote the probability space induced by the construction of $\mathcal{P}$. Let us now introduce the group action $\tht:\R^d \curvearrowright\Upxi$ which is determined by the spatial translation as a shift operator. That is, $\Poi\circ\tht_z = \big\{v-z :v \in \Poi \big\}$. The following lemma is a classical result on PPPs, which can be found for example in Meester and Roy \cite[Prop. 2.6]{meester1996}.

\begin{lemma} \label{lm:PPP.mixing}
    The homogeneous PPP is mixing on $(\Upxi,\F,\mu,\tht)$.
\end{lemma}

\begin{remark} \label{rmk:isometry}
    Let $S:\R^d\to\R^d$ be an isometry. Then, it is known that $S$ induces a $\mu$-preserving ergodic function $\Tilde{\sigma}: \Upxi \to \Upxi$ where $S[\Poi] = \Poi \circ \ \Tilde{\sigma}$.
\end{remark}

We aim to investigate the spread of infection within an infinite connected component of $\G$. It is a well-known fact from continuum percolation theory (see Meester and Roy \cite{meester1996} or Penrose \cite[Chapter~10]{penrose2003} for details) that, for all $d \geq 2$, there exists a critical $r_c(\lambda)>0$ (or $r_c$ for $\lambda=1$) such that $\G$ has an infinite component $\Hh$ $\mu$-\textit{a.s.} for all $r>r_c$. Moreover, $\Hh$ is $\mu$-\textit{a.s.} unique. As $\Hh$ is a subgraph of $\G$, we denote by $V(\Hh)$ and $E(\Hh)$ its sets of vertices and edges, respectively. To simplify notation, we often use $\Hh$ to represent the set of vertices $V(\Hh)$. 

For our purposes, it suffices to note that $r_c\geq 1/{\upupsilon_d}^{1/d}$ where $\upupsilon_d$ denotes the volume of the unit ball in the $d$-dimensional Euclidean space. Indeed, more precise lower and upper bounds can be found in Torquato and Jiao \cite{torquato2012} and $r_c$ approximates to $1/{\upupsilon_d}^{1/d}$ from above as $d \uparrow +\infty$.

Define $\mathbbm{B}(t)$ as the hypercube $[-t/2, t/2]^d$ and consider the Euclidean ball denoted as $B(x,t) := \{y \in \R^d \colon \|y-x\| < t\}$. Let us fix $\theta_{r}$ as $\mu \big( B(o,r) \cap \Hh \neq \varnothing \big)$. The following proposition presents a fundamental result concerning the volume of $\Hh$, which is a weaker version of Theorem 1 in Penrose and Pisztora~\cite{penrose1996}.

\begin{proposition}\label{prop:Hn.growth}
    Let $d \geq 2$, $r>r_c$ and $\varepsilon \in (0,1/2)$. Then, there exists $\mathtt{c}_0>0$ and $t_0>0$ such that, for all $t \geq s_0$,
    \[
        \mu\left((1-\varepsilon)\theta_{r} <\frac{|\Hh \cap \mathbbm{B}(t) |}{t^d}  <(1+\varepsilon)\theta_{r} \right) \geq 1 -\exp(-\mathtt{c}_0t^{d-1}).
    \]
\end{proposition}

As a consequence of the last result, we present the following lemma without proof (see Yao et al. \cite[Lemma~3.3]{yao2011}).

\begin{lemma} \label{lm:ball.intersecting.infinite.comp}
    Let $r>r_c$. Then, there exists $C,C'>0$ such that, for each $x \in \R^d$ and all $s>0$,
    \[\mu\big(B(x,s) \cap \Hh= \varnothing\big) \leq C\exp(-C's^{d-1}).\]
\end{lemma}

Define $\mathscr{P}(x,y)$ as the set of self-avoiding paths from $x$ to $y$ in $\G$. The simple length of a path $\gamma = (x=x_0, x_1, \dots, x_m=y) \in \mathscr{P}(x,y)$ is denoted by $|\gamma|=m$. Let $\thickbar{q}: \R^d \to \G$ be a function defined as follows:
\begin{equation} \label{def:q_bar.function}
    \thickbar{q}(x) := \argmin_{y \in V}\{\|y-x\|\}
\end{equation}
which determines the closest point to $x$ in $V$.
Note that from \eqref{def:q_bar.function}, $\thickbar{q}$ may have multiple values for certain $x \in \mathbb{R}^d$. In such cases, we presume that $\thickbar{q}(x)$ is uniquely defined by arbitrarily selecting one outcome of \eqref{def:q_bar.function}.

Let $\thickbar{D}(x,y)$ stand for the $\G$-distance between $x,y \in \R^d$ given by
\[
    \thickbar{D}(x,y) = \inf\{|\gamma|:\gamma \in \mathscr{P}\big(\thickbar{q}(x),\thickbar{q}(y)\big)\}.
\]

Similarly to \eqref{def:q_bar.function}, we set $q: \R^d \to \Hh$ to be a function that determines the closest point to $x$ in $\Hh$ defined as
\begin{equation*}
    q(x) := \argmin_{y \in \Hh}\{\|y-x\|\}.
\end{equation*}
Hence $q$ induces a Voronoi partition of $\R^d$ with respect to $\Hh$, see Figure~\ref{fig:RGG_Voronoi} for an illustration.

\begin{figure}[!htb]
    \centering
    \includegraphics[scale=0.235]{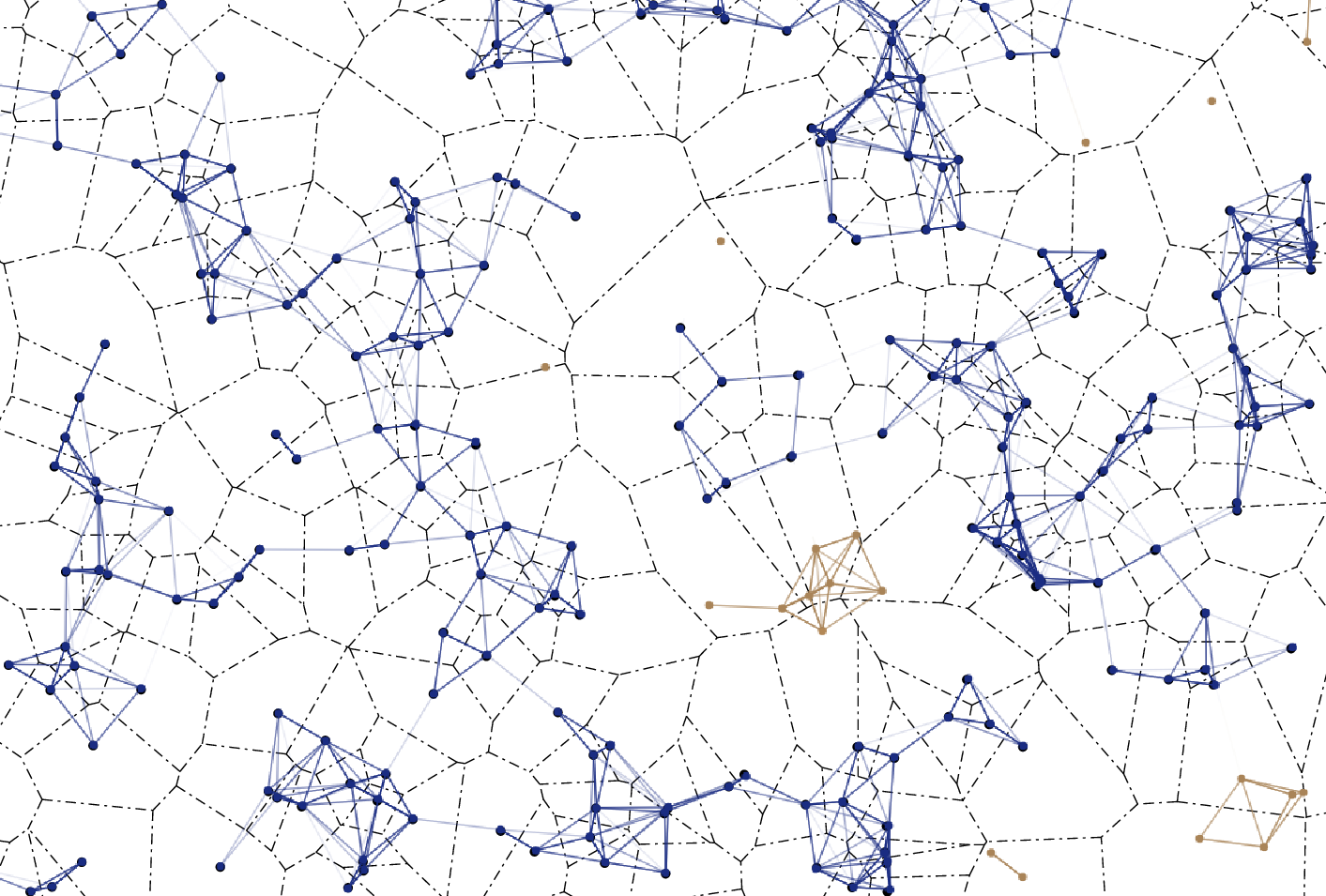}
    \caption{A random geometric graph on $\R^2$ with the Voronoi partition generated by the infinite connected component $\Hh$ (in blue). Fig. from \cite{coletti2023}.}
    \label{fig:RGG_Voronoi}
\end{figure}

We restrict the domain of $\thickbar{D}$ to $\Hh$ by defining $D(x,y):=\thickbar{D}(q(x),q(y))$ for every $x,y \in \R^d$. The proposition below can be immediately adapted from the proof of~Yao et al. \cite[Thm 2.2]{yao2011} by applying properties of Palm calculus and Lemma \ref{lm:ball.intersecting.infinite.comp}.
\begin{proposition}\hspace{-0.1cm}{\rmfamily{(Adapted from~Yao et al. \cite[Thm. 2.2]{yao2011}).}} \label{prop:linearity.chem.dist}
    Let $d \geq 2$ and $r>r_c$. Then there exists $\varrho_{r} > 0$ depending on $r$ such that, $\mu$-\textit{a.s.}, for all $x \in \R^d$,
    \[
        \lim_{\substack{\|y\|\uparrow + \infty}}\frac{D(x,y)}{\|y-x\|} = \varrho_{r}.
    \]
\end{proposition}

The constant $\varrho_{r}$ is called \emph{stretch factor} of $\Hh$. Observe that $\varrho_{r} \geq 1/r$. Due to the subadditivity of the $\Hh$-distance, one can easily see that $\E_\mu[D(o, z)]$ with $\|z\|=1$ is an upper-bound for $\varrho_{r}$.

Let $u, v \in \R^d$, we define the Palm measure $\mu_v$ as the convolution of the measure $\mu$ with the Dirac measure $\delta_v$. Specifically, $\mu_v := \mu \ast \delta_v$. Similarly, $\mu_{u,v}$ stands for $\mu \ast \updelta_u \ast \updelta_v$. Set $\mathscr{C}(x)$ to be the connected component of $\thickbar{q}(x)$ in $\G$. We have the following result about the tail behaviour of $D(o, z)$.

\begin{lemma} \label{lm:chem.all}
    Let $d \geq 2$ and $r>r_c$. Then, there exist $C,C'>0$ and $\beta^\dagger >1$ such that, for all $x \in \R^d$ and every $t> \beta^\dagger \|x\|$,
    \[
        \mu\big(D(o,x) \geq t\big)  \leq C\exp(-C' t).
    \]
\end{lemma}
\begin{proof}
    Let $\thickbar{D}$ be the simple $\G$-distance. It is clear that $D(v,w)=\thickbar{D}(v,w)$ whenever $v,w\in \Hh$. By \cite[Lemma 3.4]{yao2011}, there exist $\widetilde{C},\widetilde{C}'>0$ and $\beta^\dagger>1$ such that
    \begin{equation} \label{eq:Palm.bound}
        \mu_{v,w} \big(v \in\mathscr{C}(w) , \text{~and~} \thickbar{D}(v,w) \geq t\big) \leq \widetilde{C} \exp(-\widetilde{C}'t)
    \end{equation}
    for all $t \geq \beta^\dagger \|v-w\|/2$. Consider now $\mathlcal{B}_r(z):= B(z,r) \cap \mathcal{P}$. We apply Lemma \ref{lm:ball.intersecting.infinite.comp},  \eqref{eq:Palm.bound}, and Campbell's theorem to obtain that there exist $\thickbar{C},\thickbar{C}'>0$ such that
    \begin{align} 
        \mu\big(D(o,x) \geq t\big) &\begin{multlined}[t]\leq \mu\big(\|q(o)\| \geq t/(2\beta')\big) + \mu\big(\|q(x)-x\| \geq t/(2\beta')\big) \\+ \mu\left(\bigcup_{\substack{v \in \mathlcal{B}_{t/(2\beta')}(o), \, w \in \mathlcal{B}_{t/(2\beta')}(x)}}\{v \in \mathscr{C}(w) \text{~and~} \thickbar{D}(v,w) \geq t\}\right)\end{multlined} \nonumber\\
        &\leq 2 \thickbar{C} \exp\big({-\thickbar{C}'t/(2\beta')}\big) + \widetilde{C}\frac{ \upupsilon_d^2}{2^{2d}}t^{2d} \exp({-\widetilde{C}'t}) \label{eq:chem.prob.upper.bound}
    \end{align}
    for all $t \geq \beta^\dagger \|x\|$.
    Hence, we can conclude the proof of \eqref{eq:chem.prob.upper.bound} by suitably choosing the constants $C,C'>0$.
\end{proof}

Define $\operatorname{Cover}(t):= \mathbbm{B}(t) \cap \left( \bigcup_{v \in \Hh}B(v,r)\right)$ as the coverage of $\Hh$ in $\mathbbm{B}(t)$. A connected region in $\mathbbm{B}(t) \setminus \operatorname{Cover}(t)$ is called a hole in $\mathbbm{B}(t)$. Let $v \in \mathbbm{B}(t)\cap \Hh$. The Euclidean ball $B(v,t')$ is an spherical hole when $B(v,t) \cap \operatorname{Cover}(t) = \varnothing$. The largest diameter of a spherical hole in $\mathbbm{B}(t)$ is denoted by
\[
    \mathcal{D}(t) := \sup \left\{ t' \geq 0 \colon v \in \mathbbm{B}(t), ~B(v,t'/2) ~\text{is a spherical hole in}~ \mathbbm{B}(t) \right\}.
\]

The following proposition is an adapted version of Theorem 3.3 of Yao and Guo~\cite{yao2015}:

\begin{proposition} \label{prop:holes.H}
    Let $d \geq 2$ and $r>r_c$. Then there exist $\Cl[small]{cte.l.bound.holes}, \Cl[small]{cte.u.bound.holes}>0$ such that
    \[\mu\Big( \Cr{cte.l.bound.holes} \cdot \log(t) < \mathcal{D}(t) < \Cr{cte.u.bound.holes} \cdot\log(t) \Big) \geq 1- \frac{1}{t^2}. \]
\end{proposition}

The subsequent result appears as Lemma 4 in Yao \cite{yao2013} and as Lemma 3.4 in Yao, Chen, and Guo~\cite{yao2011}.

\begin{lemma} \label{lm:PPP.clusters}
    Let $d \geq 2$ and $r>r_c$. Then there exist $C,C'>0$ such that, for each $t>0$,
    \begin{equation*}
    \mu_v\Big(v \not\in \Hh , ~\mathscr{C}(v) \not\subseteq B(v,t)\Big) \leq C \exp(-C' t).
    \end{equation*}
    
    Furthermore, there exist $\Cl[small]{dist.bar.exp.ext},\Cl[small]{dist.bar.exp.int}>0$ and $\beta' >1$ such that, for all $u,v \in \R^d$ and every $t> \beta' \|u-v\|$,
    \begin{equation*}
        \mu_{u,v}\big(\thickbar{D}(u,v) \mathbbm{1}_{u \in \mathscr{C}(v)}> t\big)  \leq \Cr{dist.bar.exp.ext}\exp(-\Cr{dist.bar.exp.int} t).
    \end{equation*}
\end{lemma}

\chapter{The existence of the Limiting Shape of FPP Models on RGGs} \label{ch:shape.fpp.rgg}

\vspace{2cm}
First-Passage Percolation (FPP) was initially introduced to study the spread of fluids through random medium (see \cref{sec:fpp} for details).  Since then, several variations of the percolation process have been extensively investigated (see Auffinger et al. \cite{auffinger2017} for an overview of FPP models) due to their considerable amount of theoretical consequences and applications. It determines a random metric space by assigning random weights to the edges of a graph.

We consider the FPP model defined on a random geometric graph (RGG) in $\R^d$ with $d \geq 2$. Here, the RGG is defined as in \cref{ch:rgg}. Recall that the infinite component $\Hh$ is unique almost surely in the supercritical case. We define the FPP model on $\Hh$ with independent and identically distributed random variables on a joint probability space $(\Om, \A,\p)$.

The aim of this chapter is to investigate the $\p$-\textit{a.s.} existence of the limiting shape of the above defined process. In fact, we show that, under some conditions, the random balls of $\Hh$ converge $\p$-\textit{a.s.} to the deterministic shape of an Euclidean ball. The additional conditions refer to the distribution of zero passage time on the edges and the at-least linear growth of the process.

Let $\tau$ be a random variable which defines the common distribution of the i.i.d.~passage times $\tau_e$ along each edge $e \in \mathcal{E}$. Recall that $r_c(\lambda)>0$ is the critical $r$ for the existence of the infinite connected component $\Hh$ of the RGG $\G$. Let $\upupsilon_d$ be volume of the unit ball in $d$-dimensional Euclidean space. Denote by $H_t$ the random subset of $\R^d$ of points for which their closest point in $\Hh$ is reached by the FPP model up to time $t>0$. We let $H_0$ be the set of points that have the same closest point in $\Hh$ as the origin. Here is our first main theorem.

\begin{theorem}[Shape theorem for FPP on RGGs] \label{thm:shape.FPP}
    Let $d \geq 2$ and $r>r_c(\lambda)$. Consider the FPP with i.i.d.~random variables
    defined on the infinite connected component $\Hh$ of  $\G$. Suppose that the following conditions are satisfied:
    \begin{itemize}
    \item[$({A}_1)$] \ltext{${A}_1$}{A1}
    We have that
    \[
        \p(\tau=0) < \frac{1}{\upupsilon_d r^d \lambda}.
    \]
    
    \item[$({A}_2)$] \ltext{${A}_2$}{A2}
    There exists $\eta>2d+2$ such that
    \[
        \E[\tau^\eta] < +\infty.
    \]
\end{itemize}
Then, there exists $\upvarphi \in (0, +\infty)$ such that, for all $\varepsilon \in (0,1)$, one has $\p$-\textit{a.s.} that
    \begin{equation} \label{eq:asymptotic.cone}
        (1-\varepsilon) B(o,\upvarphi) \subseteq \frac{1}{n}H_{n} \subseteq (1+\varepsilon) B(o,\upvarphi)
    \end{equation}
    for sufficiently large $n \in \N$.
\end{theorem}

The existence of the limiting shape is particularly interesting because the RGG is a random graph which exhibits, $\p$-a.s., unbounded degree and holes with unbounded diameter. To avoid the possible extreme effects of such pathologies on the growth of the process, we control the growth almost surely by combining the conditions above with properties of the point process. Additionally, we note that the probability measure $\p$ is given by $\mu \otimes \nu$, where $\mu$ represents the measure associated with the PPP and $\nu$ denotes the measure related to the passage times. For further details, interested readers can refer to the construction of the probability space outlined in the proof of \cref{lem_apply_conc}.

The interest of applications for this class of models has already been pointed out by Jahnel and K\"onig \cite{jahnel2020}. In particular, they suggested the theorem for the Richardson model on telecommunication networks. The example is naturally associated with the contact process by stochastic domination as studied by M\'enard and Singh\cite{menard2016}, and Riblet \cite{riblet2019}. Another interesting application is a lower bound for the critical probability of bond percolation on the RGG. The same lower bound can be obtained by other methods (\textit{e.g.} branching processes), however it shows in comparison how good and suitable condition \eqref{A1} is.

It is worth pointing out that one can find in the literature a bigger class of random geometric graphs studied by Hirsch et al. \cite{hirsch2015} where the graph distance was interpreted as a FPP model. It suggests that the class of RGGs could also be expanded in our case. We chose to focus our attention on the standard definition in this work due to the usage of intermediate results presented in the next section.

\section{Intermediate Results}

Let us proceed with defining the first-passage percolation model on  $\G=(V,\mathcal{E})$ and i.i.d. passage times $\{\tau_e: e \in \mathcal{E}\}$. Recall properties of the RGG $\G$ presented in \cref{sec:basic.defs_RGGs} and fix $\lambda=1$. We introduce $\thickbar{T}$ as a random variable called first-passage time such that, for all $x, y \in \R^d$,  we have
\[\thickbar{T}(x,y) := \inf\left\{\sum_{e \in \gamma} \tau_e \ \ \colon \ \ \gamma \in \mathscr{P}(\thickbar{q}(x),\thickbar{q}(y))\right\}.\]

The first-passage time on $\Hh$ is defined as
\[T(u,v) := \thickbar{T}(q(u),q(v))\] 
For brevity, we denote $T(x) = T(q(o),q(x))$. To streamline our focus on the shape theorem itself and on the techniques employed in Chapter \ref{ch:moderate}, we have omitted the proofs of \cref{lm:FPP.mean.lower.bound,lm:all.growth,lm:FPP.seq.conv} from this text. Interested readers can find these proofs in detail in our published paper \cite{coletti2023}.

\begin{lemma}[{\hspace{-0.7pt}\cite[Lm.~3.1]{coletti2023}}] \label{lm:FPP.mean.lower.bound}
    Let $d \geq 2$, $r>r_c$, and $\p(\tau=0) < 1/(\upupsilon_d r^d)$. Then, there exists $a>0$ depending on $r$ such that, for all $x \in \R^d$,
    \[
        a \|x\| \leq \E[T(x)].
    \]
\end{lemma}
\begin{remark} \label{rmk:FPP.mean.upper.bound}
    Observe that {$\E[T(x)] \leq \E[D(o,x)]\E[\tau]$ due to the subadditivity and Fubini's Theorem. Moreover, condition \eqref{A2} implies that $\E[\tau] < +\infty$. One can easily see from Proposition \ref{prop:linearity.chem.dist} and $L^1$ convergence given by Kingman's Subadditive Ergodic Theorem (\cref{thm:kingman}) applied to the $\Hh$-distance, that, for all $x \in \R^d\setminus\{o\}$,
    \[
        b:= \varrho_{r}\E[\tau] \geq \lim_{n \uparrow +\infty} \E[T(nx)]/\|nx\|.
    \]}
\end{remark}

Let $\Upxi'$ stand for the event where $\Hh$ exists. Denote by $\p_\xi$ the quenched probability of the propagation model given a realization $\xi \in \Upxi'$. The lemma below ensures the at least linear growth of the passage times.
 
\begin{lemma}[{\hspace{-0.7pt}\cite[Lm.~3.2]{coletti2023}}] \label{lm:all.growth}
    Let $d \geq 2$, $r>r_c$, and assume that \eqref{A2} holds. Then, there exist deterministic $\beta>0$ and $\upkappa>1$ such that, for every $x,y \in \R^d$, and for each $\xi \in \Upxi'$,
    \[
        \p_{\xi}\big(T(x,y) \geq t\big) \leq  t^{-(d+\upkappa)}
    \]
    for all $t \geq \beta D(x,y)$.
\end{lemma}

Before proving our first main theorem of \cref{part:rgg}, we state and prove the following result. It is an annealed version of the at least linear growth from lemma above in all directions.

\begin{lemma}[{\hspace{-0.7pt}\cite[Lm.~3.3]{coletti2023}}] \label{lm:FPP.seq.conv}
    Let $d \geq 2$, $r>r_c$. Consider the i.i.d.~FPP on the RGG satisfying \eqref{A2}. Then, there exist constants $\delta, C > 0$, and $\upkappa>1$ such that for all $t>0$ and all $x \in \R^d$, one has
   \[
        \p\left(\sup_{y \in B(x,\delta t)} T(x,y) \ge t\right) \le C t^{-\upkappa}.
    \]
\end{lemma}

The forthcoming results will be utilized in Chapter \ref{ch:moderate}, where we enhance \eqref{A2} with a stronger assumption, subsequently denoted as \eqref{A2p}. Let's recall the definition of the Palm measures $\mu_{u,v}$ and consider $\p_{u,v} := \mu_{u,v} \otimes \nu$.

\begin{lemma} \label{lm:T.bar.all}
    Let $d \geq 2$ and $r>r_c$, and suppose that $\E[e^{\eta \tau}] < +\infty$ for some $\eta>0$. Then, there exist $\Cl[small]{cte.ext.Tbar},\Cl[small]{cte.int.Tbar}>0$ and $\thickbar{\beta} >1$ such that, for all $u,v \in \R^d$ and every $t> \thickbar{\beta} \|u-v\|$,
    \[
        \p_{u,v}\big(\thickbar{T}(u,v) \mathbbm{1}_{u \in \mathscr{C}(v)} \geq t\big)  \leq \Cr{cte.ext.Tbar}\exp(-\Cr{cte.int.Tbar} t).
    \]
\end{lemma}
\begin{proof}
    Fix $\beta'' > \log\big(\E[e^{\eta\tau}]^{1/\eta}\big)$ with $\eta>0$ satisfying \eqref{A2}, then
    \begin{equation} \label{eq:control.chernoff}
        \frac{\E[e^{\eta\tau}]}{e^{\eta \beta''}}<1.
    \end{equation}
    Consider $u,v \in \R^d$ to be fixed and define the event
    \[\thickbar{E}(t) := \{\thickbar{D}(u,v) \leq t\}.\]
    
    Let $\gamma \in \mathscr{P}(u,v)$ be a geodesic given by the $\thickbar{T}$-distance on $\thickbar{E}(t)$. We apply the Chernoff bound to verify that
    \begin{align}
        \p_{u,v}\big(\big\{\thickbar{T}(u,v) \geq t \big\} \cap \thickbar{E}(t) \big) &\leq  \p_{u,v}\big(\thickbar{T}(\gamma) \geq t \big) \nonumber \\
        & \leq \frac{\E[e^{\eta \tau}]^{D(x,y)}}{e^{ \eta t}} \leq \left(\frac{\E[e^{\eta \tau}]}{e^{\eta \beta'' }}\right)^{t/\beta''} \label{eq:all.DbarS}
    \end{align}
    with $t \geq \beta'' \thickbar{D}(o,x)$. Hence the result follows from \cref{lm:PPP.clusters}, \eqref{eq:control.chernoff}, and \eqref{eq:all.DbarS} since
    \[
        \p_{u,v}\big(\big\{\thickbar{T}(u,v)\mathbbm{1}_{u \in\mathscr{C}(v)}  \geq t \big\}\big) \leq \big({\E[e^{\eta\tau}]}{e^{-\eta \beta''}}\big)^{t/\beta''} + \Cr{dist.bar.exp.ext}\exp(-\Cr{dist.bar.exp.int} t/\beta'')
    \]
    for $t \geq \thickbar{\beta} \|u-v\|$ with $\thickbar{\beta} = \beta'\beta''$.
\end{proof}

\begin{lemma} \label{lm:T.bds}
    Let $d \geq 2$, $r>r_c$. Consider that $\E[e^{\eta \tau}] < +\infty$ for some $\eta>0$. Then there exist $\Cl[small]{T.exp.ext}, \Cl[small]{T.exp.int}>0$ and $\beta>1$ such that, for all $x \in \R^d$ and every $t> \beta \|x\|$,
    \begin{equation*} 
        \p(T(x)\geq t) \leq \Cr{T.exp.ext} \exp(-\Cr{T.exp.int} t), 
    \end{equation*}
\end{lemma}
\begin{proof}
    The proof of the lemma mirrors that of \cref{lm:T.bar.all}. Here, we replace $\thickbar{E}(t)$ with $E(t) := \{D(o,x) < t\}$ and $\thickbar{\beta}$ with $\beta := \beta^\dagger\beta''$. The result follows as a consequence of applying \cref{lm:chem.all}.
\end{proof}

\begin{lemma} \label{lm:T.bds.sup}
    Let $d \geq 2$, $r>r_c$, and $\beta, \thickbar{\beta}>0$ be as defined in \cref{lm:T.bar.all,lm:T.bds}. If $\E[e^{\eta \tau}] < +\infty$ for some $\eta>0$, then there exist $c,c'>0$ such that, for all $t'>1$ and any $t>\beta t'$,
    \begin{equation} \label{eq:supT.ball}
        \p\left(\sup_{\|w\| < t'}T(w) > t\right) \leq c \exp(-c' t') + c (t')^d \exp(-c't).
    \end{equation}
    Moreover, there exist $C,C'>0$ such that, for all $t,t'>1$,
    \begin{equation}  \label{eq:supT.annulus}
        \p\left(\sup_{\substack{z \in \Hh\cap B(o,t')\\y \in \Hh\cap B(z,t)}} T(z,y) > \thickbar{\beta} t\right) \leq \exp(-C' t') + C (t't)^d \exp(-C't).
    \end{equation}
\end{lemma}
\begin{proof}
    Consider the constants random variables from \cref{prop:Hn.growth} to define the event
    \[E= \left\{ \|q(o)\| \le t'/2 \text{ and } |\Hh\cap B(o,2t')| < 2^d\theta_r\cdot(t')^d\right\}.\]
    Hence, by \cref{prop:Hn.growth} and \cref{lm:T.bds},
    \begin{align*}
        \p\left(\sup_{\|w\| < t'}T(w) > t\right) &\le \p(E^c) + \p\left( \left( \bigcup_{z \in \Hh \cap B(o,2t')} \hspace{-10pt}\left\{ T(z) > t\right\}\right) \cap E \right)  \\
        &\le 2 e^{-\frac{\mathtt{c}_0}{4} t'} + 2^d\theta_r\Cr{T.exp.ext}(t')^d e^{-\Cr{T.exp.int} t},
    \end{align*}
    which proves \eqref{eq:supT.ball}. Let us now define
    \begin{align*}
        E_1'&= \left\{ \theta_r (t')^d/2 < |\Hh\cap B(o,t')| < 3\theta_r\cdot(t')^d/2\right\}, \text{ and} \\
        E_2'&= \left\{ 
        \text{for all } z \in \Hh\cap B(o,t'), \text{ we have }|\Hh\cap B(z,t)| \le \theta_r\cdot t^d \right\}.
    \end{align*}
    Observe that, by Mecke's formula, \cref{prop:Hn.growth,lm:T.bar.all},
    \begin{align}
		\nonumber &\p\left(\sup_{\substack{z \in \Hh\cap B(o,t')\\y \in \Hh\cap B(z,t)}} T(z,y) > \thickbar{\beta} t\right) \\
		\nonumber &\le \p\big((E_1')^c\big) + \p\big((E_2')^c \cap E_1\big) + \E \left[ \sum_{\substack{z \in \Hh\cap B(o,t') \\ y \in \Hh\cap B(z,t)}}\hspace{-10pt}\p_{z,y}\left(T(z,y)> \thickbar{\beta}t\right)\mathbbm{1}_{E_1'\cap E_2'} \right] \\
        \nonumber &\le e^{-\frac{\mathtt{c}_0}{2}t'} + \theta_r(t')^d e^{-\frac{\mathtt{c}_0}{2}t} + \theta_r^2(t't)^d \Cr{cte.ext.Tbar} e^{-\Cr{cte.int.Tbar}\thickbar{\beta} t}.
	\end{align}
        This inequality establishes \eqref{eq:supT.annulus} as asserted.
\end{proof}

\section{Proof of the Standard Shape Theorem}

After this preparatory work, we now proceed to prove Theorem~\ref{thm:shape.FPP}. The methods are closely related to standard techniques for shape theorems which can be found in \cite{kesten1986}, for instance.
\begin{proof}[Proof of \cref{thm:shape.FPP}]
    We begin by verifying properties of $T(nx)$. Note that for every $x \in \R^d$ one has that $\E[T(x)] < +\infty$ by \cref{lm:chem.all,lm:all.growth}. Recall that the process is mixing on $(\Om, \A,\p,\tht)$ by \cref{lm:PPP.mixing}. Then, by the subadditivity of $T$, we apply Kingman's subadditive ergodic theorem to obtain that, $\p$-\textit{a.s.}, for all $x \in \R^d$,
    \begin{equation} \label{eq.FPP.conv.T(nx)}
        \lim_{n\uparrow+\infty}\frac{T(nx)}{n} = \phi (x),
    \end{equation}
    where $\phi:\R^d\to [0,+\infty)$ is a homogeneous and subadditive function given by
    \[
        \phi(x) = \inf_{n\geq 1} \frac{\E[T(nx)]}{n} = \lim_{n\uparrow+\infty} \frac{\E[T(nx)]}{n}.
    \]
    
    Since the process is rotation invariant, there exists a constant $\upvarphi$ (the time constant) such that $\phi(x)=\upvarphi^{-1}\|x\|$ for all $x \in \R^d$. In fact, one has from Lemma \ref{lm:FPP.mean.lower.bound} and Remark \ref{rmk:FPP.mean.upper.bound} that
    \[
        0<a \leq \upvarphi^{-1} \leq b= \varrho_{r}\E[\tau] < \infty.
    \]


   Let us now prove the $\p$-\textit{a.s.} asymptotic equivalence
   \begin{equation} \label{eq:asymptotic.equivalence}
        \lim_{\|y\|\uparrow+\infty} \frac{T(y)}{\|y\|} = \frac{1}{\upvarphi}.
    \end{equation}

    For the approach from below, we prove the equivalent statement that for every $\epsilon \in (0,1)$,
    \[
        \limsup_{s \uparrow +\infty}\quad \sup_{\| y \| \le (1-\epsilon)s}\frac{T(y)}{s} = \limsup_{m \in \N ,m \uparrow +\infty}\quad \sup_{\| y \| \le (1-\epsilon)m}\frac{T(y)}{m} \le \frac{1}{\upvarphi} \quad \p-a.s.,
    \]
    where the first equation holds as $\lfloor s \rfloor/s$ converges to $1$.
    Fix $\epsilon \in (0,1)$  and let $\delta$ be given by Lemma \ref{lm:FPP.seq.conv}. Due to compactness, there exists a finite cover of open balls with centers $(y_i)_{i \in \{1,\dots,n\}} \subseteq \R^d$ with $\|y_i\| \le 1-\epsilon$ such that
    \[
        \overline{B(o,1 - \epsilon)} \subseteq \bigcup_{{i \in \{1,\dots,n}\}} B\big(y_i,\delta\epsilon/(2\upvarphi)\big).
    \]
    Furthermore $B\big(o,m(1-\epsilon)\big) \subseteq \bigcup_{{i \in \{1,\dots,n}\}} B\big(my_i, m\delta\epsilon/(2\upvarphi)\big)$ for every $m \in \N$.
    Applying Lemma $\ref{lm:FPP.seq.conv}$ we obtain
    \[
       \sum_{m \in \N} \p\left( \sup_{\|y- my_i\| \le m \delta\epsilon/(2\upvarphi)} T(m y_i, y) > m \epsilon/(2\upvarphi) \right) < \infty.
    \]
    Therefore, by the Borel--Cantelli lemma,
    \[
      \limsup_{m \in \N, m \uparrow +\infty} \quad \sup_{\|my_i-y\| \le m \delta\epsilon/(2\upvarphi)} \frac{T(my_i,y)}{m} <  \frac{\epsilon}{2\upvarphi} \quad \p\text{-a.s.}
    \]
    Applying \eqref{eq.FPP.conv.T(nx)} and subadditivity, we obtain
    \begin{align*}
       \limsup_{m \uparrow +\infty; \| y \| \le (1-\epsilon)m}\frac{T(y)}{m} &\le 
       \limsup_{m \uparrow +\infty} \begin{multlined}[t]\bigg(\max_{i \in \{1,\dots n\}} \frac{T(o,my_i)}{m} \\ \hspace{50pt}+ \sup_{\|my_i-y\| \le m \delta\epsilon/(2\upvarphi)} \frac{T(my_i,y)}{m}\bigg)\end{multlined}\\
       &\le \max_{i \in \{1,\dots, n\}}\|y_i\|/\upvarphi + \epsilon/(2\upvarphi) < 1/\upvarphi \quad \p\text{-a.s.},
    \end{align*}
    where we used that $\|y_i\| < 1-\epsilon$.
    
    For the approach from above, define  $A_t := B\big(o,t(1+2\epsilon)\big) \setminus B\big(o,t(1+\epsilon)\big)$ 
   and observe that it suffices to prove
    \[
        \liminf_{m \in \N, m \uparrow +\infty}\quad \inf_{y \in {A}_t}\frac{T(y)}{m} \ge \frac{1}{\upvarphi} \quad \p\text{-a.s.}
    \]
    for arbitrary but fixed $\epsilon>0$, as for (for $t > \epsilon$) any $x$ with $\|x\|>t(1+2\epsilon)$ there exists an $\tilde x \in A_t$ with $T(\tilde x) \le T(x)$.
    
    Similar as in the approach from below, fix $\epsilon > 0$ and $\delta > 0$ small enough that Lemma \ref{lm:FPP.seq.conv} holds. There exists a set of centers $(y_i)_{i \in \{1,\dots,n\}} \subseteq \R^d$ with $\|y_i\| \ge 1+\epsilon$ such that
        \[
        A_t \subseteq \bigcup_{{i \in \{1,\dots,n}\}} B\big(y_i,\delta\epsilon/(2\upvarphi)\big)
    \]
and hence
    \begin{align*}
       \liminf_{m\in \N,m \uparrow +\infty}\quad\inf_{y \in {A}_m}\frac{T(y)}{m} &\ge 
        \liminf_{m\in \N,m \uparrow +\infty} \begin{multlined}[t]\bigg(\min_{i \in \{1,\dots n\}} \frac{T(o,my_i)}{m} \\ \hspace{40pt} - \sup_{\|my_i-y\| \le m \delta\epsilon/(2\upvarphi)} \frac{T(my_i,y)}{m}\bigg)\end{multlined}\\
       &\ge\min_{i \in \{1,\dots n\}}\|y_i\|/\upvarphi - \epsilon/(2\upvarphi) > 1/\upvarphi,
    \end{align*}
    which concludes the proof of the asymptotic equivalence \eqref{eq:asymptotic.equivalence}. The proof of the theorem is now complete by standard arguments of the $\p$-\textit{a.s.} uniform convergence given by \eqref{eq:asymptotic.equivalence}.
\end{proof}

\section{Illustrative Examples}

In this section, we delve into specific applications of the Shape Theorem established in previous chapters.

\begin{example}[Bond percolation]
    We define the bond percolation by considering the clusters of the Bernoulli FPP only at time zero (see Figure~\ref{fig:bond_perc_RGG}). For this, let us call $e \in E(\Hh)$ an \emph{open edge} when $\tau_e =0$. Set $\tau_e \sim \operatorname{Ber}(1-p)$ independently for every $e \in E(\Hh)$ and observe that \eqref{A2} is immediately satisfied.
    
    Then, the open clusters are maximally connected components defined by sites with passage time zero between them. Let us define the critical probability $p_c$ for the bond percolation on the $d$-dimensional RGG by
    \[
        p_c := \inf\{p \in [0,1]: \p(\exists \text{ an infinite open cluster in } \Hh)>0,\tau_e \sim \operatorname{Ber}(1-p)\}.
    \]
    \begin{figure}[!ht]
        \centering
        \fbox{\includegraphics[width=325pt]{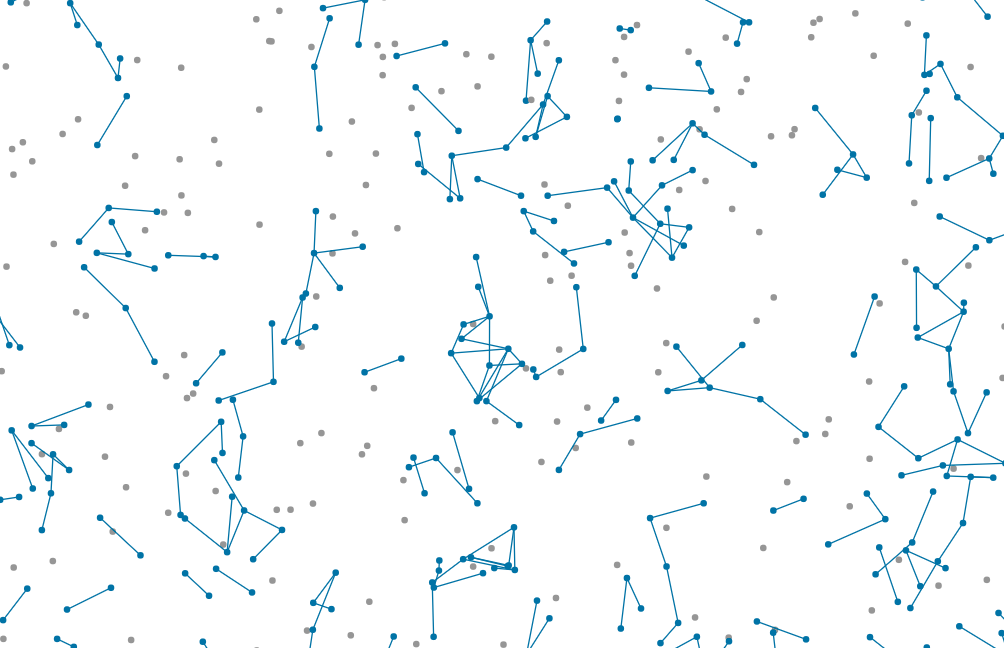}}
        \caption{Simulation of the open clusters for a bond percolation model on a 2-dimensional RGG with $p<1/(\upupsilon_d r^d\lambda)$. Fig. from \cite{coletti2023}.}
        \label{fig:bond_perc_RGG}
    \end{figure}
    Note that by Theorem \ref{thm:shape.FPP} the case $p< 1/(\upupsilon_d r^d \lambda)$ implies the existence of the limiting shape. Thus, an immediate consequence of the theorem is the following lower bound for the critical probability
    \[
        p_c \geq 1/(\upupsilon_d r^d\lambda),
    \]
    and for $p=0$ we recover $\Hh$. 
    We observe that the same lower bound can also be obtained by exploration methods.
\end{example}

\begin{example}[Richardson's growth model] 
    Consider the interacting particle system known as the Richardson model defined on the infinite connected component $\Hh$ of the RGG with parameter $\lambda_{\mathrm I}>0$. It is a random growth process based on a model introduced by Richardson \cite{richardson1973} and illustrated in Figure~\ref{fig:Richardson_RGG}. It is commonly referred to as a model for the spread of an infection or for the growth of a population.
    
    At each time $t \geq 0$, a site of $\Hh$ is in either of two states, healthy (vacant) or infected (occupied). Let $\zeta_t:\Hh\to \{0,1\}$ indicate the state of the sites at time $t$ assigning the values $0$ and $1$ for the healthy and infected states, respectively. The process evolves as follows:
    \begin{itemize}
        \item A healthy particle becomes infected at rate $\lambda_{\mathrm I}\sum_{y \sim x}\zeta_t(y)$ and
        \item An infected particle remains infected forever.
    \end{itemize}

    It is easily seen that the process is determined by FPP with edge passage times ${\tau_e \sim \mathrm{Exp}(\lambda_{\mathrm I})}$ independently for each $e \in E(\Hh)$. In particular, this version of the Richardson model conventionally stochastically dominates the basic contact process.
    \begin{figure}[!ht]
        \centering
        \includegraphics[scale=0.27]{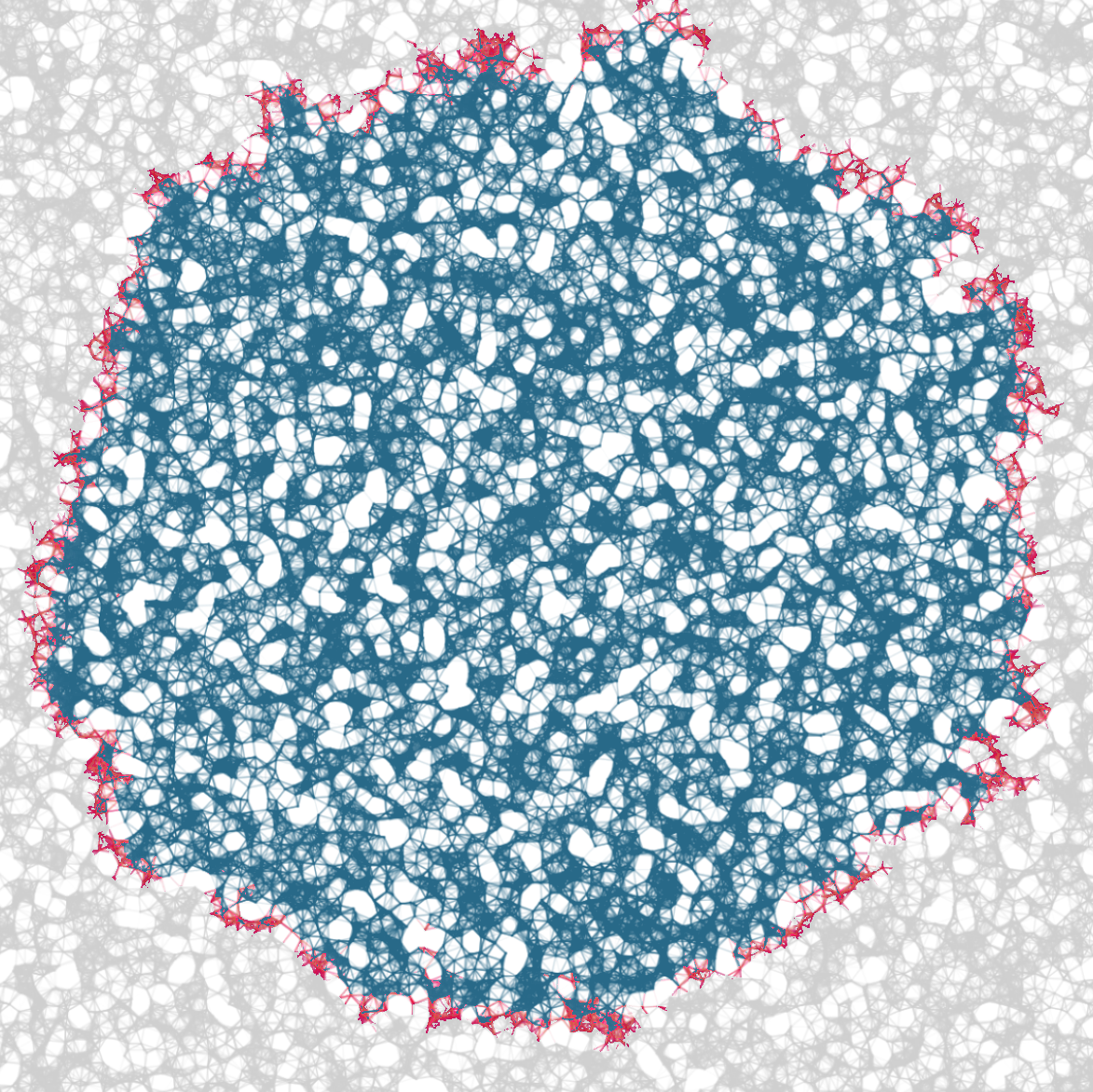}
        \caption{Simulation of the spread of an infection given by the Richardson model on a bidimensional RGG. Fig. from \cite{coletti2023}.}
        \label{fig:Richardson_RGG}
    \end{figure}

    Conditions \eqref{A1} and \eqref{A2} are straightforward since $\p(\tau=0)=0<1/(\upupsilon_d r^d\lambda)$ and since  $\E[\exp(\alpha\tau)] < +\infty$ for $\alpha \in (0, \lambda_{\mathrm I})$. Hence, Theorem~\ref{thm:shape.FPP} is valid for the Richardson model on $\Hh$ for any supercritical $r> r_c(\lambda)$.
    
    Futhermore, it is immediate to see that Theorem \ref{thm:shape.FPP} still holds for any initial configuration $\mathcal{Z} \subseteq \R^d$ of infected particles whenever $\mathcal{Z} \subseteq B(o,s')$ for some $s'>0$. In that case, we simply replace $H_t$ by $H_t^{\mathcal{Z}} := \bigcup_{z \in \mathcal{Z}}H_t^z$.
\end{example}

%% file: texts/moderate_deviation.tex
\chapter{Moderate Deviation for FPP Models on RGGs and its Implications} \label{ch:moderate}

\vspace{2cm}
In this chapter, we extend the research initiated in~\cref{ch:shape.fpp.rgg} and \cite{coletti2023} by conducting further analysis of first-passage percolation on random geometric graphs. 
Our primary focus is to examine the properties of geodesic paths and moderate deviations, which enhance our understanding of the convergence behavior of FPP models towards their limiting shape ---commonly referred to as the quantitative shape theorem.

Kesten \cite{kesten1993} established the groundwork for understanding the speed of convergence in FPP models on the hypercubic lattice $\Z^d$ using martingales. Recent advancements by Tessera \cite{tessera2018} have further refined these results through the application of Talagrand's concentration inequality. Additionally, similar investigations have been pursued for FPP on random structures. For example, Pimentel \cite{pimentel2011} explored FPP on two-dimensional Delaunay graphs, and Howard and Newman \cite{howard2001} studied Euclidean FPP.

Here, we establish a more informative version of the shape theorem under a slightly modified set of assumptions. Specifically, we assume that
\begin{itemize}
    \item[(${A}_1'$)] \ltext{${A}_1'$}{A1p} $\p(\tau=0) =0$;
    \item[(${A}_2'$)] \ltext{${A}_2'$}{A2p}
    $\E[e^{\eta \tau}] < +\infty$ for some $\eta>0$.
\end{itemize}
Our first main result is the following.
\begin{theorem}[Quantitative Shape Theorem for FPP on RGGs] \label{thm:speed.FPP}
	Let $d \geq 2$,~$\lambda > 0$ and $r>r_c(\lambda)$, and consider first-passage percolation on the random geometric graph on~$\R^d$ with parameters~$\lambda$ and~$r$, with passage times satisfying~\eqref{A1p} and~\eqref{A2p} above. 
Then, there exists $c'>0$ and $\upvarphi>0$ such that almost surely, for~$t$ large enough we have
    \begin{equation} \label{eq:asymptotic.cone2}
	    \left(1-c'\frac{\log(t)}{\sqrt{t}}{}\right) B(o, \upvarphi) \subseteq \frac{1}{t}H_{t} \subseteq \left(1+c'\frac{\log(t)}{\sqrt{t}}{}\right) B(o, \upvarphi).
    \end{equation}
\end{theorem}

We derive this theorem as a consequence of two results of independent interest, which we now present.

\begin{theorem}[Moderate Deviations of First-Passage Times] \label{thm_new_moderate_deviations}
	Consider first-passage percolation as in Theorem~\ref{thm:speed.FPP}, under the same assumptions as in that theorem. There exist~$C,C', c > 0$ such that for any~$x \in \mathbb R^d$ with~$\|x\|$ large enough, we have
	\begin{equation*}
		\p\left(\frac{|T(x) - \E[T(x)]|}{\sqrt{\|x\|}} > \ell \right) \le Ce^{-c \ell} \quad \text{for any } ~\ell \in \left[C'\log\big(\|x\|\big),\sqrt{\|x\|}\right].
	\end{equation*}
\end{theorem}

\begin{theorem}[Asymptotic Expectation and Variance of First-Passage Times] \label{thm:moderate.dev.FPP}
Consider first-passage percolation as in Theorem~\ref{thm:speed.FPP}, under the same assumptions as in that theorem. There exist~$\upvarphi > 0$ and~$C > 0$ such that, for~$x \in \R^d$ with~$\|x\|$ large enough,
	\begin{equation}\label{eq_asymp_exp}
		\frac{\|x\|}{ \upvarphi} \leq  \E\left[ T(x) \right]   \leq \frac{\|x\|}{ \upvarphi} + C \sqrt{\|x\|} \log\big(\|x\|\big)
	\end{equation}
	and
    \begin{equation} \label{eq:VarT}
    \operatorname{Var} T(x) \le C\|x\| \log \big(\|x\|\big).
    \end{equation}
\end{theorem}

Building upon the well-established understanding that the limiting shape is an Euclidean ball, we aim gaining valuable insights into the fluctuations of geodesic paths as a consequence of the moderated deviation of the passage times. Our approach is influenced by the seminal work of Howard and Newman \cite{howard2001}, from which we adapt their techniques to our model of FPP on random geometric graphs. This enables the investigation and quantification of fluctuations of the geodesic paths within our specific framework.

Let us define, for all $x, y \in \R^d$, the $T$-geodesics $\upgamma_x(y)$ to be chosen as follows. The path $\upgamma_x(y) = \big(q(x)=q_0,q_1, \dots, q_{n'}=q(y)\big)$ is a $T$-geodesic from $q(x)$ to $q(y)$ constructed inductively so that, for each $i \in \{0, 1, \dots, n'\}$, $\upgamma_x(y)$ is the concatenation of $\upgamma_x(q_i)$ and $(q_i, q_{i+1}, \dots, q_{n'})$.

In what follows, $\overline{xy}$ stands for the straight line segment from $x$ to $y$. Recall the definition of the Hausdorff distance, denoted as $d_H$, as introduced in \cref{sec:polynomial.growth.group}. Subsequently, we present a theorem that provides bounds for the fluctuations of the geodesics $\upgamma_x(-)$.

\begin{figure}[!htb]
    \centering
    \includegraphics[scale=0.28]{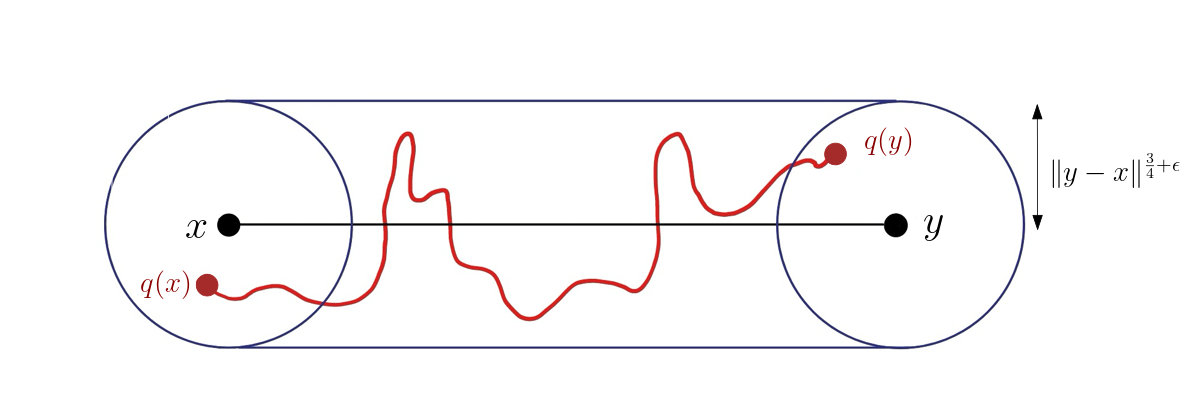}
    \caption{A geodesic path with its fluctuations in a cylindrical region of space given by \cref{thm:fluctuations.of.geodesics}.}
\end{figure}

\begin{theorem}[Fluctuations of the geodesics] \label{thm:fluctuations.of.geodesics}
    Let $d \geq 2$, $r>r_c(\lambda)$, and fix $\epsilon\in (0,1/4)$. Consider the FPP with i.i.d.~random variables
    defined on $\Hh$ of $\G$ satisfying \eqref{A1p} and \eqref{A2p}. Then there exist $\check{c}_1,\check{c}_2>0$ such that, for all $x,y \in \R^d$,
    \begin{equation} \label{eq:geodesic.fluctuations}
        \p\left(d_H\big(\upgamma_x(y),\overline{xy}\big) \geq \|y-x\|^{\frac{3}{4} + \epsilon}\right) \le \check{c}_1 \exp(-\check{c}_2 \|y-x\|^{2\epsilon}).
    \end{equation}
\end{theorem}

In what follows, we present a fundamental result concerning the properties of the spanning trees $\mathlcal{T}_x$ of $\Hh$ rooted at $q(x)$, where the edges of $\mathlcal{T}_x$ are induced by the union of geodesic paths $\upgamma_x(y)$ for all $y \in \R^d$. To characterize this spanning tree, we introduce several key concepts.

Let $\uptheta(x,y)$ denote the angle between the vectors $\vec{x}$ and $\vec{y}$. Define $\operatorname{Cone}(x, s)$ as the set of points $ y \in \R^d$ such that $\uptheta(x,y) \le s$ where $s$ varies in the interval $[0,\pi]$.

\begin{figure}[!htb]
    \centering
    \includegraphics[scale=0.38]{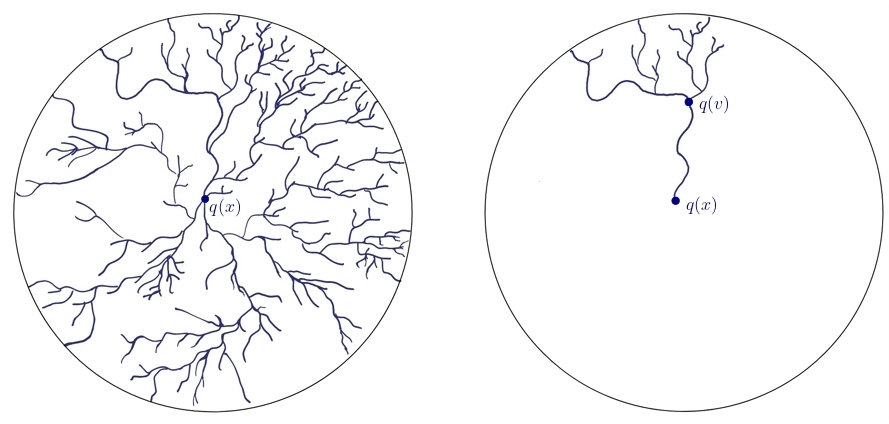}
    \caption{Section of a spanning tree $\mathlcal{T}_x$ (left) and the subtree given by $\mathlcal{T}_x^{\mathsf{out}}(v)$ (right).}
\end{figure}

We now introduce a subset of $\Hh$ which determines a subtree of $\mathlcal{T}_x$
\[\mathlcal{T}_x^{\mathsf{out}}(v) := \left\{u \in \Hh \colon v \in \upgamma_x(u)\right\}.\] 

Additionally, we define for a real function $f$
\[\mathlcal{V}_{x, f} := \left\{v \in \Hh ~\colon~ \mathlcal{T}_x^{\mathsf{out}}\big(v) \not\subseteq x + \operatorname{Cone}\big(v-x, f(\|v-x\|)\big)\right\}.\]

A tree $\mathlcal{T}_x$ is called $f$-straight at $x$ when $\mathlcal{V}_{x,f}$ is finite. The next theorem establishes conditions under which $\mathlcal{T}_x$ if $f$-straight.

\begin{figure}[!htb]
    \centering
    \includegraphics[scale=0.13]{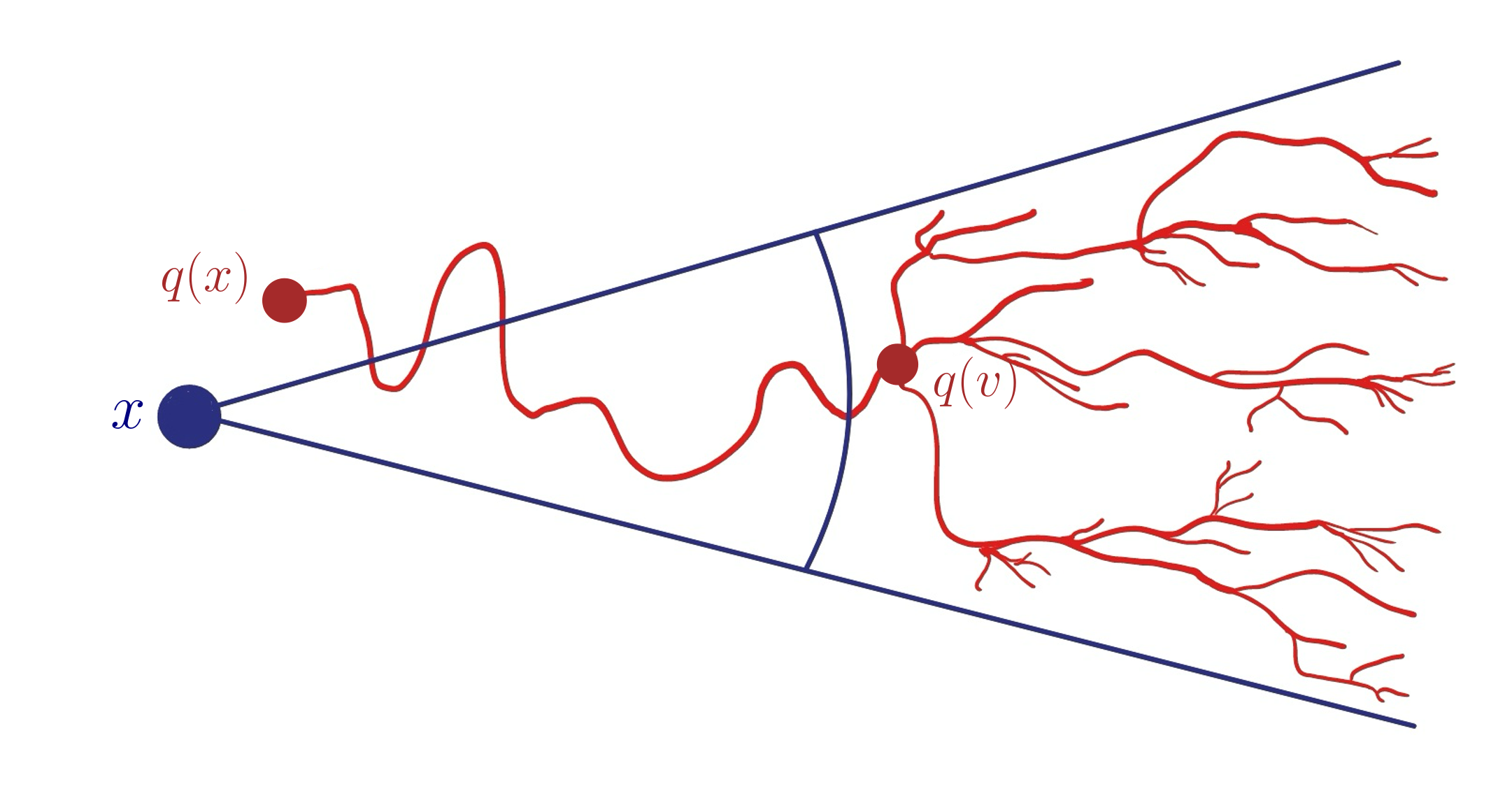}
    \caption{The subtree associated with $\mathlcal{T}_x^{\mathsf{out}}(v)$ approaching the $f_\epsilon$-straight property (see \cref{thm:f.straight.spanning.trees}).}
\end{figure}

\begin{theorem}[Asymptotic behaviour of the spanning trees $\mathlcal{T}_x$] \label{thm:f.straight.spanning.trees}
    Let $d \geq 2$ and $r>r_c(\lambda)$. Consider the FPP  defined on $\Hh$ with i.i.d.~random variables satisfying \eqref{A1p} and \eqref{A2p}. Fix $x \in \R^d$ and $\epsilon\in (0,1/4)$. Then there exists, $\p$-a.s., a finite $\mathlcal{F}_{x,f_\epsilon} \subseteq \R^d$ with $f_\epsilon(s)=s^{\epsilon-1/4}$ such that $\mathlcal{V}_{x,f_\epsilon} \subseteq \mathlcal{F}_{x,f_\epsilon}$ for any choice of $~\mathlcal{T}_x$. 
    
    Hence, all spanning trees of the type $\mathlcal{T}_x$ are $\p$-a.s. $f_\epsilon$-straight at $x$.
\end{theorem}

Now that we have established certain properties of the spanning trees, we will now delve into the properties of the asymptotic behavior of their semi-infinite paths. Let $\mathlcal{S}_x$ be the set of semi-infinite paths in $\mathlcal{T}_x$ starting from $q(x)$, \textit{i.e.},
\[\mathlcal{S}_x:=\Big\{\big(q(x), q_1, q_2, \dots\big) \in(\R^d)^{\N_0} \colon \forall n \in \N\big( \upgamma_x(q_n) = (q(x), q_1, \dots, q_n)\big)\Big\}.\]
Now, let's consider $\upsigma \in \mathlcal{S}_x$, where $\upsigma=(q(x), q_1, q_2, \dots)$.  We define $\eth$ as the asymptotic direction for the semi-infinite paths, which is given by
\[\eth(\upsigma) := \lim_{n \uparrow +\infty} \frac{q_n}{\|q_n\|}.\]
It is readily apparent that if the asymptotic direction $\eth(\upsigma)$ exists, then $\eth(\upsigma) \in\partial B(o,1)$.

\begin{corollary} \label{cor:asymp.dir}
    Consider $d \ge 2$ and $r> r_c(\lambda)$. Let the FPP be defined on $\Hh$ of $\mathcal{G}_{\lambda,r}$ such that \eqref{A1p} and \eqref{A2p} hold true. Then one has $\p$-a.s. that, for all $x \in \R^d$:
    \begin{enumerate}[(a)]
        \item Every semi-infinite path $\upsigma\in\mathlcal{S}_x$ has an asymptotic direction $\eth(\upsigma) \in \partial B(o,1)$.
        \item For all directions $\Hat{x} \in \partial B(o,1)$, there exist at least one semi-infinite path $\upsigma\in\mathlcal{S}_x$ with $\eth(\upsigma)=\Hat{x}$.
        \item The set of asymptotic directions with non-unique associated semi-infinite paths
        \[
            \mathlcal{D}_x =\big\{\Hat{x} \in \partial B(o,1)\colon\exists \upsigma_1,\upsigma_1 \in \mathlcal{S}_x \text{ with } \upsigma_1\neq\upsigma_2 \text{ s.t. } \eth(\upsigma_1)=\eth(\upsigma_2)=\Hat{x}\big\}
        \]
        is dense in $\partial B(o,1)$
    \end{enumerate}
\end{corollary}

\section{Approximation Scheme for First-Passage Times}
To study the first-passage times on $\Hh$, we introduce a new random variable $T^t$ and a random graph $\G^t$, which builds upon $\G$ by incorporating additional vertices and edges. Subsequently, the first-passage times will be approximated by $T^t$ for a given $t>0$. We define $\G^t= (V^t, \mathcal{E}^t)$ with $t>0$ as follows:
\[
    V^t := V \cup (t\Z^d) \quad \text{and} \quad \mathcal{E}^t := \mathcal{E} \cup \mathcal{E}'(t).
\]
where
\[
    \mathcal{E}'(t):=\left\{ \{u,v\} \ \colon \ u \in t\Z^d \text{ and }\begin{array}{c}
    v \in \left(u + [t/2, ~t/2)^d\right) \cap V\\ 
    \text { or }\\
    v \in t\Z^d \text { with } \|u-v\|=t
    \end{array}\right\}.
\]

The vertices corresponding to $t\Z^d$ are referred to as extra vertices, while $\mathcal{E}'(t)$ represents the set of extra edges. Since $V \cap (t\Z^d) = \varnothing$ with probability one, the unions determining $V^t$ and $\mathcal{E}^t$ are $\p$-a.s. disjoint. Analogous to $q(x)$ and $\thickbar{q}(x)$, we define, for all $x \in \R^d$,
~$q^t(x) := \argmin_{y \in V^t}\big\{ \| y-x \| \big\}$ and note that $q^t(o)=o$.

\begin{figure}[htb!]
    \centering
    \includegraphics[scale=0.3,trim={50 0 250 250},clip]{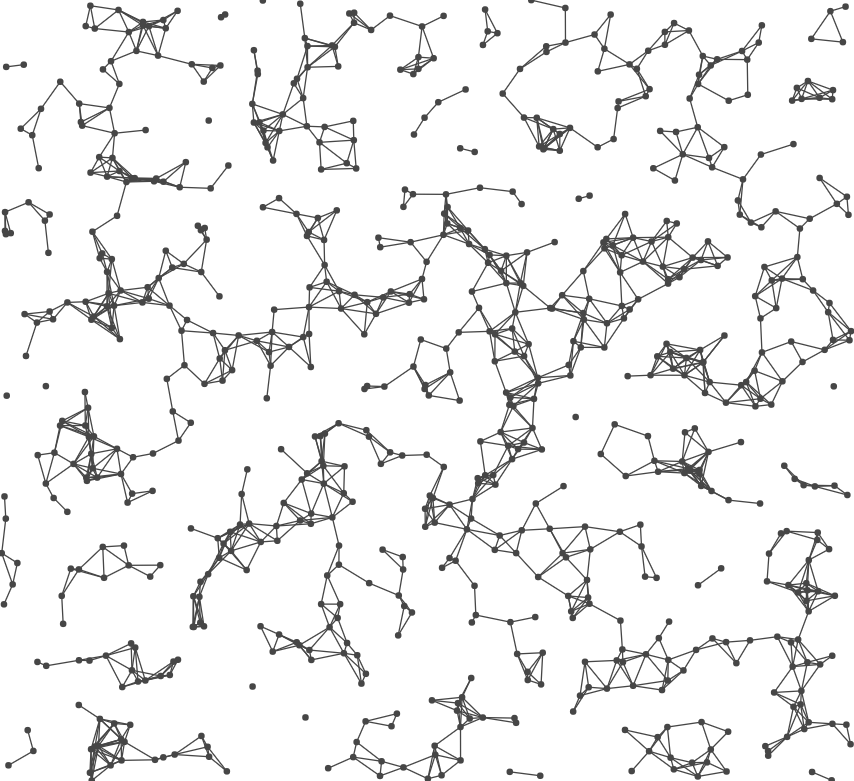} \hspace{0.4cm}\includegraphics[scale=0.3,trim={50 0 250 250},clip]{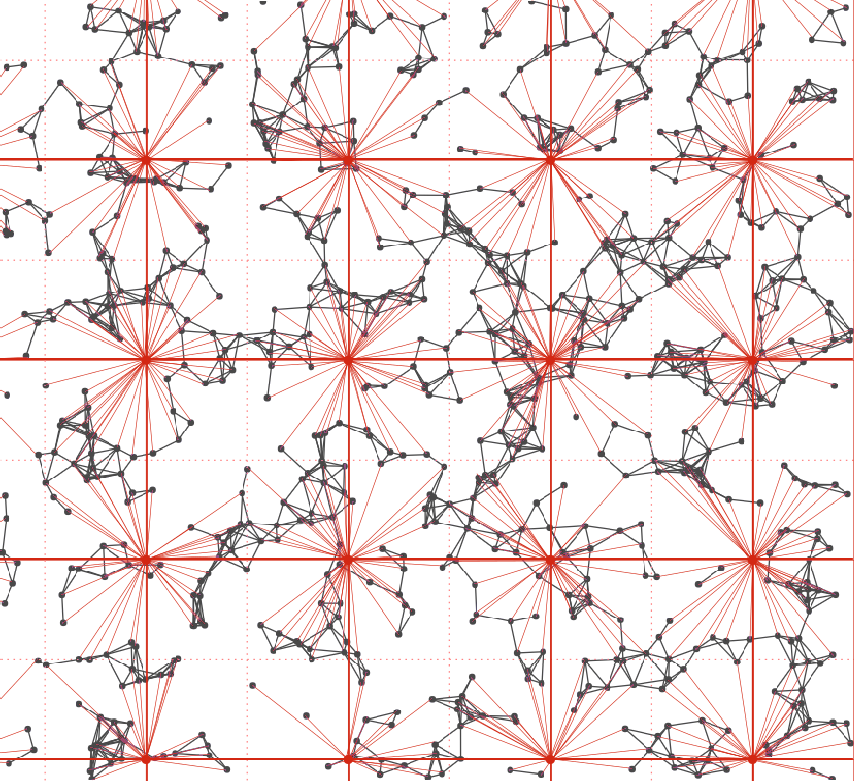}
    \caption{The image depicts the same region of a standard RGG, denoted as $\mathcal{G}$ (left), alongside an RGG with extra vertices and extra edges  $\mathcal{G}^t$ (right).}
    \label{fig:rgg_extra}
\end{figure}

The passage times along the extra edges are considered deterministically determined. Let $\K > 4 d(\thickbar{\beta} \vee \beta)$ be a fixed constant, and define the passage times for the edges of $\G^t$ as
\[\tau_e^t := \left\{ \begin{array}{ll}
     \tau_e,& \text{ if }~e \in \mathcal{E},  \\ 
     \K t,& \text{ if }~e \in \mathcal{E}'(t).
\end{array}\right. \]

Let $\gamma$ be a (self-avoiding) path in $\mathcal G^t$. The passage time along $\gamma$ is given by
\[
    T^t(\gamma) := \sum_{e \in \gamma} \tau_e^t.
\]
We define the modified first-passage time~$T^t(x,y)$ between vertices~$x$ and~$y$ by~$T^t(x,y):= \inf_\gamma T^t(\gamma)$, where the infimum is taken over all paths $\gamma$ from $q^t(x)$ to $q^t(y)$ in $\G^t$. Henceforth, we assume that \eqref{A1p} and \eqref{A2p} hold true.

Now, we present the first results on the passage times on $\G^t$.

\begin{lemma}\label{lem_cheap_hop}
	There exists~$\delta > 0$  such that  for any~$n \in \N$, any~$\ell > 0$, and any~$t \ge 1$, 
	\begin{equation*}
		\p \left( \begin{array}{l} \text{there is a path $\gamma$ in $\mathcal{G}^t$  starting in  } B(o,\ell)\\[.1cm] \text{ with } |\gamma| = n \text{ and } \sum_{e \in \gamma} \tau^t_e \le \delta n
		\end{array}\right) \le  \frac{\max\{\ell^d,1\}}{2^{n}}.
	\end{equation*}
\end{lemma}
\begin{proof}
	Throughout this proof, we write, for~$k \in \N$ and~$s \in \R$,
	\[F_k(s):=\p(Z_1 + \cdots + Z_k \le s), \quad \text{where } Z_1,\ldots,Z_k \text{ are i.i.d. } \sim \tau,\]
	that is,~$F_k$ is the cumulative distribution function of the~$k$-fold convolution of the passage time through one (non-extra) edge.

	Fix~$n \in \N$. Let~$\varepsilon > 0$ be a small constant to be chosen later. Using the assumption that~$\lim_{s \to 0} F_1(s) = 0$ and elementary large deviations considerations, it is easy to see that we can choose~$\delta = \delta(\varepsilon)$ such that
	\begin{equation}\label{eq_choice_of_bar_delta}
		F_k(\delta k) < \varepsilon^k \quad \text{for all } k \in \N.
	\end{equation}

	For now, we condition on a realization of the graph~$\mathcal G^t = (V^t, \mathcal{E}^t)$, so that the only randomness left is that of the passage times. Let~$\gamma$ be a (self-avoiding) path in~$\G^t$ with~$|\gamma| = n$ and let~$m$ denote the number of extra edges traversed by~$\gamma$. The probability that~$T^t(\gamma) \le \delta n$ is $F_{n-m}(\delta n - \K t m)$. This is zero in case~$m \ge \frac{\delta}{\K t} n$; otherwise, we bound:
			\[F_{n-m}(\delta n - \K t m) \le F_{n-m}(\delta(n - m)) \le \varepsilon^{n-m} \le \varepsilon^{n/2}, \]
			where the first inequality follows from~$\K t \ge \delta$, the second inequality follows from~\eqref{eq_choice_of_bar_delta}, and the third inequality follows from~$m < \frac{\delta}{\K t} n \le n/2$ (if~$\delta < 1/2$, since~$\K t \ge 1$).

	This shows that in all cases, the probability that a path of graph length~$n$ has passage time (with respect to~$T^t$) smaller than~$\delta n$ is smaller than~$\varepsilon^{n/2}$.

	Now including also the randomness in the choice of the graph, a union bound over paths shows that the probability in the statement of the lemma is smaller than
			\begin{align} \label{eq_epsilon_with_E}
				\varepsilon^{n/2} \cdot \E  \left| \left\{ \text{paths $\gamma$ in $\mathcal{G}^t$  starting in  } B(o,\ell) \text{ with } |\gamma| = n				 \right\}\right|.
			\end{align}
			Defining
			\[\mathsf{v}_t(s) := \sup_{x \in \R^d} \E[|V^t \cap B(x,s)|],\quad s > 0,\]
			by Mecke's formula, the expression in~\eqref{eq_epsilon_with_E} is smaller than
	\begin{equation}\label{eq_epsilon_with_E2}
			\varepsilon^{n/2} \cdot \mathsf{v}_t(\ell)\cdot (\mathsf{v}_t(r))^n.
	\end{equation}
Recalling that~$\upupsilon_d$ denotes the volume of the unit ball in~$\mathbb R^d$, we bound, for any~$t \ge 1$ and~$s > 0$:
	\begin{align*}
		&\mathsf v_t(s) \le \upupsilon_d s^d + \lceil s/t\rceil^d \le \upupsilon_d s^d + \lceil s \rceil^d  \\
		&\hspace{2cm}\le \upupsilon_d s^d + \lceil s \rceil^d \cdot \mathbbm{1}_{\{s > 1\}} +  \mathbbm{1}_{\{s \le 1\}} \le (\upupsilon_d + 2^d) s^d + 1.
	\end{align*}
	Hence, the expression in~\eqref{eq_epsilon_with_E2} is smaller than
	\[\varepsilon^{n/2} \cdot ((\upupsilon_d + 2^d) \ell^d + 1) \cdot ((\upupsilon_d + 2^d)r^d + 1)^n.\]
	It is now easy to see that taking~$\varepsilon$ small enough (depending on~$r$ and~$d$, but not on~$\ell$ or~$n$), the right-hand side is smaller~$\max\{1,\ell^d\}/2^n$.
\end{proof}

For any~$u,v \in V^t$, let~$\gamma^t_{u \leftrightarrow v}$ denote the shortest path in~$\mathcal G^t$ from~$u$ to~$v$ that only uses extra edges.  Writing~$u=(u_1,\ldots,u_d)$ and~$v=(v_1,\ldots,v_d)$, we can bound
\begin{equation}
	\label{eq_hopcount_extra}
	|\gamma^t_{u \leftrightarrow v}| \le \sum_{i=1}^d \left\lceil \frac{u_i - v_i}{t}\right\rceil \le \frac{\|u-v\|_1}{t} + d \le  \frac{\sqrt{d}}{t} \|u-v\| + d.
\end{equation}
This gives
\begin{equation}
	\label{eq_T_t_and_l1}
		T^t(u,v) \le \K t |\gamma^t_{u\leftrightarrow v}| \le \K  \sqrt{d}\|u-v\| + \K t d.
\end{equation}

Concerning the special case of~$u = o$ and~$v = q^t(x)$ for~$x \in \R^d$, we will need the following.
\begin{claim}
	For any~$x \in \R^d$ with~$\|x\| \ge 1$ and any~$t \in [1,\|x\|]$, we have
	\begin{equation}\label{eq_newer_bound_path}
		|\gamma^t_{o \leftrightarrow q^t(x)}| \le \frac{3d}{t}\|x\|.
	\end{equation}
\end{claim}
\begin{proof}
We bound
\begin{equation}
	\label{eq_bound_qt}
	\|q^t(x)\| \le  \|x\| + \|q^t(x)-x\| \le \|x\| +\sqrt{d} \|q^t(x)-x\|_\infty \le \|x\|+ \tfrac{\sqrt{d}t}{2},
\end{equation}
	and combining this with~\eqref{eq_hopcount_extra} and the assumptions~$\|x\| \ge 1$ and~$t \in [1,\|x\|]$ gives
	\begin{align*}
		|\gamma^t_{o \leftrightarrow q^t(x)}| \le \frac{\sqrt{d}}{t} \|q^t(x)\| + d &\le  \frac{\sqrt{d}}{t} \|x\| + \frac{3d}{2} \\[.2cm]
		&\le \frac{d}{t}\|x\| + \frac{3d}{2}\cdot \frac{\|x\|}{t} \le \frac{3d}{t}\|x\|.
	\end{align*}
\end{proof}

Letting~$\delta$ be the constant given in Lemma~\ref{lem_cheap_hop}, we define
\begin{equation}
\label{eq_def_Kprime}
\Kprime := \frac{3d \K}{\delta},
\end{equation}
and
\begin{equation*}
	\begin{split}
		&Y_{t,x}:= \inf \left\{
			\sum_{e \in \gamma} \tau_e^t:\; \gamma \text{ is a path in $\mathcal{G}^t$ from $o$ to $q^t(x)$ with } |\gamma| \le  \Kprime \|x\| 
	\right\},\\
&\hspace{6.5cm} \text{for }x \in \R^d \text{ with } \|x\| \ge 1,\; t \in [1, \|x\|].
	\end{split}
\end{equation*}
Since~$\frac{3d}{t} \le \frac{3d \K}{\delta}$, it follows from the above claim that~$\gamma^t_{o \leftrightarrow q^t(x)}$ belongs to the set of paths whose infimum is taken in the definition of~$Y_{t,x}$. In particular,
	\begin{equation}
		\label{eq_bound_on_Y}
		Y_{t,x} \le \K t |\gamma^t_{o \leftrightarrow q^t(x)}| \stackrel{\text{\footnotesize\eqref{eq_newer_bound_path}}}{\le}  3d \K\|x\|.
	\end{equation}

We now compare the truncated passage time $Y_{t,x}$ with $T^t(o, q^t(x))$.

\begin{lemma}
	\label{lem_Y_and_T_t}
	If~$x \in \mathbb R^d$ with~$\|x\| \ge 1$ and~$t \in [1,\|x\|]$, then
	\[\p(Y_{t,x} \neq T^t(o,q^t(x))) \le  2^{-\Kprime\|x\|}.\]
\end{lemma}
\begin{proof}
	In the event~$\{Y_{t,x} \neq T^t(o,q^t(x))\}$, there exists a path~$\gamma$ in~$\mathcal G^t$ from~$o$ to~$q^t(x)$ which has~$|\gamma| > \Kprime\|x\|$ and
	\begin{align*}
		\sum_{e \in \gamma} \tau^t_e = T^t(o,q^t(x)) <  Y_{t,x} &\stackrel{\text{\footnotesize\eqref{eq_bound_on_Y}}}{\le} 3d\K \|x\|.
	\end{align*}
	We then have
	\[\frac{\sum_{e \in \gamma} \tau^t_e}{|\gamma|} \le \frac{3d \K\|x\|}{|\gamma|}<   \frac{3d \K\|x\|}{\Kprime \|x\|} = \delta.\]
	By Lemma~\ref{lem_cheap_hop} (with~$\ell = 1$), the existence of such a path has probability smaller than~${2^{-|\gamma|} <  2^{-\Kprime \|x\|}}$.
\end{proof}

The following result is a comparison of the $T^t$-distance between vertices provided by $q$ and $q^t$.

\begin{lemma}
	\label{lem_q_and_no_q}
	There exist~$\mathfrak{C} \ge 1$, $\mathfrak{c} > 0$ such that for any~$x \in \R^d$ and any~$t \ge 1$,
	\[\p(|T^t(o,q^t(x)) - T^t(q(o),q(x))|\ge  \mathfrak{C}t) < e^{-\mathfrak{c}t}.\]
\end{lemma}
\begin{proof}
Recall that~$q^t(o) = o$. For any~$s > 0$,  by  the triangle inequality, 
 \begin{align}
\nonumber     &\p(|T^t(o,q^t(x)) - T^t(q(o),q(x))|\ge  s)\\ &\le \p(T^t(o, q(o)) \ge s/2)+ \p(T^t(q^t(x), q(x)) \ge s/2). \label{eq_new_triangle}
 \end{align}
 Let us deal with the second term on the right-hand side. By~\eqref{eq_T_t_and_l1}, it is smaller than
	\begin{equation}
		\label{ex_norms2i}
		\p \left( \|q^t(x) - q(x)\|\ge \frac{ \tfrac{s}{2} - \K td}{\K\sqrt{d}}  \right).
	\end{equation}
	Since~$V \subseteq V^t$, we have~$\|q^t(x) - x\|\le \|q(x) - x\|$, so the triangle inequality gives~${\|q^t(x)-q(x)\| \le 2 \|q(x)-x\|}$. 
	Using this, the probability in~\eqref{ex_norms2i} is at most
	\begin{equation}\label{eq_weird_s}
		\p \left(\|q(x)-x\| \ge \frac{\tfrac{s}{2} - \K td }{2\K \sqrt{d}} \right) = \p \left(\|q(o)\| \ge \frac{\tfrac{s}{2} - \K td }{2\K \sqrt{d}} \right).
	\end{equation}
A similar argument shows that the first term in~\eqref{eq_new_triangle} is bounded by the same value. We have thus proved that
\begin{equation}\label{eq_new_new_triangle}
    \p\big(|T^t(o,q^t(x)) - T^t(q(o),q(x))|\ge  s\big) \le 2 \p \left(\|q(o)\| \ge \frac{\tfrac{s}{2} - \K td }{2\K \sqrt{d}} \right).
\end{equation}
 
	Now, using~\cref{prop:Hn.growth}, there exists~$\bar{s}$ (not depending on~$x$) such that, for any~$s' \ge \bar{s}$,
	\begin{equation}\label{eq_apply_pisz}\p \big(\|q(o)\| \ge s'\big) < \frac12 e^{-c s'}.\end{equation}
		By taking~$s=\mathfrak{C}t$ with~$\mathfrak{C}:=  2(d + 2\sqrt{d} \cdot  \bar{s}) \K $, the right-hand side of~\eqref{eq_new_new_triangle} equals~$2\p(\|q(o)\|\ge  \bar{s}\cdot t)$.
	Since~$t \ge 1$, we have~$\bar{s} \cdot t \ge \bar{s}$, so by~\eqref{eq_apply_pisz}, we can bound
	\[ \p(\|q(o)\| \ge \bar{s} \cdot  t) < \frac12 e^{-c \bar{s} t},\]
  so we set~$\mathfrak{c}:= c \bar{s}$ to complete the proof.
\end{proof}

The lemma below analyses the first-passage time $T$ and the modified random variable $T^t$. The proof is somewhat involved, and we postpone it to~\cref{s_appendix_proofs}.

\begin{lemma}\label{lem_T_t_and_T}
	 	There exists~$\mathfrak{c}_1 > 0$ such that, for~$x \in \R^d$ with~$\|x\|$ large enough and~$t$ large enough (not depending on~$x$) with~$t \le \|x\|$, we have
	\[\p\big(T^t(q(o),q(x)\big) \neq T(x)) < \|x\|^{4d} e^{-\mathfrak{c}_1 t}.\]
\end{lemma}

We derive the following corollaries from the results above.

\begin{corollary}\label{cor_join_bounds}
	For~$x$ and~$t$ as in Lemma~\ref{lem_T_t_and_T}, we have
	\begin{equation}\label{eq_all_together_now}
		\p (| T(x) - Y_{t,x}| \ge \mathfrak C t ) \le 2^{-\Kprime \|x\|} + e^{-\mathfrak{c} t} + \|x\|^{4d} e^{-\mathfrak{c}_1 t},
	\end{equation}
	where~$\Kprime$ is defined in~\eqref{eq_def_Kprime},~$\mathfrak C$ and~$\mathfrak c$ are the constants of Lemma~\ref{lem_q_and_no_q}, and~$\mathfrak c_1$ is the constant of Lemma~\ref{lem_T_t_and_T}.
\end{corollary}
\begin{proof}
	This follows from putting together Lemma~\ref{lem_Y_and_T_t}, Lemma~\ref{lem_q_and_no_q} and Lemma~\ref{lem_T_t_and_T}.
\end{proof}

\begin{corollary}\label{cor_join_expectation}
	There exists~$\mathfrak C_1 > 0$ such that the following holds. Let~$x \in \R^d$ be such that~$\|x\|$ is large enough (as required in Lemma~\ref{lem_T_t_and_T}), and let~$t \in [\mathfrak C_1 \log\big(\|x\|\big),\|x\|]$.  Then,
	\begin{equation}
		\label{eq_for_second_mom}
	\E [ (T(x)- Y_{t,x})^2] \le 2(\mathfrak Ct)^2.
	\end{equation}
\end{corollary}

	\begin{proof}
Let~$\beta^* := \max\left( 3d\K,\; \beta \right)$,
	where~$\beta$ is the constant of~\cref{lm:T.bds}. We bound
	\begin{align}
		\nonumber &\E[(T(x)-Y_{t,x})^2] \\
		\nonumber &\le (\mathfrak Ct)^2 + \E[ (T(x) - Y_{t,x})^2 \cdot \mathbbm{1}{\{|T(x) - Y_{t,x}| \ge \mathfrak Ct \}}]\\
		\label{eq_with_fraks}&\le (\mathfrak Ct)^2 + \E[ (T(x) + \beta^* \|x\|)^2 \cdot \mathbbm{1}{\{|T(x) - Y_{t,x}| \ge \mathfrak Ct \}}],
	\end{align}
	where the second inequality follows from
	\[|T(x) - Y_{t,x}| \le T(x) + Y_{t,x} \stackrel{\text{\footnotesize\eqref{eq_bound_on_Y}}}{\le} T(x) + \beta^* \|x\|.\] We will now bound the expectation in~\eqref{eq_with_fraks} by breaking into the cases where~$T(x) \le \beta^* \|x\|$ and~$T(x) > \beta^* \|x\|$. First,
	\begin{equation*}\begin{split}
		&\E[ (T(x) + \beta^* \|x\|)^2 \cdot \mathbbm{1}{\{|T(x) - Y_{t,x}| \ge \mathfrak Ct, \; T(x) \le \beta^* \|x\| \}}]\\
		&\quad \le 4 (\beta^*)^2 \|x\|^2 \cdot \p(|T(x) - Y_{t,x}| \ge \mathfrak Ct) \\
		&\quad \stackrel{\text{\footnotesize\eqref{eq_all_together_now}}}{\le} 4 (\beta^*)^2 \|x\|^2 \cdot (2^{-\Kprime \|x\|} + e^{-\mathfrak{c} t} + \|x\|^{4d} e^{-\mathfrak{c}_1 t}).
	\end{split}
	\end{equation*}
Second,
	\begin{equation*}\begin{split}
		&\E[ (T(x) + \beta^* \|x\|)^2 \cdot \mathbbm{1}{\{|T(x) - Y_{t,x}| \ge \mathfrak Ct, \; T(x) > \beta^* \|x\| \}}]\\
		&\quad \le \E[ (2T(x))^2 \cdot \mathbbm{1} \{T(x) > \beta^* \|x\|\}\\
		&\quad \le 4 \left( \E[(T(x))^2]\right)^{1/2} \cdot \left( \p(T(x) > \beta^* \|x\|)\right)^{1/2}\\
		&\quad \le 4 \left( \E[(T(x))^2]\right)^{1/2} \cdot \left(\Cr{T.exp.ext} \exp(-\Cr{T.exp.int} \beta^* \|x\|) \right)^{1/2},
	\end{split}
	\end{equation*}
where the second inequality is Cauchy-Schwarz, and the third inequality is given in \cref{lm:T.bds}. Next, we bound
\begin{align*}
	\E[(T(x))^2] &=2 \int_0^\infty s \cdot  \p(T(x) > s) \;\mathrm{d}s\\
	&\le 2 \beta \|x\| + 2 \int_{\beta\|x\|}^\infty s \cdot \Cr{T.exp.ext} \exp(-\Cr{T.exp.int} s) \;\mathrm{d}s.
\end{align*}

Putting things together, we have shown that~$\E[(T(x)- Y_{t,x})^2]$ is bounded from above by
\begin{align*}
	&(\mathfrak C t)^2 + 4 (\beta^*)^2 \|x\|^2 \cdot (2^{-\Kprime \|x\|} + e^{-\mathfrak{c} t} + \|x\|^{4d}e^{-\mathfrak{c}_1 t})\\
	& + 4 \left( 2 \beta \|x\| + 2 \int_{\beta\|x\|}^\infty s \cdot \Cr{T.exp.ext} \exp(-\Cr{T.exp.int} s) \;\mathrm{d}s \right)^{1/2} \cdot \left(\Cr{T.exp.ext} \exp(-\Cr{T.exp.int} \beta^* \|x\|) \right)^{1/2}.
\end{align*}
It is now easy to see that if~$\mathfrak C_1$ is large and~$t \ge \mathfrak C_1 \log \big(\|x\|\big)$ with~$\|x\|$ large enough, the expression above  is smaller than~$2(\mathfrak C t)^2$.
\end{proof}

\section{Intermediate Results for the Approximation of First-Passage Times}
This section presents preparatory results for the approximation of first-passage times. Let~$\mathscr B$ denote the collection of all boxes of the form~$z + [-t/2,t/2)^d$, with~$z \in t\Z^d$. The following lemma offers a bound for the boxes crossed by the truncated passage time.

\begin{lemma}\label{lem_number_of_boxes}
	Let~$x \in \R^d$ with~$\|x\| \ge 1$ and~$t \in [1,\|x\|]$. Let~$\gamma$ be a path in~$\mathcal G^t$ from~$o$ to~$q^t(x)$ with~$|\gamma| \le \Kprime \|x\|$ and such that~$Y_{t,x} = \sum_{e \in \gamma} \tau^t_e$. Then, the number of boxes of~$\mathscr B$ intersected by~$\gamma$ is at most~$(3^d+1)\left(\frac{(3d + \Kprime r)\|x\|}{t} + 1\right)$.
\end{lemma}
\begin{proof}
We write~$\gamma = (\gamma_0, \ldots, \gamma_n)$, where~$n = |\gamma|$,~$\gamma_0 = o$, and~$\gamma_n = q^t(x)$. We will now define an increasing sequence~$J_0,\ldots,J_m$ of indices in the path. We first let~$J_0 := 0$. Next, assuming~$J_i$ has been defined and is smaller than~$n$, we define (with the convention that the minimum of an empty set is infinity):
\begin{align*}
	J_{i+1}:= n &\wedge \min\{j > J_i: \{\gamma_{j-1},\gamma_j\} \text{ is an extra edge}\} \\
	&\wedge \min\{j > J_i: \|\gamma_j - \gamma_{J_i}\|_\infty > t\}.
\end{align*}
We let~$m$ be the index such that~$J_m = n$ (which is the last one). For cleanliness of notation, we write
\[a_i := J_{i+1} - J_{i},\quad i \in \{0,\ldots, m-1\},\]
and
\[\mathsf y(i,k) := \gamma_{J_i + k},\quad i \in \{0,\ldots,m-1\},\; k \in \{0,\ldots,a_i\},\]
so that
\[(\mathsf y(i,0),\ldots, \mathsf y(i,a_i)) = (\gamma_{J_i},\ldots, \gamma_{J_{i+1}}).\]
By construction, for any~$i$, the set~$\{\mathsf y(i,0),\ldots, \mathsf y(i,a_i)\}$ intersects at most~$3^d + 1$ boxes of~$\mathscr B$ (that is: at most the box containing~$\mathsf y(i,0)$, the boxes that are adjacent to it in~$\ell_\infty$-norm, and the box containing~$\mathsf y(i,a_i)$). Hence,
\begin{align*}
&|\{ \text{boxes of $\mathscr B$ intersected by $\gamma$}\}|\\ &\le  \sum_{i=0}^{m-1}|\{ \text{boxes of $\mathscr B$ intersected by $\{\mathsf y(i,0),\ldots, \mathsf y(i,a_i)\}$}\}|  \le (3^d+1)m.
\end{align*}
We now want to give an upper bound for~$m$. For this, we write
	\[m = m_{\mathrm{basic}} + m_{\mathrm{extra}}+1,\]
where~$m_{\mathrm{extra}}$ is the number of~$i < m-1$ such that the last step in the sub-path~$(\mathsf y(i,0),\ldots, \mathsf y(i,a_i))$ is an extra edge, and~$m_{\mathrm{basic}}$ is the number of other~$i <m-1$ (that is, for which the sub-path~$(\mathsf y(i,0),\ldots, \mathsf y(i,a_i))$ does not traverse any extra edge). Noting that
\[3d \K \|x\| \stackrel{\text{\footnotesize\eqref{eq_bound_on_Y}}}{\ge} Y_{t,x} = \sum_{e \in \gamma} \tau^t_e \ge \K t |\{e \in \gamma: e \text{ is extra}\}|,\]
we have~$m_{\mathrm{extra}} \le \frac{3d \K \|x\|}{\K t}= \frac{3d  \|x\|}{t}$. Next, if~$i$ is an index that contributes to~$m_{\mathrm{basic}}$, then
\[t \le \|\mathsf y(i,a_i) - \mathsf y(i,0)\|_\infty \le \sum_{k=0}^{a_i-1} \|\mathsf y(i,k+1) - \mathsf y(i,k)\| \le ra_i,\]
so~$a_i \ge t/r$, and then,
\[\Kprime \|x\| \ge |\gamma| = \sum_{i=0}^{m-1} a_i \ge m_{\mathrm{basic}} \cdot \frac{t}{r},\]
so~$m_{\mathrm{basic}} \le \frac{\Kprime r \|x\|}{t}$. This gives~$m \le \frac{(3d + \Kprime r)\|x\|}{t}+1$.
\end{proof}

The proposition below is based on the results established in Boucheron, Lugosi, and Massart~\cite{boucheron2003}. It sets the stage for obtaining a concentration inequality for our truncated random variable $Y_{t,x}$, which will be derived later.

\begin{proposition}\label{prop_concentration}
    Let~$(S,\mathcal{S})$ be a measurable space and~$n \in \N$. Let~$\sigma$ be a probability measure on~$(S,\mathcal{S})$ and~$\p$ be the probability on~$(S^n,\mathcal{S}^n)$ given by the~$n$-fold product measure of~$\sigma$. Set $X_1, \dots, X_n$ to be independent random elements taking values in~$S$ and let $X_i'$ be an independent copy of $X_i$. Set~$h: S^n \to\R$ to be a measurable function on $(S^n,\mathcal{S}^n,\p)$ and fix
    \[Z:=h(X_1,\dots,X_n), \ \ \text{and} \ \ Z^{(i)}:=h(X_1,\dots,X_{i-1},X_{i}',X_{i+1},\dots,X_n).\]

    Consider that there exist~$\kappa > 0$ and measurable sets~$E_1,\ldots, E_n \in \mathcal{S}^n$ such that, $\p$-a.s., for each $i \in \{1,\ldots,n\}$,
    \begin{equation} \label{eq_condition_measurable}           Z^{(i)}-Z \le \kappa \cdot \mathbbm{1}_{E_i}.
    \end{equation}
	Then, 
        \begin{equation} \label{eq:new_steele-efron-stein}
            \mathrm{Var}Z \leq \kappa^2 \sum_{i=1}^n\p(E_i).
        \end{equation}
    Furthermore, for any~$\alpha > 0$ such that $\E[e^{\alpha Z}]< +\infty$, one has
    \begin{equation} \label{eq:exp.Vminus.bound}
       \E\left[\exp\left\{\alpha (\E[Z]-Z)\right\} \right] \le \E \left[ \exp \left\{ 2 \alpha^2 \kappa^2\sum_{i=1}^n \mathbbm{1}_{E_i}\right\} \right] 
    \end{equation}
    and, for all~$\alpha > 0$  in an open interval such that $\E[e^{\alpha Z}]< +\infty$, one has
    \begin{equation} \label{eq:exp.W.bound}
        \E\left[\exp\left\{\lambda (Z-\E[Z])\right\} \right] \le \E \left[ \exp \left\{ 2\lambda^2 \kappa^2 e^{\alpha\kappa} \sum_{i=1}^n \mathbbm{1}_{E_i}\right\} \right].
    \end{equation}
\end{proposition}
\begin{proof}
    The bounds for \eqref{eq:new_steele-efron-stein} and \eqref{eq:exp.Vminus.bound} are directly obtained by applying \eqref{eq_condition_measurable} in the Steele-Efron-Stein inequality (see \cite[p.~1585]{boucheron2003}) and in Theorem 2 of Boucheron et al. \cite{boucheron2003} with $\lambda=\alpha$ and $\theta = 1/(2\alpha)$.  Additionally, item \eqref{eq:exp.W.bound} is a specific case addressed in Lemma 3.2 of Garet and Marchand \cite{garet2010}. Below, we demonstrate how it can be attained from Theorem 2 of \cite{boucheron2003}.
    
    Let us set $\psi(s)= s(e^s-1)$ and note that $\psi(-s) \le s^2$ for $s \ge 0$. Hence, by Lemma 8 of Massart \cite{massart2000},
    \begin{align*}
        \alpha \E[e^{\alpha Z}] -\E[e^{\alpha Z}]\log\left(\E[e^{\alpha Z}]\right) &\le  \sum_{i=1}^n\E\left[e^{\alpha Z}\psi\big(-\alpha(Z -Z^{(i)})\big)\mathbbm{1}_{Z-Z^{(i)} \ge 0}\right] \\
        &\le  \alpha^2\sum_{i=1}^n\E\left[e^{\alpha Z}(Z -Z^{(i)})_+^2\right] \\
        &=  \alpha^2\sum_{i=1}^n\E\left[e^{\alpha Z^{(i)}} (Z^{(i)}-Z)_+^2\right] \\
        &\le   \alpha^2\kappa^2\sum_{i=1}^n\E\left[e^{\alpha Z} e^{\alpha( Z^{(i)}-Z)} \mathbbm{1}_{E_i}\right] \\
        &\leq \E\left[ e^{\alpha Z} \left(\alpha^2 \kappa^2 e^{\alpha\kappa}\sum_{i=1}^n\mathbbm{1}_{E_i}\right)\right]
    \end{align*}
    and the remaining steps of the proof follow the same methodology as outlined in Theorem 2 of Boucheron et al. \cite{boucheron2003}.
\end{proof}

We use the inequalities obtained in the proposition above for the truncated random variable resulting in the following lemma.

\begin{lemma}\label{lem_apply_conc}
	There exists~$C_{\mathrm{dev}} \ge 1$ such that the following holds. Let~$x \in \R^d$ with~$\|x\| \ge 1$ and~$t \in [1,\|x\|]$. Then, 
	\begin{equation*}
		\E\left[\exp \left\{ \alpha \cdot | \E[Y_{t,x}]-Y_{t,x} |\right\}\right] \le   \exp \left\{C_{\mathrm{dev}} \alpha^2 t~\|x\| \right\} \quad \text{for any }\alpha \in \left(0,\tfrac{1}{4\K t}\right).
	\end{equation*}
\end{lemma}
\begin{proof}
In order to apply Proposition~\ref{prop_concentration} to~$Y_{t,x}$, we need to define this random variable in a probability space with a suitable product measure. We achieve this by constructing~$\mathcal G$ (and consequently also~$\mathcal G^t$) in a probability space where the randomness is given in blocks, each block corresponding to a box of~$\mathscr B$. This has to be done with some care, because we need to describe how to encode the randomness corresponding to the passage times of all edges, including those that touch distinct boxes.

We take a probability space with probability measure~$\p$ where we have defined a family of independent random elements
\[
	X_i := (\mathscr V_i, (\mathscr T_i(x,y): x \in [-t/2,t/2)^d,\; y \in \mathbb R^d)),\quad  i \in \N,
\]
where~$\mathscr V_i$ is a Poisson point process on~$[-t/2,t/2)^d$ with intensity~$1$, and independently for each pair~$(x,y) \in [-t/2,t/2)^d \times \R^d$ (and independently of~$\mathscr V_i$),~$\mathscr T_i(x,y)$ is a random variable with the law of the passage time~$\tau$.

Now, fix an arbitrary enumeration~$\{z_1,z_2,\ldots\}$ of~$\Z^d$. We use the above random elements to construct~$\mathcal G = (V,\mathcal{E})$ and the passage times~$T(\cdot,\cdot)$ as follows. First construct the set of vertices as
\[V= \{t z_i + u:\; i \in \N,\; u \in \mathscr{V}_i\}.\]
Next, construct the set of edges as prescribed for a random geometric graph with range parameter~$r$: an edge is included between any two vertices within Euclidean distance~$r$ of each other.

For the definition of the passage times, fix two vertices of~$V$, written as
\[u = tz_i + x \quad \text{ and }\quad v = tz_j + y.\]  First assume that~$i = j$, that is,~$u$ and~$v$ are in the same box, and without loss of generality, also assume that~$x$ is smaller than~$y$ in the lexicographic order in~$[-t/2,t/2)^d$. Then, set~$T(x,y) = \mathscr{T}_i(x,y-x)$. Now assume that~$i < j$ (without loss of generality); in that case, set~$T(u,v) = \mathscr T_i(x,v-u) = \mathscr{T}_i(x, tz_j +y - tz_i - x)$. This completes the construction of~$\mathcal G$, and then~$\mathcal G^t$ is obtained from it as before (by adding extra vertices and edges).

We have exhibited an infinite product space, but in fact, by the definition of~$Y_{t,x}$, there exists a deterministic constant~$ N_{t,x}$ such that we can write
\begin{equation*}
Y_{t,x} = h(X_1,\ldots, X_{N_{t,x}})
\end{equation*}
for some function~$h$.

The next step is to exhibit a constant~$\kappa$ and events~$E_i$ for~\eqref{eq_condition_measurable}. Let~$\gamma$ be a path in~$\mathcal G^t$ from~$o$ to~$q^t(x)$ such that~$|\gamma| \le \Kprime \|x\|$ and such that~$Y_{t,x} = \sum_{e \in \gamma} \tau^t_e$ (in case there are multiple paths with this property, we choose one using some arbitrary procedure). Then, let~$E_i$ be the event that~$\gamma$ intersects~$z_i + [-t/2,t/2)^d$, for~$i \in \{1,\ldots, N_{t,x}\}$. Set
\[X:=(X_1,\dots, X_{N_{t,x}}), \quad X_{(i)}:=(X_1,\dots,X_{i-1},X_i',X_{i+1},\dots, X_{N_{t,x}}),\]
and recall that $Y_{t,x}^{(i)} = h(X_{(i)})$ for any independent copy $X_i'$ of $X_i$. It is easy to check that, if~$X$ and $X_{(i)}$ are two elements of our probability space that agree in all entries except possibly the~$i$-th, then
\begin{equation} \label{eq:new_V-minus}
    Y_{t,x}^{(i)} - Y_{t,x} \le 4\K t \cdot \mathbbm{1}_{E_i}.
\end{equation}
Moreover,
	\begin{equation} \label{eq:sum.indicators_Y}
            \sum_{i=1}^{N_{t,x}} \mathbbm{1}_{E_i} \le 3^{d+1}(3d + \Kprime r)\frac{\|x\|}{t} + 3^{d+1} \le 2\cdot3^{d+1}(3d+\Kprime r)\frac{\|x\|}{t},
        \end{equation}
	where the first inequality follows from Lemma~\ref{lem_number_of_boxes}, and the second from the choice of~$t$. Now, Proposition~\ref{prop_concentration} states that, for any~$\alpha \in \left(0,~1/(4\K t)\right)$, and
		\begin{equation*}
			\E[\exp\{ \alpha \cdot |Y_{t,x} - \E[Y_{t,x}]|\}] \le   \exp \left\{32 \cdot 3^{d+1} \K^2 e~(3d + \Kprime r) \alpha^2 \cdot  t~\|x\| \right\}
	\end{equation*}
    Hence, given the conditions on $\|x\|$ and $t$, our assertion holds for $C_{\mathrm{dev}}$ sufficiently large, uniformly across all $x$ and $t$, thereby establishing the lemma.
\end{proof}

By combining the above lemma with Jensen's and Markov's inequality, we have
\begin{equation}\label{eq_Markov_for_dev}
	\p\big(|Y_{t,x} - \E[Y_{t,x}]| > s \big) \le 2 \inf_{\alpha \in \left(0,~1/(4\K t)\right)} \exp\left\{ C_{\mathrm{dev}} \alpha^2 t~\|x\| - s \alpha\right\}
\end{equation}
for any~$x$ with~$\|x\| \ge 1$,~$t \in [1,\|x\|]$ and~$s > 0$.

\section{Moderate Deviations of First-Passage Times}

In this section, we present the proof of our theorem regarding the moderate deviations of first-passage times, utilizing the preparatory results established earlier.

\begin{proof}[Proof of Theorem~\ref{thm_new_moderate_deviations}]
	Using Jensen's inequality and~\eqref{eq_for_second_mom}, for~$x$ with~$\|x\|$ large enough and any~$t \in [\mathfrak{C}_1 \log(\|x\|),\|x\|]$ we have
	\begin{align*}
		|\E[T(x)] - \E[Y_{t,x}]| \le \E[|T(x)-Y_{t,x}|] \le (\E[(T(x) - Y_{t,x})^2])^{1/2} \le \mathfrak{C}\sqrt{2} t.
	\end{align*}
	In particular, if~$\mathfrak{C}\sqrt{2}t \le s/2$, then~$|\E[T(x)] - \E[Y_{t,x}]|\le s/2$ and
	\begin{equation*}
		\p(|T(x) - \E[T(x)]| > s) \le \p(|T(x) - \E[Y_{t,x}]| > s/2).
	\end{equation*}

	Next, in case~$ \mathfrak{C} t \le s/4$ we can bound
	\begin{align*}
		&\p\big(|T(x) - \E[Y_{t,x}]| > s/2\big) \\[.2cm]
		&\le \p\big(|T(x) - Y_{t,x}| > \mathfrak{C} t\big) + \p\big(|Y_{t,x} - \E[Y_{t,x}]| > s/4\big)\\[.2cm]
		&\le 2^{-\Kprime \|x\|} + e^{-\mathfrak{c}t} + \|x\|^{4d} e^{-\mathfrak{c}_1 t} + 2 \inf_{\alpha \in (0,1/(4 \K t))} \exp\{C_{\mathrm{dev}} \alpha^2 t~\|x\| - s \alpha/4\},
	\end{align*}
	by~\eqref{eq_all_together_now} and~\eqref{eq_Markov_for_dev}. Therefore, we have proved that
\begin{equation}\label{eq_dev_almost}
	\begin{split}
		&\p\big(|T(x) - \E[T(x)]| > s\big) \\[.2cm]
		&\le 2^{-\Kprime \|x\|} + e^{-\mathfrak{c}t} + \|x\|^{4d}e^{-\mathfrak{c}_1 t} + 2 \inf_{\alpha \in (0,1/(4 \K t))} \exp\{C_{\mathrm{dev}} \alpha^2 t~\|x\|- s \alpha/4\},
	\end{split}
\end{equation}
	under the restrictions
\begin{equation}
	\label{eq_to_optimize}
	\|x\| \text{ large},\quad s > 0,\quad \mathfrak C_1 \log\big(\|x\|\big) \le t \le \min \left\{ \|x\|, \frac{s}{4\mathfrak{C}}\right\}.
\end{equation}
    Let us now fix $t=s/\sqrt{\|x\|}$, then, for sufficiently large $\|x\|$,
    \begin{equation*}
	\mathfrak C_1 \log\big(\|x\|\big) \sqrt{\|x\|}\le s \le \|x\|.
    \end{equation*}
    We also set~$\alpha = \frac{1}{4\K C_{\mathrm{dev}} \sqrt{\|x\|}}$ and note that $\alpha<\tfrac{1}{4\K t}$. With these choices for~$t$ and~$\alpha$, the right-hand side of~\eqref{eq_dev_almost} becomes
\[2^{-\Kprime \|x\|} + e^{-\mathfrak{c}s/\sqrt{\|x\|}} + e^{-\mathfrak{c}_1 s/\sqrt{\|x\|}} + 2  \exp\left\{  - ~\frac{\K-1}{16 \K^2 C_{\mathrm{dev}} } \cdot \frac{s}{\sqrt{\|x\|}} \right\}.\]
Clearly, there exist~$C,c > 0$ such that the expression above is smaller than~$Ce^{-cs/\sqrt{\|x\|}}$, uniformly over~$s \in [ \mathfrak{C}_1 \log\big(\|x\|\big)\sqrt{\|x\|}, ~ \|x\|]$. 
\end{proof}

\section{Expectation and Variance Bounds}
In this section, we prove Theorem~\ref{thm:moderate.dev.FPP}. We start with the proof of the variance bound.
\begin{proof}[Proof of Theorem~\ref{thm:moderate.dev.FPP}, item \eqref{eq:VarT}]
	First, observe that, by \cref{cor_join_expectation} for all $t>0$ and sufficiently large $\|x\|$,
    \begin{eqnarray*}
        \operatorname{Var} T(x) &\leq 2 \operatorname{Var} Y_{t,x} &+ ~2\operatorname{Var} \big(Y_{t,x} - T(x)\big)\\
        &\leq 2 \operatorname{Var}Y_{t,x} &+~ 4 (\mathfrak C t)^2.
    \end{eqnarray*}
    
    By \eqref{eq:new_steele-efron-stein}, \eqref{eq:new_V-minus}, and \eqref{eq:sum.indicators_Y}, one has
    \[\operatorname{Var} Y_{t,x} \leq 16 \K^2 t^2 \cdot \E\left[\sum_{i=1}^{N_{t,x}}\mathbbm{1}_{E_i}\right] \leq \left(16 \K^2 3^{d+1}(3d+\Kprime r) \right)\cdot t~ \|x\|.\] 
    
    Let us fix $t=\mathfrak C_1 \log\big(\|x\|\big)$ for sufficiently large $\|x\|$ and it completes the proof.
\end{proof}

We now turn to the proof of the expectation bound~\eqref{eq_asymp_exp}. For this, we follow Howard and Newman~\cite{howard2001} and Yao~\cite{yao2013}. We will need the following preliminary result.
\begin{lemma} \label{lm_expectation.double}
	There exists~$\bar{C} > 0$ such that for all sufficiently large $t>0$ and all ${x \in \partial B(o, 1)}$,
    \[
	    2 \E[T(t\cdot x)] \leq \E[T(2t \cdot  x)] + \bar{C} \sqrt{t} \log(t).
    \]
\end{lemma}
\begin{proof}
    Since the distribution of $T$ is rotation invariant, it suffices to verify the result for $T(t \cdot e_1)$ with $e_1 \in \partial B(o,1)$ an element of the canonical basis of $\R^d$. We take~$t \ge 0$ which will be assumed large throughout the proof.

	It is straightforward to see that, if~$t$ is large, there exist~$n_t \le t^d$ and (deterministic) points~$x_1,\ldots,x_{n_t} \in \partial B(o,t)$ such that
	\begin{equation}
		\label{eq_cover1}
		B(o,t)\backslash B(o,t-r) \subseteq \bigcup_{i=1}^{n_t} B(x_i,t^{1/4}).
	\end{equation}
	We set~$y_i:=2te_1 + x_i \in \partial B(2te_1,t)$, so that
	\begin{equation}
		\label{eq_cover2}
		B(2te_1,t)\backslash B(2te_1,t-r) \subseteq \bigcup_{i=1}^{n_t} B(y_i,t^{1/4}).
	\end{equation}
	Define~$\mathsf F_t$ as the event that all of the following happen:
	\begin{itemize}
		\item $q(o) \in B(o,t)$,
		\item $q(2te_1) \in B(2te_1,t)$,
		\item for every~$z \in \{x_1,\ldots,x_{n_t},y_1,\ldots,y_{n_t}\}$ and all~$w \in B(z,t^{1/4})$, we have~$T(z,w) \le t^{1/2}$.
	\end{itemize}

	Now, fix a realization of the random graph and the passage times, and let~$\gamma$ be a path from~$q(o)$ to~$q(2te_1)$ that minimizes the passage time between these two vertices. On~$\mathsf F_t$, this path starts in~$B(o,t)$ and ends in~$B(2te_1,t)$. Traversing~$\gamma$ from~$q(o)$ to~$q(2te_1)$, let~$u^*$ be the first vertex in the annulus~$B(o,t)\backslash B(o,t-r)$, and let~$v^*$ be the first vertex in the annulus~$B(2te_1,t)\backslash B(2te_1,t-r)$. Then, by~\eqref{eq_cover1} and~\eqref{eq_cover2}, there exist~$i_*$ such that~$u^* \in B(x_{i^*},\sqrt{t})$ and~$v^* \in B(y_{j^*},\sqrt{t})$. We then have
	\begin{align*}
		T(2te_1) \cdot \mathbbm{1}_{\mathsf F_t} &\ge  (T(u^*) + T(v^*,2te_1)) \cdot  \mathbbm{1}_{\mathsf F_t} \\
		&\ge (T(x_{i^*}) + T(y_{j^*},2te_1) - 2\sqrt{t}) \cdot \mathbbm{1}_{\mathsf F_t}.
	\end{align*}
	Taking the expectation and using symmetry, this gives
	\begin{equation}\label{eq_split_two_xs}
		\begin{split}
			\E[T(2te_1)] &\ge 2 \E\left[\mathbbm{1}_{\mathsf F_t} \cdot  \min_{1 \le i \le n_t} T(x_i)  \right] - 2\sqrt{t} \\
			&= 2\E\left[  \min_{1 \le i \le n_t} T(x_i)  \right] - 2 \E\left[\mathbbm{1}_{\mathsf F_t^c} \cdot  \min_{1 \le i \le n_t} T(x_i)  \right] - 2\sqrt{t}.
		\end{split}
	\end{equation}

	We will deal with the two expectations on the right-hand side above separately. Define~$\E[T(x_i)] = \E[T(te_1)] =: \nu_t$ for every~$i$. We will need the fact that there exists some~$C> 0$ such that, for~$t$ large enough,
	\begin{equation}\label{eq_fast_for_nu}
		\nu_t \le C t.
	\end{equation}
	This can be easily obtained from~\cref{prop:Hn.growth}.

	Let us give a lower bound for the first expectation on the right-hand side of~\eqref{eq_split_two_xs}. Using Theorem~\ref{thm_new_moderate_deviations}, we can choose~$C > 0$ such that, for~$t$ large enough, we have~$\p(T(te_1) \le \nu_t - C\sqrt{t}\log(t)) < e^{-t}$. We then bound
	\begin{align*}
		\E\left[  \min_{1 \le i \le n_t} T(x_i)  \right] &\ge (\nu_t - C\sqrt{t}\log(t))\cdot  \p(T(x_i) > \nu_t - C \sqrt{t} \log(t) \text{ for all }i)\\
		&\ge (\nu_t - C\sqrt{t}\log(t))\cdot (1-n_t\cdot \p(T(te_1) \le \nu_t - C\sqrt{t}\log(t))\\
		&\ge (\nu_t - C\sqrt{t}\log(t))\cdot (1-t^d\cdot e^{-t}).
	\end{align*}
	Using~\eqref{eq_fast_for_nu}, we see that the right-hand side above is larger than~$\nu_t - C'\sqrt{t}\log(t)$ for some constant~$C' > 0$ and~$t$ large enough.

	To give an upper bound for the second expectation on the right-hand side of~\eqref{eq_split_two_xs}, we first bound
	\[\E\left[\mathbbm{1}_{\mathsf F_t^c} \cdot  \min_{1 \le i \le n_t} T(x_i)  \right] \le \p(\mathsf F_t^c)^{1/2}\cdot  \E[T(te_1)^2]^{1/2}\]
	using Cauchy-Schwarz. Next, we have
	\begin{align*}
		\p(\mathsf F_t^c) &\le 2 \p(\mathcal H \cap B(o,t) = \varnothing) + n_t\cdot \p\left(\max_{w \in B(o,t^{1/4})} T(o,w) > \sqrt{t}\right)\\
		&\le 2Ce^{-ct} + t^d \cdot C e^{-ct^{1/4}}
	\end{align*}
for some~$C,c> 0$, by~\cref{prop:Hn.growth,lm:T.bds.sup}. We also bound
	\begin{align*}
		\E[T(te_1)^2] = \mathrm{Var}(T(te_1)) + \E[T(te_1)]^2 \le Ct^2
	\end{align*}
	for some~$C > 0$, by~\eqref{eq:VarT} and~\eqref{eq_fast_for_nu}. 
	
	Putting things together, we have proved that the right-hand side of~\eqref{eq_split_two_xs} is larger than
	\begin{align*}
		&2(\nu_t - C'\sqrt{t}) - 2\cdot (2Ce^{-ct} + t^d \cdot C e^{-ct^{1/4}})^{1/2} \cdot (Ct^2)^{1/2} -2\sqrt{t} \\[.2cm]
		&\ge 2\nu_t - C'' \sqrt{t}\log(t)
	\end{align*}
	for some~$\bar{C} > 0$ and~$t$ large enough.
\end{proof}

We can now conclude the proof of Theorem~\ref{thm:moderate.dev.FPP} by applying the lemmas above.

\begin{proof}[Proof of Theorem~\ref{thm:moderate.dev.FPP}, item \eqref{eq_asymp_exp}]
	Recall the isotropic properties and subadditivity of $T(x)$. By Kingman's subadditive ergodic theorem
    \[
        \lim_{n \uparrow +\infty}\frac{T(nx)}{n} = \lim_{n \uparrow +\infty}\frac{\E\big[T(nx)\big]}{n} = \inf_{n \in \N}\frac{\E\big[T(nx)\big]}{n} = \phi(x)
    \]
	where $\phi(x)= \|x\|/\upvarphi$ with $\upvarphi>0$ (see \cite[Eq.~(3.6)]{coletti2023} for details). 

	Since $\E\big[T(x)\big] = \E\big[T(te_1)\big]$ for $t=\|x\|$, it suffices to prove the result for $\E\big[T(te_1)\big]$ with large $t>0$. Let $\bar{C}>0$ be the constant of Lemma~\ref{lm_expectation.double}. One can obtain the result by applying Lemma 4.2 of Howard and Newman~\cite{howard2001}. To be self-contained and more precise, let us define
	\[f(t):=\E\big[ T(te_1) \big]\quad  \text{and} \quad \Tilde{f}(t):= f(t)-10\bar{C}\sqrt{t}\log(t), \quad t > 0.\]
	We have
	\[\Tilde{f}(2t) = f(2t) - 10 \bar{C} \sqrt{2t} \log(2t) \ge 2f(t) - \bar{C}\sqrt{t}\log(t) - 10\bar{C}\sqrt{2t}\log(2t),\]
	where the inequality holds for~$t$ large enough, by~Lemma~\ref{lm_expectation.double}. When~$t$ is large we have~$10\sqrt{2t}\log(2t) \le 19\sqrt{t}\log(t)$ (since~$10\sqrt{2} < 19$ and~$\log(2t)/\log(t) \to 1$ as~$t \uparrow +\infty)$, so we obtain
	\[\Tilde{f}(2t) \ge 2 f(t) -  \bar{C}\sqrt{t}\log(t) - 19\bar{C}\sqrt{t}\log(t) = 2f(t) - 20\bar{C}\sqrt{t}\log(t) = 2\Tilde{f}(t).\]
	Iterating~$\Tilde{f}(2t) \ge 2 \Tilde f(t)$, we get
    \[
        \frac{\Tilde{f}(2^n t)}{2^n t} \geq \frac{\Tilde{f}(t)}{t} \quad \text{for all large } t \text{ and all }n \in \N.
    \]
    Taking a limit of the left-hand side as~$n \uparrow +\infty$ we obtain that, for all large~$t$,
	\[ \frac{1}{\upvarphi} \ge \frac{\Tilde{f}(t)}{t} = \frac{\E\big[T(t\cdot e_1)\big] - 10 \bar{C} \sqrt{t}\log(t)}{t},\]
    which yields the desired conclusion.
\end{proof}

\section{Quantitative Shape Theorem}\label{s_quantitative}

Building upon the results established in Theorems \ref{thm_new_moderate_deviations} and \ref{thm:moderate.dev.FPP}, this section focuses on the speed of convergence of the asymptotic shape presented in Theorem \ref{thm:speed.FPP}.

\begin{proof}[Proof of Theorem~\ref{thm:speed.FPP}]
	By putting together Theorem~\ref{thm_new_moderate_deviations} and~\eqref{eq_asymp_exp}, we can obtain~$C^* > 0$ such that $\p$-almost surely, for~$x$ with~$\|x\|$ large enough we have
	\begin{equation}\label{eq_choice_of_R}
		\left| T(x) - \frac{\|x\|}{\upvarphi}\right| \le C^* \sqrt{\|x\|}\log\big(\|x\|\big).
	\end{equation}
	Writing
	\[g_+(a):= \frac{a}{\upvarphi} + C^* \sqrt{a} \log(a),\quad g_-(a):= \frac{a}{\upvarphi} - C^* \sqrt{a} \log(a),\quad a > 0,\]
	the above inequality can be written as
	\begin{equation}
		\label{eq_becomes}
	g_-\big(\|x\|\big) \le T(x) \le g_+\big(\|x\|\big).
	\end{equation}

	Fix a realization of the random graph and passage times, and let~$R$ be large enough that~\eqref{eq_becomes} holds for all~$x$ with~$\|x\| \ge R$. We can take~$R' \ge R$ such that~$g_-$ is increasing on~$[R',\infty)$.  Now fix~$t$ large enough that
	\begin{equation*}
		t \ge \max\left\{\max_{x \in B(o,R')} T(x),\; \frac{R'}{\upvarphi}\right\}.
	\end{equation*}
We will now prove that, letting~$b := 4\upvarphi \sqrt{\upvarphi} C^*$, we have, $\p$-almost surely,
	\begin{equation}\label{eq_two_inclusions} 
		B(o, \upvarphi t - b \sqrt{t}\log(t)) \subseteq H_t \subseteq B(o, \upvarphi t + b \sqrt{t}\log(t)).
	\end{equation}

	To prove the first inclusion, we fix
	\begin{equation}
		\label{eq_my_choice_x}
	x \in B(o,\upvarphi t - b \sqrt{t}\log(t))
	\end{equation}
	and show that~$T(x) \le t$. This inequality is automatic from the choice of~$t$ if~$x \in B(o,R')$, so assume that~$\|x\| > R'$. Then, since~$g_+$ is increasing,
	\begin{align*}
		T(x) &\stackrel{\eqref{eq_becomes}}{\le} g_+\big(\|x\|\big)  \stackrel{\eqref{eq_my_choice_x}}{\le} g_+(\upvarphi t - b\sqrt{t}\log(t)) \\[.2cm]
		&= t - \frac{b}{\upvarphi} \sqrt{t}\log(t)+ C^* \sqrt{\upvarphi t - b\sqrt{t} \log(t)} \cdot \log(\upvarphi t - b \sqrt{t}\log(t)).
		 	\end{align*}
	Increasing~$t$ if necessary, we have
	\[\sqrt{\upvarphi t - b\sqrt{t} \log(t)} \cdot \log(\upvarphi t - b \sqrt{t}\log(t)) \le 2\sqrt{\upvarphi} \cdot \sqrt{t}\log(t),\]
so we get
	\[T(x)\le t - \frac{b}{\upvarphi}\sqrt{t} \log(t) + 2C^* \sqrt{\upvarphi} \cdot \sqrt{t} \log(t) \le t\]
	by the choice of~$b$.

	To prove the second inclusion in~\eqref{eq_two_inclusions}, fix~$x \notin B(o,\upvarphi t + b \sqrt{t}\log(t))$, and let us show that~$T(x) > t$. We start with
	\[T(x) \ge g_-\big(\|x\|\big).\]
	Since~$\|x\| \ge \upvarphi t + b\sqrt{t}\log(t) \ge R'$, and~$g_-$ is non-increasing in~$[R',\infty)$, we have
	\begin{align*}
		g_-\big(\|x\|\big) &\ge g_-(\upvarphi t + b\sqrt{t}\log(t))\\[.2cm]
		&= t + \frac{b}{\upvarphi} \sqrt{t}\log(t)- C^* \sqrt{\upvarphi t + b\sqrt{t} \log(t)} \cdot \log(\upvarphi t + b \sqrt{t}\log(t)).
	\end{align*}
Increasing~$t$ if necessary, we have
	\[\sqrt{\upvarphi t + b\sqrt{t} \log(t)} \cdot \log(\upvarphi t + b \sqrt{t}\log(t)) \le 2\sqrt{\upvarphi} \cdot \sqrt{t}\log(t),\]
	so
	\[T(x) \ge t + \frac{b}{\upvarphi} \sqrt{t}\log(t) - 2C^*\sqrt{\upvarphi} \sqrt{t}\log(t) \ge t\]
	again by the choice of~$b$, completing the proof.
\end{proof}

\section{Fluctuations of the Geodesics and Spanning Trees}

Having established the framework of moderated deviations, we now investigate the probabilistic behavior of the geodesic paths. Specifically, we leverage the derived properties to characterize the probabilities of geodesics residing in specific regions of the space. Additionally, we gain insights regarding the spanning trees generated by the passage times.

\begin{proof}[Proof of \cref{thm:fluctuations.of.geodesics}]
Let us fix $s := \|y-x\|$ for given $x, y \in \R^d$. Recall the notation defined previously $[U]_{\varepsilon}$ for the $\varepsilon$-neighbourhood of $U \subseteq \R^d$. To obtain an upper bound for the probability in first part of the theorem, consider the events 
\begin{align*}
    \mathsf{A}_{x,y}^\epsilon &:= \left\{v \in V(\Hh) \setminus [\overline{xy}]_{s^{3/4+\epsilon}} \colon d_H(v, [\overline{xy}]_{s^{3/4+\epsilon}})< r\right\} \quad\text{and}\\
    \mathsf{E}_{x,y}^\epsilon &:= \big\{\exists v \in \mathsf{A}_{x,y}^\epsilon \colon T(x,y) = T(x, v) +T(v,y) \big\},
\end{align*}
then $\p\left(d_H(\upgamma_x(y), \overline{xy}) \geq s^{\frac{3}{4}+\epsilon}\right) \leq \p\big(\mathsf{E}_{x,y}^\epsilon\big)$.

Set $\Delta_{x,y}^v := \left(\|v-x\|+\|y-v\|-s\right)/\upvarphi$ with $\upvarphi>0$ the constant determining the asymptotic shape in \cref{thm:speed.FPP}. Observe that, if $T(x,y) = T(x,v)+T(v,y)$, then
\[
    T(x,y) -s/\upvarphi - T(x,v) + \|v-x\|/\upvarphi - T(v,y) + \|y-v\|/\upvarphi = \Delta_{x,y}^v. 
\]
Thence,
\[
    \vert T(x,y) -s/\upvarphi\vert + \vert T(x,v) - \|v-y\|/\upvarphi\vert + \vert T(v,y) - \|y-v\|/\upvarphi \vert \geq \Delta_{x,y}^v. 
\]
It thus follows that
\begin{align*}
    \p\big(\mathsf{E}_{x,y}^\epsilon\big)&\leq \p\big(\left\vert T(x, y) - s/\upvarphi \right\vert \geq \Delta_{x,y}^v/3\big)\\
    & \quad\quad  + \p\big(\exists v \in \mathsf{A}_{x,y}^\epsilon \colon \left\vert T(x, v) - {\|v-x\|}/{\upvarphi} \right\vert \geq {\Delta_{x,y}^v}/{3}\big)\\
    & \quad\quad + \p\big(\exists v \in \mathsf{A}_{x,y}^\epsilon \colon \left\vert T(v, y) - {\|y-v\|}/{\upvarphi} \right\vert \geq {\Delta_{x,y}^v}/{3})\\
    &\leq \p\big(\left\vert T(x, y) - s/\upvarphi \right\vert \geq \Delta_{x,y}^v/3\big)\\
    & \quad\quad +2\p\left(\bigcup_{v \in \mathsf{A}_{x,y}^\epsilon}\left\{\left\vert T(x, v) - {\|v-x\|}/{\upvarphi} \right\vert \geq {\Delta_{x,y}^v}/{3}\right\}\right).
\end{align*}

By the Pythagorean Theorem, there exist $\Tilde{c}_0,\Tilde{c}_1>0$ such that, if $s>r$, then
\[\Tilde{c}_0 s^{\frac{1}{2}+2\epsilon} \le \Delta_{x,y}^v \le \Tilde{c}_1 s^{\frac{3}{4}+\epsilon}.\]
Therefore \eqref{eq:geodesic.fluctuations} is verified as a direct consequence of \cref{prop:Hn.growth}, \cref{thm_new_moderate_deviations,thm:moderate.dev.FPP}, and the translation invariance of $T$.
\end{proof}

Before proving \cref{thm:f.straight.spanning.trees}, we state below Lemma 3.7 of \cite{howard2001}.

\begin{lemma} \label{lm:3.7.howard.newman}
    Consider $d \geq 2$, $\uppsi \in (0,1/4)$, and $\{q_i\}_{i \in \N}$ such that $\|q_j\| \uparrow +\infty$ as $j\to+\infty$. Suppose that, for all large $j \in \N$,
    \begin{equation} \label{eq:lm.qn}
        \|q_j-q_{j+1}\| \le \|q_j\|^{3/4} \quad\text{and}\quad \inf\big\{\|y -q_k\| \colon y \in \overline{o q_j}\big\}\leq \|q_j\|^{1-\uppsi} \ \ \text{for } k<j. 
    \end{equation}
    Then there exist $c_\uppsi,\Tilde{k}>0$ such that, for all $k \in \big\{\Tilde{k},\Tilde{k}+1,\dots,j-1\big\}$, one has
    \begin{equation} \label{eq:angle.theta.ineq}
        \uptheta(q_k,q_j) \le c_\uppsi~\|q_k\|^{-\uppsi}.
    \end{equation}
\end{lemma}

\begin{proof}[Proof of \cref{thm:f.straight.spanning.trees}]
    First, let $\mathcal{A}(x,r_1,r_2)$ stand for an annulus $B(x,r_2)\setminus B(x, r_1)$. Consider now, for $n\in \N$,
    \begin{gather*}
        \mathsf{X}_{x,n}^{\epsilon} := \left\{ y \in \mathcal{A}(x, n-1,n) \cap V(\Hh) \colon d_H\big( \upgamma_x(y),~\overline{xy} \big) \geq \|y-x\|^{\frac{3}{4}+ \frac{\epsilon}{2}} \right\}, \\
        \text{and} \quad \mathsf{X}_x^{\epsilon} := \Dot{\bigcup_{n\in \N}}\mathsf{X}_{x,n}^{\epsilon}.
    \end{gather*}
    Observe that, by \cref{lm:3.7.howard.newman} with $\uppsi= 1/4-\epsilon/2$, it suffices to define
    \[\mathlcal{F}_{x,f_\epsilon}= \mathsf{X}_x^\epsilon \cup \left( 
B\big(x, r^{4/3} \vee c_\uppsi^{2/\epsilon}\big) \cap V(\Hh)\right)\]
    to obtain $\mathlcal{V}_{x, f_\epsilon} \subseteq \mathlcal{F}_{x, f_\epsilon}$ for any choice of $\mathlcal{T}_x$. One can easily see from \cref{prop:Hn.growth} that, $\p$-a.s., any finite region of $\R^d$ contains a finite number of $v \in V(\Hh)$. Therefore, it remains to verify that $|\mathsf{X}_x^\epsilon|$ is $\p$-a.s. finite.

    It follows from \cref{thm:fluctuations.of.geodesics}, \cref{prop:Hn.growth}, and Lebesgue's Dominated Convergence Theorem that there exists $\Tilde{c}>0$ such that
    \[
        \E\big[\vert\mathsf{X}_x^{\epsilon}\vert\big] = \sum_{n \in \N}\E\big[\vert\mathsf{X}_{x,n}^{\epsilon}\vert\big]\leq \sum_{n \in \N} \frac{\Tilde{c}~ \cdot n^d}{\exp(n^{\epsilon})} < +\infty
    \] 
    Hence, by Markov inequality, there exists $\Bar{c}>0$, for all $x\in\R^d$, and $s \geq 0$,
    \begin{equation} \label{eq:Markov.number.non.out.sites}
        \p\big(|\mathsf{X}_x^\epsilon| \geq s^2 \big) \leq \frac{\Bar{c}}{s^2},
    \end{equation}
    which completes the proof by an application of Borel-Cantelli Lemma.
\end{proof}

The corollary naturally arises as a straightforward consequence of the results above, with the application of a proposition found in Howard and Newman \cite{howard2001}.

\begin{proof}[Proof of \cref{cor:asymp.dir}]
    It follows from \cref{prop:Hn.growth} that any spanning tree $\mathlcal{T}_x$ of $\Hh$ is locally finite. By \cref{prop:holes.H}, the spherical holes $\mathcal{D}(s) \in \mathcal{O}\big(\log(s)\big)$ $\p$-a.s. as $s \uparrow +\infty$. Thus $V(\Hh)$ is asymptotic omnidirectional, \textit{i.e.}, for all $n \in \N$, 
    \[\left\{\frac{v}{\|v\|} \colon v \in V(\Hh) \setminus B(o,n) \right\} \text{ is dense in } \partial B(o,1).\]
    
    Since \cref{thm:f.straight.spanning.trees} implies that $\mathlcal{T}_x$ is  $\p$-a.s. $f_\epsilon$-straight for all $\epsilon \in (0, 1/4)$ with ${ f_\epsilon(s)= s^{\epsilon-\frac{1}{4}} }$, \cref{cor:asymp.dir} is a straightforward application of Proposition 2.8 of Howard and Newman \cite{howard2001}.
\end{proof}

%% file: texts/competition.tex
\section{A Two-Species Competition Model on RGGs} \label{ch:competition}

Using the theorems above, we investigate a two-species competition model originally formulated by Kordzakhia and Lalley \cite{kordzakhia2005} for the hypercubic lattice $\Z^d$.

Designating the two competing species as red and blue, each establishes territories within the spatial domain $\Hh$. The occupancy of sites by the red species at any given time $t \in [0, +\infty)$ is symbolized by $\upxi(t)$, while $\upzeta(t)$ represents the analogous territory held by the blue species.

At the outset, $\upchi(t) := \upxi(t) \dot{\cup} \upzeta(t) \subseteq \Hh$ is defined as the combined inhabited territory. The dynamics governing growth and competition are determined by Richardson's and voter's models. Within this framework, the competition unfolds as follows:
\begin{itemize}
    \item Unoccupied sites at time $t$, denoted as $x \not\in \upchi(t)$, are subject to occupation by either species. The rate of occupation by the red or blue species is determined by the presence of neighboring sites already occupied by each species. More specifically, the rates are
    \[\sum_{y \sim x}\mathbbm{1}_{\upxi(t)}(y) \quad \text{and} \quad \sum_{y \sim x}\mathbbm{1}_{\upzeta(t)}(y).\]

    \item Occupied sites at time $t$, specifically $x \in \upchi(t)$, witness potential invasion. Here, the transition to a different species occurs at rates contingent upon neighboring sites' current occupants.
    \begin{itemize}
        \item If $x \in \upxi(t)$, it undergoes a color shift to blue at rate $\sum_{y \sim x}\mathbbm{1}_{\upzeta(t)}(y)$.
        \item Likewise, if $x \in \upzeta(t)$, the site is becomes inhabited by the red species at rate $\sum_{y \sim x}\mathbbm{1}_{\upxi(t)}(y)$.
    \end{itemize}
\end{itemize}

\begin{figure}[htb!]
    \centering
    \includegraphics[width=0.8\linewidth]{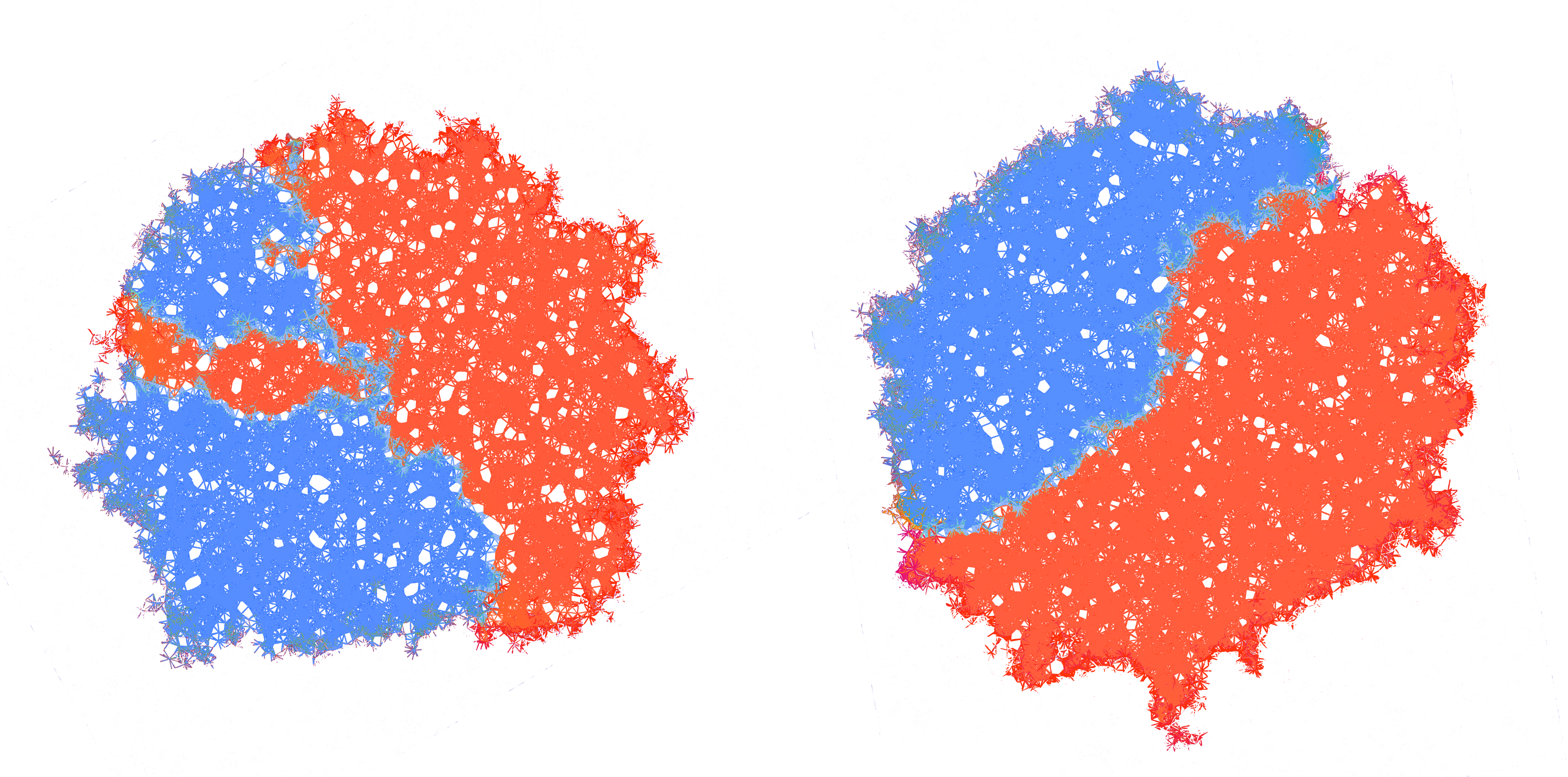}
    \caption{Two simulations of the competition model on a RGG.}
\end{figure}

The initial conditions for $\upxi(0)$ and $\upzeta(0)$ are established by two given sets in the $d$-dimensional space. The two-species competition model is considered to have a \emph{finite initial configuration} if there exist two non-empty disjoint sets $W$ and $W'$ in a finite region of $\R^d$ such that
\[\upxi(0) \subseteq q(W), \ \upzeta(0) \subseteq q(W'), \ \text{and} \ \upchi(0)=q(W \cup W').\]
The configuration of the set $q(W) \cap q(W')$ can be chosen arbitrarily. Our primary objective with this model is to determine whether both species coexist with positive probability. To this end, we introduce the event
\[\mathrm{Coex}(\upxi,\upzeta):=\left\{\text{for all }t\geq 0, \ \upxi(t) \neq \varnothing \text{ and } \upzeta(t) \neq \varnothing \right\}.\]

We then present the following result:
\begin{theorem} \label{thm:coexistence}
    Let $d \geq 2$ and $r>r_c(\lambda)$, and consider the two-species competition model defined above. Then, for any finite initial configuration, one has
    \[\p\left(\operatorname{Coex}(\upxi,\upzeta)\right)>0.\]
\end{theorem}

\subsection{The Intermediate Condition}

To investigate the coexistence of two species within this competition model over time, we will demonstrate that for a given sequence $t_n$ of time instances, it is possible to observe both species inhabiting specific regions of the space. This phenomenon is ensured by what we refer to as the \emph{intermediate condition}.

Let $\mathfrak{a}, \mathfrak{b} \in \left(\frac{3}{4}, ~1\right)$ be such that $\mathfrak{b}<2\mathfrak{a}-1$, and fix $\mathfrak{d} \in (0,1)$. Consider $\{t_n\}_{n\in\N_0}$ and $\{r_n\}_{n\in\N_0}$ to be sequences of times and angles, respectively, taking values in $(0,+\infty)$ such that, for all $n \in \N_0$, \[t_{n+1}-t_n = \mathfrak{d} t_n \quad \text{and} \quad r_n-r_{n+1}= (t_n)^{\mathfrak{a}-1}\] with $t_0>0$ to be determined by \eqref{intermediate.condition}. Then,  $t_n$ is given by $t_n :=t_0(1+\mathfrak{d})^n$ and we set \[r_n:=\frac{1+(1+\mathfrak{d})^{(\mathfrak{a}-1)\cdot n}}{1-(1+\mathfrak{d})^{(\mathfrak{a}-1)}}(t_0)^{\mathfrak{a}-1}.\]
Observe that $r_n$ decreases to $r_0/2>0$ as $n \uparrow +\infty$.

Fix $\bar{s} := \left(\frac{2(1+\mathfrak{d})}{\mathfrak{d}}\right)^{\frac{1}{1-\mathfrak{a}}}$. Recall the definition of the annulus $\mathcal{A}$ given in the proof of \cref{thm:f.straight.spanning.trees}. For $z \in \partial B(o,1)$ and $t_0 > \bar{s}$, define $\Upphi_n(z)$ as the random set of vertices in a region given by 
\[
    \Upphi_n(z) := \upvarphi\cdot\mathcal{A}\left(o, ~\tfrac{1}{1+\mathfrak{d}}t_n,~t_n-(t_n)^{\mathfrak{b}}\right) \cap \operatorname{Cone}(z, r_n)\cap \Hh.
\]
Let $\Uptheta_{w,w'}$ be the  event that guarantees the existence of vertices in regions of interest, defined as follows:
\[\Uptheta_{w,w'}:=\bigcap_{n\in\N_0} \big\{\Upphi_n(w)\neq\varnothing \ \text{ and } \ \Upphi_n(w')\neq \varnothing\big\}.\]
Consider $\widehat{H}$ to be the set $q^{-1}(H)$ for $H \subseteq \Hh$. Define $\Gamma_n$ as the event associated with convergence to the limiting shape, adjusted for $\upchi(0)$:
\[\Gamma_n :=\left\{B\left(o, \upvarphi(t_n - (t_n)^{\mathfrak{b}})\right) \subseteq \widehat{\upchi}_n \subseteq B\left(o, \upvarphi(t_n + (t_n)^{\mathfrak{b}})\right)\right\}.\]

Before stating the intermediate condition, we will restrict our attention to a specific part of the occupied sites $\upchi$. Note that, since $t_0>\bar{s}$, it follows that  $r_n \in(0, \pi/2)$ for $n\in \N_0$. Let us write, for all $n \in \N_0$, 
\[\upchi_{z,n}:=\upchi(t_n)\cap \left(\mathrm{Cone}(z,r_n)\left\backslash B\left(o,\tfrac{\upvarphi}{1+\mathfrak{d}}t_n\right)\right.\right).\] 

\begin{figure}[htb!]
    \centering
    \includegraphics[trim={70pt 70pt 70pt 70pt},clip,width=0.72\linewidth]{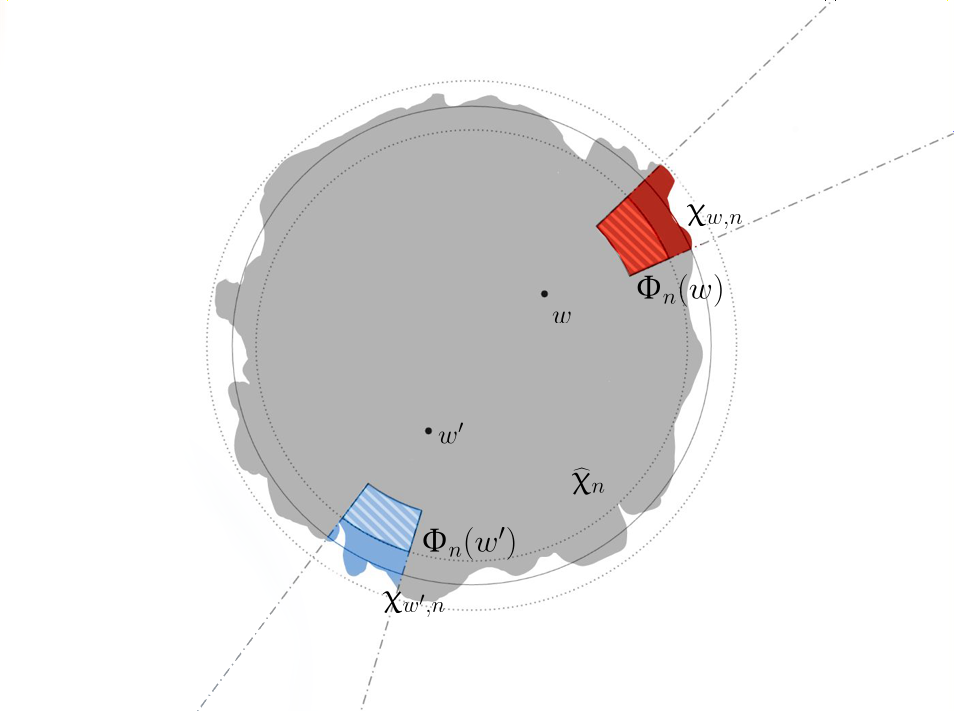}
    \caption{The regions $\upchi_{w,n}$, $\upchi_{w',n}$, $\Upphi_n(w)$, $\Upphi_n(w')$ on the event $\Gamma_n$.}
    \label{fig:comp_intermediate_condition}
\end{figure}

Now, we define $\Uppsi_{w,w'}$ by combining the events above, focusing on the occupation of selected regions of space by a unique species in each region. Specifically,
\[
    \Uppsi_{w,w'}:=\left\{
   \upchi_{w,0}\subseteq\upxi(t_0) \ \text{ and } \ \upchi_{w',0}\subseteq\upzeta(t_0)\right\} \cap \Uptheta_{w,w'} \cap \Gamma_0.
\]

Let the constant $s_0>\bar{s}$ be determined later (see \cref{lm:intermediate.condition} and the Proof of \cref{thm:coexistence} for details). Recall that $\uptheta(u,v)$ denotes the angle between $\vec{ou}$ and $\vec{ov}$. We can now state the intermediate condition for the competition model as follows: There exist $w,w' \in \partial B(o,1)$ and $t_0 \geq s_0$ such that $\uptheta(w,w')>2r_0$ and 
\begin{equation} \tag{\small \scshape I$_0$} \label{intermediate.condition}
    \p\big(\Uppsi_{w,w'}\big)>0.
\end{equation}

From now on, we consider the event inside the probability in \eqref{intermediate.condition} to denote the conditional probability of $\Uppsi_{w,w'}$ as
\[P(\cdot):=\p(~\cdot \mid\Uppsi_{w,w'}\big).\]
We will show that we can control the coexistence of both species within suitable regions of space on this event. Additionally, observe that $\Upphi_n(w) \subseteq\upchi_{w,n}$ and $\Upphi_n(w') \subseteq\upchi_{w',n}$ (see \cref{fig:comp_intermediate_condition}).

The following lemma ensures that the intermediate condition is reasonable for finite initial configurations of the competition model.

\begin{lemma} \label{lm:intermediate.condition}
    Let $W,W'$ be two non-empty disjoint sets of $~\R^d$ determining a finite initial configuration of the competition model. Then,for all $\varepsilon\in(0,1)$, there exists $s_0 > \bar{s}$ such that
    \[\p\left(\text{item } \eqref{intermediate.condition} \text{ holds true for } t_0 \geq s_0 \mid \upxi(0)\neq \varnothing \text{ and }\upzeta(0)\neq \varnothing\right)>1-\varepsilon.\]
\end{lemma}
\begin{proof}
    First, we verify that $\mathsf{p}_0:=\p\big(\upxi(0)\neq \varnothing \text{ and }\upzeta(0)\neq \varnothing\big)>0$ by fixing any $x \in W$ and $y\in W'$, then the event $q(x) \neq q(y)$ has strictly positive probability.

    Let us fix, without loss of generality, an arbitrary $w \in\partial B(o,1)$ and let $w'=-w$. Define $\Gamma$ to be the event $\bigcap_{n\in\N}\Gamma_n$.
    Since $W$ and $W'$ are in a finite region of $\R^d$, \cref{thm:speed.FPP,thm_new_moderate_deviations,prop:Hn.growth} ensure that we can approximate $\p(\Gamma \cap \Uptheta_{w,w'})$ to 1 as closely as we want as $s_0 \uparrow+\infty$. Let $s_0$ be large so that $\p(\Gamma^c \cup \Uptheta_{w,w'}^c) < \mathsf{p}_0\cdot\varepsilon /2$.

    It remains to verify that $\upchi_{w,0}\subseteq\upxi(t_0)$ and $\upchi_{w',0}\subseteq\upzeta(t_0)$ occur with positive probability. Set
    \[
    \Upphi_n^+(z) := \upvarphi\cdot\mathcal{A}\left(o, ~\tfrac{1}{1+\mathfrak{d}}t_n,~t_n+(t_n)^{\mathfrak{b}}\right) \cap \operatorname{Cone}(z, r_n)\cap \Hh.
\]

    Note that $\Upphi_n(z) \subseteq \upchi_{z,n} \subseteq \Upphi_n^+(z)$ on $\Gamma_n$. Consider $\upxi(0)$ and $\upzeta(0)$ to be non-empty. Let $\mathcal{W}_{\upxi,w}$ to be a connected subgraph of $\Hh$ that connects $\upxi(0)$ to $\Upphi_0^+(w)$ with $T$-geodesics. Similarly, define $\\mathcal{W}_{\upzeta,w'}$ to be a subgraph connecting $\upzeta(0)$ to $\Upphi_0^+(w')$.

    If $\mathcal{W}_{\upxi,w}$ and $\mathcal{W}_{\upzeta,w'}$ do not intersect, then $\upchi_{w,0}\subseteq\upxi(t_0)$ and $\upchi_{w',0}\subseteq\upzeta(t_0)$ happens with positive probability by forbidding the invasion dynamics from occurring  with the elements of $\mathcal{W}_{\upxi,w}$ and $\mathcal{W}_{\upzeta,w'}$ before time $t_0$. Let $\mathcal{W}$ be the graph union of $\mathcal{W}_{\upxi,w}$ and $\mathcal{W}_{\upzeta,w'}$. Consider now that case where $\mathcal{W}$ is connected.

    The worst case scenario is when $\mathcal{W}$ is a line graph that can prevent a invasion/occupation path from $\upxi(0)$ and $\upzeta(0)$ to $\upchi_{w,0}$ and $\upchi_{w',0}$, respectively (see \cref{fig:comp_line}). The referred paths can be obtained using, Harris' graphical construction with the invasion and occupation being determined by arrows that appear with distribution $\operatorname{Exp}(1)$. We can restrict or impose the arrows that paint $\upchi_{w,0}$ red and $\upchi_{w',0}$ blue when there exists at least one vertex in $\mathcal{W}$ with degree greater than $2$.

    \begin{figure}[htb!]
        \centering
        \includegraphics[width=0.5\linewidth]{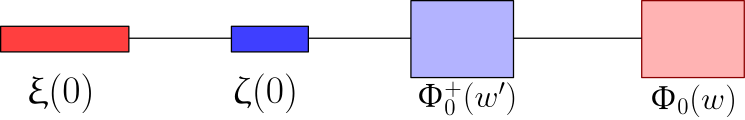}
        \caption{Schematics for worst case of the invasion/occupation dynamics in $\mathcal{W}$.}
        \label{fig:comp_line}
    \end{figure}

    To verify the claim above, consider that there exists $x \in \mathcal{W}$ with degree greater than $2$. Fix $x_a$, $x_b$, and $x_c$ to be three distinct neighbours of $x$. Choose one of $x_a$ and $x_b$ to be painted red and the other blue. If we remove the edges $\{x,x_a\}$ and $\{x,x_b\}$, $\mathcal{W}$ is decomposed into at most three connected components $\mathcal{K}_a$, $\mathcal{K}_b$, and $\mathcal{K}_c$ (see \cref{fig:comp_W}) containing $x_a$, $x_b$, and $x_c$, respectively. In the case with three distinct connected components, there is an invasion/occupation path from $x_a$ and $x_b$ to any configuration $\mathcal{K}_c$. Observe that the vertices $x$, $x_a$, $x_b$, and $x_c$ can assume any configuration. Consider $x_c$ red or blue to obtain any configuration of $\mathcal{K}_a$ and $\mathcal{K}_b$. This entire process can occur up to the given time $t_0$. The probability of such event may be particularly small, but it guarantees the strictly positive probability of $\upchi_{w,0}\subseteq\upxi(t_0)$ and $\upchi_{w',0}\subseteq\upzeta(t_0)$.

    \begin{figure}[htb!]
        \centering
        \includegraphics[width=0.35\linewidth]{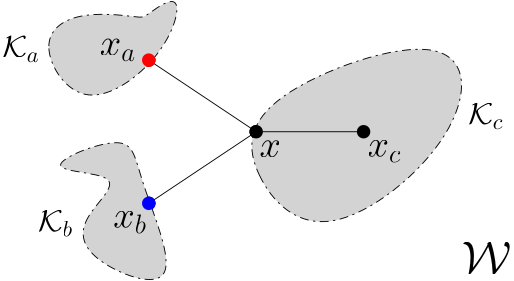}
        \caption{An example of $\mathcal{W}$ with $x$ having a degree greater than $2$.}
        \label{fig:comp_W}
    \end{figure}

    In the case with less than three connected components, consider a spanning tree in the component that contains more than one site among $\{x_a,x_b,x_c\}$. Suppose that $\mathcal{K}_c$ is equal to $\mathcal{K}_a$ (or $\mathcal{K}_b$), if when we cut the branches at $x_a$ (and/or at $x_b$) the vertices $x$ and $x_c$ are not in the branch cut, proceed just as in the case with three distinct connected components with the branches cut and a disjoint connected subgraph of $\mathcal{K}_c$ of the remaining vertices of $\mathcal{W}$. Otherwise, if $x$ or $x_c$ is removed of the spanning tree with a branch cut, we simply take cut the branch at $x$ and/or $x_c$ and keep the edge ${x,x_c}$, then the rest of the sites in $\mathcal{K}_c$ will be reached through $x_a$ (or $x_b$). Similar arguments apply in the case where $\mathcal{K}_a=\mathcal{K}_b\neq\mathcal{K}_c$.

    Let $E_{w,w'}$ be the event that every site in $\Upphi_0(w) \cup \Upphi_0(w')$ has degree at most $2$. Due to the properties of the homogeneous Poisson point process, one can choose $s_0> \bar{s}$ such that, for $t_0 \geq s_0$, $\p(E_{w,w'})< \mathsf{p}_0\cdot\varepsilon/2$. Therefore, 
\begin{align*}
  &\p\left(\text{item } \eqref{intermediate.condition} \text{ does not hold true for a }t_0 \geq s_0 \mid \upxi(0)\neq \varnothing \text{ and }\upzeta(0)\neq \varnothing\right) \\ & \hspace{5cm}
  \leq \p(\Gamma^c \cup \Uptheta_{w,w'})/\mathsf{p}_0 + \p(E_{w,w'})/\mathsf{p}_0 < \varepsilon ,
\end{align*}
which yields the desired conclusion.
\end{proof}

\subsection{Control of invasion times}

One of the key aspects of the competition model is controlling the invasion dynamics within the graph. We will utilize its correspondence to random walks to derive the necessary bounds that enable effective control of the invasion.

Consider $S^u_t$ to be a continuous-time simple random walk on a graph ${\G'=(V',\mathcal{E}')}$ with . Let $\deg_{\max}(\mathcal{G}')$ be the maximum vertex degree of $\mathcal{G}'$ and denote by $D'(u, v)$ the graph metric between $u, v \in V'$. The following lemma is a straightforward consequence of Corollary 11 of Davies \cite{davies1993} by using heat kernel techniques (see \cite{barlow2017} for details).

\begin{lemma} \label{lm:heat.kernel} 
    Let $\G'$ be a graph with bounded vertex degree, then
        \[\p(S^u_t=v) \leq \operatorname{deg}_{\max}(\mathcal{G}')\cdot\exp\left(-\frac{D'(u,v)^2}{2t} \left(1-\frac{D'(u,v)^2}{10t^2}\right)\right).\]
\end{lemma}

We apply this result in the proof of the lemma below. First, we define the random ball $\bar{B}_D(x, s):=\{y\in\Hh \colon D(x,y) \leq s\}$. The following result bounds the probability that a region occupied by one species will be invaded within a given time, in relation to its volume.

\begin{lemma} \label{lm:control.invasion}
    Let $x \in \R^d$ and $\rho > 0$ with $B(x,\rho^{\mathfrak{b}}) \cap \Hh \subseteq \upxi(0)~$. Then there exist constants $C,C'>0$ such that, for all $t \in [0,\rho]$,
    \begin{equation} \label{eq:invasion.q(x)}
        \p\big(q(x) \in \upzeta (t)\big) \leq C \exp\left( -C' ~\rho^{2{\mathfrak{b}}-3/2} \right).
    \end{equation}
\end{lemma}
\begin{proof}
     It suffices to consider by stochastic domination that $\upxi(0) = B(x,\rho^{\mathfrak{b}}) \cap \Hh$ and $\upzeta(0) = \Hh\setminus \upxi(0)$. Recall that the dual process of the voter model is the coalescent random walk. Therefore, one can prove the lemma verifying properties of random walks on the infinite component. Let us denote by $S_t^{x'}$ a continuous-time simple random walk on $\Hh$ starting at $q(x')$. Set 
     \[\uptau^x_s := \inf \{t>0 : S_t^x \not \in B(x,s^{\mathfrak{b}})\}.\]
     We will verify \eqref{eq:invasion.q(x)} applying the following claim.
     
    \begin{claim} \label{cl:first.exit}
         There exist $c_1',c_2'>0$ such that
        \begin{equation} \label{eq:first.exit.p.bound}
            \p\left(\uptau^x_\rho < t\right) \leq c_1' \exp\left( -c_2' ~\rho^{2{\mathfrak{b}}-1} \right).
        \end{equation}
    \end{claim}
    \begin{proof} \renewcommand\qedsymbol{$\blacksquare$}
    Consider w.l.o.g. that $\rho \geq 1$, we will treat the case $\rho <1$ will separately. In order to verify \eqref{eq:first.exit.p.bound}, we will first study $S_t^x$ on events with suitable properties for $\Hh$. The strategy is to control the distribution of points and vertex degrees within a given region of the graph.
    
    By Theorem 2.2 of \cite{yao2011} and Palm calculus, there exist $c'_0>1$ and $c''_1,c''_2>0$ such that, for all $s\geq1$,
    \begin{equation} \label{eq:D.Euclid.Lm.Yao}
        \p\big(B(x,s)\cap\Hh \not\subseteq \bar{B}_D(x, c'_0s)\big) \leq c''_1\exp(-c''_2 s).
    \end{equation}
    Set
    \[
        E_1 := \left\{\|q(x)-x\|<\rho^{\mathfrak{b}}/2,~ B(x,\rho^{\mathfrak{b}})\cap\Hh \subseteq \bar{B}_D(x, c'_0~\rho^{\mathfrak{b}}) \right\}.
    \]
    Let $c'':= \frac{2e+1}{2}c_0'r$ and set $\widetilde{\mathbbm{B}}(x,\rho) := x+c''\rho^2[-1,1]^d$ and denote by $\deg(u)$ the vertex degree of $u \in \Hh$. Write $\Hh_{x,\rho}$ for the subgraph of $\Hh$ restricted to the vertices $\widetilde{\mathbbm{B}}(x,\rho) \cap \Hh$. 
    Define 
    \[
        E_2 := \big\{\deg_{\max}(\Hh_{x,\rho}) \leq \sqrt{\rho}~\big\}.
    \]
    Note that $\widetilde{\mathbbm{B}}(x,\rho)$ can be embedded in a partition with $\left\lceil \frac{2e+1}{4}c_0'\rho^2\right\rceil^d$ hypercubes whose side has length $2r$. We obtain an upper bound for $\p(E_2^c)$ by considering the event in which the PPP assigns more than $\sqrt{\rho}/3^d$ points to at least one of the above-referred hypercubes. By Chernoff bound, if $X \sim \operatorname{Poi}(\lambda')$ with respect to $P$, then $P(X \geq s) \leq \exp\left(\lambda'(e^{s'}-1)-s's\right)$. Let $s'=\log(s/\lambda')>0$ in the previous inequality, then $P(X \geq s)< (e \cdot \lambda'/s)^s$ for all $s>0$. Hence,
    \begin{equation} \label{eq:p.E2.comp.bound}
        \p\big(E_2^c\big) \leq (3c_0')^d \rho^{2d} \left( \frac{e \cdot 6^d r^d }{\sqrt{\rho}}\right)^{\sqrt{\rho}}.
    \end{equation}

    It is a well known fact that $\deg_{\max}(\Hh)$ is $\p$-a.s. unbounded. Let $S_t^{x,\rho}$ stand for $S_t^x$ restricted to $\Hh_{x,\rho}$. Denote by $\mathsf{N}^{x}_\rho$ the number of jumps performed by $S^{x,\rho}_t$ up to time $\rho$. Let $\mathsf{N}(\rho)$ be the counting of the homogeneous PPP on $[0,\rho]$ with rate $\sqrt{\rho}$ such that $\mathsf{N}^x_\rho \le \mathsf{N}(\rho)$ on $E_2$. 
    
    Define $E_3 := \{\mathsf{N}(\rho) < \rho^2 \cdot e\}$ and let $E := E_1\cap E_2 \cap E_3$. Then, by the same inequality used above for the Poisson distributions, one has that
    \begin{equation} \label{eq:p.E3.c.bound}
        \p(E_3^c) \leq \rho^{-\rho^2/2}.
    \end{equation}
    Observe that, since $B(x, \rho^{\mathfrak{b}}/2+r\cdot e\cdot c_0' \rho^2 +r) \subseteq \widetilde{\mathbbm{B}}(x,\rho)$, it follows that $S_t^x = S_t^{x,\rho}$ on $E$ for all $t \in [0, \rho]$. We will estimate
    \begin{equation} \label{eq:St.ball.bound}
        \p\left(S_t^x \not \in B\left(x,\frac{\rho^{\mathfrak{b}}}{2}\right)\right) \le \p\left(\left\{S_t^{x,\rho} \not \in \bar{B}_D\left(x,\frac{c_0'}{2}\rho^{\mathfrak{b}}\right)\right\} \cap E_1 \cap E_2 \right) + \p(E^c).
    \end{equation}
    
    By \cref{lm:heat.kernel}, one has that, for all $u,v \in \Hh_{x, \rho}$,
    \[
        \p(S_t^{u,\rho}=v \mid E_1\cap E_2) \leq \deg_{\max}(\Hh_{x,\rho})\mathsf{q}_t\big(q(x),v\big)
    \]
    with
    \[
        \mathsf{q}_t\big(q(x),v\big) \le 2\sqrt{10}c_0'\rho^{\mathfrak{b}} \cdot\exp\left( - \frac{D\big(q(x),v\big)^2}{4t}\right)
    \]
    Let us write $\mathlcal{S}(x,\rho):= \left\{S_t^{x,\rho} \not \in \bar{B}_D\left(x,\frac{c_0'}{2}\rho^{\mathfrak{b}}\right)\right\} \cap E_1 \cap E_2$  and set ${\mathlcal{V}_{x,\rho} := \Hh_{x,\rho}\setminus \bar{B}_D(x,\frac{c_0'}{2}\rho^{\mathfrak{b}})}$. Then
    \begin{align*}
        \p\big(\mathlcal{S}(x,\rho) \big) &\le \E\left[\#(\mathlcal{V}_{x,\rho}) \cdot \sqrt{\rho} \cdot \sup_{v \in \mathlcal{V}_{x,\rho}}\big\{ \mathsf{q}_t(x,v)\big\} \cdot \mathbbm{1}_{(E_1\cap E_2)}\right]\\
        &\leq 2\sqrt{10}(e+1)^d \rho^{2d+{\mathfrak{b}}+1/2}\cdot\exp\left( -\frac{c_0'}{8} ~\rho^{2{\mathfrak{b}}-1} \right).
    \end{align*}
    
    We combine the result above with \cref{prop:Hn.growth}, \eqref{eq:D.Euclid.Lm.Yao}, \eqref{eq:p.E2.comp.bound}, \eqref{eq:p.E3.c.bound}, and \eqref{eq:St.ball.bound} to verify the existence of $c_1',c_2'>0$ such that, for all $x \in \R^d$ and $\rho \geq 1$ fixed,
    \[
        \p\left(S_t^x \not \in B\left(x,\frac{\rho^{\mathfrak{b}}}{2}\right)\right) \le  \frac{c_1'}{2} \exp\left(-c_2' \rho^{2{\mathfrak{b}}-1}\right) =: \mathsf{p}_\rho.
    \]
    
    We now turn to the proof of \eqref{eq:first.exit.p.bound}. Let $S_\uptau:= S^x_{\uptau_\rho^x}$. Note that, since $S_\uptau$ is in the boundary of the subgraph $ \Hh \cap B(x,\rho^{\mathfrak{b}})$,
    \begin{align*}
        \p\left(\uptau^x_\rho < t\right) &= \p\big(\uptau^x_\rho < t, S^x_t \not\in B(x,\rho^{\mathfrak{b}}/2) \big) + \p\big(\uptau^x_\rho < t, S^x_t \in B(x,\rho^{\mathfrak{b}}/2)\big)\\
        &\leq \mathsf{p}_\rho + \E\left[ \mathbbm{1}_{(\uptau^x_\rho < t)} \p\big(S^{S_\uptau}_{t-\uptau^x_\rho} \in B(x,\rho^{\mathfrak{b}}/2)\big) \right]\\
        &\leq \mathsf{p}_\rho + \E\left[ \mathbbm{1}_{(\uptau^x_\rho < t)} \max_{y \in \partial B(x,\rho^{\mathfrak{b}}) \cap \Hh} \sup_{s \in [0, \rho]}\p\big(S^y_{s} \not\in B(y,\rho^{\mathfrak{b}} /2)\big) \right]\\
        &\leq \mathsf{p}_\rho + \mathsf{p}_\rho \cdot \p\left(\uptau^x_\rho < t\right).
    \end{align*}
    
    It then follows that $\p\left(\uptau^x_\rho < t\right) \le 2\mathsf{p}_\rho$ which proves \eqref{eq:first.exit.p.bound} for $\rho\geq 1$.
    The case $\rho<1$ is covered by adjusting $c_1'$ and the proof is complete.
    \end{proof}
    
    Consider now the coalescent random walks dominating the invasion dynamics by duality. Our aim is to determine if the blue species invades $q(x)$ up to time $t\le\rho$. Observe that any admissible path of invasion is stochastically dominated by the existence of a time-reversed random walk $S^x_{t'}$ with $t'\in[0,t]$ reaching $\upzeta$.
    
    Regard $q(x)$ as a source of random walks $\{S^{x, (i)}_{t'}: t' \in [0,t], i \in \N\}$. Set $\mathbf{n}(x)$ to be the set of neighbouring sites of $q(x)$. Write $\mathsf{N}_{x,y}^{[t]}$  for the number of points of the PPP with respect to the edge $\{q(x),y\}$ on $[0,t]$. The number of offspring of random walks starting at $q(x)$ is $\sum_{y \in \mathbf{n}(x)}\mathsf{N}^{[t]}_{x,y}$. Define the events
    \[
        E_1':=\big\{\deg\big(q(x)\big)< \sqrt{\rho}\big\}, \quad \text{and} \quad E_2':=\left\{\forall y \in \mathbf{n}\big(q(x)\big)\left(\mathsf{N}^{[t]}_{x,y}< e\cdot \rho^{3/2}\right)\right\}
    \]
    
    By applying the same upper-bounds for Poisson distributions used in the proof of \cref{cl:first.exit}, 
    \[
        \p\left((E_1')^c\right)\leq \left(\frac{ 2^d r^d }{\sqrt{\rho}}e\right)^{\sqrt{\rho}}, \quad \p\left((E_2')^c\cap E_1\right)\le \rho^{-\rho^{3/2}}.
    \]
    
    Let $E' := E_1' \cap E_2'$. then one has by \cref{cl:first.exit} that
    \[
        \p\left(\{x \in \upzeta(t)\} \cap E'\right) \le e \cdot \rho^2 \cdot c_1' \exp\left( -c_2' \rho^{2 {\mathfrak{b}} -1}\right).
    \]
    We complete the proof of \eqref{eq:invasion.q(x)} by choosing suitable $C, C'>0$.
\end{proof}

The established lemma above provides a method to control the invasion times in a region occupied by another species. The next result offers a more refined approach to studying the competition dynamics of the model. First, let us introduce some notation. Let $\bar{\upxi}, \bar{\upzeta} \subseteq \R^d$ be two disjoint sets and denote
\[P_{\bar{\upxi}, \bar{\upzeta}}(\cdot):=\p\left(~\cdot \mid\left\{\upxi(0)=\bar{\upxi}\cap\Hh, ~\upzeta(0)=\bar{\upzeta}\cap\Hh\right\}\cap\Uptheta_{w,w'}\right)\]

\begin{figure}[htb!]
    \centering
    \includegraphics[width=0.65\linewidth]{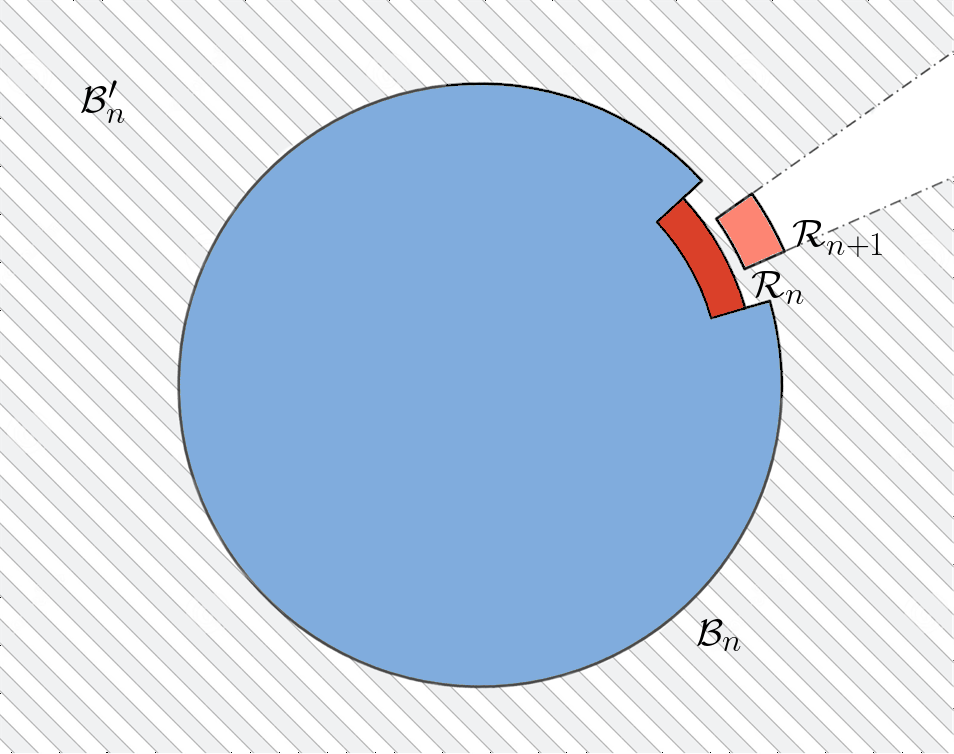}
    \caption{The regions $\mathcal{R}_n$, $\mathcal{R}_{n+1}$, $\mathcal{B}_n$, and $\mathcal{B}_n'$ before intersecting $\Hh$.}
\end{figure}
Furthermore, let us consider the following regions of $\Hh$:
\begin{align*}
\mathcal{R}_n&:=\Upphi_n(w)\\
\mathcal{B}_n &:= \left(\upvarphi \cdot B\left(o,\tfrac{1}{1+\mathfrak{d}}t_n\right) \cup\left( \left.\upvarphi \cdot B\left(o, t_n+ (t_n)^{\mathfrak{b}}\right)\right\backslash \operatorname{Cone}(w, r_n)\right)\right)\cap\Hh\text{, and}\\
    \mathcal{B}_n' &:= \left(\upvarphi \cdot B(o,t_n) \cup \big(\operatorname{Cone}(w, r_{n+1})\big)^c\right) \cap \Hh.
\end{align*}

Now, we can obtain the following property regarding the competition model.

\begin{lemma} \label{lm:control.species.growth}
    There exist constants $C,C'>0$  such that the following holds for all  $t_0>\bar{s}$. If ~$~\mathcal{R}_n \subseteq \bar{\upxi}$ and $\bar{\upzeta} \subseteq \mathcal{B}_n$, then
    \begin{equation} \label{eq:upper.bdn.blue.steps}
        P_{\bar{\upxi},\bar{\upzeta}}\big(\upzeta(\mathfrak{d}t_n)\not\subseteq \mathcal{B}_n'\big) \leq C t_n^{3d}\exp\big(-C'(t_n)^{2\mathfrak{b}-3/2}\big).
    \end{equation}
\end{lemma}

\begin{proof}
    The demonstration closely follows the proof of Lemma 3 in \cite{kordzakhia2005}, with an emphasis on the key differences and necessary adjustments for our model with random structures.
    
    The sets $\mathcal{R}_n$, $\mathcal{R}_{n+1}$, $\mathcal{B}_n$, $\mathcal{B}_n'$ correspond to $\mathscr{R}_0^n$, $\mathscr{R}_1^n$, $\mathscr{B}_0^n$, $\mathscr{B}_1^n$, from \cite{kordzakhia2005}, respectively, scaled by $\upvarphi$ and intersected with $\Hh$. Similarly, the constants $\mathfrak{a}$, $\mathfrak{b}$, and $\mathfrak{d}$ represent $\alpha$, $\beta$, and $\delta$ in the same lemma.

    Their proof consists of five claims. Since the limiting shape is a Euclidean ball, which is uniformly curved, Claim 1 is immediately satisfied. The remaining four claims depend on the polynomial growth of $\Hh$, estimated by \cref{prop:Hn.growth}, and the moderate deviation given by \cref{thm_new_moderate_deviations}. Noting that $0 < 2\mathfrak{b} - \frac{3}{2} < \mathfrak{b} - \frac{1}{2}$ ensures that the upper bounds for the probabilities in Claims 2 and 3 hold true. In Claim 4, we replace the application of Lemma 1 in \cite{kordzakhia2005} with \cref{lm:control.invasion}, resulting in \eqref{eq:upper.bdn.blue.steps}.
\end{proof}

\subsection{The coexistence of the species}

In this subsection, we address the proof of the theorem on the coexistence of species in the competition model. This result follows from the intermediate condition, along with the lemmas obtained in the previous subsection, moderate deviations, and the shape being a Euclidean ball.

\begin{proof}[Proof of \cref{thm:coexistence}]
    First, since the sets $W$ and $W'$ that determine the initial configuration are non-empty and disjoint, $\upxi(0)$ and $\upzeta(0)$ are non-empty with strictly positive probability. Moreover, \cref{lm:intermediate.condition} ensures that the intermediate condition \eqref{intermediate.condition} holds true. We begin by finding a lower bound for the probability  $\p\big(\mathrm{Coex}(\upxi,\upzeta)\big) = P\big(\mathrm{Coex}(\upxi,\upzeta)\big) \cdot \p(\Uppsi_{w,w'})$. Suppose, w.l.o.g., that \[P(\text{red dies out})\geq P(\text{blue dies out}),\] where `\emph{red dies out}' is the event when there exists a time $t>0$ with $\upxi(t)=\varnothing$ (similar for `\emph{blue dies out}'). Note that Boole's inequality yields
    \begin{align}
        P\big(\mathrm{Coex}(\upxi,\upzeta)\big) & \geq 1 - P\left( \exists t \geq 0 \ \text{s.t.} \ \upxi(t) = \varnothing \ \text{or} \ \upzeta(t) = \varnothing \right) \nonumber\\
        & \geq 1 - 2 P\left( \exists t \geq 0 \ \text{s.t.} \ \upxi(t) = \varnothing \right). \label{eq:coex.bound}
    \end{align}
    Let us prove that, for a $\varepsilon \in (0,1/4)$,
    \begin{equation*}
        P(\text{red dies out})=P\left( \exists t \geq 0 \ \text{s.t.} \ \upxi(t) = \varnothing \right) < 2 \varepsilon
    \end{equation*}
    to establish the theorem. Let $\Gamma:=\bigcap_{n\in\N_0} \Gamma_n$ and repeat the arguments in the proof of \cref{lm:intermediate.condition} to select a large $s_0>\bar{s}$ such that, for all $t_0\geq s_0$,
    \begin{equation} \label{eq:first.small.compr.bdn}
        P(\Gamma^c) \leq \sum_{n \in \N} P\big(\Gamma_n^c\big) < \varepsilon.
    \end{equation}
    Next, by applying \cref{lm:control.species.growth}, we find a possibly larger $s_0$ so that, for all $t_0 \geq s_0$,
    \begin{equation} \label{eq:second.small.compr.bdn}
    \sum_{n \in \N_0}P_{\mathcal{R}_n,\mathcal{B}_n}\big(\upzeta(\mathfrak{d}t_n)\not\subseteq \mathcal{B}_n'\big) <\varepsilon.
    \end{equation}
   
    By stochastic domination, \eqref{eq:first.small.compr.bdn} and \eqref{eq:second.small.compr.bdn} imply
    \begin{equation*}
        P\left( \text{red dies out} \right) \leq P(\Gamma^c)+\sum_{n\in \N_0}  P\big(\{\upzeta(t_{n+1})\not\subseteq \mathcal{B}_n'\}\cap\Gamma \mid \upzeta(t_{n})\subseteq \mathcal{B}_n\big)< 2 \varepsilon.
    \end{equation*}
    
   The inequality above leads to  $\p\big(\mathrm{Coex}(\upxi,\upzeta)\big) >(1-4\varepsilon)\p(\Uppsi_{w,w'})>0$, which is the statement of the theorem.
\end{proof}

%% file: texts/conclusion.tex
In this concluding chapter, we reflect on the exploration undertaken in this doctoral thesis, focusing on limiting shape theorems across diverse mathematical structures. Throughout our research, we have derived significant theorems that providing insights into different aspects of subadditive processes, virtually nilpotent groups, and First-Passage Percolation (FPP) models. As we delve into the main results and discussions presented in the preceding chapters,  we gain a deeper understanding of the behavior of these mathematical phenomena and their implications. Furthermore, we consider the avenues for future research, contemplating the potential advancements and challenges that lie ahead in this area of study.

\section{Limiting Shape of Subadditive Cocycles on Groups}

In \cref{part:groups}, we have successfully established the Asymptotic Shape Theorem for random subadditive processes on both nilpotent and virtually nilpotent groups (see \cref{shape.thm,thm:shape.polynomial}). By extending existing results in the literature, we have achieved a comprehensive understanding of the behavior of these processes under more relaxed growth conditions—both at least and at most linear growth. This broadening of applicability enhances the utility of our results in diverse mathematical contexts.

A noteworthy contribution of our work lies in the exploration of FPP models, a crucial class of processes meeting the considered conditions, especially the innerness property. Leveraging this, we were able to derive \cref{cor:fpp.virt.nil} for the limiting shape in FPP models, thereby extending the reach of our results to encompass this important and widely studied class of random processes.

Moreover, our presentation of examples generalizes previously known results in shape theorems. These examples illustrate scenarios where the strong restriction of $L^\infty$ cocycles is alleviated, emphasizing the versatility of our established theorems in capturing a broader range of applications.

Looking forward, possible future research may involve the exploration of other types of random variables exhibiting almost subadditive behavior. Additionally, considering point processes on nilpotent Lie groups to define random graphs opens up intriguing possibilities for further investigation. An interesting direction for future works could involve refining our theorems based on the generating set, recognizing the crucial role it plays in certain key aspects. Such refinements could leverage quasi-isometric properties, offering a more nuanced understanding of the interplay between the generating set and the behavior of random subadditive processes.

\section{Asymptotic Behaviour of FPP Models on RGGs}

In this thesis, we have made significant strides in understanding First Passage Percolation models on Random Geometric Graphs (RGGs). Our research has revealed several key findings that describe the behavior of random growth in these graph structures.

First, we established the existence of the limiting shape of FPP models on RGGs using standard techniques derived from the subadditive ergodic theorem. Building upon this foundation, we advanced our understanding by obtaining an improvement in the form of a quantitative shape theorem, or speed of convergence. This approach utilizes moderate deviations estimates, eschewing reliance on ergodic techniques. This methodology not only provided deeper insights into the convergence behavior but also paved the way for the study of geodesics within this context.

Moreover, our investigation into the fluctuations of geodesic paths and their spanning trees, inspired by the work of Howard and Newman \cite{howard2001}. The application of these results in the study of the competition model highlights the relevance of the obtained theorems and their potential implications. For instance, we applied these results to study a two-species competition model, determining the positive probability of both species coexisting indefinitely.

Looking ahead, our research suggests several avenues for future exploration. It would be particularly intriguing to explore the possibility of weakening the hypotheses governing our models, perhaps by leveraging subadditive cocycles or relaxing assumptions regarding the common distribution or independence of random variables.

In the broader context of mathematics, our findings underscore the significance of randomness in the study of RGGs. These graph structures find applications across various areas, from network analysis to spatial statistics. The techniques developed in this thesis derive as natural extensions of the study of infinite connected components, akin to bond percolation in $\Z^d$, to RGGs, contributing to our understanding of random growth.

\section{Final Remarks and Discussion}

Comparing the two main themes explored in this thesis, we observe intriguing parallels and distinctions. Both investigations contribute to our understanding of random growth processes, albeit in different mathematical contexts. The study of subadditive cocycles on groups offers insights into the behavior of random processes in algebraic structures. On the other hand, the analysis of FPP models on RGGs delves into the behavior of random growth within the framework of random graphs.

The avenues for future research outlined present exciting opportunities for further exploration and advancement. Whether it be the refinement of theorems based on generating sets in the context of subadditive cocycles on groups or the exploration of alternative growth models and hypotheses in FPP models on RGGs, there remains ample room for innovation and discovery.

For instance, one can follow the studies of Coletti, Miranda, and Mussini \cite{coletti2016boolean} and Coletti, Miranda, and Grynberg \cite{coletti2020boolean} to investigate RGGs on a Carnot group $G_\infty$ and verify if similar techniques are applicable in this context. Furthermore, Tessera \cite{tessera2018} utilized the main approach employed in \cite{benjamini2015} for nilpotent groups to study an improved version for the speed of convergence in $\mathbb{Z}^d$. This highlights the continued potential for further exploration and refinement of results in these distinguished structures.

In conclusion, this thesis contributes to the broader landscape of mathematical research by elucidating fundamental principles underlying random growth processes. By bridging the gap between algebraic structures and random graph-theoretic frameworks, we aim to facilitate interdisciplinary collaborations and foster the development of novel methodologies in the study of stochastic systems.

%% file: texts/appendix.tex
\chapter{Appendix}

\section{Proof of the Approximation Bounds of FPP Models on RGGs} \label{s_appendix_proofs}

This section is dedicated to proving \cref{lem_T_t_and_T}. We assume the conditions specified in the lemma. To commence, we focus on the event where the $T^t$-geodesics are confined within a given ball.

\begin{lemma}\label{eq_new_hopcount}
	There exists~$c > 0$ such that, for any~$x \in \R^d$ with~$\|x\|$ large enough and~$t$ large enough (not depending on~$x$) with~$t \le \|x\|$, we have
	\begin{align*}
		\p \left(\begin{array}{l} \text{any path in $\mathcal G^t$ from $q(o)$ to $q(x)$ that minimizes}\\ \text{the $T^t$-passage time is contained in $B(o,\|x\|^3)$} \end{array} \right) > 1- e^{-ct}.
	\end{align*}
\end{lemma}
\begin{proof}
	Recall that~$\gamma^t_{u \leftrightarrow v}$ is the shortest path from~$u$ to~$v$ in~$\mathcal G^t$ that only uses extra edges. We bound
	\[T^t(q(o),q(x)) \le \K t |\gamma^t_{q(o) \leftrightarrow q(x)}| \stackrel{\text{\footnotesize\eqref{eq_hopcount_extra}}}{\le} \K t \left(\frac{\sqrt{d}}{t}\|q(o) - q(x)\| +d	\right).\]
	Define the event
	\begin{align*}
		E_1 &:= \{\|q(o)\| \le t/2,\; \|q(x) - x\| \le t/2\}.
	\end{align*}
	 On~$E_1$, we have~$\|q(o)-q(x)\| \le \|q(o)\| + \|x\| + \|q(x) - x\| \le  \|x\| + t$, so
	\begin{align}\label{eq_for_minimizer}
		T^t(q(o),q(x)) \le \K t \left( \frac{\sqrt{d}}{t}(\|x\| + t) + d\right) \le 3 \K d \|x\|,
	\end{align}
	where we used~$\sqrt{d} \le d$ and~$t \le \|x\|$.

	Next, let
	\begin{align*}
		E_2 := \left\{ \begin{array}{l}\text{there is no path $\gamma$ in $\mathcal G^t$ starting in $B(o,t)$}\\[.2cm] \text{with $|\gamma| = \Big\lceil 3\K d \|x\| / \delta\Big\rceil$ and $\sum_{e \in \gamma} \tau^t_e \le 3\K d \|x\|$} \end{array} \right\},
	\end{align*}
	where~$\delta$ is the constant of Lemma~\ref{lem_cheap_hop}.

	Assume that~$E_1 \cap E_2$ occurs and let~$\gamma^*$ be a path in~$\mathcal G^t$ from~$q(o) = x_0$ to~$q(x) = x_m$ minimizing the~$T^t$-passage time, that is, such that~$T^t(q(o),q(x)) = \sum_{e \in \gamma^*} \tau^t_e$. Then,
	\[\sum_{e \in \gamma^*} \tau^t_e \le \K t |\gamma_{q(o) \leftrightarrow q(x)}^t| \stackrel{\text{\footnotesize\eqref{eq_for_minimizer}}}{\le} 3 \K d \|x\|,\]
	so by the definition of~$E_2$, we have that~$|\gamma^*| \le \lceil 3\K d \|x\| / \delta\rceil$. Then, writing~$\gamma^* = (x_0,\ldots,x_m)$ (so that~$x_0 = q(o)$,~$x_m = q(x)$ and~$m = |\gamma^*|$), for any~$i \in \{0,\ldots, m\}$ we have
	\begin{align*}
		\|x_i\| &\le \|x_0\| + \|x_i - x_0\| \\[.2cm]
		&\le \frac{t}{2} + \sum_{j=0}^{i-1} \|x_{j+1} - x_j\| \\[.2cm]
		&\le \frac{t}{2} + i\cdot \max(r,t) \le \frac{t}{2} + m \cdot  \max(r,t) \le \frac{t}{2} + \lceil 3\K d \|x\| / \delta\rceil \cdot \max(r,t).
	\end{align*}
	Using~$t \le \|x\|$, the right-hand side is smaller than~$\|x\|^3$ if~$\|x\|$ is large enough. We have thus proved that on~$E_1 \cap E_2$,~$\gamma^*$ is entirely contained in~$B(o,\|x\|^3)$.

	To conclude, we note that
	\begin{align*}
		\p(E_1) = \p( \mathcal H \cap B_{t/2}(o) \neq \varnothing,\; \mathcal H \cap B_{t/2}(x) \neq \varnothing) > 1 - e^{-c t}
	\end{align*}
	for some~$c > 0$ and~$t$ large enough, by~\cref{prop:Hn.growth}), and
	\begin{align*}
		\p(E_2) > 1 - \frac{t^d}{2^{\lceil 3\K d \|x\| / \delta\rceil}} > 1 -e^{-c\|x\|}
	\end{align*}
	for some~$c > 0$ if~$\|x\|$ is large enough, by Lemma~\ref{lem_cheap_hop}.
\end{proof}

\begin{lemma}
	\label{lem_hops}
	Let~$\gamma = (x_0,\ldots, x_n)$ be a path in~$\mathcal G^t$. Assume that~$x_0$ and~$x_n$ belong to~$\mathcal H$ (the infinite cluster of~$\mathcal G$), and~$x_1,\ldots,x_{n-1}$ do not belong to~$\mathcal H$ (so that each of them is either in~$t\Z^d$ or in some finite cluster of~$\mathcal G$). Let
	\[b:= \max\left\{\|y-x\|_\infty:\; \begin{array}{l} x, y \in V,\; x \text{ and } y \text{ are in the same finite} \\ \text{cluster of }\mathcal G, \text{ and this cluster is visited by } \gamma \end{array}\right\},\]
		that is,~$b$ is the maximum~$\ell_\infty$-diameter of all finite clusters of~$\mathcal G$ intersected by~$\gamma$. Then,
	\begin{equation*}
		\sum_{e \in \gamma} \tau^t_e \ge \frac{\K t}{2t+b}\cdot \|x_n -x_0\|_\infty.
	\end{equation*}
\end{lemma}
\begin{proof}
Let~$\mathcal{X}$ be the set containing $x_0$, $x_n$, and all points of $t \mathbb Z^d$ among $\{x_1,\ldots,x_{n-1}\}$. Take indices~$0=i_0 < i_1 < \cdots < i_m = n$ such that~$\mathcal X = \{x_{i_0},x_{i_1},\ldots, x_{i_m}\}$. It is easy to check that the number of extra edges traversed by~$\gamma$ is at least~$m$, so
	\begin{align}\label{eq_first_nice1}
	\sum_{e \in \gamma} \tau^t_e \ge \K t \cdot m.
\end{align}
	Moreover, for any~$j \in \{0,\ldots,m-1\}$, we have~$\|x_{i_{j+1}}-x_{i_j}\|_\infty \le 2t+b$. This is because the portion of~$\gamma$ from~$x_{i_j}$ to~$x_{i_{j+1}}$ stays inside a single finite cluster of~$\mathcal G$, apart from possibly traversing an extra edge when departing from~$x_{i_j}$, and another extra edge when arriving at~$x_{i_{j+1}}$. Hence,
	\begin{align}\label{eq_second_nice1}
	\|x_n - x_0\|_\infty \le \sum_{j=0}^{m-1} \|x_{i_{j+1}} - x_{i_j}\|_1 \le (2t+b)m.
\end{align}
	Combining~\eqref{eq_first_nice1} and~\eqref{eq_second_nice1}, we obtain the desired inequality.

\end{proof}

We will now proceed to the proof of the lemma employed to establish the approximation bounds:

\begin{proof}[Proof of Lemma~\ref{lem_T_t_and_T}]
	Fix~$x$ and~$t$ large enough, as required by Lemma~\ref{eq_new_hopcount}. Define the events
		\begin{align*}
			&E_1 := \left\{ \begin{array}{l} \text{any path in $\mathcal G^t$ from $q(o)$ to $q(x)$ that minimizes}\\ \text{the $T^t$-passage time is contained in $B(o,\|x\|^3)$} \end{array} \right\}, \\[.2cm]
				&E_2 := \left\{ \begin{array}{l} \text{for any } z \in \mathcal H \cap  B(o,\|x\|^3) \text{ and } y \in \mathcal{H}\cap B(z,t),\\ \text{ we have } \quad T(z,y) \le \thickbar{\beta}\cdot t \end{array} \right\},\\[.2cm]
					&E_3 := \left\{\begin{array}{l}\text{any finite cluster of $\mathcal G$ that intersects $B(o,\|x\|^3)$} \\ \text{ has diameter (in $\ell^\infty$ norm) smaller than $t$} \end{array}\right\}.
		\end{align*}
	We claim that if~$E_1$,~$E_2$,~$E_3$ all occur, then any path in~$\mathcal G^t$ from~$q(o)$ to~$q(x)$ that minimizes the~$T^t$-passage time does not traverse any extra edge, so that
	\[E_1 \cap E_2 \cap E_3 \subseteq \{T^t(q(o),q(x)) = T(x)\}.\]
	Let us prove this. Assume that the three events occur, and fix a path~$\gamma^*=(x_0,\ldots,x_m)$ with the properties stated in the claim. Assume for a contradiction that~$\gamma^*$ traverses some extra edge, and let
	\[\mathcal I:= \min\{i: x_{i+1} \notin \mathcal H\},\quad \mathcal J:= \min\{j > \mathcal I: x_j \in \mathcal H\}.\]
	Note that these are well defined with~$0 \le \mathcal I < \mathcal J \le m$, since~$x_0 = q(o) \in \mathcal H$ and~$x_m = q(x) \in \mathcal H$. Define the sub-path~$\gamma := (x_{\mathcal{I}}, x_{\mathcal{I}+1},\ldots, x_{\mathcal J})$. Now there are two cases. First, if~$\|x_{\mathcal J} - x_{\mathcal I}\|\le t$, then (since~$\gamma$ traverses at least one extra edge) we have
	\[T^t(x_{\mathcal I},x_{\mathcal J}) = \sum_{e \in \gamma} \tau^t_e \ge \K t > \thickbar{\beta} t = \thickbar{\beta} \cdot \max(\|x_{\mathcal I},x_{\mathcal J}\|,t),\]
	contradicting the assumption that~$E_2$ occurs. Second, if~$\|x_{\mathcal I} - x_{\mathcal J}\| > t$, then we apply Lemma~\ref{lem_hops} (with~$b \le t$) to obtain
	\[T^t(x_{\mathcal I},x_{\mathcal J}) = \sum_{e \in \gamma} \tau^t_e \ge \frac{\K t}{3t}\cdot \|x_{\mathcal I} - x_{\mathcal J}\| > \thickbar{\beta}\|x_{\mathcal{I}} - x_{\mathcal J}\|,\]
	again contradicting the occurrence of~$E_3$. This completes the proof of the claim.

	It remains to show that the three events occur with high probability. First, by Lemma~\ref{eq_new_hopcount},
	\[\p(E_1) > 1 - e^{-ct}.\]
	Second, by~\cref{lm:T.bds.sup}, 
	\[\p(E_2) > 1 - C\|x\|^{4d} \cdot e^{-C't}\]
	if~$\|x\|$ is large enough. Third, by~\cref{lm:PPP.clusters} and Mecke's formula,
	\[\p(E_3) > 1 - (2\|x\|^3)^d\cdot e^{-ct} > 1 - \|x\|^{4d}\cdot e^{-ct}\]
	when~$\|x\|$ is large enough.
\end{proof}